%% file: monotone-complement-II.tex
\documentclass[a4paper,11pt]{amsart}
\usepackage[utf8x]{inputenc}
\usepackage[margin=1.25in]{geometry}
\usepackage[T1]{fontenc}
\usepackage{ae,aecompl}
\usepackage{tikz-cd}
\usepackage{mathrsfs} 
\usepackage{verbatim,amsmath,amsthm,amsfonts,amssymb,latexsym,graphicx,mathtools,extpfeil,color,mathabx}
\usepackage{tikz}
\usetikzlibrary{cd}
\usepackage[all]{xy}
\usepackage{graphicx}
\usepackage{booktabs}
\usepackage{caption}
\usepackage{tabularx}
\usepackage{subcaption}
\usepackage{float}
\usepackage{pdflscape}
\usepackage{cancel}
\usepackage{slashed}
\DeclareMathAlphabet{\mathpzc}{OT1}{pzc}{m}{it}
\usepackage[colorlinks,pagebackref,hypertexnames=false]{hyperref} \usepackage[alphabetic,backrefs,msc-links]{amsrefs}

\usepackage{aliascnt}
\numberwithin{equation}{section}

\input{init}

\theoremstyle{introthm}
\newtheorem{introthm}{Theorem}

\newcommand\bbZ{{\mathbb Z}}
\newcommand\bbR{{\mathbb R}}
\newcommand\bbC{{\mathbb C}}

\title[Lagrangian Floer theory in divisor complement]{Monotone Lagrangian Floer theory in smooth divisor complements: II}
\author{Aliakbar Daemi, Kenji Fukaya}
\thanks{The work of the first author was supported by NSF Grant DMS-1812033. The work of the second author was supported by NSF Grant DMS-1406423 and the Simons Foundation through its Homological Mirror Symmetry Collaboration grant.}

\date{}

\begin{document}
\address{Department of Mathematics and Statistics, Washington University in St. Louis, MO 63130}\email{adaemi@wustl.edu}
\address{Simons Center for Geometry and Physics, State University of New York, Stony Brook, NY 11794-3636, USA} \email{kfukaya@scgp.stonytbrook.edu}

\begin{abstract}
In \cite{part1:top}, the first part of the present series of papers, 
we studied the moduli spaces of holomorphic discs and strips into an open symplectic manifold, isomorphic to the complement of a smooth divisor in a closed symplectic manifold. In particular, we introduced a compactification of this moduli space, which is called the {\it RGW compactification}. The goal of this paper is to show that the RGW compactifications admit Kuranishi structures. This result provides the crucial ingredient for the main construction of \cite{part1:top,part3:FH}: Floer homology for monotone Lagrangians in a smooth divisor complement.
\end{abstract}
\maketitle
{
  \hypersetup{linkcolor=black}
  \tableofcontents
}

\section{Introduction}
\label{sec:kura}

In this series of papers, the authors study Lagrangian intersection Floer homology of a pair of monotone Lagrangians in an open symplectic manifold, which is isomorphic to a divisor complement. At the heart of the construction, there is a compactification of the moduli spaces of holomorphic discs and strips satisfying Lagrangian boundary condition. The main purpose of this sequel to \cite{part1:top} is to show that this compactification, called the {\it RGW compactification}, admits a {\it Kuranishi structure}. The virtual count of the elements of these Kuranishi spaces is used in \cite{part3:FH} to define the desired Lagrangian Floer homology.

To be more detailed, let $(X,\omega)$ be a symplectic manifold and $\mathcal D$ be a symplectic submanifold of $X$ with codimension $2$. 
Using a compatible almost complex structure of $\mathcal D$ and a unitary connection of the normal 
bundle, we defined in \cite[Subsections 3.1 and 3.2]{part1:top} a compatible almost complex structure 
in a neighborhood of $\mathcal D$ in $X$, which is invariant under a {\it partial $\bbC_*$-action}. Then we extend this into an $\omega$-compatible almost complex structure $J$ to $X$.

%Another probably non-essential\footnote{In \cite{part1:top}, this assumption is made to simplify the arguments. %However, the authors believe that this assumption can be removed at the expense of a more complicated analysis %of holomorphic curves in a neighborhood of $\mathcal D$ in $X$.} assumption is that $\mathcal D$ and $\mathcal %N_{\mathcal D}(X)$, the normal bundle of $\mathcal D$ in $X$, 
%admit integrable complex structures  compatible with the symplectic structure. The complex structure on $\mathcal N_{\mathcal D}(X)$ induces an integrable complex structure in a neighborhood of $\mathcal D$ and we extend this complex structure to an almost complex structure $J$ which is tamed by $\omega$. (See \cite[Subsection 3.3]{part1:top} for more details on the choice of almost complex structures.) 

Let $L_0$ and $L_1$ be compact orientable and transversal Lagrangians in $X\setminus \mathcal D$. For any $\beta \in \Pi_2(X;L_i)$ with $\beta \cap \mathcal D = 0$, let $\mathcal M_{k+1}^{\rm reg}(L_i;\beta)$ be the moduli space of $J$-holomorphic disks of homology class $\beta$ with $k+1$ boundary marked points and Lagrangian boundary condition associated to $L_i$. In \cite{part1:top}, we introduced the RGW compactification $\mathcal M_{k+1}^{\rm RGW}(L_i;\beta)$ of $\mathcal M_{k+1}^{\rm reg}(L_i;\beta)$. (See \cite[Section 3]{part1:top} for the definition of this moduli space as a set and  \cite[Section 4]{part1:top} for the definition of topology on this moduli space. Note that this compactification is different from the stable map compactification.) We also defined the RGW compactification $\mathcal M_{k_1,k_0}^{\rm RGW}(L_1,L_0;p,q;\beta)$ of the moduli space $\mathcal M_{k_1,k_0}^{\rm reg}(L_1,L_0;p,q;\beta)$ of $J$-holomorphic strips of homology class $\beta$ with boundary marked points and Lagrangian boundary condition associated to $L_0$, $L_1$ in a similar way as in \cite[Section 3]{part1:top}.

\begin{introthm}\label{Kura-exists}
	The moduli spaces $\mathcal M_{k+1}^{\rm RGW}(L_i;\beta)$ and 
	$\mathcal M_{k_1,k_0}^{\rm RGW}(L_1,L_0;p,q;\beta)$ admit Kuranishi structures.
\end{introthm}

A topological space with a Kuranishi structure is locally modeled by the vanishing locus of an equation defined on a manifold, or more generally an orbifold. These orbifolds and equations for different points need to satisfy some compatibility conditions. (See \cite[Definition A.1.1]{fooobook2} for a more precise definition of Kuranishi structures.) Given a point of a space with Kuranishi structure, the zeros of the corresponding equation might be cut down transversally. In that case, our space looks like an orbifold in a neighborhood of such point. The main point is that such equations might not be transversal to zero and we might end up with a space which is not as regular as an orbifold. Nevertheless, a Kuranishi structure would be sufficient to have some of the interesting properties of smooth orbifolds. For example, it makes sense to talk about a space with Kuranishi structure which has boundary and corners. In fact, the Kuranishi structures of Theorem \ref{Kura-exists} have boundary and corners, which can be described in terms of similar moduli spaces. We will construct a system of Kuranishi structures which are compatible at the boundary and the corners in \cite{part3:FH}.

One of the novel features of the RGW compactification is that it has some strata which consist of obstructed objects by default. To make this point more clear, we make a comparison with the stable map compactification. In the stable map compactification of the moduli space of holomorphic discs, each stratum is described by a fiber product of the moduli spaces of holomorphic discs and strips. If each of the moduli spaces appearing in such a fiber product consists of Fredholm regular elements and the fiber product is transversal, then the moduli space in a neighborhood of this stratum consists of regular objects and hence it is a smooth orbifold in this neighborhood. However, the situation in the case of the RGW compactification looks significantly different. There are strata of the compactification which belong to the singular locus of the moduli space, even if each element of the associated fiber product is Fredholm regular and the fiber product is cut down transversely. This subtlety is the main point that our treatment diverges from the proof of the analogues of Theorem \ref{Kura-exists} for the stable map compactification of the moduli spaces of holomorphic discs and strips. (See \cite{fooobook2,foootech,fooo:tech2-1,fooo:tech2-2} for such results in the context of the stable map compactification.)

We resolve the above issue by introducing the notion of {\it inconsistent solutions} to the Cauchy-Riemann equation. Under the assumption of the previous paragraph, the space of inconsistent solutions forms a smooth orbifold. Moreover, the elements of the moduli spaces $\mathcal M_{k+1}^{\rm RGW}(L_i;\beta)$ and $\mathcal M_{k_1,k_0}^{\rm RGW}(L_1,L_0;p,q;\beta)$ can be regarded as the zero sets of appropriate equations on the moduli space of inconsistent solutions. We treat these equations as extra terms for Kuranishi maps. We believe this approach could be also useful for the analysis of the relative Gromov-Witten invariants in symplectic category. We also believe that this idea as well as some of the arguments provided in this paper can be generalized to study various conjectures proposed in \cite[Section 6]{part3:FH}.

The main steps of the construction of the Kuranishi structures required for the proof of Theorem \ref{Kura-exists} are parallel to the ones for the stable map compactification. Throughout the paper, we point out relevant references for the corresponding results in the context of the stable map compactification. At the same time, we try to make our exposition as self-contained as possible. One of the exceptions is the exponential decay result of \cite{foooexp} where the same arguments can be used to deal with the exponential decay result which we need for this paper.
(However, in our application to prove \cite[Theorem 1]{part1:top}, we do not use the smoothness of our
Kuranishi structure and so one can avoid using \cite{foooexp}.)

\begin{remark}
The fact that certain strata of the RGW compactification consists of obstructed objects by default
(that is to say, such a strata are singular) was observed inndependetly by M. Tehrani \cite{T}.
%{\color{red} (This point was not mentioned in his earlier paper \cite{T1} studying the disk case.)}\AD{Is it OK if we remove this sentence?}
A similar gluing analysis is studied in \cite{LS} which considers only the case that the neck region is connected.
In such a case, the main point of the concern of this paper does not appear.
B. Parker \cite{Pa} studies a related problem (in the case of pseudo-holomorphic map from 
curves without boundary) in a different way. For example, the `Kuranishi structure' obtained 
is different from those in the sense of \cite{fooonewbook}. In fact the Kuranishi neighborhood obtained in 
\cite{Pa} is an exploded manifold (or rather its orbifold analogue), which may not be a manifold or an orbifold.
By a similar reason the method of this paper does not give a Kuranishi structure for 
log compactification in \cite{T}.
\end{remark}

\vspace{20pt}\noindent
{\bf Outline of Contents.}
In order to make the main ideas of the construction more clear, we devote the first part of the paper to the construction of a Kuranishi chart around a special point in the RGW compactification of moduli spaces of discs. This special point, described in Section \ref{subsec:gluing1}, belongs to a stratum of the moduli space which is always obstructed. Motivated by this example, we introduce the notion of inconsistent solutions in Section \ref{sub:statement}. The stratum of this special example is given by the fiber product of a moduli space of discs and two moduli spaces of spheres. In Sections \ref{sub:Fred} and \ref{sub:Obst}, we study the deformation theory of the elements of the moduli space within this stratum. A Kuranishi chart for each element of this stratum is constructed in Section \ref{sub:kuraconst0}. The main analytical results required for the construction of the Kuranishi chart is verified in Section \ref{sub:proofmain}. 

In Section \ref{sub:kuraconst}, we explain how the method of the first part of the paper can be used to construct a Kuranishi chart around any point of the RGW moduli space. Section \ref{sub:kuracont} is devoted to showing that these Kuranishi charts are compatible with each other using appropriate coordinate changes. This completes 
 the proof of Theorem \ref{Kura-exists}.

\section{A Special Point of the Moduli Spaces of Discs}
\label{subsec:gluing1}

In the first half of the paper, we focus on the analysis of a special case. We hope that this allows the main features of our construction stand out. The special case can be described as follows. Let $\Sigma$ be a surface with nodal singularities, which has three irreducible components $\Sigma_{\rm d}$, $\Sigma_{\rm s}$ and $\Sigma_{\rm D}$. The irreducible component $\Sigma_{\rm d}$ is a disc and the remaining ones are spheres. The components $\Sigma_{\rm d}$, $\Sigma_{\rm s}$ and $\Sigma_{\rm D}$ are respectively called the {\it disc component}, the {\it sphere component} and the {\it divisor component}. The divisor component $\Sigma_{\rm D}$ intersects $\Sigma_{\rm d}$ and $\Sigma_{\rm s}$ at the points, $z_{\rm d}$ and $z_{\rm s}$, respectively. When we want to emphasize that we consider these points as elements of $\Sigma_{\rm D}$, we denote them by $z_{\rm D, d}$ and $z_{\rm D, s}$. There is no intersection between $\Sigma_{\rm d}$ and $\Sigma_{\rm s}$.

We are given a $J$-holomorphic map $u : (\Sigma,\partial \Sigma) \to (X,L)$, where $J$ is given as in \cite[Subsection 3.2]{part1:top}. The restriction of this map to $\Sigma_{\rm d}$, $\Sigma_{\rm s}$, $\Sigma_{\rm D}$ are denoted by $u_{\rm d}$, $u_{\rm s}$, $u_{\rm D}$. We assume that the image of $u_{\rm D}$ is contained in the divisor ${\mathcal D}$. The images of $\Sigma_{\rm d}$, $\Sigma_{\rm s}$ intersect $\mathcal D$ only at the points $z_{\rm d}$ and $z_{\rm s}$ with multiplicities $p_1$ and $p_2$, respectively. 
Here $p_1,p_2$ are positive integers. (See \cite[Lemma 3.8]{part1:top}.)
Following \cite[Section 3]{part1:top}, we also associate a {\it level function} $\lambda$ that  evaluates to $0$ at the components $\Sigma_{\rm d}$ and $\Sigma_{\rm s}$ and to $1$ at $\Sigma_{\rm D}$. We assume that there is one boundary marked point $z_0$ on $\Sigma_{\rm d}$. 

We also assume that the homology class $(u_{\rm D})_*([\Sigma_{\rm D}])$ satisfies the following identity (compare with \cite[Condition (3.28)]{part1:top}):
\[
  p_1 + p_2 + c_1(\mathcal N_{\mathcal D}(X)) \cap (u_{\rm D})_*([\Sigma_{\rm D}]) = 0.
\]
This condition implies that there exists a meromorphic section\footnote{
Using the $U(1)$ connection, the pulled back bundle 
$(u_{\rm D})_*([\Sigma_{\rm D}])$ has a canonical holomorphic structure.
We use this to define meromorphicity.
} $\frak s$ of $(u_{\rm D})^*\mathcal N_DX$ such that
$\frak s$ has a pole of order $p_1$ (resp. $p_2$) at $z_{\rm d}$, (resp. $z_{\rm s}$), and $\frak s$ has no other pole or zero.
The choice of this section $\frak s$ is unique up to a multiplicative constant in $\bbC_*$. We fix one such section $\frak s$ and define:
\begin{equation}\label{U-D}
  U_{\rm D}: \Sigma_{\rm D}\setminus \{z_{\rm d},z_{\rm s}\}\to 
  \mathcal N_{\mathcal D}(X) \setminus {\mathcal D} = \bbR \times S\mathcal N_{\mathcal D}(X)
\end{equation}
where $U_{\rm D}(z)$, for $z\in \Sigma_{\rm D}\setminus \{z_{\rm d},z_{\rm s}\}$, is defined to be $(u_{\rm D}(z),\frak s(z))$.
This map is $J$-holomorphic by \cite[Lemma 3.7]{part1:top}.

The nodal curve $\Sigma$ and the {\it detailed ribbon tree} corresponding to $u$ are sketched in Figures \ref{FIgsec6-1}, \ref{Figuresec6-2}. These data define an element of $\mathcal M^{\rm RGW}_{1}(L;\beta)$
for an appropriate choice of $\beta \in H_2(X,L;\bbZ)$. See \cite[Section 3]{part1:top} for the definitions of detailed ribbon trees and moduli spaces $\mathcal M^{\rm RGW}_{k+1}(L;\beta)$. Constructing a Kuranishi neighborhood for this element of $\mathcal M^{\rm RGW}_{1}(L;\beta)$ is the main goal of the first half of the paper.

\begin{figure}[h]
\centering
\includegraphics[scale=0.5]{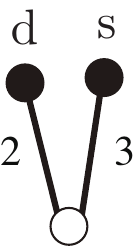}
\caption{Detailed tree of the element we study.}
\label{FIgsec6-1}
\end{figure}

\begin{figure}[h]
\centering
\includegraphics[scale=0.5]{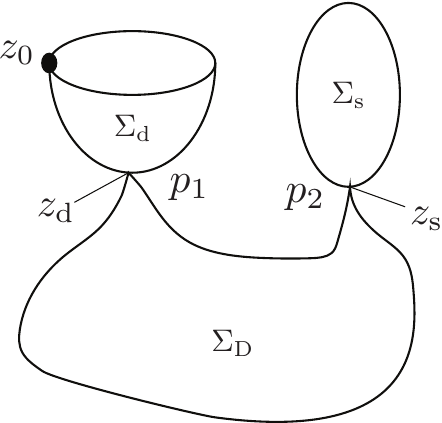}
\caption{The source curves of $(u_{\rm d},u_{\rm D},u_{\rm s})$}
\label{Figuresec6-2}
\end{figure}

\section{Fredholm theory of the Irreducible Components}
\label{sub:Fred}

In this section, we shall be concerned with the deformations of the restrictions of the map $u$ to the irreducible components $\Sigma_{\rm d}$, $\Sigma_{\rm s}$ and $\Sigma_{\rm D}$. We will see that the deformation theory of each irreducible component is governed by a Fredholm operator. 

Throughout this section, we use cylindrical coordinates both for the target and the source. There is a neighborhood $\frak U$ of the divisor ${\mathcal D}$ with a {\it partial $\bbC_*$-action} such that the almost complex structure $J$ on $\frak U$ is invariant under the  partial $\bbC_*$-action. (See \cite[Subsection 3.1]{part1:top} for the definition of partial $\bbC_*$-actions and \cite[Subsections 3.1 and 3.2]{part1:top} for the existence of $\frak U$.) We may assume that the open set $\frak U$ is chosen such that its closure minus $\mathcal D$ is diffeomorphic to:
\begin{equation}\label{subset61}
[0,\infty)_{\tau} \times S\mathcal N_{\mathcal D}(X)
\end{equation}
where $S\mathcal N_{\mathcal D}(X)$ is the unit $S^1$-bundle associated to the normal bundle $\mathcal N_{\mathcal D}(X)$ of $\mathcal D$ in $X$. We use $\tau$ to denote the standard coordinate on $[0,\infty)$. The 1-form $\theta=-d\tau\circ J$ determines a connection 1-form for the $S^1$-bundle $S\mathcal N_{\mathcal D}(X)$. %Let $\theta$ be a ${\rm U}(1)$-connection 1-form on $S\mathcal N_{\mathcal D}(X)$ and 
Let $g'$ be a metric on $\mathcal D$ which is fixed for the rest of the paper. We also fix a metric $g$ on $X \setminus {\mathcal D}$ such that its restriction to \eqref{subset61} is given by:
\begin{equation}\label{g-cylinder-end}
  g{\vert}_{[0,\infty)_{\tau} \times S\mathcal N_{\mathcal D}(X)}=d\tau^2+\theta^2+g'.
\end{equation}
In particular, $g$ is invariant with respect to the partial $\bbC_* \cong (-\infty,\infty) \times S^1$-action on \eqref{subset61}, where $(-\infty,\infty)$ acts (partially) by translation along the factor $[0,\infty)_{\tau} $ and the action of $S^1$ is induced by the obvious circle action on $S\mathcal N_{\mathcal D}(X)$. We also fix another metric $g_{NC}$ on $X\setminus \mathcal D$ whose restriction to \eqref{subset61} has the following form:
\begin{equation}\label{g-smooth}
  g_{NC}{\vert}_{[0,\infty)_{\tau} \times S\mathcal N_{\mathcal D}(X)}=e^{-2\tau}(d\tau^2+\theta^2)+g'.
\end{equation}
This non-cylindrical metric extends to $\mathcal D$ to give a smooth metric on $X$ which is also denoted by $g_{NC}$.

\begin{remark}
	We do not make any assumption on compatibility of the metric $g'$
%	metrics $g$ and $g_{NC}$ 
	with the almost complex structure or the symplectic structure.
\end{remark}

The unitary connection on the normal bundle $\mathcal N_{\mathcal D}X$ determines the decomposition
\begin{equation} \label{decom-tan-bdle}
  TX{\vert}_{\frak U}=\underbrace{\bbR\oplus \bbR}_\bbC \oplus \pi^*(T\mathcal D),
\end{equation}
where the first factor is given by the action of $(-\infty,\infty)\times \{1\} \subset (-\infty,\infty) \times S^1$, the second factor is given by the action of $ \{1\}\times S^1 \subset (-\infty,\infty) \times S^1$, and the third factor is given by the vectors orthogonal to the first two factors. Note that the last factor and the direct sum of the first two factors determine complex subspaces of $TX{\vert}_{\frak U}$.  %Since this decomposition is invariant with respect to the $\bbC_*$-action, it also determines a connection on $S\mathcal N_{\mathcal D}(X)$.

\subsection{The Disk Component}
\label{subsub:disk}
The surface $\Sigma_{\rm d}$ can be identified uniquely with the standard unit disc $D^2\subset \bbC$ such  that $z_0$ and $z_{\rm d}$ are mapped to $1$ and $0$.
The map $u_{\rm d} : D^2 \to X$ induces a map from $D^2 \setminus \{0\}$ to $X \setminus \mathcal D$, which we also denote by $u_{\rm d}$. We identify $D^2 \setminus \{0\}$ with $[0,\infty)_{r_1} \times S^1_{s_1}$ and denote the standard coordinates on the $[0,\infty)$ and $S^1$ factors with $r_1$ and $s_1$. Namely, the point $(r_1,s_1)\in [0,\infty) \times S^1$ is mapped to $\exp(-r_1-\sqrt{-1} s_1)$. (Here and in what follows, $S^1$ is  identified with $\bbR/2\pi\bbZ$.)

\begin{lemma}
There exist $R_{\rm d}\in \bbR$ and $x_{\rm d} \in S\mathcal N_{\mathcal D}(X)$ such that:
\begin{equation}\label{form62}
d_{C^m}( u_{\rm d}(r_1, s_1),(p_1r_1+R_{\rm d},p_1 s_1+x_{\rm d}))
\le C_m e^{-\delta_1r_1}
\end{equation}
for some $C_m,\delta_1 > 0$. The  constant $\delta_1$ is independent of $m$. 
\end{lemma}
Here we regard $(p_1r_1+R_{\rm d},p_1 s_1+x_{\rm d})$ as an element of $[0,\infty)_{\tau}\times S\mathcal N_{\mathcal D}(X)$ using the partial action of $(-\infty,\infty)\times S^1$ on $[0,\infty)_{\tau}\times S\mathcal N_{\mathcal D}(X)$. The expression on the left hand side of \eqref{form62} is the $C^m$ distance between the following two maps from $[0,\infty)_{r_1} \times S^1_{s_1}$ to $X\setminus \mathcal D$
\[
  (r_1, s_1) \mapsto u_{\rm d}(r_1, s_1) \hspace{1cm}(r_1, s_1) \mapsto (p_1r_1+R_{\rm d},p_1 s_1+x_{\rm d})
\]
Note that there exists $R'_{\rm d}>0$ such that $u_{\rm d}( s_1,r_1)$ is an element of \eqref{subset61} for $r_1 > R'_{\rm d}$. The $C^m$ norm is defined with respect to the cylindrical metric on $D^2\setminus \{0\}$ and the metric $g$ on $X\setminus \mathcal D$. 
\begin{proof}
The claim is a consequence of \cite[Lemma 3.7]{part1:top} and the fact that the multiplicity of the intersection of $u_d$ and $\mathcal D$ at $z_{\rm d}$ is $p_1$. 
%In the case when the almost complex structure of $X$ is integrable near $\mathcal D$,
%(\ref{form62}) can be proved by taking Tayler expansion with respect to the 
%complex coordinate and change variables. 
%In more general situation it follows from \cite[Lemma 3.4]{IP}.
\end{proof}

\begin{definition}\label{sections-V}
	We define $C^{\infty}(([0,\infty)\times S^1, \{0\}\times S^1);(u_{\rm d}^*TX,u_{\rm d}^*TL))$ to be the space of
	all smooth sections $V$ of $u_{\rm d}^*TX$ on the space $[0,\infty)\times S^1$ with the 
	boundary condition
	\[
	  V(0,s_1) \in T_{u_{\rm d}(0,s_1)}L.
	\]	
\end{definition}

We extend each vector $v \in  T_{u_{\rm d}(z_{\rm d})}{\mathcal D}$ to a vector field defined  on a neighborhood of $u_{\rm d}(z_{\rm d})$ in ${\mathcal D}$. Let $\hat v$ be the horizontal lift of this vector field using the decomposition in \eqref{decom-tan-bdle} to a neighborhood of $u_{\rm d}(z_{\rm d})$ in $X$.  We may assume that the map $v \mapsto \hat v$ is linear. Using  \eqref{decom-tan-bdle}, we can also obtain a vector field $[\frak r_{\infty},\frak s_{\infty}]$ on $X\setminus \mathcal D$ for each $(\frak r_{\infty},\frak s_{\infty}) \in \bbR \times \bbR$. These vector fields are also $(-\infty,\infty) \times S^1$-invariant.

\begin{definition}\label{defn6262}
	Let $C^{\infty}_0(([0,\infty)\times S^1, \{0\}\times S^1);(u_{\rm d}^*TX,u_{\rm d}^*TL))^+$
	be the space of all triples $(V,(\frak r_{\infty},\frak s_{\infty}),v)$ such that 
	$V \in C^{\infty}(([0,\infty)\times S^1, \{0\}\times S^1);(u_{\rm d}^*TX,u_{\rm d}^*TL))$, 
	$(\frak r_{\infty},\frak s_{\infty}) \in \bbR \times \bbR$, $v \in T_{u_{\rm d}(z_{\rm d})}{\mathcal D}$, and
	\[
	V - [\frak r_{\infty},\frak s_{\infty}] -\hat v
	\]
	has compact support. We define a weighted Sobolev norm on this vector space as follows:
	\begin{equation}\label{form6464}
        \aligned
        &\Vert(V,(\frak r_{\infty},\frak s_{\infty}),v)
        \Vert_{W^2_{m,\delta}}^2 \\
        =
        &\Vert V\Vert_{L^2_{m}([0,R'_{\rm d}]\times S^1)}^2 \\
        &+
        \sum_{j=0}^m\int_{[R'_{\rm d},\infty)\times S^1} e^{\delta r_1} 
        \vert \nabla^j(V - [\frak r_{\infty},\frak s_{\infty}] -\hat v)\vert^2 d r_1d  s_1
        \\
        &+ \vert(\frak r_{\infty},\frak s_{\infty})\vert^2 + \vert v\vert^2.
        \endaligned
        \end{equation}
	Later we shall be concerned with the case that $\delta >0$ is a sufficiently small positive number and 
	$m$ is a sufficiently large 
        positive integer.
        We denote by
        \[
          W^2_{m,\delta}(\Sigma_{\rm d} \setminus \{z_{\rm d}\};(u_{\rm d}^*TX,u_{\rm d}^*TL))
        \]
        the completion of 
        $C^{\infty}_0(([0,\infty)_{ r_1} \times S^1_{ s_1},\{0\}\times S^1_{s_1});(u_{\rm d}^*TX,u^*_{\rm d}TL))^+$ 
        with respect to the norm $\Vert \cdot \Vert_{W^2_{m,\delta}}$.
        This completion is a Hilbert space and is independent of how we extend the vectors $v$ to $\hat v$.
\end{definition}

\begin{definition}\label{coker-weighted-sob}
	Let $C^{\infty}_0([0,\infty) \times S^1;u_{\rm d}^*TX \otimes \Lambda^{0,1})$
	be the space of all smooth sections with compact supports, and define a weighted Sobolev 
	norm on it by:
	\[
	\Vert V\Vert_{L^2_{m,\delta}}^2= \Vert V\Vert_{L^2_{m}([0,R'_{\rm d}]\times S^1)}^2
	+\sum_{j=0}^m\int_{[R'_{\rm d},\infty)\times S^1} e^{\delta r_1} \vert \nabla^j(V)\vert^2 d r_1d  s_1.
	\]
	The completion of $C^{\infty}_0([0,\infty) \times S^1;u_{\rm d}^*TX \otimes \Lambda^{0,1})$ with 
	respect to the norm $\Vert \cdot \Vert_{L^2_{m,\delta}}$ is denoted by
	\begin{equation}\label{hilb6566}
		L^2_{m,\delta}(\Sigma_{\rm d} \setminus \{z_{\rm d}\};u_{\rm d}^*TX \otimes \Lambda^{0,1}).
	\end{equation}
\end{definition}
\begin{lemma}\label{lem6363}
Linearization of the Cauchy-Riemann equation at $u_{\rm d}$ gives a first order differential operator
\[
  \aligned D_{u_{\rm d}}\overline \partial :
  &C_0^{\infty}(([0,\infty) \times S^1,\{0\}\times S^1);(u_{\rm d}^*TX,u^*_{\rm d}TL)) 
  \\&\to C^{\infty}_0([0,\infty) \times S^1;u_{\rm d}^*TX \otimes \Lambda^{0,1}),
  \endaligned
\]
which has the following properties.

	\begin{enumerate}
	\item The operator $D_{u_{\rm d}}\overline \partial$ induces a continuous linear map
		\begin{equation}\label{fredholmmap1}
            		\aligned
            		D_{u_{\rm d}}\overline \partial :
            		&W^2_{m+1,\delta}(\Sigma_{\rm d} \setminus \{z_{\rm d}\};(u_{\rm d}^*TX,u_{\rm d}^*TL))  \\
            		&\to
            		L^2_{m,\delta}(\Sigma_{\rm d} \setminus \{z_{\rm d}\};u_{\rm d}^*TX \otimes \Lambda^{0,1}).
            		\endaligned
		\end{equation}
		In particular, for an element
		\[(V,(\frak r_{\infty},\frak s_{\infty}),v) \in C^{\infty}_0(([0,\infty)\times S^1, \{0\}\times S^1);(u_{\rm d}^*TX,u_{\rm d}^*TL))^+\]
		we have $D_{u_{\rm d}}\overline \partial(V,(\frak r_{\infty},\frak s_{\infty}),v)=D_{u_{\rm d}}\overline \partial(V)$.
	\item \eqref{fredholmmap1} is a Fredholm operator.
	\item The index of the operator \eqref{fredholmmap1} is equal to the virtual dimension of the moduli space
	$\mathcal M^{\rm reg,d}_1(\beta_{\rm d};(p_1))$\footnote{See \cite[Definition 3.57]{part1:top}
	for the definition of this moduli space. Here $(p_1)$ stands for the multiplicity number $p_1$.} which contains $u_{\rm d}$.
\end{enumerate}
\end{lemma}
\begin{proof}
	(1) is a consequence of (\ref{form62}).
	(We choose $\delta$ to be smaller than the constant $\delta_1$ in \eqref{form62}.)
	The differential operator $D_{u_{\rm d}}\overline \partial$ is 
	asymptotic to an operator of the form
	\[
	  \frac{\partial}{\partial r_1} + P
	\]
	as $ r_1$ goes to infinity. Furthermore, $P = J \partial/\partial s_1$ and the kernel of this operator
	can be identified with $\bbR \oplus \bbR \oplus  T_{u_{\rm d}(z_{\rm d})}{\mathcal D}$.
	Part (2) is a consequence of this observation and general results about elliptic operators on manifolds with cylindrical ends \cite[Theorem (3.10)]{APS:I}. 

	To prove Part (3), we need to show that the Fredholm index of two elliptic operators defined agree with each other:
	one is the linearized Cauchy-Riemann operator on $D^2$ and the other one is 
	the linearized Cauchy-Riemann operator on $[0,\infty)\times S^1$. We relate the indices of these operators to the indices
	of two other operators with smaller indices. First consider the subspace 
	\begin{equation}\label{domin-decay}
	  L^2_{m+1,\delta}(\Sigma_{\rm d} \setminus \{z_{\rm d}\};(u_{\rm d}^*TX,u_{\rm d}^*TL))
	 \end{equation}
	of 
	$W^2_{m+1,\delta}(\Sigma_{\rm d} \setminus \{z_{\rm d}\};(u_{\rm d}^*TX,u_{\rm d}^*TL))$.
	This subspace consists of the closure of elements 
	$(V,(\frak r_{\infty},\frak s_{\infty}),v)$ with extra condision that $v=0$.
	In particular, its codimension is equal to $\dim(\mathcal D)$.
	Let $D'_{u_{\rm d}}\overline \partial$ be the elliptic complex given by the variation of \eqref{fredholmmap1} where 
	we replace the domain with \eqref{domin-decay}. In particular, we have
	\[
	  \ind(D'_{u_{\rm d}}\overline \partial)=\ind(D_{u_{\rm d}}\overline \partial)-\dim(\mathcal D).
	\]
	Let $L^2_m(D^2;u_{\rm d};X,L)$ be the space of $L^2_m$ sections of $u_{\rm d}^*TX$ whose boundary values are
	in $u_{\rm d}\vert_{\partial D^2}^*TL$.
	Consider the subspace of this space consisting of elements whose derivatives 
	up to order $p_1-1$ vanish at the origin, and denote this space by $L^2_m(D^2;u_{\rm d};X,L;(p_1))$.
	Then the linearization of the Cauchy-Riemann operator defines an operator
	\begin{equation}\label{38chanedtodisk}
		D''_{u_{\rm d}}\overline \partial : L^2_m(D^2;u_{\rm d};X,L;(p_1)) \to 
		L^2_{m-1}(D^2;u_{\rm d}^*TX \otimes \Lambda^{0,1}),
	\end{equation}	
	This elliptic complex controls the deformation theory of a 
	subspace of $\mathcal M^{\rm reg,d}_1(\beta_{\rm d};(p_1))$ given by $J$-holomorphic curves that the intersection point
	with the divisor is constrained to be the fixed point $u_{\rm d}(D^2) \cap \mathcal D$. In particular, we have
	\[
	  \ind(D''_{u_{\rm d}}\overline \partial)={\rm virdim}(\mathcal M^{\rm reg,d}_1(\beta_{\rm d};(p_1)))-\dim(\mathcal D).
	\]
	
	To prove the claim in Part (3), it suffices to show that the indices of the operators 
	$D'_{u_{\rm d}}\overline \partial$ and $D''_{u_{\rm d}}\overline \partial$ agree with each other. Any element in the kernel (resp. 
	the cokernel) of $D''_{u_{\rm d}}\overline \partial$ determines an element of the 
	the kernel (resp. the cokernel) of $D'_{u_{\rm d}}\overline \partial$ by a reparametrization of the domain. 
	In the other direction, we may use removability of singularity. Consequently, the kernels and the cokernels of these 
	two operators are in correspondence with each other.
\end{proof}

%\AD{I'd suggest removing this paragraph.}
%To complete the Fredholm theory we consider the Hilbert manifold 
%consisting of a pair $(u_{\rm d},p)$ where $u_{\rm d} :  [0,\infty) \times S^1 \to X \setminus \mathcal D$ 
%is an $L^2_m$ map with boundary condition $u(0,s) \in L$ and $p \in \mathcal D$ 
%such that $e^{\tau \delta}u_{\rm d}(\tau,s) - p$ together with its $1$,\dots,$m$-th derivatives 
%are of $L^2$ class. 
%The tangent space of this Banach manifold is nothing but 
%the Hilbert space $W^2_{m,\delta}(\Sigma_{\rm d} \setminus \{z_{\rm d}\};u_{\rm d}^*TX)$.
%The family of Hilbert spaces $L^2_m(D^2;u_{\rm d};X,L;(p_1)) \to L^2_{m-1}(D^2;u_{\rm d}^*TX \otimes \Lambda^{0,1})$
%becomes a Hilbert bundle and 
%$(u_{\rm d},p) \mapsto \overline\partial u_{\rm d}$ is its section.
%The Fredholm operator (\ref{38chanedtodisk}) fis the derivative of this 
%section. 
%We can formulate Fredholm theory in a similar way in the 
%situation of Subsections \ref{subsub:sphere}, \ref{subsub:divisor}.\AD{Add a brief 
%discussion of Hilbert manifold as we discussed before.}

\subsection{The Sphere Component}
\label{subsub:sphere}
In this part, we study the linearization of the problem governing the map $u_{\rm s}$. This can be done similar to the case of $u_{\rm d}$. We take a compact subset $K_{\rm s}$ of $\Sigma_{\rm s} \setminus \{z_{\rm s}\}$
such that $u_{\rm s}(\Sigma_{\rm s} \setminus K_{\rm s})$ is contained in \eqref{subset61}. 
We may assume that $\Sigma_{\rm s} \setminus K_{\rm s}$ is a disk.
We take a coordinate $( r_2,s_2) 
\in \bbR \times S^1$ of $\Sigma_{\rm s} \setminus (K_{\rm s}\cup \{z_{\rm s}\})$ such that 
$( r_2,s_2)$ is identified with 
$\exp(- r_2-\sqrt{-1}s_2) \in D^2 \setminus \{0\} \cong 
\Sigma_{\rm s} \setminus (K_{\rm s} \cup \{z_{\rm s}\})$. In the same way as in \eqref{form62}, we have the inequality
\begin{equation}\label{form62rev}
d_{C^m}(u_{\rm s}( r_2,s_2),(p_2 r_2+R_{\rm s},p_2 s_2+x_{\rm s}))
\le C_m e^{-\delta_1  r_2},
\end{equation}
for a constant $R_{\rm s}$ and $x_{\rm s} \in S\mathcal N_{\mathcal D}(X)$.

\begin{definition}{\rm (Compare to \cite[Lemma 7.1.5]{fooobook2}.)}\label{defn64444}
	Let
	\[
	  C^{\infty}_0(\Sigma_{\rm s} \setminus \{z_{\rm s}\};u_{\rm s}^*TX)^+
	\] 
	be the space of 
	all triples $(V,(\frak r_{\infty},\frak s_{\infty}),v)$ such that 
	$V\in C^{\infty}(\Sigma_{\rm s} \setminus \{z_{\rm s}\};u_{\rm s}^*TX)$, 
	$(\frak r_{\infty},\frak s_{\infty}) \in \bbR \times \bbR$, $v \in T_{u_{\rm s}(z_{\rm s})}{\mathcal D}$
	and
	\[
	  V - [\frak r_{\infty},\frak s_{\infty}] -\hat v
	\]
	is compactly supported.\footnote
	{Note that $(\frak r_{\infty},\frak s_{\infty})$ and $v$ determine a vector fields 
	$[\frak r_{\infty},\frak s_{\infty}]$ and $\hat v$ on \eqref{subset61} 
	in the same way as in the last subsection.}
	Analogous to \eqref{form6464}, we define a Sobolev norm on this space as follows:
	\begin{equation}\label{form6464-2}
		\aligned
                &\Vert(V,(\frak r_{\infty},\frak s_{\infty}),v)
                \Vert_{W^2_{m,\delta}}^2 \\
                =
                &\Vert V\Vert_{L^2_{m}(K_{\rm s})}^2 \\
                &+
                \sum_{j=0}^m\int_{[0,\infty)\times S^1} e^{\delta r_2} 
                \vert \nabla^j(V - [\frak r_{\infty},\frak s_{\infty}] -\hat v)\vert^2 d r_2d s_2
                \\
                &+ \vert(\frak r_{\infty},\frak s_{\infty})\vert^2 + \vert v\vert^2.
                \endaligned
        \end{equation}
        We shall be concerned with the case that $\delta$ is a sufficiently small positive number and 
        $m$ is a sufficiently large positive integer. We denote by
        \[
          W^2_{m,\delta}(\Sigma_{\rm s} \setminus \{z_{\rm s}\};u_{\rm s}^*TX)
        \]
        the completion of 
        $C^{\infty}_0(\Sigma_{\rm s} \setminus \{z_{\rm s}\};u_{\rm s}^*TX)^+$ with respect to the norm 
        $\Vert \cdot \Vert_{W^2_{m,\delta}}$.
        This completion is a Hilbert space.

	We can also define the Hilbert space:
	\[
	  L^2_{m,\delta}(\Sigma_{\rm s} \setminus \{z_{\rm s}\};u_{\rm s}^*TX \otimes \Lambda^{0,1}),
	\]
	in the same way as in \eqref{hilb6566}.
\end{definition}

\begin{lemma}
        \begin{enumerate}
        \item The linearization of the Cauchy-Riemann equation at $u_{\rm s}$
	        defines a continuous linear map
                \begin{equation}\label{fredholmmap2ss}
                		D_{u_{\rm s}}\overline \partial :
                		W^2_{m+1,\delta}(\Sigma_{\rm s} \setminus \{z_{\rm s}\};u_{\rm s}^*TX) \to
	                L^2_{m,\delta}(\Sigma_{\rm s} \setminus \{z_{\rm s}\};u_{\rm s}^*TX \otimes \Lambda^{0,1}).
                \end{equation}
        \item \eqref{fredholmmap2ss} is a Fredholm operator.
        \item The index of the operator \eqref{fredholmmap2ss} is equal to 4 plus the virtual dimension of the 
        		moduli space $\mathcal M^{\rm reg,s}(\beta_{\rm s};(p_2))$\footnote{
		See \cite[Definition 3.60]{part1:top}.
	        $(p_2)$ stands for the multiplicity number $p_2$.} which contains $u_{\rm s}$.
        \end{enumerate}
\end{lemma}
The proof is similar to the proof of Lemma \ref{lem6363}.
The number $4$, that appears in Item (3), is the dimension of the group of automorphisms of $(S^2,z_{\rm s})$.

\subsection{The Divisor Component}
\label{subsub:divisor}
Finally, we analyze the deformation theory of $u_{\rm D}$.  Note that $u_{\rm D}$ is a map to the symplectic manifold ${\mathcal D}$. So we firstly describe a Fredholm theory for the deformation of $u_{\rm D}$ as a map to ${\mathcal D}$. This is a standard task in Gromov-Witten theory. We have a Fredholm operator
\begin{equation} \label{CR-uD}
D_{u_{\rm D}}\overline \partial : L^2_{m+1}(\Sigma_{\rm D};u_{\rm D}^*T{\mathcal D})
\to 
L^2_{m}(\Sigma_{\rm D};u_{\rm D}^*T{\mathcal D}\otimes \Lambda^{0,1}).
\end{equation}
To perform gluing analysis, we compare this Fredholm operator with another Fredholm operator associated to the  map $U_{\rm D}$ in \eqref{U-D}. 

\begin{definition}\label{defn66666}
	As in previous two subsections, we extend 
	any $v_{\rm d} \in T_{u_{\rm d}(z_{\rm d})}\mathcal D$, 
	$v_{\rm s} \in T_{u_{\rm s}(z_{\rm s})}\mathcal D$, to vector fields 
	$\hat v_{\rm d}$, $\hat v_{\rm s}$
	on open neighborhoods of  the fibers of $\bbR_{\tau} \times S\mathcal N_{\mathcal D}(X)$ over $u_{\rm d}(z_{\rm d})$ and $u_{\rm s}(z_{\rm s})$.
	For any $(\frak r_{\infty},\frak s_{\infty}) \in \bbR \times \bbR$,
	we can also define a vector field $[\frak r_{\infty},\frak s_{\infty}]$ in a neighborhood of any of the points $u_{\rm d}(z_{\rm d})$ and $u_{\rm s}(z_{\rm s})$,
	as in the last two subsections. %We identify
%	$\Sigma_{\rm D} \setminus \{z_{\rm d}, z_{\rm s}\}$ with $(-\infty,\infty)\times S^1$ such that a punctured 
%	neighborhood of $z_{\rm d}$ (resp. $z_{\rm s}$) is identified with $(-\infty,0)\times S^1$
%	(resp. $(0,\infty)\times S^1$). 
	Let
	\[
	  C^{\infty}_0(\Sigma_{\rm D} \setminus \{z_{\rm d}, z_{\rm s}\};
	  U_{\rm D}^*T(\bbR_{\tau} \times S\mathcal N_{\mathcal D}(X) ))^+
	\]
	be the space of all 5-tuples $(V,(\frak r_{\infty,\rm d},\frak s_{\infty,\rm d}),
	(\frak r_{\infty,\rm s},\frak s_{\infty,\rm s}),v_{\rm d},v_{\rm s})$ such that
	\[V\in C^{\infty}(\Sigma_{\rm D} \setminus \{z_{\rm s}\} \setminus \{z_{\rm d}\};U_{\rm D}^*T(\bbR_{\tau} \times S\mathcal N_{\mathcal D}(X) ))\] 
	and that:
	\begin{enumerate}
		\item[(i)] The restriction of $V - [\frak r_{\infty,\rm d},\frak s_{\infty,\rm d}] -\hat v_{\rm d}$ 
		to a punctured neighborhood of $z_{\rm d}$ in $\Sigma_{\rm D} \setminus \{z_{\rm d}, z_{\rm s}\}$
		vanishes;
		\item[(ii)] The restriction of $V - [\frak r_{\infty,\rm s},\frak s_{\infty,\rm s}] -\hat v_{\rm s}$ 
			to a punctured neighborhood of $z_{\rm s}$ in $\Sigma_{\rm D} \setminus \{z_{\rm d}, z_{\rm s}\}$
			vanishes.
			%to $(C,\infty)\times S^1$ vanishes for a large enough value of $C$.
	\end{enumerate}
	We define a weighted  Sobolev norm on this space as follows:
	\begin{equation}\label{form64641rev}
                \aligned
                &\Vert(V,(\frak r_{\infty,\rm d},\frak s_{\infty,\rm d}),
                (\frak r_{\infty,\rm s},\frak s_{\infty,\rm s}),v_{\rm d},v_{\rm s})
                \Vert_{W^2_{m,\delta}}^2 \\
                =
                &\Vert V\Vert_{L^2_{m}(K_{\rm D})}^2 \\
                &+
                \sum_{j=0}^m\int_{[0,\infty)\times S^1} e^{\delta r_1} 
                \vert \nabla^j(V - [\frak r_{\infty,\rm d},\frak s_{\infty,\rm d}] -\hat v_{\rm d})\vert^2 d r_1d  s_1
                \\
                &+
                \sum_{j=0}^m\int_{[0,\infty)\times S^1} e^{\delta r_2} 
                \vert \nabla^j(V - [\frak r_{\infty,\rm s},\frak s_{\infty,\rm s}] -\hat v_{\rm s})\vert^2 d r_2d s_2
                \\
                &+ \vert(\frak r_{\infty,\rm d},\frak s_{\infty,\rm d})\vert^2  
                + \vert(\frak r_{\infty,\rm s},\frak s_{\infty,\rm s})\vert^2 +\vert v_{\rm d}\vert^2+\vert v_{\rm s}\vert^2.
                \endaligned
        \end{equation}
        In order to clarify the notation in \eqref{form64641rev}, the following comments are in order.
        We take a compact subset $K_{\rm D} \subset \Sigma_{\rm D} \setminus \{z_{\rm d}, z_{\rm s}\}$
        such that %$u_{\rm D}(\Sigma_{\rm D} \setminus K_{\rm D}) \subset \eqref{subset61}$ and 
        $\Sigma_{\rm D} \setminus K_{\rm D}$ is the union of two discs.
        We fix coordinates $(r_1,s_1) \in [0,\infty) \times \bbR/2\pi\bbZ$ and $(r_2,s_2) \in [0,\infty) \times \bbR/2\pi\bbZ$
        on the complement of the origins of these two discs. That is, we identify $[0,\infty) \times \bbR/2\pi\bbZ$ with $D^2\setminus \{0\}$ using
        $(r_i,s_i) \mapsto \exp(-(r_i+\sqrt{-1} s_i))$.

        We denote the completion of 
        $C^{\infty}_0(\Sigma_{\rm D} \setminus \{z_{\rm d}, z_{\rm s}\};
        U_{\rm D}^*T(\bbR_{\tau} \times S\mathcal N_{\mathcal D}(X) 
        ))^+$ with respect to the norm $\Vert\cdot \Vert_{W^2_{m,\delta}}$ by:
        $$
        H_1:=W^2_{m,\delta}(\Sigma_{\rm D} \setminus \{z_{\rm d},z_{\rm s}\};
        U_{\rm D}^*T(\bbR_{\tau} \times S\mathcal N_{\mathcal D}(X))).
        $$
        This completion is a Hilbert space.

        We also define a Hilbert space 
        \[
          H_2:=L^2_{m,\delta}(\Sigma_{\rm D} \setminus \{z_{\rm d},z_{\rm s}\};U_{\rm D}^*T(\bbR_{\tau} \times S\mathcal N_{\mathcal D}(X)) \otimes \Lambda^{0,1}),
        \]
        in the same way as in \eqref{hilb6566}.
\end{definition}

We have a short exact sequence of holomorphic bundles on $\Sigma_{\rm D} \setminus \{z_{\rm d},z_{\rm s}\}$ as follows:
\[
  0 \to \underline {\bbC} \to U_{\rm D}^*T(\bbR_{\tau} \times S\mathcal N_{\mathcal D}(X)) \to u_{\rm D}^*T\mathcal D \to 0.
\]
Here the first map is defined by the $\bbC_*$-action. This short exact sequence induces a diagram of the following form:
\begin{equation} \label{diagram}
	\begin{tikzcd}
		0 \ar[r]&A_1 \ar[r]\ar[d,"f"]&H_1 \ar[r]\ar[d,"g"]&B_1 \ar[r]\ar[d,"h"]&0\\
		0 \ar[r]&A_2\ar[r]&H_2\ar[r]&B_2\ar[r]&0
	\end{tikzcd}
\end{equation}
where we have:
\[
  A_1=W^2_{m+1,\delta}(\Sigma_{\rm D} \setminus \{z_{\rm d},z_{\rm s}\};\underline {\bbC}), \hspace{1cm}
   A_2=L^2_{m,\delta}(\Sigma_{\rm D} \setminus \{z_{\rm d},z_{\rm s}\};\Lambda^{0,1}),
\]
and 
\[
  B_1=W^2_{m+1,\delta}(\Sigma_{\rm D} \setminus \{z_{\rm d},z_{\rm s}\};u_{\rm D}^*T\mathcal D), \hspace{.5cm}
   B_2=L^2_{m,\delta}(\Sigma_{\rm D} \setminus \{z_{\rm d},z_{\rm s}\};u_{\rm D}^*T\mathcal D \otimes \Lambda^{0,1}).
\]
The spaces $A_1$ and $B_1$ are defined similar to $H_1$ in an obvious way. In the same way as in the proof of Lemma \ref{lem6363}, we can show that the linearization of the Cauchy-Riemann equation at $U_{\rm D}$ defines a continuous linear map:
\begin{equation}\label{fredholmmap2ssrev}
        \aligned
	D_{U_{\rm D}}\overline \partial :
	&W^2_{m+1,\delta}(\Sigma_{\rm D} \setminus \{z_{\rm d},z_{\rm s}\};U_{\rm D}^*T(\bbR_{\tau} 
	\times S\mathcal N_{\mathcal D}(X))) \\        
	&\to L^2_{m,\delta}(\Sigma_{\rm D} \setminus \{z_{\rm d},z_{\rm s}\};U_{\rm D}^*T(\bbR_{\tau} 
	\times S\mathcal N_{\mathcal D}(X)) \otimes \Lambda^{0,1}).\endaligned
\end{equation}
which is the map $g$ in \eqref{diagram}. The map $f$ is the standard Cauchy-Riemann operator  and $h$ is the linearized Cauchy-Riemann operator associated to the map $u_{\rm D}$. The diagram  in \eqref{diagram} commutes and each row of the diagram forms an exact sequence.

%We fix a splitting:
%\begin{equation}\label{form61200}
%	T(\bbR_{\tau} \times S\mathcal N_{\mathcal D}(X))\cong
%	\bbR \oplus \bbR \oplus \pi^*(T\mathcal D),
%\end{equation}
%using a unitary connection on $S\mathcal N_{\mathcal D}(X)$. Here $\pi : S\mathcal N_{\mathcal D}(X) \to \mathcal D$ is the projection map. This splitting allows us to define the following two linear maps:
%\begin{equation}\label{form612}
%L^2_{m+1}(\Sigma_{\rm D};u_{\rm D}^*T\mathcal D) 
%\to
%W^2_{m+1,\delta}(\Sigma_{\rm D} \setminus \{z_{\rm d},z_{\rm s}\};U_{\rm D}^*T(\bbR_{\tau} \times S\mathcal N_{\mathcal D}(X))),
%\end{equation}
%\begin{equation}\label{form613}
%L^2_{m}(\Sigma_{\rm D};u_{\rm D}^*T\mathcal D\otimes \Lambda^{0,1})
%\to 
%L^2_{m,\delta}(\Sigma_{\rm D} \setminus \{z_{\rm d},z_{\rm s}\};U_{\rm D}^*T(\bbR_{\tau} \times S\mathcal N_{\mathcal D}(X)) \otimes \Lambda^{0,1}).
%\end{equation}
%We choose $\delta$ sufficiently small such that these maps are continuous and well defined.

\begin{lemma}
        \begin{enumerate}
        \item The operator in \eqref{fredholmmap2ssrev} is Fredholm.
        \item The kernel and the cokernel of the operator $h$ in \eqref{diagram} can be identified with the kernel and the cokernel of $D_{u_{\rm D}}\overline\partial$ in \eqref{CR-uD}.
 		Moreover, \eqref{diagram} induces a short exact sequence of the following form:
	        \[
	        0 \to \bbC\to {\rm Ker} D_{U_{\rm D}}\overline\partial \to {\rm Ker} D_{u_{\rm D}} \overline\partial  \to 0.
        		\]        
	        and an isomorphism
		\[
	          {\rm CoKer} D_{u_{\rm D}}\overline\partial \cong {\rm CoKer} D_{U_{\rm D}}\overline\partial.
		\] 
        \end{enumerate}
\end{lemma}
\begin{proof}
	The proof of the claim in (1) is similar to the proof of Lemma \ref{lem6363}.
	Identification of the kernels and cokernels of the operators $h$ and $D_{u_{\rm D}}\overline\partial$ in \eqref{CR-uD} is straightforward. Similarly, 
	the kernels and cokernels of the operators $f$
	and the Cauchy-Riemann operator associated to the trivial bundle on the sphere $\Sigma_{\rm D}$ can be identified with each other. The latter operator is surjective and its kernel is a copy of $\bbC$, consists of constant sections
	of the trivial bundle. Using this observation, the remaining claims in part (2) follow from the 
	diagram chase of the diagram \eqref{diagram}.
\end{proof}
\section{Stabilization of the Source Curves and the Obstruction Bundles}
\label{sub:Obst}

The operators $D_{u_{\rm d}}\overline \partial$, $D_{u_{\rm s}}\overline \partial$, $D_{u_{\rm D}}\overline \partial$ are not necessarily surjective. If these operators are not surjective, then the deformation theories of $u_{\rm d}, u_{\rm s}, u_{\rm D}$ are obstructed. Following a general idea due to Kuranishi, we introduce {\it obstruction spaces}.
\begin{definition}\label{defn6868}
            A triple $E = (E_{\rm d},E_{\rm s},E_{\rm D})$ of finite dimensional vector spaces that 
            \[
            \aligned
            E_{\rm d} &\subset C^{\infty}(\Sigma_{\rm d} \setminus \{z_{\rm d}\};u_{\rm d}^*TX \otimes \Lambda^{0,1}),
            \\
            E_{\rm s} &\subset C^{\infty}(\Sigma_{\rm s} \setminus \{z_{\rm s}\};u_{\rm s}^*TX \otimes \Lambda^{0,1}),
            \\
            E_{\rm D} &\subset C^{\infty}(\Sigma_{\rm D}\setminus \{z_{\rm d},z_{\rm s}\};u_{\rm D}^*T{\mathcal D}
            \otimes \Lambda^{0,1})
            \endaligned
            \]
            is called an {\it obstruction space} for $u$ if it satisfies the following properties.
            \begin{enumerate}
            	\item Elements of $E_{\rm d}$, $E_{\rm s}$, $E_{\rm D}$ have compact supports away from 
            		$z_{\rm d}$, $z_{\rm s}$, $\{z_{\rm d},z_{\rm s}\}$, respectively;
            	\item $ {\rm Im} (D_{u_{\rm d}}\overline \partial) + E_{\rm d}= 
            		L^2_{m,\delta}(\Sigma_{\rm d} \setminus \{z_{\rm d}\};u_{\rm d}^*TX \otimes \Lambda^{0,1})$;
            	\item ${\rm Im} (D_{u_{\rm s}}\overline \partial) + E_{\rm s}
            		= L^2_{m,\delta}(\Sigma_{\rm s} \setminus \{z_{\rm s}\};u_{\rm s}^*TX \otimes \Lambda^{0,1})$;
            	\item ${\rm Im}(D_{u_{\rm D}}\overline \partial) + E_{\rm D} =
            		L^2_{m}(\Sigma_{\rm D};u_{\rm D}^*T{\mathcal D}\otimes \Lambda^{0,1})$.         
            \end{enumerate}
\end{definition}
For the purpose of the gluing analysis, we need our obstruction spaces satisfy the {\it mapping transversality condition} defined as follows.

\begin{definition}\label{defn6969}
	Let 
	\[
	  {\mathcal{EV}}_{\rm d} :
	  W^2_{m+1,\delta}(\Sigma_{\rm d} \setminus \{z_{\rm d}\};(u_{\rm d}^*TX,u_{\rm d}^*TL)) 
	\to T_{u_{\rm d}(z_{\rm d})}\mathcal D
	\]
	be the continuous linear map that associates to a triple $(V,(\frak r_{\infty},\frak s_{\infty}),v)$ the vector $v$.
	The map
	\[
	{\mathcal{EV}}_{\rm s} : W^2_{m+1,\delta}(\Sigma_{\rm s} \setminus \{z_{\rm s}\};u_{\rm s}^*TX) 
	\to T_{u_{\rm s}(z_{\rm s})}\mathcal D,
	\]
	is defined similarly. Finally, let:
	\[
	{\mathcal{EV}}_{\rm D} := ({\rm EV}_{\rm D,d},{\rm EV}_{\rm D,s}) :
	L^2_{m+1}(\Sigma_{\rm D};u_{\rm D}^*T{\mathcal D})
	\to  T_{u_{\rm d}(z_{\rm d})}\mathcal D \oplus T_{u_{\rm s}(z_{\rm s})}\mathcal D.
	\]
	be the map that associates to $V\in L^2_{m+1}(\Sigma_{\rm D};u_{\rm D}^*T{\mathcal D})$ the pair of vectors $(V(z_{\rm d}),V(z_{\rm s}))$.
	We say that an obstruction space $E=(E_{\rm s}, E_{\rm d}, E_{\rm D})$ satisfies the 
	{\it mapping transversality condition}, if the following map is surjective:
        \begin{equation}\label{form61616161}
                \aligned
                ({\mathcal{EV}}_{\rm d}+{\mathcal{EV}}_{\rm D,d},&{\mathcal{EV}}_{\rm s}+{\mathcal{EV}}_{\rm D,s}): \\
                &
                (D_{u_{\rm d}}\overline \partial)^{-1}E_{\rm d}
                \oplus
                (D_{u_{\rm s}}\overline \partial)^{-1}E_{\rm s}
                \oplus (D_{u_{\rm D}}\overline \partial)^{-1}E_{\rm D} \\
                &\to 
                T_{u_{\rm d}(z_{\rm d})}\mathcal D \oplus T_{u_{\rm s}(z_{\rm s})}\mathcal D.
                \endaligned        
        \end{equation}
\end{definition}

We shall use obstruction spaces to define Kuranishi neighborhoods of 
the elements represented by $u_{\rm d}$, $u_{\rm s}$, $u_{\rm D}$, respectively. 
Note that at this stage we are studying three irreducible components separately. The process of gluing them will be discussed in the next stage.
\par\smallskip

We next introduce the notion of {\it added marked points} and {\it transversal submanifolds}.
The primary purpose of these auxiliary data is to fix coordinate on the domain of maps close to $u$.
Note that the domain of $u$ is unstable, i.e., there are non-trivial automorphisms of the domain.
This automorphism group gives rise to an ambiguity when we want to fix coordinate for the domain of $u$ and nearby maps.
We remove this ambiguity using added marked points and transversal submanifolds. 
This is a standard technique which is used, for example, in \cite[appendix]{FO}. (See also \cite[Section 20]{foootech}, \cite[Subsection 9.3]{fooo:const1}.) 

The source curve $\Sigma_{\rm d}$ of $u_{\rm d}$ comes with one interior nodal point $z_{\rm d}$ and one boundary marked point $z_0$. The group of isometries of $\Sigma_{\rm d}$ preserving $z_{\rm d}$ and $z_0$ is trivial and hence $\Sigma_{\rm d}$ together with these marked points is stable. Moreover, this source curve does not have any deformation parameter. However, the source curve $\Sigma_{\rm s}$ of $u_{\rm s}$ comes with only one interior nodal point $z_{\rm s}$. Therefore, it is unstable and we add two extra marked points $w_{{\rm s},1}$, $w_{{\rm s},2}$ such that it becomes stable with no deformation parameter. Similarly, the source curve $\Sigma_{\rm D}$ of $u_{\rm D}$ comes with two interior nodal points
$z_{\rm s}$, $z_{\rm d}$ and is unstable. We add one marked point $w_{\rm D}$ so that it becomes stable without any deformation parameter. The transversal submanifolds in the following definition are used for the purpose of killing the extra freedom of moving the auxiliary marked points.

\begin{definition}\label{conds610}
	Suppose $\mathcal N_{{\rm s},1}$, $\mathcal N_{{\rm s},2}$ are codimension $2$ smooth submanifolds of $X$ and $\mathcal N_{\rm D}$ is a codimension $2$ smooth submanifold of $\mathcal D$.
	We say the data $\Xi$ of the extra marked points $w_{{\rm s},1}$, $w_{{\rm s},2}$, $w_{\rm D}$ as above and the submanifolds $\mathcal N_{{\rm s},1}$, $\mathcal N_{{\rm s},2}$ and $\mathcal N_{\rm D}$ form {\it stabilization
	data} for $u$, if they satisfy the following properties.
\begin{enumerate}
	%\item $\mathcal N_{{\rm s},1}$, $\mathcal N_{{\rm s},2}$ are codimension $2$ smooth submanifolds of $X$,
	%	and $\mathcal N_{\rm D}$ is a codimension $2$ smooth submanifold of $\mathcal D$.
	\item For $i=1,2$, there exists an open neighborhood $\mathcal V_{\rm s}(w_{{\rm s},i})$ of $w_{{\rm s},i}$ 
		such that $\mathcal V_{\rm s}(w_{{\rm s},i}) \cap u_{\rm s}^{-1}(\mathcal N_{{\rm s},i}) = \{w_{{\rm s},i}\}$ and 
		$u_{\rm s}\vert_{\mathcal V_{\rm s}(w_{{\rm s},i})}$ is transversal to $\mathcal N_{{\rm s},i}$ at $w_{{\rm s},i}$.
	\item There exists an open neighborhood $\mathcal V_{\rm D}(w_{\rm D})$ of $w_{\rm D}$ such that
		$\mathcal V_{\rm D}(w_{\rm D}) \cap u^{-1}_{\rm D}(\mathcal N_{\rm D}) = \{w_{\rm D}\}$ and 
		$u_{\rm D}\vert_{\mathcal V_{\rm D}(w_{\rm D})}$ is transversal to $\mathcal N_{\rm D}$ at $w_{\rm D}$.
	\end{enumerate}
	We say a pair $\Upsilon:=(\Xi,E)$ of stabilization data $\Xi$ together with an obstruction space $E$ as in Definition \ref{defn6868} provide {\it stabilization and obstruction data} for $u$. %We use the symbol $\Upsilon$ for it.
%We say the data consisting of extra marked points $w_{{\rm s},1}$, $w_{{\rm s},2}$, $w_{\rm D}$ as above and $\mathcal N_{{\rm s},1}$, $\mathcal N_{{\rm s},2}$ and $\mathcal N_{\rm D}$ satisfying the following conditions 
%{\it stabilization data}.
%We use the symbol $\Xi$ for the notation of stabilization data. 
%\begin{enumerate}
%	\item $\mathcal N_{{\rm s},1}$, $\mathcal N_{{\rm s},2}$ are codimension $2$ smooth submanifolds of $X$,
%		and $\mathcal N_{\rm D}$ is a codimension $2$ smooth submanifold of $\mathcal D$.
%	\item For $i=1,2$, there exists an open neighborhood $\mathcal V_{\rm s}(w_{{\rm s},i})$ of $w_{{\rm s},i}$ 
%		such that $\mathcal V_{\rm s}(w_{{\rm s},i}) \cap u_{\rm s}^{-1}(\mathcal N_{{\rm s},i}) = \{w_{{\rm s},i}\}$ and 
%		$u_{\rm s}\vert_{\mathcal V_{\rm s}(w_{{\rm s},i})}$ is transversal to $\mathcal N_{{\rm s},i}$ at $w_{{\rm s},i}$.
%	\item There exists an open neighborhood $\mathcal V_{\rm D}(w_{\rm D})$ of $w_{\rm D}$ such that
%		$\mathcal V_{\rm D}(w_{\rm D}) \cap u^{-1}_{\rm D}(\mathcal N_{\rm D}) = \{w_{\rm D}\}$ and 
%		$u_{\rm D}\vert_{\mathcal V_{\rm D}(w_{\rm D})}$ is transversal to $\mathcal N_{\rm D}$ at $w_{\rm D}$.
%\end{enumerate}
%We say a pair $(\Xi,E)$ of stabilization data $\Xi$ 
%together with obstruction space $E$ as in Definition \ref{defn6868}
%{\it stabilization and obstruction data}.
%We use the symbol $\Upsilon$ for it.
\end{definition}

To define Kuranishi neighborhoods for the components of $u$, we need to transfer the obstruction space $E$ of $u$ to nearby maps. This is done using {\it target parallel transportation}.
%We next introduce the notion of target parallel transportation.
%Note that we take an obstruction space which is a section 
%of the pull back of the tangent bundle associated to a point of the moduli space.
%To define its Kuranishi neighborhood we need to send the obstruction spaces 
%to nearby objects. The target parallel transportation is the way we use 
%in this paper for this purpose.
Let  $u'_{\rm d} : (\Sigma_{\rm d} \setminus \{z_{\rm d}\},\partial \Sigma_{\rm d}) \to (X \setminus \mathcal D,L)$
(resp. $u'_{\rm s} : \Sigma_{\rm s} \setminus \{z_{\rm s}\} \to X \setminus \mathcal D$, $u'_{\rm D} : \Sigma_{\rm D} \to \mathcal D$) be an $L^2_{m+1,loc}$ map such that:
\begin{equation}\label{form617}
	\aligned 
	d(u_{\rm d}(x),u'_{\rm d}(x)) \le \epsilon
	\quad &\text{(resp. $d(u_{\rm s}(x),u'_{\rm s}(x)) \le \epsilon$},\\
	&\qquad\text{ $d(u_{\rm D}(x),u'_{\rm D}(x)) \le \epsilon$.)}
	\endaligned
\end{equation}
for any $x \in {\rm Int}(\Sigma_{\rm d} \setminus \{z_{\rm d}\})$ (resp. $x \in \Sigma_{\rm s}\setminus \{z_{\rm s}\}$, $x \in \Sigma_{\rm D}$). Here $d$ is defined with respect to the metric $g$ on $X\setminus \mathcal D$ or the metric $g'$ on $\mathcal D$, introduced at the beginning of Section \ref{sub:Fred}. We wish to define:
%We next introduce the notion of target parallel transportation.
%Note that we take an obstruction space which is a section 
%of the pull back of the tangent bundle associated to a point of the moduli space.
%To define its Kuranishi neighborhood we need to send the obstruction spaces 
%to a nearby objects. The target parallel transportation is the way we use 
%in this paper for this purpose.
%
%Let  $u'_{\rm d} : (\Sigma_{\rm d} \setminus \{z_{\rm d}\},\partial \Sigma_{\rm d}) \to (X \setminus \mathcal D,L)$
%(resp. $u'_{\rm s} : \Sigma_{\rm s} \setminus \{z_{\rm s}\} \to X \setminus \mathcal D$, $u'_{\rm D} : \Sigma_{\rm D} \to \mathcal D$) be an $L^2_{m+1,loc}$ map such that:
%\begin{equation}\label{form617}
%	\aligned 
%	d(u_{\rm d}(x),u'_{\rm d}(x)) \le \epsilon
%	\quad &\text{(resp. $d(u_{\rm s}(x),u'_{\rm s}(x)) \le \epsilon$},\\
%	&\qquad\text{ $d(u_{\rm D}(x),u'_{\rm D}(x)) \le \epsilon$.)}
%	\endaligned
%\end{equation}
%for any $x \in {\rm Int}(\Sigma_{\rm d} \setminus \{z_{\rm d}\})$ (resp. $x \in \Sigma_{\rm s}\setminus \{z_{\rm s}\}$, $x \in \Sigma_{\rm D}$). Here $d$ is defined with respect to the metric $g$ on $X\setminus \mathcal D$ or the metric $g'$ on $\mathcal D$, introduced at the beginning of Section \ref{sub:Fred}. We wish to define:
\begin{align}
&E_{\rm d}(u'_{\rm d}) \subset L^2_m(\Sigma_{\rm d} \setminus \{z_{\rm d}\};(u_{\rm d}')^*TX \otimes \Lambda^{0,1})\nonumber\\
&E_{\rm s}(u'_{\rm s}) \subset L^2_m(\Sigma_{\rm s} \setminus \{z_{\rm s}\};(u_{\rm s}')^*TX \otimes \Lambda^{0,1})\label{form618}\\
&E_{\rm D}(u'_{\rm D}) \subset L^2_m(\Sigma_{\rm D};(u_{\rm D}')^*T\mathcal D \otimes \Lambda^{0,1})\nonumber
\end{align}
which are finite dimensional subspaces consisting of elements with compact supports. 
To define target parallel 
transportation, we need to impose an additional constraint on $E_{\rm d}$, $E_{\rm s}$, $E_{\rm D}$.
\begin{definition}\label{cond4.4}
	Given an obstruction space $E = (E_{\rm d},E_{\rm s},E_{\rm D})$, if for any 
	$x \in \Sigma_{\rm d}$ (resp. $x \in \Sigma_{\rm s}$, $x \in \Sigma_{\rm D}$)
	in the support of an element of $E_{\rm d}$ (resp. $E_{\rm s}$, $E_{\rm D}$), the map
	$u_{\rm d}$ (resp. $u_{\rm s}$, $u_{\rm D}$) is an immersion at $x$, then $E$ is called {\it support-immersive}. %\hfill $\blacklozenge$
\end{definition}
This condition in particular implies that $E_{\rm d}$ (resp. $E_{\rm s}$, $E_{\rm D}$) is zero, if $u$ is constant on $\Sigma_{\rm d}$ (resp. $\Sigma_{\rm s}$, $\Sigma_{\rm D}$). Using the fact that $\Sigma$ has genus $0$, we can always take $E_{\rm d}$, $E_{\rm s}$, $E_{\rm D}$ satisfying this additional condition. In the following, we assume that our obstruction spaces satisfy the mapping transversality condition of Definition \ref{defn6969} and are support-immersive in the sense of Definition \ref{cond4.4}.\footnote{There are several other methods to define (\ref{form618}), where
we do not need to assume Condition \ref{cond4.4}.}
%\begin{conds}\label{cond4.4}
%If $x \in \Sigma_{\rm d}$ (resp. $x \in \Sigma_{\rm s}$, $x \in \Sigma_{\rm D}$)
%is in the support of an element of $E_{\rm d}$ (resp. $E_{\rm s}$, $E_{\rm D}$), then
%$u_{\rm d}$ (resp. $u_{\rm s}$, $u_{\rm D}$) is an immersion at $x$.
%\end{conds}
%This condition in particular implies that $E_{\rm d}$ (resp. $E_{\rm s}$, $E_{\rm D}$) is zero, if $u$ is constant on $\Sigma_{\rm d}$ (resp. $\Sigma_{\rm s}$, $\Sigma_{\rm D}$). Using the fact that $\Sigma$ has genus $0$, we can always take $E_{\rm d}$, $E_{\rm s}$, $E_{\rm D}$ satisfying this additional condition.

In the following definition, ${\rm Supp} (E_{\rm d})$ (resp. ${\rm Supp} (E_{\rm s})$, ${\rm Supp} (E_{\rm D})$)
denotes the union of the supports of the elements of  $E_{\rm d}$ (resp. $E_{\rm s}$, $E_{\rm D}$).
%We denote by ${\rm Supp} (E_{\rm d})$ (resp. ${\rm Supp} (E_{\rm s})$, ${\rm Supp} (E_{\rm D})$)
%the union of the supports of elements of  $E_{\rm d}$ (resp. $E_{\rm s}$, $E_{\rm D}$).

\begin{definition}\label{defn481}
	 For any triple of maps $u'_{\rm d}$, $u'_{\rm s}$, $u'_{\rm D}$ as above, let $I^{\rm t}_{\rm d} : {\rm Supp} (E_{\rm d}) \to \Sigma_{\rm d}$ 
	(resp.  $I^{\rm t}_{\rm s} : {\rm Supp} (E_{\rm s}) \to \Sigma_{\rm s}$, 
	$I^{\rm t}_{\rm D} : {\rm Supp} (E_{\rm D}) \to \Sigma_{\rm D}$) be the maps that is defined as follows.
	 For $x \in {\rm Supp} (E_{\rm d})$ (resp. $x \in {\rm Supp} (E_{\rm s})$, $x \in {\rm Supp} (E_{\rm D})$), the point $I^{\rm t}_{\rm d}(x)$ (resp. $I^{\rm t}_{\rm s}(x)$, $I^{\rm t}_{\rm D}(x)$) is 
	the unique point which satisfies the following two conditions.
	\begin{enumerate}
	\item The distance between $x$ and $I^{\rm t}_{\rm d}(x)$ 
		(resp. $I^{\rm t}_{\rm s}(x)$, $I^{\rm t}_{\rm D}(x)$) is smaller than the constant $\epsilon$.
		We choose $\epsilon$ small enough such that \eqref{form617} and this condition imply that
		$$ \aligned d(u_{\rm d}(x),u'_{\rm d}(I^{\rm t}_{\rm d}(x))) \le o,
		   \quad &\text{(resp. $d(u_{\rm s}(x),u'_{\rm s}(I^{\rm t}_{\rm s}(x))) \le o$},\\
		  &\qquad\text{ $d(u_{\rm D}(x),u'_{\rm D}(I^{\rm t}_{\rm D}(x))) \le o$.)}
		  \endaligned
		$$
		where $o$ is a constant smaller than 
		the injectivity radii of $X \setminus \mathcal D$ and $\mathcal D$.
		\item Condition (1) implies that 
%		We choose $\epsilon$ small enough such that $o(\epsilon)$ is smaller than 
%		the injectivity radius of $X$, $X \setminus \mathcal D$ or $\mathcal D$.
%		Therefore, 
		there exists a unique minimal geodesic $\gamma_{\rm d} : [0,1] \to X\setminus \mathcal D$ 
		(resp. $\gamma_{\rm s} : [0,1] \to X\setminus \mathcal D$, $\gamma_{\rm D} : [0,1] \to \mathcal D$) 
		joining $u_{\rm d}(x)$ to $u'_{\rm d}(I^{\rm t}_{\rm d}(x))$
		(resp. $u_{\rm s}(x)$ to $u'_{\rm s}(I^{\rm t}_{\rm s}(x))$, 
		$u_{\rm D}(x)$ to $u'_{\rm D}(I^{\rm t}_{\rm D}(x))$.)\footnote{
		Here the geodesics are defined with respect to the metric $g$ on $X\setminus \mathcal D$  
		and the metric $g'$ on $\mathcal D$.} We require that the vector $(d\gamma_{\rm d}/dt)(0)$
		(resp. $(d\gamma_{\rm s}/dt)(0)$, $(d\gamma_{\rm D}/dt)(0)$)
		is perpendicular to the image of $u_{\rm d}$ 
		(resp. $u_{\rm s}$, $u_{\rm D}$) at $t=0$.
	\end{enumerate}
	Here ${\rm t}$ in $I^{\rm t}_{\rm d}$, $I^{\rm t}_{\rm s}(x)$ and $I^{\rm t}_{\rm D}(x)$ stands for target. %\hfill $\blacklozenge$
	%We remark that those maps depend on $u'_{\rm d}$, $u'_{\rm s}$, $u'_{\rm D}$.
\end{definition}

We fix a unitary connection on $T (X\setminus \mathcal D)$, whose restriction to $\frak U$ is given by the direct sum of the trivial connection on $\underline \bbC$ and a unitary connection on $T\mathcal D$. In particular, this connection is invariant with respect to the partial $\bbC_*$-action. The parallel transport along the geodesics $\gamma_{\rm d}$, $\gamma_{\rm s}$ with respect to this unitary connection induces complex linear maps:
%We fix a unitary connection on $T (X\setminus \mathcal D)$, whose restriction to $\frak U$ is given by the direct sum of the trivial connection on $\underline \bbC$ and a unitary connection on $T\mathcal D$. In particular, this connection is invariant with respect to the partial $\bbC_*$-action. The parallel transport along the geodesics $\gamma_{\rm d}$, $\gamma_{\rm s}$ with respect to this unitary connection induces complex linear maps:
\[
  T_{u_{\rm d}(x)}X \to T_{u'_{\rm d}(I^{\rm t}_{\rm d}(x))}X,
  \quad
  T_{u_{\rm s}(x)}X \to T_{u'_{\rm s}(I^{\rm t}_{\rm s}(x))}X.
\]
We thus obtain bundle maps:
\[
  u_{\rm d}^*TX \to (u'_{\rm d}\circ I^{\rm t}_{\rm d})^*TX, \quad
  u_{\rm s}^*TX \to (u'_{\rm s}\circ I^{\rm t}_{\rm s})^*TX.
\]
By differentiating and projecting to the $(0,1)$ part, we also obtain bundle maps:
\[
  d^{0,1}I^{\rm t}_{\rm d} : \Lambda^{0,1} \to (I^{\rm t}_{\rm d})^* \Lambda^{0,1},\quad 
  d^{0,1}I^{\rm t}_{\rm s} : \Lambda^{0,1} \to (I^{\rm t}_{\rm s})^*\Lambda^{0,1}.
\]
We may assume that these maps are isomorphisms by choosing $\epsilon$ to be small enough. Taking tensor product gives rise to the maps:
\begin{align*}
&u_{\rm d}^*TX\otimes \Lambda^{0,1} \to (u'_{\rm d}\circ I^{\rm t}_{\rm d})^*TX\otimes (I^{\rm t}_{\rm d})^*\Lambda^{0,1}, \\
&u_{\rm s}^*TX\otimes \Lambda^{0,1} \to  (u'_{\rm s}\circ I^{\rm t}_{\rm s})^*TX\otimes (I^{\rm t}_{\rm s})^*\Lambda^{0,1}.
\end{align*}
which induce linear maps:
\begin{align}
\mathcal{PAL} : 
&L^2_m({\rm Supp}E_{\rm d};u_{\rm d}^*TX \otimes \Lambda^{0,1})
\to
L^2_{m,loc}(\Sigma_{\rm d} \setminus \{z_{\rm d}\};(u_{\rm d}')^*TX \otimes \Lambda^{0,1}), \label{PAL1}\\
\mathcal{PAL} : 
&L^2_m({\rm Supp}E_{\rm s};u_{\rm s}^*TX \otimes \Lambda^{0,1})
\to
L^2_{m,loc}(\Sigma_{\rm s} \setminus \{z_{\rm s}\};(u_{\rm s}')^*TX \otimes \Lambda^{0,1}).\label{PAL2}
\end{align}

We now define
\begin{equation}\label{newform620}
	E_{\rm d}(u'_{\rm d})= \mathcal{PAL}(E_{\rm d}),\quad E_{\rm s}(u'_{\rm s})=\mathcal{PAL}(E_{\rm s}).
\end{equation}
\begin{remark}
	Since $\Sigma_{\rm d}$, $\Sigma_{\rm s}$ and $\Sigma_{\rm D}$ with the added marked points are stable, 
	we can use the identity map instead of $I_{\rm d}^{\rm t}$, $I_{\rm s}^{\rm t}$ and $I_{\rm D}^{\rm t}$. 	
	(Note that adding marked points is essential, otherwise we need to fix representatives for the components $\Sigma_{\rm s}$, $\Sigma_{\rm D}$, and then the identity maps depend on these representatives.)
	This approach of using the identity map after adding marked points is employed in \cite{FO,fooobook2} and many other places in the literature. 
	We call our choice here the {\it target parallel transportation}. (A similar method is used 
	in \cite[page 250, Condition 4.3.27]{foooast}.)
	This method works better for the construction of \cite{part3:FH}.  The advantage of 
	target parallel transport lies in the fact that it is more canonical and independent of the choice of 
	domain coordinate. This fact is useful to obtain a system of Kuranishi structures which are compatible at the 
	boundary and corners.
	For example, the maps $I_{\rm d}^{\rm t}$, 
	$I_{\rm S}^{\rm t}$, $I_{\rm D}^{\rm t}$  do {\it not} change when we slightly 
	perturb added marked points or transversals.
%	Since $\Sigma_{\rm d}$, $\Sigma_{\rm s}$ and $\Sigma_{\rm D}$ with the marked points are stable, 
%	we can use the identity map instead of $I_{\rm d}^{\rm t}$, $I_{\rm s}^{\rm t}$ and $I_{\rm D}^{\rm t}$. 
%	This is the approach used in \cite{FO,fooobook2} and many other places in the literature.
%	We call our choice here the {\it target parallel transportation}.\footnote{A similar method was used 
%	in \cite[page 250, Condition 4.3.27]{foooast}.}
%	This method works better for our construction of \cite{part3:FH}.  The advantage of 
%	target parallel transport lies in the fact that it is more canonical and independent of the the choice of 
%	domain coordinate.\footnote{The `identity map' actually depends on the 
%	choice of representatives of the elements of the moduli spaces. We can 
%	fix the representative by using extra marked points and transversals. 
%	However that method 
%	is less canonical than the one we take here.}  This fact is useful to obtain a system of Kuranishi structures which are compatible at the 
%	boundary and corners.
%	For example, the maps $I_{\rm d}^{\rm t}$, 
%	$I_{\rm S}^{\rm t}$, $I_{\rm D}^{\rm t}$  do {\it not} change when we slightly 
%	perturb added marked points or transversals.
\end{remark}

We now define (Kuranishi) neighborhoods of $u_{\rm d}$, $u_{\rm s}$, $u_{\rm D}$
as follows.

\begin{definition}
	We denote by $\mathcal U_{\rm d}$  (resp. $\mathcal U_{\rm s}$) the set of all $L^2_{m+1,loc}$
	maps $u'_{\rm d} : (\Sigma_{\rm d} \setminus \{z_{\rm d}\},\partial \Sigma_{\rm d}) \to (X \setminus \mathcal D,L)$
	(resp.  $u_{\rm s}' : \Sigma_{\rm s}\setminus \{z_{\rm s}\} \to X\setminus\mathcal D$) with the following properties:
	\begin{enumerate}
	\item The $C^2$-distance between $u_{\rm d}$ and $u'_{\rm d}$ (resp. $u_{\rm s}$ and $u'_{\rm s}$) is less than
		$\epsilon$.
	\item The equation 
		\[
		  \overline \partial u_{\rm d}' \in E_{\rm d}(u'_{\rm d}),\quad \text{(resp. $\overline \partial u_{\rm s}' \in E_{\rm s}(u'_{\rm s})$)}
		\]
		is satisfied.
	\item There exists $p \in \mathcal D$ such that
		\[
		  \lim_{x \to z_{\rm d}} u_{\rm d}'(x) = p,\quad \text{(resp. $\lim_{x \to z_{\rm s}} u_{\rm s}'(x) = p$)}.
		\]
%		{\color{red} Moreover, the symplectic volume $\int_{\Sigma_{\rm d}\setminus \{z_{\rm d}\}} (u_{\rm d}')^*\omega$ 
%		(resp. $\int_{\Sigma_{\rm s}\setminus \{z_{\rm s}\}} (u_{\rm s}')^*\omega$) is finite.
%		Here $\omega$ is the symplectic form of $X$.}
	\item In the latter case,  $u_{\rm s}'(w_{{\rm s},1}) \in \mathcal N_{{\rm s},1}$ 
	and $u_{\rm s}'(w_{{\rm s},2}) \in \mathcal N_{{\rm s},2}$.
	\end{enumerate}
	We define $\mathcal U^+_{\rm s}$ to be the set of maps $u_{\rm s}'$ satisfying (1), (2) and (3), but not necessarily (4).
\end{definition} 
Note that standard regularity results imply that elements of $\mathcal U_{\rm d}$ and $\mathcal U_{\rm s}$ are smooth.

In the same way as in the case of $u'_{\rm d}$, $u'_{\rm s}$, for $u'_{\rm D} : \Sigma_{\rm D} \to \mathcal D$  with $d(u_{\rm D}(x),u'_{\rm D}(x)) \le \epsilon$, we define: 
\begin{equation}\label{PAL-D}
	\mathcal{PAL} : L^2_m({\rm Supp}E_{\rm D};u_{\rm D}^*T\mathcal D \otimes \Lambda^{0,1})\to L^2_m(\Sigma_{\rm D};(u_{\rm D}')^*T\mathcal D \otimes \Lambda^{0,1})
\end{equation}
using the map $I^{\rm t}_{\rm D}$ and parallel transport with respect to the chosen unitary connection on $T\mathcal D$. %In order to have a well-defined map, we assume that the sub-bundle $T\mathcal D$ is parallel with respect to the unitary connection on $TX$, which we used before. We use the restriction of this connection to define the parallel transport involved in the definition of $\mathcal{PAL}$ in \eqref{PAL-D}. 
We also define:
\begin{equation}\label{formula621}
	E_{\rm D}(u'_{\rm D})= \mathcal{PAL}(E_{\rm D}).
\end{equation}

\begin{definition}
	We denote by $\mathcal U_{\rm D}$ the set of $L^2_m$ maps $u_{\rm D}' : \Sigma_{\rm D} \to \mathcal D$ 
	with the following properties:
	\begin{enumerate}
		\item The $C^2$-distance between $u_{\rm D}$ and $u'_{\rm D}$ is less than $\epsilon$.
		\item The equation 
			\begin{equation}
				\overline \partial u_{\rm D}' \in E_{\rm D}(u'_{\rm D})
			\end{equation}
			is satisfied.
		\item $u_{\rm D}'(w_{{\rm D}}) \in \mathcal N_{{\rm D}}$.
	\end{enumerate}
	We define $\mathcal U^+_{\rm D}$ to be the set of maps $u_{\rm D}'$ satisfying (1) and (2), but not necessarily (3).
\end{definition}
We define maps:
\[
  {\rm ev}_{\rm d} : \mathcal U_{\rm d} \to \mathcal D, \quad{\rm ev}_{\rm s} : \mathcal U_{\rm s} \to \mathcal D,  \quad({\rm ev}_{\rm D,d},{\rm ev}_{\rm D,s}) : 
  \mathcal U_{\rm D} \to \mathcal D \times \mathcal D,
\]
by
 \begin{align*}
  &{\rm ev}_{\rm d}(u_{\rm d}') := u_{\rm d}'(z_{\rm d}),\quad {\rm ev}_{\rm s}(u_{\rm s}') := u_{\rm s}'(z_{\rm s}), \\
  &{\rm ev}_{\rm d}(u_{\rm D}') := u_{\rm D}'(z_{\rm d}),\quad{\rm ev}_{\rm s}(u_{\rm D}') := u_{\rm D}'(z_{\rm s}).
\end{align*}
We summarize their properties as follows.
\begin{lemma}\label{smooth-nbhd-comp}
	If $\epsilon$ is small enough, then we have:
	\begin{enumerate}
		\item $\mathcal U_{\rm d}$, $\mathcal U_{\rm D}$, $\mathcal U_{\rm s}$ are smooth manifolds.
		\item The maps  ${\rm ev}_{\rm d}$, ${\rm ev}_{\rm D,d}$, ${\rm ev}_{\rm D,s}$, ${\rm ev}_{\rm s}$ are smooth.
		\item The fiber product 
			\begin{equation}\label{form6210}
				\mathcal U_{\rm d} \,\,{}_{{\rm ev}_{\rm d}}\times_{{\rm ev}_{\rm D,d}}  \mathcal U_{\rm D}\,\,{}_{{\rm ev}_{\rm D,s}}\times_{{\rm ev}_{\rm s}}\mathcal U_{\rm s}
			\end{equation}
			is transversal.
	\end{enumerate}
\end{lemma}
\begin{proof}
	Part (1) is a consequence of the implicit function theorem using the assumptions in Definition \ref{defn6868}.
	Part (2) follows from the way we set up Fredholm theory.
	Part (3) follows from the surjectivity of the map \eqref{form61616161}.
\end{proof}
The fiber product \eqref{form6210} describes a Kuranishi neighborhood of any element  $[\Sigma,z_0,u]$ of the stratum of $\mathcal M^{\rm RGW}_1(L;\beta)$, consisting of objects with the combinatorial data given in Section \ref{subsec:gluing1}. Next, we include the gluing construction and construct a Kuranishi neighborhood of $[\Sigma,z_0,u]$ in the moduli space $\mathcal M^{\rm RGW}_1(L;\beta)$. Let $D^2$ be the unit disk in the complex plane and $D^2(r)$ denote $r\cdot D^2$. We fix coordinate charts:
\begin{align}\label{newform624}
  \phantom{i + j + k}
  &\begin{aligned}
   &\varphi_{\rm d} : {\rm Int}(D^2) \to \Sigma_{\rm d}, \quad &\varphi_{\rm D,\rm d} : {\rm Int}(D^2) \to \Sigma_{\rm D}, \\
      &\varphi_{\rm D,\rm s} : {\rm Int}(D^2) \to \Sigma_{\rm D}, \quad&\varphi_{\rm s} : {\rm Int}(D^2) \to \Sigma_{\rm s},
  \end{aligned}
\end{align}
which are bi-holomorphic maps onto the image and $\varphi_{\rm d}(0) = z_{\rm d}$, $\varphi_{\rm D,\rm d}(0) = z_{\rm D,\rm d}$, $\varphi_{\rm D,\rm s}(0) = z_{\rm D,\rm s}$, $\varphi_{\rm s}(0) = z_{\rm s}$. We assume that the marked points $w_{\rm D}$, $w_{{\rm s},i}$ do not belong to the image of the above coordinate charts. 
Moreover we assume that the image of $\varphi_{\rm D,\rm d}$ is disjoint from 
the image of $\varphi_{\rm D,\rm s}$.

For $\sigma_1,\sigma_2 \in D^2\backslash \{0\}$, we form the disk $\Sigma(\sigma_1,\sigma_2)$ as follows. %We fix a sufficiently small $\epsilon$ such that the marked points $w_{\rm D}$ (resp. $w_{{\rm s},i}$) does not belong to $\varphi_{\rm D,\rm d}(D^2(\epsilon))\cup \varphi_{\rm D,\rm s}(D^2(\epsilon))$ (resp. $\varphi_{\rm s}(D^2(\epsilon))$). For $\sigma_1,\sigma_2 \in D^2(\epsilon^{\color{red} 2})\backslash \{0\}$, we form the disk $\Sigma(\sigma_1,\sigma_2)$ as follows. 
Consider the disjoint union:
\begin{equation}\label{form622}
\aligned
(\Sigma_{\rm d} \setminus \varphi_{\rm d}(D^2(\vert\sigma_1\vert)))
&\sqcup
(\Sigma_{\rm D} \setminus (\varphi_{\rm D,\rm d}(D^2(\vert\sigma_1\vert)) 
\cup \varphi_{\rm D,\rm s}(D^2(\vert\sigma_2\vert))) \\
&\sqcup(\Sigma_{\rm s} \setminus \varphi_{\rm s}(D^2(\vert\sigma_2\vert))).
\endaligned
\end{equation}
and define the equivalence relation $\sim$ on \eqref{form622} as follows:
\begin{enumerate}
	\item[(gl-i)] If $z_1z_2 = \sigma_1$, $z_1,z_2 \in D^2$, then $\varphi_{\rm d}(z_1) \sim \varphi_{\rm D,\rm d}(z_2)$.
	\item[(gl-ii)] If $z_1z_2 = \sigma_2$, $z_1,z_2 \in D^2$, then $\varphi_{\rm s}(z_1) \sim \varphi_{\rm D,\rm s}(z_2)$.
\end{enumerate}
Then $\Sigma(\sigma_1,\sigma_2)$ is the quotient space of (\ref{form622}) by this equivalence relation. 
See Figure \ref{Figuresec6-3} below. The above definition can be extended to the case that $\sigma_1$ or $\sigma_2$ vanishes. For example, if $\sigma_2=0$, then \eqref{form622} is replaced with:
\begin{equation}\label{form6220}
	(\Sigma_{\rm d} \setminus \varphi_{\rm d}(D^2(\vert\sigma_1\vert))) \sqcup (\Sigma_{\rm D} \setminus \varphi_{\rm D,\rm d}(D^2(\vert\sigma_1\vert))) \sqcup
	\Sigma_{\rm s} .
\end{equation}
 where we use the identification in {\rm (gl-i)}, and the identification in {\rm (gl-ii)} is replaced with $\varphi_{\rm s}(0) \sim \varphi_{\rm D,\rm s}(0)$.
\begin{figure}[h]
\centering
\includegraphics[scale=0.6]{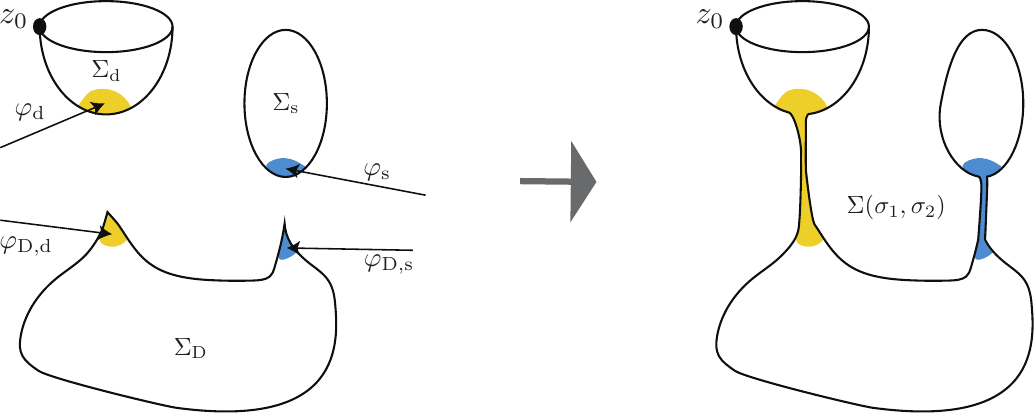}
\caption{$\Sigma(\sigma_1,\sigma_2)$}
\label{Figuresec6-3}
\end{figure}

We also define:
\begin{equation}\label{metaform626}
	\aligned
	\Sigma_{\rm d}(\sigma_1) &= \Sigma_{\rm d} \setminus \varphi_{\rm d}(D^2(\vert\sigma_1\vert)), \\
	\Sigma_{\rm s}(\sigma_2) &= \Sigma_{\rm s} \setminus \varphi_{\rm s}(D^2(\vert\sigma_2\vert)), \\
	\Sigma_{\rm D}(\sigma_1,\sigma_2)
	&=\Sigma_{\rm D} \setminus (\varphi_{\rm D,\rm d}(D^2(\vert\sigma_1\vert)) \cup \varphi_{\rm D,\rm s}(D^2(\vert\sigma_2\vert))).
	\endaligned
\end{equation}
By construction, there exist bi-holomorphic embeddings:
\begin{equation}
	\aligned
	&I_{\rm d} : \Sigma_{\rm d}(\sigma_1) \to \Sigma(\sigma_1,\sigma_2), \quad
	I_{\rm s} : \Sigma_{\rm s}(\sigma_2) \to \Sigma(\sigma_1,\sigma_2), \\
	&I_{\rm D} : \Sigma_{\rm D}(\sigma_1,\sigma_2) \to \Sigma(\sigma_1,\sigma_2).
	\endaligned
\end{equation}
Let $u'_{\rm d} : \Sigma_{\rm d}(\sigma_1) \to X \setminus \mathcal D$, $u'_{\rm s} : \Sigma_{\rm s}(\sigma_2) \to X \setminus \mathcal D$, $U'_{\rm D} : \Sigma_{\rm D}(\sigma_1,\sigma_2) \to \frak U\subset X \setminus \mathcal D$
be $L^2_{m+1}$ maps such that $u_{\rm d}'$, $u_{\rm s}'$, $u'_{\rm D}:=\pi \circ U'_{\rm D}$ are close to the restrictions of $u_{\rm d}$, $u_{\rm s}$, $u_{\rm D}$ in the same sense as in \eqref{form617}. We define:
\begin{align}\label{Ed-s}
  \phantom{i + j + k}
  &\begin{aligned}
   &E_{\rm d}(u'_{\rm d}) \subset L^2_m(\Sigma_{\rm d}(\sigma_1);(u_{\rm d}')^*TX \otimes \Lambda^{0,1}) \\
&E_{\rm s}(u'_{\rm s}) \subset L^2_m(\Sigma_{\rm s}(\sigma_2);(u_{\rm s}')^*TX \otimes \Lambda^{0,1})
  \end{aligned}
\end{align}
similar to \eqref{newform620}, using target parallel transportations. Next, we define:
\begin{equation}
	E_{\rm D}( U'_{\rm D}) \subset L^2_m(\Sigma_{\rm D}(\sigma_1,\sigma_2);(U_{\rm D}')^*TX \otimes \Lambda^{0,1})
\end{equation}
as follows.
Since $u'_{\rm D}$ is close to $u_{\rm D}$, we can use the same construction as in \eqref{formula621} to define:
\[
  E'_{\rm D}(u'_{\rm D})\subset L^2_m(\Sigma_{\rm D}; (u_{\rm D}')^*T\mathcal D \otimes \Lambda^{0,1})
\]
Then the decomposition in \eqref{decom-tan-bdle} allows us to define:
\begin{equation}\label{ED}
E_{\rm D}(U'_{\rm D})
\subset
L^2_m(\Sigma_{\rm D};(U_{\rm D}')^*TX \otimes \Lambda^{0,1}).
\end{equation}
By construction, we have isomorphisms
\begin{equation}\label{form63000}
	\mathcal P_{\rm d} : E_{\rm d}(u'_{\rm d}) \to E_{\rm d}, \hspace{.5cm}
	\mathcal P_{\rm s} : E_{\rm s}(u'_{\rm s}) \to E_{\rm s}, \hspace{.5cm}
	\mathcal P_{\rm D} : E_{\rm D}(u'_{\rm D}) \to E_{\rm D}.
\end{equation}

Recall that we fixed a codimension 2 submanifold $\mathcal N_{\rm D} \subset \mathcal D$.
We define $\widehat{\mathcal N}_{\rm D} \subset X$ to be its inverse image in the tubular neighborhood $\frak U$ of $\mathcal D$ in
$X$ by the projection map $\pi$. In the following definition $\epsilon$ is the same constant as in Lemma \ref{smooth-nbhd-comp}. We may make this constant smaller as we move through the paper whenever it is necessary.
\begin{definition}\label{defnn614}
	We denote by $\mathcal U_0$ the set of all triples $(u',\sigma_1,\sigma_2)$ where 
	$\sigma_1, \sigma_2 \in D^2(\epsilon)$.
	In the case that $\sigma_1$ and $\sigma_2$ are non-zero, $(u',\sigma_1,\sigma_2)$ 
	needs to satisfy the following properties:
\begin{enumerate}
	\item $u' : \Sigma(\sigma_1,\sigma_2) \to X\backslash \mathcal D$ is a smooth map.
	\item Let: 
	\[
	u'_{\rm d} := u' \circ I_{\rm d},\quad u'_{\rm s} := u' \circ I_{\rm s},\quad U'_{\rm D} = u' \circ I_{\rm D}.
	\]
	Then the $C^2$ distance of $u'_{\rm d}$ (resp. $u'_{\rm s}$) with the restriction of $u_{\rm d}$ (resp. $u_{\rm s}$) to $\Sigma_{\rm d}(\sigma_1)$ (resp. $\Sigma_{\rm s}(\sigma_2)$) 
	is less than $\epsilon$.
	The maps $U'_{\rm D}$ and $U_{\rm D}$ are also $C^2$-close to each other in the sense 
	that the image of $U_{\rm D}'$ is contained 
	in the open set $\frak U$ and there is a constant $r$ such that the $C^2$ distance of $U'_{\rm D}$ and ${\rm Dil}_r\circ U_{\rm D}$, restricted to $\Sigma_{\rm D}(\sigma_1,\sigma_2)$, 
	is less than $\epsilon$.\footnote{
	The $C^2$ distance in part (2) of the definition are defined with respect to the metric $g$ on $X\backslash \mathcal D$ and the metric on 
	$\mathcal N_{\mathcal D}(X)\setminus \mathcal D$ which has the form in \eqref{g-cylinder-end}.}
\item
	(Modified non-linear Cauchy-Riemann equation) $u'_{\rm d}$, $u'_{\rm s}$, $U'_{\rm D}$ satisfy the equations:
\begin{equation}\label{eq630}
\overline \partial u_{\rm d}' \in E_{\rm d}(u'_{\rm d}), 
\quad
\overline \partial u_{\rm s}' \in E_{\rm s}(u'_{\rm s}),
\quad
\overline \partial U_{\rm D}' \in E_{\rm D}(U'_{\rm D}).
\end{equation}
\item
(Transversal constraints) We also require:
\begin{equation}\label{eq631}
u'(w_{\rm D}) \in \widehat{\mathcal N}_{\rm D},
\quad
u'(w_{\rm s,1}) \in \mathcal N_{{\rm s},1},
\quad
u'(w_{\rm s,2}) \in {\mathcal N}_{{\rm s},2}.
\end{equation}
	Here we use $I_{\rm D}$, $I_{\rm s}$ to regard $w_{\rm D}$, $w_{{\rm s},i}$ as elements of $\Sigma(\sigma_1,\sigma_2)$.
	In the case that one of the constants $\sigma_1$ and $\sigma_2$ vanishes, the other one is also zero,
	and $u'$ is an element of the fiber product \eqref{form6210}.
\end{enumerate}
\end{definition}

One might hope that the intersection of the space $\mathcal U_0$
with $(\sigma_1,\sigma_2) = (\sigma^0_1,\sigma^0_2)$ for 
each given $(\sigma^0_1,\sigma^0_2)$ is cut down transversely by \eqref{eq630} and \eqref{eq631}, and hence the space $\mathcal U_0$ could be used to define a Kuranishi neighborhood of $[\Sigma,z_0,u]$ in  $\mathcal M^{\rm RGW}_1(L;\beta)$. However, this naive expectation does not hold. Roughly speaking, if that would hold, then one should obtain a solution for any element of the fiber product \eqref{form6210} close to $[\Sigma,z_0,u]$ and any small values of $\sigma_1$, $\sigma_2$. On the other hand, as a consequence of \cite[Remark 4.69]{part1:top},
the stratum in \eqref{form6210} has real codimension $2$ in our case, which is a contradiction. Note that this is in contrast with the stable map compactification, where a fiber product of the form \eqref{form6210} has codimension $4$. To resolve this issue, we introduce a space $\mathcal U$ larger than $\mathcal U_0$ such that $\mathcal U$ is a smooth manifold and $\mathcal U_0$ is cut out from $\mathcal U$ by an equation of the following form:
\begin{equation}\label{equation631}
	\sigma_1^{p_1} = c\sigma_2^{p_2}.
\end{equation}
Here $c$ is a complex valued function on $\mathcal U$ which never vanishes.
The space $\mathcal U$ is realized as the moduli space of {\it inconsistent solutions}, which will be defined in the next section.
Note that the set of solutions of \eqref{equation631} has a singularity at the locus $\sigma_1 = \sigma_2 = 0$. 

\section{Inconsistent Solutions and the Main Analytical Result}
\label{sub:statement}

In this section, we discuss the main step where the construction of the Kuranishi chart in our situation is different from the case of the stable map compactification. 
\begin{definition}\label{defn615inconsis}
	For $\sigma_1,\sigma_2\in D^2(\epsilon)$, an {\it inconsistent solution} is a 7-tuple \[(u_{\rm d}',u_{\rm s}',U_{\rm D}',\sigma_1,\sigma_2,\rho_1,\rho_2)\] 
	satisfying the following properties.
	\begin{enumerate}
		\item $u_{\rm d}' : \Sigma_{\rm d}(\sigma_1) \to X \setminus \mathcal D$, $u_{\rm s}' : \Sigma_{\rm s}(\sigma_2) \to X \setminus \mathcal D$,
			$U_{\rm D}' : \Sigma_{\rm D}(\sigma_1,\sigma_2) \to \mathcal N_{\mathcal D}(X)\setminus \mathcal D$. 
			The $C^2$ distances of $u_{\rm d}'$, $u_{\rm s}'$ and $U_{\rm D}'$ with $u_{\rm d}$, $u_{\rm s}$ and $U_{\rm D}$ are less than $\epsilon$.\footnote{
			Here we use the same convention as in Definition \ref{defnn614} to defined the $C^2$-distances.}
		\item The following equations are satisfied:
			\begin{equation}\label{eq630-def}
				\overline \partial u_{\rm d}' \in E_{\rm d}(u'_{\rm d}), \quad\overline \partial u_{\rm s}' \in E_{\rm s}(u'_{\rm s}),\quad\overline \partial U_{\rm D}' \in E_{\rm D}( U'_{\rm D}).
			\end{equation}
			Here $E_{\rm d}(u'_{\rm d})$, $E_{\rm d}(u'_{\rm s})$ and $E_{\rm d}(U'_{\rm D})$ are defined as in \eqref{Ed-s} and \eqref{ED} using target parallel transport.
		\item We require the following transversal constraints: 
			\begin{equation}\label{eq631-p}
				\pi \circ U_{\rm D}'(w_{\rm D}) \in\mathcal N_{\rm D},\quad 
				u_{\rm s}'(w_{\rm s,1}) \in \mathcal N_{{\rm s},1},\quad u_{\rm s}'(w_{\rm s,2}) \in {\mathcal N}_{{\rm s},2}.
			\end{equation}
		\item Let $z_1,z_2 \in D^2$.
			\begin{enumerate}
			\item If $z_1z_2 = \sigma_1$, then:
			\[
			  u_{\rm d}'(\varphi_{\rm d}(z_1)) =({\rm Dil}_{\rho_1} \circ U_{\rm D}')(\varphi_{\rm D,1}(z_2)).
			\]
			In particular, we assume that the left hand side is contained in the open neighborhood $\frak U$ of $\mathcal D$.
			\item If $z_1z_2 = \sigma_2$, then:
			\[
			   u_{\rm s}'(\varphi_{\rm s}(z_1)) =  ({\rm Dil}_{\rho_2} \circ U_{\rm D}')(\varphi_{\rm D,2}(z_2)).
			\]
			\end{enumerate}
	\end{enumerate}
	We say two inconsistent solutions $(u^{(j)}_{\rm s},u^{(j)}_{\rm d},U^{(j)}_{\rm D},\sigma^{(j)}_1,\sigma^{(j)}_2,\rho^{(j)}_1,\rho^{(j)}_2)$, 
	$j=1,2$, are {\it equivalent} if the following holds:
	\begin{enumerate}
	\item[(i)] $u^{(1)}_{\rm d} =u^{(2)}_{\rm d}$, $u^{(1)}_{\rm s} =u^{(2)}_{\rm s}$, $\sigma^{(1)}_1 = \sigma^{(2)}_1$, $\sigma^{(1)}_2 = \sigma^{(2)}_2$.
	\item[(ii)] There exists a nonzero complex number $c$ such that:
		\[
		  U^{(2)}_{\rm D} =  {\rm Dil}_{1/c} \circ U^{(1)}_{\rm D},\qquad \rho^{(2)}_1 = c\rho^{(1)}_1, \qquad\rho^{(2)}_2 = c\rho^{(1)}_2.
		\]
	\end{enumerate}
	We will write $\mathcal U$ for the set of all equivalence classes of inconsistent solutions.
\end{definition}
\begin{remark}
	In the above definition, we include the case that $\sigma_1$ or $\sigma_2$ is $0$ in the following way:
	\begin{enumerate}
	\item If $\sigma_1 = 0$ (resp. $\sigma_2 = 0$), then the condition (4) (a) (resp. (b)) is replaced by the condition that  
	$u_{\rm d}'(\varphi_{\rm d}(0)) = \pi \circ U_{\rm D}'(\varphi_{\rm D,1}(0))$ 
	(resp. $u_{\rm s}'(\varphi_{\rm s}(0)) =\pi \circ  U_{\rm D}'(\varphi_{\rm D,1}(0))$);
	\item If $\sigma_1 = 0$ (resp. $\sigma_2 = 0$), then $\rho_1 = 0$ (resp. $\rho_2 = 0$).
	\end{enumerate}
	In the case that exactly one of $\sigma_1$ and $\sigma_2$ is zero, 
	the source curve $\Sigma(\sigma_1,\sigma_2)$ has only one node.
	Such source curves do not appear in $\mathcal U_0$.
	However, there are elements of this form in $\mathcal U$.
\end{remark}
Below we state our main analytic results about $\mathcal U$:
\begin{prop}\label{prop617}
	If $\epsilon$ is small enough, then the moduli space $\mathcal U$ is a smooth manifold diffeomorphic to\footnote{See Remark \ref{rem619} for the 
	definition of the smooth structure of $D^2(\epsilon)$.}: 
	\begin{equation}\label{form633}
		(\mathcal U_{\rm d} \,\,{}_{{\rm ev}_{\rm d}}\times_{{\rm ev}_{\rm D,d}}  \mathcal U_{\rm D}
		\,\,{}_{{\rm ev}_{\rm D,s}}\times_{{\rm ev}_{\rm s}}\mathcal U_{\rm s})
		\times D^2(\epsilon) \times D^2(\epsilon).
	\end{equation}
	The diffeomorphism has the following properties:
	\begin{enumerate}
	\item This diffeomorphism identifies the projection to the factor $D^2(\epsilon) \times D^2(\epsilon)$ with:
		\[
		  [u'_{\rm s},u'_{\rm d},u'_{\rm D},\sigma_1,\sigma_2,\rho_1,\rho_2] \mapsto (\sigma_1,\sigma_2).
		\]
	\item There exists $\hat\rho_i : \mathcal U \to \bbC$ such that  
		any element $q$ of $\mathcal U$ has a representative whose $\rho_i$ component 
		is equal to $\hat\rho_i(q)$. The functions $\hat\rho_i$ are smooth. Moreover, there exists a homeomorphism:
		\begin{equation}\label{rho1rho2hito}
			\mathcal U_0 \cong \{ \frak y \in \mathcal U \mid \hat\rho_1(\frak y) = \hat\rho_2(\frak y)\}.
		\end{equation}
		This homeomorphism is given as follows. Let:
		\[
		  \frak y = [u'_{\rm d},u'_{\rm D},u'_{\rm s},\sigma_1,\sigma_2,\hat\rho_1,\hat\rho_2]
		\]
		be an element of the right hand side of \eqref{rho1rho2hito} with $c = \hat\rho_1 = \hat\rho_2$.
		Then we can glue the three maps $u'_{\rm d},u'_{\rm s},{\rm Dil}_{c}\circ U'_{\rm D}$ as in {\rm (gl-i),(gl-ii)} to obtain 
		a map $u' : \Sigma(\sigma_1,\sigma_2) \to X$. This gives the desired element of the left hand side of \eqref{rho1rho2hito}.
\end{enumerate}
\end{prop}
\begin{remark}
	We can take our diffeomorphism so that its restriction to 
	$(\mathcal U_{\rm d} \,\,{}_{{\rm ev}_{\rm d}}\times_{{\rm ev}_{\rm D,d}}  \mathcal U_{\rm D}
	\,\,{}_{{\rm ev}_{\rm D,s}}\times_{{\rm ev}_{\rm s}}\mathcal U_{\rm s})\times \{(0,0)\}$
	is the obvious one. We can also specify the choice of $\hat\rho_i$ in (2) above by requiring
	\begin{equation}\label{formula634}
		\hat\rho_1 = \sigma_1^{p_1}.
	\end{equation}
	From now on, we will take this choice unless otherwise mentioned explicitly. The proof we will give implies that:
	\[
	  \hat{\rho}_2(\frak y) = f(\frak y) \sigma_2^{p_2},
	\]
	where $f$ is a nonzero smooth function.
\end{remark}

The next proposition is the exponential decay estimate similar to those in the case of the stable map compactification. (See \cite{foooexp} for the detail of the proof of this exponential decay estimate in the case of the stable map compactification.) To state our exponential decay estimate, we need to introduce some notations.
We define $T_i \in [0,\infty)$, $\theta_i \in \bbR/2\pi\sqrt{-1}\bbZ$ by the formula:
\begin{equation}\label{form635}
\sigma_i = \exp(-(T_i+\sqrt{-1}\theta_i)).
\end{equation}
The exponential decay estimate is stated in terms of $T_i$ and $\theta_i$.

Let $\xi = (u^{\xi}_{\rm d},u^{\xi}_{\rm D},u^{\xi}_{\rm s})$ be an element of the fiber product \eqref{form6210}. The triple $(\xi,\sigma_1=\exp(-(T_1+\sqrt{-1}\theta_1)),\sigma_2=\exp(-(T_2+\sqrt{-1}\theta_2)))$ determines an element of $\mathcal U$, which is denoted by:
\begin{equation}\label{form637777new}
	\aligned
	&\frak x(\xi,T_1,T_2,\theta_1,\theta_2)\\ 
	&=(u^{\xi}_{\rm d}(T_1,T_2,\theta_1,\theta_2;\cdot),u^{\xi}_{\rm D}(T_1,T_2,\theta_1,\theta_2;\cdot),u^{\xi}_{\rm s}(T_1,T_2,\theta_1,\theta_2;\cdot),\\
	&\qquad\sigma_1,\sigma_2,\rho_1(\xi,T_1,T_2,\theta_1,\theta_2),
	\rho_2(\xi,T_1,T_2,\theta_1,\theta_2))
	\endaligned
\end{equation}
Here we fix the representative by requiring (\ref{formula634}), namely, $\rho_1(\xi,T_1,T_2,\theta_1,\theta_2) = \sigma_1^2$.
Let $\frak R_2 \in [0,\infty)$, $\eta_2 \in \bbR/2\pi\bbZ$ be functions of $\xi$, $T_1,T_2,\theta_1,\theta_2$ given by
$$
\rho_2(\xi,T_1,T_2,\theta_1,\theta_2)
= \exp(-(\frak R_2+\sqrt{-1} \eta_2)).
$$
\begin{prop}\label{prop618}
	\begin{enumerate}
	\item Let $u_\circ^{\xi}$ be one of $u^{\xi}_{\rm d}$, $u^{\xi}_{\rm s}$,  $u^{\xi}_{\rm D}$. Then for any compact subset $K$ of 
		$\Sigma_{\rm s} \setminus \{z_{\rm s}\}$ (resp. $\Sigma_{\rm d} \setminus \{z_{\rm d}\}$, 
		$\Sigma_{\rm D} \setminus \{z_{\rm D,s},z_{\rm D,d}\}$) we have the following exponential decay estimates:
		\begin{equation}
			\left\Vert
			\frac{\partial^{m_1}}{\partial T_1^{m_1}}\frac{\partial^{m'_1}}{\partial \theta_1^{m'_1}}
			\frac{\partial^{m_2}}{\partial T_2^{m_2}}
			\frac{\partial^{m'_2}}{\partial \theta_2^{m'_2}}u_\circ^{\xi}\right\Vert_{L^2_{\ell}(K)}
			\le
			C \exp(-c\upsilon_1T_1 - c\upsilon_2T_2).
		\end{equation}
	Here $\upsilon_1 = 0$ if $m_1=m'_1 =0$. Otherwise $\upsilon_1 = 1$. 
	Similarly, $\upsilon_2$ equals to $0$ if $m_2=m'_2 =0$ and is equal to $1$ otherwise.
	Here $C,c$ are positive constants depending on $K$, $\ell$, $m_1,m'_1,m_2,m'_2$. 
	The same estimate holds for the derivatives of $u_\circ^{\xi}$ with respect to $\xi$.
	\item We also have the following estimates
%	\begin{equation}
%		\aligned
%		&\left\vert \frac{\partial^{m_1}}{\partial T_1^{m_1}}
%		\frac{\partial^{m'_1}}{\partial \theta_1^{m'_1}}\frac{\partial^{m_2}}{\partial T_2^{m_2}}
%		\frac{\partial^{m'_2}}{\partial \theta_2^{m'_2}}{(\frak R_2 - p_2T_2)}\right\vert\le
%		C \exp(-c\upsilon_1T_1 - c\upsilon_2T_2) \\
%		&\left\vert \frac{\partial^{m_1}}{\partial T_1^{m_1}}
%		\frac{\partial^{m'_1}}{\partial \theta_1^{m'_1}}\frac{\partial^{m_2}}{\partial T_2^{m_2}}
%		\frac{\partial^{m'_2}}{\partial \theta_2^{m'_2}}{(\eta_2 - p_2\theta_2)}\right\vert\le C \exp(-c\upsilon_1T_1 - c\upsilon_2T_2).
%		\endaligned
%	\end{equation}
\begin{align}\label{}
  \phantom{i + j + k}
  &\begin{aligned}
   &\left\vert \frac{\partial^{m_1}}{\partial T_1^{m_1}}\frac{\partial^{m'_1}}{\partial \theta_1^{m'_1}}\frac{\partial^{m_2}}{\partial T_2^{m_2}}
   \frac{\partial^{m'_2}}{\partial \theta_2^{m'_2}}{(\frak R_2 - p_2T_2)}\right\vert\le C \exp(-c\upsilon_1T_1 - c\upsilon_2T_2) \\
      &\left\vert \frac{\partial^{m_1}}{\partial T_1^{m_1}}\frac{\partial^{m'_1}}{\partial \theta_1^{m'_1}}\frac{\partial^{m_2}}{\partial T_2^{m_2}}
	\frac{\partial^{m'_2}}{\partial \theta_2^{m'_2}}{(\eta_2 - p_2\theta_2)}\right\vert\le C \exp(-c\upsilon_1T_1 - c\upsilon_2T_2).
  \end{aligned}
\end{align}
	Here $\upsilon_1$ and $\upsilon_2$ are defined as in the first part. 
	The same estimate holds for the derivatives of $\frak R_2$, $\eta_2$ with respect to $\xi$.
	\end{enumerate}
\end{prop}

We will discuss the proofs of Propositions \ref{prop617}, \ref{prop618} in Section \ref{sub:proofmain}.

\begin{remark}\label{rem619}
	We follow \cite[Subsection A1.4]{fooobook2}, \cite[Section 8]{foooexp}, \cite[Subsection 9.1]{fooo:const1} to use a smooth 
	structure on $D^2$ different from the standard one as follows.
	For $z \in D^2$, let $T,\theta$ be defined by the following identity:
	\[
	  z= \exp(-(T+\sqrt{-1}\theta)).
	\]
	We define a homeomorhism from a neighborhood of the origin in $D^2$ to $D^2$  as follows:
	\[
	z\mapsto w = \frac{1}{T} \exp(-\sqrt{-1}\theta).
	\]	
	We define a smooth structure on $D^2$, temporarily denoted by $D^2_{\text{new}}$, such that $z \mapsto w$ becomes a diffeomorphism from $D^2_{\text{new}}$ to $D^2$ with 
	the standard smooth structure.
	This new smooth structure $D^2_{\text{new}}$ is used to define a smooth structure on the factors $D^2$ in \eqref{form633}. 
	(We drop the term `new' from $D^2_{\text{new}}$ hereafter.)
	The Proposition \ref{prop618} implies smoothness of various maps at the origin of $D^2$ with respect to the new smooth structure. 
	See for example \cite[Lemma 22.6]{foootech}, \cite[Subsection 8.2]{foooexp}, \cite[Section 10]{fooo:const1} for further discussions related to this point.
\end{remark}

\section{Kuranishi Charts: a Special Case}
\label{sub:kuraconst0}

In this section we use Propositions \ref{prop617} and \ref{prop618} to obtain a Kuranishi chart at the point $[\Sigma,z_0,u] \in \mathcal M_1^{\rm RGW}(L;\beta)$. By definition, a Kuranishi chart of a point $p$ in a space $M$ consists of $(V_p,\Gamma_p, \mathcal E_p, \frak s_p,\psi_p)$ where $V_p$, the {\it Kuranishi neighborhood}, is a smooth manifold containing a distinguished point $\tilde p$, $\Gamma_p$, the {\it isotropy group}, is a finite group acting on $V_p$, $\mathcal E_p$, {the \it obstruction bundle} is a vector bundle over $V_p$ and $\frak s_p$, the {\it Kuranishi map}, is a section of $\mathcal E_p$ over $V$. Moreover, the action of $\Gamma_p$ at $\tilde p$ is trivial and the action of this group on $V_p$ is lifted to $\mathcal E_p$. The section $\frak s_p$ is $\Gamma_p$-equivariant and vanishes at $\tilde p$. Finally, $\psi_p$ is a homeomorphism from $\frak s_p^{-1}(0)/\Gamma_p$ to a neighborhood of $[\Sigma,z_0,u]$ in $\mathcal M_1^{\rm RGW}(L;\beta)$, which maps $\tilde p$ to $p$.

In the present case, we define the Kuranishi neighborhood to be the manifold $\mathcal U$ in Proposition \ref{prop617}, and define the isotropy group to be the trivial one. The obstruction bundle $\mathcal E$ on $\mathcal U$ is a trivial bundle whose fiber is 
\begin{equation}
E_{\rm d} \oplus E_{\rm D} \oplus E_{\rm s} \oplus \bbC.
\end{equation}
The Kuranishi map 
$$
\frak s = (\frak s_{\rm d},\frak s_{\rm D},\frak s_{\rm s},\frak s_{\rho}) : 
\mathcal U \to E_{\rm d} \oplus E_{\rm D} \oplus E_{\rm s} \oplus \bbC
$$
is defined by
\begin{equation}\label{ob-map}
\aligned
& 
\frak s_{\rm d}(\xi,T_1,T_2,\theta_1,\theta_2)
= \mathcal P_{\rm d}(\overline\partial u_{\rm d}^{\xi}(T_1,T_2,\theta_1,\theta_2;\cdot)) \\
&\frak s_{\rm D}(\xi,T_1,T_2,\theta_1,\theta_2)
= \mathcal P_{\rm D}(\overline\partial u_{\rm D}^{\xi}(T_1,T_2,\theta_1,\theta_2;\cdot)) \\
&
\frak s_{\rm s}(\xi,T_1,T_2,\theta_1,\theta_2)
= \mathcal P_{\rm s}(\overline\partial u_{\rm s}^{\xi}(T_1,T_2,\theta_1,\theta_2;\cdot))
\\
&\frak s_{\rho}(\xi,T_1,T_2,\theta_1,\theta_2)
= \sigma_1^{p_1} - \hat\rho_2(\xi,T_1,T_2,\theta_1,\theta_2).
\endaligned
\end{equation}
Here $\mathcal P_{\rm s}, \mathcal P_{\rm D}, \mathcal P_{\rm d}$ are as in \eqref{form63000}. The maps $u_{\rm s}^{\xi}$, $u_{\rm D}^{\xi}$, $u_{\rm d}^{\xi}$ are as in \eqref{form637777new}. Therefore,
$\overline\partial u_{\rm s}^{\xi}(T_1,T_2,\theta_1,\theta_2;\cdot)
\in E_{\rm s}( u_{\rm s}^{\xi}(T_1,T_2,\theta_1,\theta_2;\cdot))$
is a consequence of (\ref{eq630}). Since $E_{\rm s}( u_{\rm s}^{\xi}(T_1,T_2,\theta_1,\theta_2;\cdot))$
is in the domain of $\mathcal P_{\rm s}$, the first map is well-defined. Similarly, we can show that the second and the third maps are also well-defined. The last map is equivalent to $\hat\rho_1 - \hat\rho_2$, because of \eqref{formula634}.
\begin{lemma}
The map $\frak s$ is smooth.
\end{lemma}
\begin{proof}
	Proposition \ref{prop617} (1) implies that $\sigma_1$ is a smooth function.
	Proposition \ref{prop617} (2) implies that $\hat\rho_2$ is a smooth function.
	Therefore $\frak s_{\rho}$ is smooth.
	Smoothness of the maps $\frak s_{\rm s},\frak s_{\rm d},\frak s_{\rm D}$ for non-zero values of 
	$\sigma_1$ and $\sigma_2$ is a consequence of standard elliptic regularity.
	Smoothness for $\sigma_i=0$ follows from part (1) of Proposition \ref{prop618}.
	For similar results in the context of the stable map compactification, see 
	\cite[Lemma 22.6]{foootech}, \cite[Theorem 8.25]{foooexp}, \cite[Proposition 10.4]{fooo:const1},
	\cite[Section 26]{foootech} and \cite[Section 12]{fooo:const1}.
	The first three references concerns the $C^m$ property of the relevant maps whereas the last two discuss smoothness.
\end{proof}
We finally construct the parametrization map
\[
  \psi : \frak s^{-1}(0) \to \mathcal M_1^{\rm RGW}(L;\beta).
\]
Let $\frak x= [u'_{\rm d},u'_{\rm D},u'_{\rm s},\sigma_1,\sigma_2,\rho_1,\rho_2]\in \mathcal U$ be an element such that $\frak s(\frak x) = 0$. Firstly, let $\sigma_1$ and $\sigma_2$ be both non-zero. Equation $\frak s_{\rho}(\frak x) = 0$ implies that $\rho_1 = \rho_2$. Therefore, we can glue $u'_{\rm d},u'_{\rm D},u'_{\rm s}$, as in Proposition \ref{prop617} (2), to obtain 
$u' : \Sigma(\sigma_1,\sigma_2) \to X \setminus \mathcal D$. We use $\frak s_{\rm d}(\frak x) = \frak s_{\rm D}(\frak x) =\frak s_{\rm s}(\frak x) =0$ to conclude that $u'$ is $J$-holomorphic. We define $\psi(\frak x) \in \mathcal M_1^{\rm RGW}(L;\beta)$ to be the element determined by $u'$ and $z_0 \in \partial D_{\rm d} \subset \partial \Sigma(\sigma_1,\sigma_2)$.
In the case that $\sigma_1=0$, $\rho_1$ vanishes by definition. Equation $\frak s(\frak x) = 0$ implies that $\rho_2$ is also zero. We can also conclude from Definition \ref{defn615inconsis} that $\sigma_2=0$. Finally the first three equations in \eqref{ob-map} imply that $\frak x$ determines an element of $\mathcal M_1^{\rm RGW}(L;\beta)$ in the stratum described in Section \ref{subsec:gluing1}. The case that $\sigma_2=0$ can be treated similarly. It is easy to see that $\psi$ is a homeomorphism to a neighborhood of $[\Sigma,z_0,u]$
in $\mathcal M_1^{\rm RGW}(L;\beta)$. Given Propositions \ref{prop617} and \ref{prop618}, we thus proved the following result:
\begin{prop}
	$(\mathcal U,\mathcal E,\frak s,\psi)$ provides a Kuranshi chart for the 
	moduli space $\mathcal M_1^{\rm RGW}(L;\beta)$ at $[\Sigma,z_0,u]$.
\end{prop}

\section{Proof of the Main Analytical Result}
\label{sub:proofmain}

The purpose of this section is to prove Proposition \ref{prop617}.
The proofs are similar to the arguments in \cite{foooexp}. However, there is one novel point, which is related to the fact that we need the notion of inconsistent solutions.  In this section, we go through the construction of the required family of inconsistent solutions, emphasizing on this novel point. Then the estimates claimed in Proposition \ref{prop618} can be proved in the same way as in \cite[Section 6]{foooexp}.

Throughout this section, we use a different convention for our figures to sketch pseudo-holomorphic curves in $X$. In our figures in this section (e.g. Figure \ref{Figuresec6-4}), we regard the divisor $\mathcal D$ as a vertical line on the right. This is in contrast with our convention in Figure \ref{Figuresec6-3} and \cite{part1:top}, where we regard the divisor as a horizontal line on the bottom. Our new convention is more consistent with the previous literature, especially \cite{foooexp}.

\subsection{Cylindrical Coordinates}
\label{subsub:cylindricalcoordinate}
In \eqref{newform624}, we fix coordinate charts on $\Sigma_{\rm d}$, $\Sigma_{\rm s}$, $\Sigma_{\rm D}$ near the nodal points and parametrized by the disc ${\rm Int} (D^2)$. In this section, it is convenient to use a cylindrical coordinates on the domain of these coordinate charts. Thus we modify the definition of the maps in \eqref{newform624} as follows:
\[
  \begin{array}{ccc}
 	 \varphi_{\rm d}:[0,\infty) \times S^1 \to \Sigma_{\rm d},&\hspace{1cm}&
	 \varphi_{\rm D,d}:(-\infty,0]  \times S^1 \to \Sigma_{\rm D},\\
	 \varphi_{\rm s}:[0,\infty) \times S^1 \to \Sigma_{\rm s},&\hspace{1cm}&
	 \varphi_{\rm D,s}:(-\infty,0]  \times S^1 \to \Sigma_{\rm D},
 \end{array}	 
\]
where
\[
  \begin{array}{ccc}
 	 \varphi_{\rm d}(r_1',s_1'),&\hspace{1.5cm}&
	 \varphi_{\rm D,d}(r_1'',s_1''),\\
	 \varphi_{\rm s}(r_2',s_2'),&\hspace{1.5cm}&
	 \varphi_{\rm D,s}(r_2'',s_2''),
 \end{array}	 
\]
for $(r'_i,s'_i) \in [0,\infty) \times S^1$, $(r''_i,s''_i) \in (-\infty,0] \times S^1$, is defined to be what we denoted by
\begin{equation}\label{form643}
  \begin{array}{cc}
 	 \varphi_{\rm d}(\exp(-(r'_1+\sqrt{-1} s'_1))),&
	 \varphi_{\rm D,d}(\exp(r''_1+\sqrt{-1} s''_1)),\\
	 \varphi_{\rm s}(\exp(-(r'_2+\sqrt{-1} s'_2))),&
	 \varphi_{\rm D,s}(\exp(r''_2+\sqrt{-1} s''_2)),
 \end{array}	 
\end{equation}
in Section \ref{sub:Obst}.

The equations $z_1z_2 = \sigma_1$ or $z_1z_2 = \sigma_2$
appearing in (gl-i) and (gl-ii)\footnote{See the discussion about the construction of $\Sigma(\sigma_1,\sigma_2)$ around \eqref{form6220}.} can be rewritten as:
\begin{equation}\label{form644}
	\aligned
	r''_1 = r'_1 - 10T_1, \qquad s''_1 = s'_1 - \theta_1,  \\
	r''_2 = r'_2 - 10T_2, \qquad s''_2 = s'_2 - \theta_2, 
\endaligned
\end{equation}
where\footnote{We use the coefficient $10$ here to be consistent with \cite{foooexp}. Otherwise, they are not essential.}
\begin{equation}
\sigma_i = \exp(-(10T_i + \sqrt{-1}\theta_i)).
\end{equation}
We define
\begin{equation}\label{r1r2s1s2}
	r_i = r'_i - 5T_i = r''_i + 5T_i, \quad s_i = s'_i - \theta_i/2 = s''_i + \theta_i/2.
\end{equation}
We also slightly change our convention for the polar coordinate of $\rho_i$ of Definition \ref{defn615inconsis} ($i=1,2$) and define $\frak R_i$, $\eta_i$ as follows:
\[
  \rho_i = \exp(-(10\frak R_i + \sqrt{-1}\eta_i)).
\]
See Figure \ref{Figuresec6-4} below and compare with \cite[(6.2) and (6.3)]{foooexp}\footnote{In \cite{foooexp}, the letter $\tau$ is used for the variables that we denote by $r_i$ here. In this paper, we use $\tau$ to denote the $\bbR$ factor appearing in the target space.}.
\begin{figure}[h]
\centering
\includegraphics[scale=0.6]{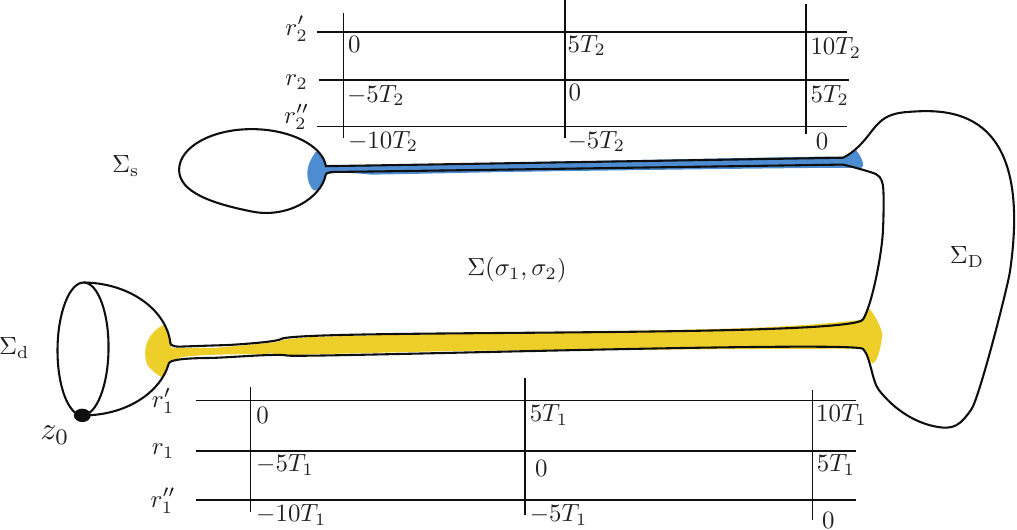}
\caption{$r_i,r'_i,r''_i$}
\label{Figuresec6-4}
\end{figure}

\subsection{Bump Functions}
\label{subsub:Bump}

For the purpose of constructing approximate solutions (pre-gluing) and for each step of the Newton's iteration used to solve our variant of non-linear Cauchy-Riemann equation, we use bump functions. Here we review various bump functions that we need. We may use the maps $\varphi_{\rm d}$, $\varphi_{\rm s}$, $\varphi_{\rm D,d}$ and $\varphi_{\rm D,s}$ to regard the following spaces as subspaces of $\Sigma(\sigma_1,\sigma_2)$:
$$
\aligned
\mathcal X_{i,T_i}= [-1,1]_{r_i} \times S^1_{s_i} 
&= [5T_i-1,5T_i+1]_{r'_i} \times S^1_{s'_i} \\
&= [-5T_i-1,-5T_i+1]_{r''_i} \times S^1_{s''_i},
\endaligned
$$
$$
\aligned
\mathcal A_{i,T_i} = [-T_i-1,-T_i+1]_{r_i} \times S^1_{s_i} 
&= [4T_i-1,4T_i+1]_{r'_i} \times S^1_{s'_i} \\
&= [-6T_i-1,-6T_i+1]_{r''_i} \times S^1_{s''_i},
\endaligned
$$
$$
\aligned
\mathcal B_{i,T_i} = [T_i-1,T_i+1]_{r_i} \times S^1_{s_i} 
&= [6T_i-1,6T_i+1]_{r'_i} \times S^1_{s'_i} \\
&= [-4T_i-1,-4T_i+1]_{r''_i} \times S^1_{s''_i}.
\endaligned
$$
Using $\varphi_{\rm d}$ (resp. $\varphi_{\rm s}$), the spaces $\mathcal X_{1,T_1}$, $\mathcal A_{1,T_1}$, 
$\mathcal B_{1,T_1}$ (resp. $\mathcal X_{2,T_2}$, $\mathcal A_{2,T_2}$, $\mathcal B_{2,T_2}$) can be identified with subspaces of $\Sigma_{\rm d} \setminus \{z_{\rm d}\}$ (resp. $\Sigma_{\rm s} \setminus \{z_{\rm s}\}$). Similarly, the map $\varphi_{\rm D,d}$ (resp $\varphi_{\rm D,s}$) allows us to regard $\mathcal X_{1,T_1}$, $\mathcal A_{1,T_1}$, 
$\mathcal B_{1,T_1}$, $\mathcal X_{2,T_2}$, $\mathcal A_{2,T_2}$, $\mathcal B_{2,T_2}$ as subspaces of $\Sigma_{\rm D} \setminus \{z_{\rm d},z_{\rm s}\}$. (See Figure \ref{Figuresec6-5} below.)
\begin{figure}[h]
\centering
\includegraphics[scale=0.6]{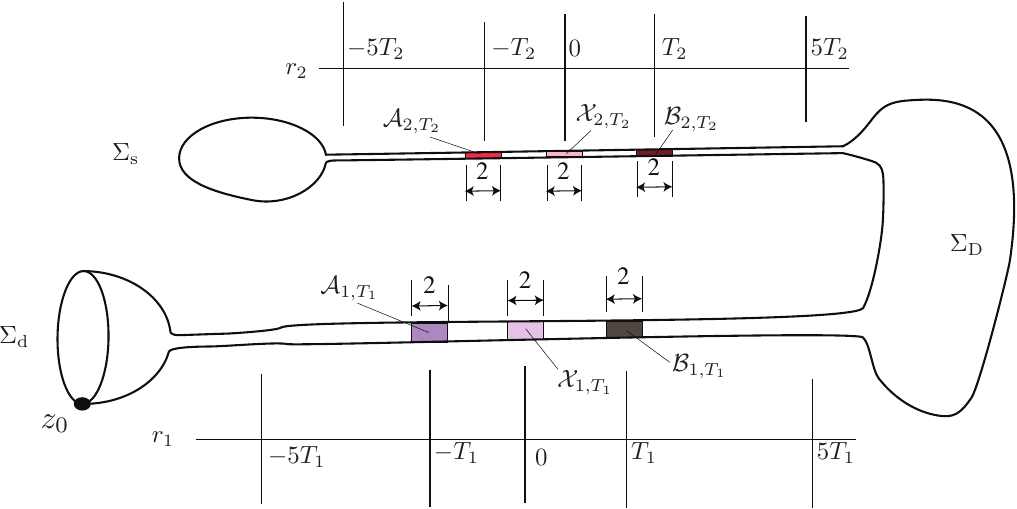}
\caption{$\mathcal X_{i,T_i}$, $\mathcal A_{i,T_i}$, $\mathcal B_{i,T_i}$}
\label{Figuresec6-5}
\end{figure}

We fix a non-increasing smooth function $\chi : \bbR \to [0,1]$ such that
$$
\chi(r)=
\begin{cases}
1   &{r< -1} \\
0 & {1 < r},
\end{cases}
$$
and $\chi(-r) = 1 - \chi(r)$.
We now define
\begin{equation}\label{form646new}
\aligned
&\chi_{i,\mathcal X}^{\leftarrow}(r_i,s_i) = \chi(r_i), 
\qquad
&\chi_{i,\mathcal X}^{\rightarrow}(r_i,s_i) = \chi(-r_i), \\
&\chi_{i,\mathcal A}^{\leftarrow}(r_i,s_i) = \chi(r_i+T_i), 
\qquad
&\chi_{i,\mathcal A}^{\rightarrow}(r_i,s_i) = \chi(-(r_i+T_i)),
\\
&\chi_{i,\mathcal B}^{\leftarrow}(r_i,s_i) = \chi(r_i-T_i), 
\qquad
&\chi_{i,\mathcal B}^{\rightarrow}(r_i,s_i) = \chi(-(r_i-T_i)).
\endaligned
\end{equation}
The functions $\chi_{1,\mathcal X}^{\leftarrow}$, $\chi_{1,\mathcal A}^{\leftarrow}$ and $\chi_{1,\mathcal B}^{\leftarrow}$ can be extended to smooth functions on $\Sigma_d$ which are locally constant outside of the spaces $\mathcal X_{1,T_1}$, $\mathcal A_{1,T_1}$ and $\mathcal B_{1,T_1}$, respectively. We use the same notations to denote these extensions. Similarly,  we can define functions $\chi_{2,\mathcal X}^{\leftarrow}$, $\chi_{2,\mathcal A}^{\leftarrow}$ and $\chi_{2,\mathcal B}^{\leftarrow}$ on $\Sigma_{\rm s}$. These functions can be also regarded as functions defined on $\Sigma(\sigma_1,\sigma_2)$ in the obvious way.

We use $\chi_{i,\mathcal X}^{\rightarrow}$ (resp. $\chi_{i,\mathcal A}^{\rightarrow}$ and $\chi_{i,\mathcal B}^{\rightarrow}(r_i,s_i)$), for $i=1,2$, to define a smooth function $\chi_{\mathcal X}^{\rightarrow}$ (resp. $\chi_{\mathcal A}^{\rightarrow}$ and $\chi_{\mathcal B}^{\rightarrow}$) on $\Sigma(\sigma_1,\sigma_2)$ as follows.
On the neck regions where the coordinate $(r_i,s_i)$, for $i=1$ or $2$, is defined, we set $\chi_{\mathcal X}^{\rightarrow}$ (resp. $\chi_{\mathcal A}^{\rightarrow}$, $\chi_{\mathcal B}^{\rightarrow}$) to be the function $\chi_{i,\mathcal X}^{\rightarrow}(r_i,s_i)$ (resp. $\chi_{i,\mathcal A}^{\rightarrow}(r_i,s_i)$ and $\chi_{i,\mathcal B}^{\rightarrow}(r_i,s_i)$)
given in \eqref{form646new}. This function is defined to be locally constant on the complement of the above space. %Thus we have functions $\chi_{i,\mathcal X}^{\leftarrow}$, $\chi_{i,\mathcal A}^{\leftarrow}$, $\chi_{i,\mathcal B}^{\leftarrow}$, for $i=1,2$, and $\chi_{\mathcal X}^{\rightarrow}$ $\chi_{\mathcal A}^{\rightarrow}$, $\chi_{\mathcal B}^{\rightarrow}$.
See Figures \ref{Figuresec6-6} and  \ref{Figuresec6-7}.
\begin{figure}[h]
\centering
\includegraphics[scale=0.6]{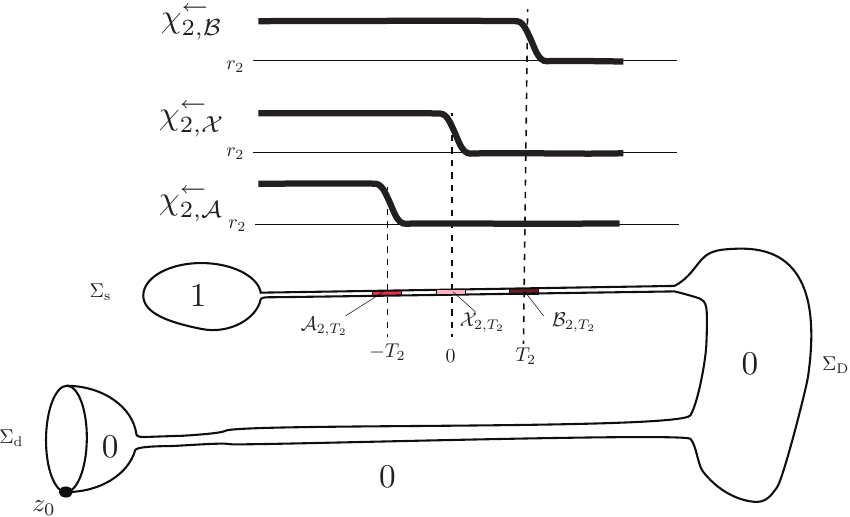}
\caption{$\chi_{2,\mathcal X}^{\leftarrow}$, $\chi_{2,\mathcal A}^{\leftarrow}$, $\chi_{2,\mathcal B}^{\leftarrow}$}
\label{Figuresec6-6}
\end{figure}
\begin{figure}[h]
\centering
\includegraphics[scale=1.2]{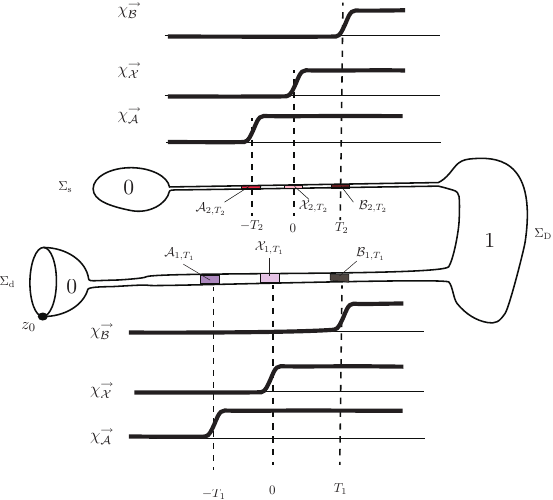}
\caption{ $\chi_{\mathcal X}^{\rightarrow}$ $\chi_{\mathcal A}^{\rightarrow}$,
$\chi_{\mathcal B}^{\rightarrow}$}
\label{Figuresec6-7}
\end{figure}

Note that the supports of the first derivatives of 
$\chi_{i,\mathcal X}^{\leftarrow}$, $\chi_{i,\mathcal A}^{\leftarrow}$, 
$\chi_{i,\mathcal B}^{\leftarrow}$
are subsets of $\mathcal X_{i,T_i}$, $\mathcal A_{i,T_i}$, 
$\mathcal B_{i,T_i}$, respectively.
The supports of the first derivatives of 
$\chi_{\mathcal X}^{\rightarrow}$ $\chi_{\mathcal A}^{\rightarrow}$,
$\chi_{\mathcal B}^{\rightarrow}$
are subsets of $\mathcal X_{1,T_1} \cup \mathcal X_{2,T_2}$, $\mathcal A_{1,T_1} \cup \mathcal A_{2,T_2}$, 
$\mathcal B_{1,T_1} \cup \mathcal B_{2,T_2}$, respectively.

\subsection{Weighted Sobolev Norms}
\label{subsub:Wsobolev}

In Section \ref{sub:Fred}, we define weighted Sobolev norms on several function spaces on $\Sigma_{\rm d}$, $\Sigma_{\rm s}$, $\Sigma_{\rm D}$. Here we use weighted Sobolev norms to define a function space on $\Sigma(\sigma_1,\sigma_2)$.
Since $\Sigma(\sigma_1,\sigma_2)$ is compact and the weight functions that we will define are smooth, the resulting weighted Sobolev norm is equivalent to the usual Sobolev norm. In other words, the ratio between the two norms is bounded as long as we fix $\sigma_1,\sigma_2$. However, the ratio depends on $\sigma_1,\sigma_2$ and is unbounded as $\sigma_1,\sigma_2$ goes to zero. Therefore, using weighted Sobolev norm is crucial to show that various estimates are independent of $\sigma_1,\sigma_2$.

We decompose $\Sigma(\sigma_1,\sigma_2)$ as follows:
\begin{align*}
	\Sigma(\sigma_1,\sigma_2)=
	&(\Sigma_{\rm d} \setminus {\rm Im}\varphi_{\rm d})\cup (\Sigma_{\rm s} \setminus {\rm Im}\varphi_{\rm s})
	  \cup 
	  (\Sigma_{\rm D} \setminus ({\rm Im}\varphi_{\rm D,d} \cup {\rm Im}\varphi_{\rm D,s}))\\
	  &\cup ([-5T_1,5T_1]_{r_1} \times S^1_{s_1}) \cup ([-5T_2,5T_2]_{r_2} \times S^1_{s_2}).
\end{align*}
Here we identify $[-5T_1,5T_1]_{r_1} \times S^1_{s_1}$ and
$[-5T_2,5T_2]_{r_2} \times S^1_{s_2}$ with their images in $\Sigma(\sigma_1,\sigma_2)$.
We also introudce the following notations for various subspaces of $\Sigma(\sigma_1,\sigma_2)$: (See Figures \ref{Figuresec6-8}, \ref{Figuresec6-9} and \ref{Figuresec6-10}.)
\begin{equation}\label{formu647}
\aligned
\Sigma_{\rm d}^-(\sigma_1,\sigma_2) &= \Sigma_{\rm d} \setminus {\rm Im}\varphi_{\rm d}, \\
\Sigma_{\rm d}^+(\sigma_1,\sigma_2) &= \Sigma_{\rm d} \setminus \varphi_{\rm d}(D^2(\vert \sigma_1\vert)), \\
\Sigma_{\rm s}^-(\sigma_1,\sigma_2) &= \Sigma_{\rm s} \setminus {\rm Im}\varphi_{\rm s}, \\
\Sigma_{\rm s}^+(\sigma_1,\sigma_2) &= \Sigma_{\rm s} \setminus \varphi_{\rm s}(D^2(\vert \sigma_2\vert)), \\
\Sigma_{\rm D}^-(\sigma_1,\sigma_2) &= \Sigma_{\rm D} \setminus ({\rm Im}\varphi_{\rm D,d}\cup  {\rm Im}\varphi_{\rm D,s}), \\
\Sigma_{\rm D}^+(\sigma_1,\sigma_2) &= \Sigma_{\rm D} \setminus (\varphi_{\rm D,d}(\vert \sigma_1\vert)
\cup  {\rm Im}\varphi_{\rm D,s}(\vert \sigma_2\vert)).
\endaligned
\end{equation}
\begin{figure}[h]
\includegraphics[scale=0.6]{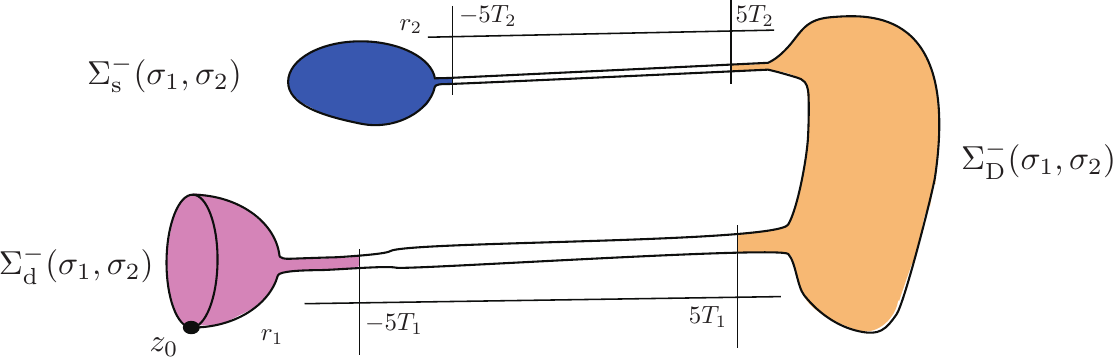}
\caption{$\Sigma_{\rm d}^-(\sigma_1,\sigma_2)$, $\Sigma_{\rm s}^-(\sigma_1,\sigma_2)$,
$\Sigma_{\rm D}^-(\sigma_1,\sigma_2)$}
\label{Figuresec6-8}
\end{figure}
\begin{figure}[h]
\includegraphics[scale=0.6]{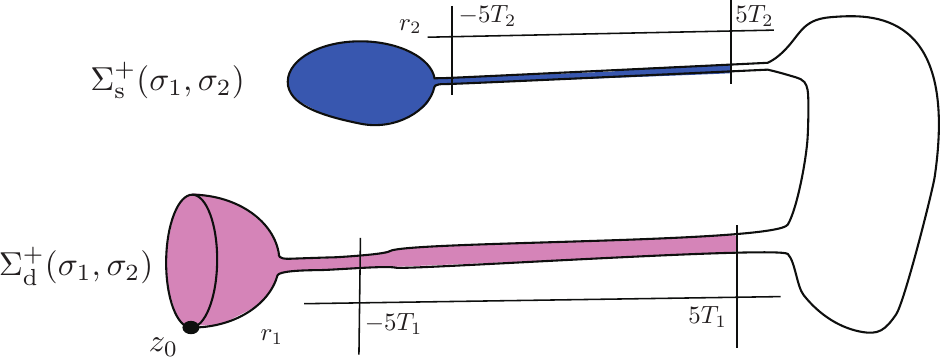}
\caption{$\Sigma_{\rm d}^+(\sigma_1,\sigma_2)$, $\Sigma_{\rm s}^+(\sigma_1,\sigma_2)$
}
\label{Figuresec6-9}
\end{figure}
\begin{figure}[h]
\includegraphics[scale=0.6]{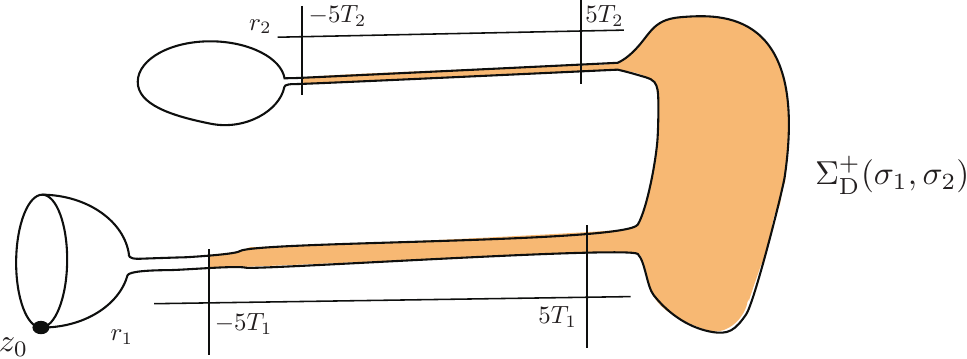}
\caption{$\Sigma_{\rm D}^+(\sigma_1,\sigma_2)$}
\label{Figuresec6-10}
\end{figure}
Note that $\Sigma^+_{\rm d}(\sigma_1,\sigma_2)$, $\Sigma^+_{\rm s}(\sigma_1,\sigma_2) $ and $\Sigma^+_{\rm D}(\sigma_1,\sigma_2)$ are respectively equal to the spaces $\Sigma_{\rm d}(\sigma_1)$, $\Sigma_{\rm s}(\sigma_2)$ and
$\Sigma_{\rm D}(\sigma_1,\sigma_2)$ defined in \eqref{metaform626}.

\begin{lemma} \label{formula648}  
	There exists a smooth function   $e^{\sigma_1,\sigma_2}_{\delta} :  \Sigma(\sigma_1,\sigma_2) \to [0,\infty)$  
	satisfying the following properties (see Figure \ref{Figuresec6-11}): 
	\begin{itemize}
            	\item[(i)] If $x\in \Sigma^-_{\rm d}(\sigma_1,\sigma_2)\cup 
            			\Sigma^-_{\rm s}(\sigma_1,\sigma_2)
            			\cup \Sigma^-_{\rm D}(\sigma_1,\sigma_2)$, then $e^{\sigma_1,\sigma_2}_{\delta}(x)= 1$;
            	\item[(ii)] If $r_i \in [1-5T_i,-1]$, then $e^{\sigma_1,\sigma_2}_{\delta}(r_i,s_i)=e^{\delta (r_i+5T_i)}$;
            	\item[(iii)] If $r_i \in [1,5T_i-1]$, then $e^{\sigma_1,\sigma_2}_{\delta}(r_i,s_i)=e^{\delta (-r_i+5T_i)}$; 
            	\item[(iv)] If $\vert \vert r_i\vert-5T_i\vert \le 1$, then $e^{\sigma_1,\sigma_2}_{\delta}(r_i,s_i)\in [1,10]$;
            	\item[(v)] If $\vert r_i\vert \le 1$, then $e^{\sigma_1,\sigma_2}_{\delta}(r_i,s_i)\in [e^{5T_i\delta},10e^{5T_i\delta}]$.
	\end{itemize}
\end{lemma}
\begin{figure}[h]
\includegraphics[scale=0.8]{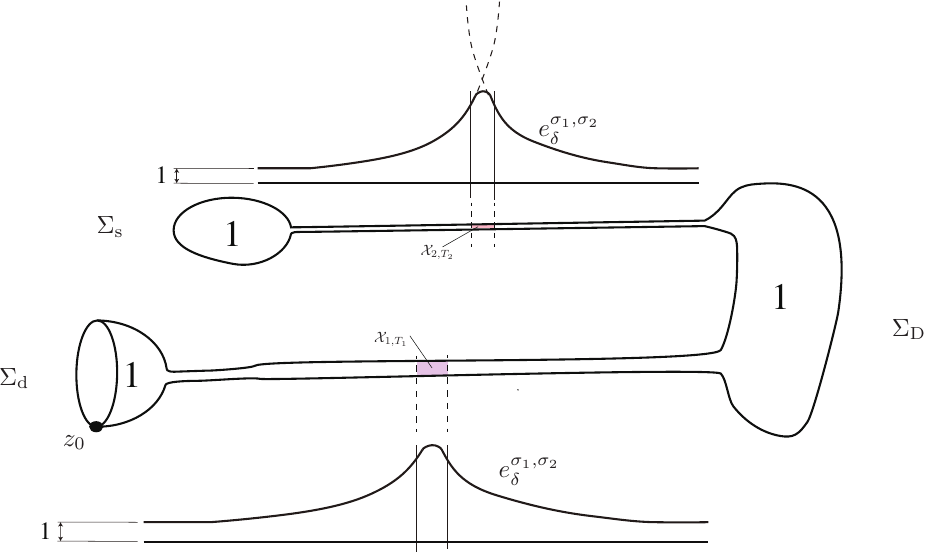}
\caption{$e^{\sigma_1,\sigma_2}_{\delta}$}
\label{Figuresec6-11}
\end{figure}

We fix a smooth map $u' : \Sigma(\sigma_1,\sigma_2)  \to X \setminus \mathcal D$
and assume that the diameters of:
\begin{equation}\label{neck-img}
  u'([-5T_1,5T_1]_{r_1} \times S^1_{s_1})\qquad \text{and} \qquad u'([-5T_2,5T_2]_{r_2} \times S^1_{s_2})
\end{equation}
with respect to the metric $g$ are less than a given positive real number $\kappa$. We require that the above sets are contained in $\frak U$, introduced in the beginning of Section \ref{subsec:gluing1}, where the partial $\bbC_*$-action is defined. Assuming $\kappa$ is small enough, to any:
\[
  V \in L^2_{m}(\Sigma(\sigma_1,\sigma_2);u^{\prime *}TX).
\]
we associate sections $\hat v_1$ and $\hat v_2$ of $u^{\prime *}TX$ over the subspaces
$[-5T_1,5T_1]_{r_1} \times S^1_{s_1}$ and $[-5T_2,5T_2]_{r_2} \times S^1_{s_2}$ in the following way.

Let $(0,0)_i \in [-5T_i,5T_i]_{r_i} \times S^1_{s_i}$ 
be the point whose $r_i,s_i$ coordinates are $0$.
By choosing $m$ to be greater than $1$, the following vector is well-defined:
\begin{equation}\label{defnvivi}
v_i = V((0,0)_i) \in T_{u'((0,0)_i)}X.
\end{equation}
Suppose $v_i = v_{i,\bbR} + v_{i,S^1} + v_{i,\rm D}$ is the decomposition of this vector with respect to \eqref{decom-tan-bdle}. If $\kappa$ is small enough, we can assume that the distance between any two points of the projection of \eqref{neck-img} to $\mathcal D$ is less than the injectivity radius of $\mathcal D$. In particular, we can extend $v_{i,\rm D}$ to a vector field $\hat v_{i,\rm D}$ in a neighborhood of $(0,0)_i$ using parallel transport along geodesics based at $u'((0,0)_i)$ with respect to the unitary connection on $T\mathcal D$, which we fixed before. Then the vector $\hat v_i$ is defined to be:
\begin{equation}\label{form651651newnew}
  \hat v_i = v_{i,\bbR} + v_{i,S^1} + \hat v_{i,\rm D}.
\end{equation}

Now we define 
\begin{equation}\label{form6464rev}
\aligned
&\Vert V
\Vert_{W^2_{m,\delta}}^2 \\
=
&\Vert V\Vert_{L^2_{m}((\Sigma_{\rm d} \setminus {\rm Im}\varphi_{\rm d})
\cup 
(\Sigma_{\rm s} \setminus {\rm Im}\varphi_{\rm s})
\cup 
(\Sigma_{\rm D} \setminus ({\rm Im}\varphi_{\rm D,d} \cup {\rm Im}\varphi_{\rm D,s}))}^2 \\
&+
\sum_{i=1}^2\sum_{j=0}^m\int_{[-5T_i,5T_i]_{r_i}\times S^1_{s_i}} e^{\delta}_{\sigma_1,\sigma_2}(r_i,s_i) 
\vert \nabla^j(V -  \hat v_i)\vert^2 dr_id s_i
\\
&+ \vert v_1 \vert^2 + \vert v_2 \vert^2.
\endaligned
\end{equation}
We use the cylindrical metric on $\Sigma(\sigma_1,\sigma_2)$ and the metric $g$ on $X \setminus \mathcal D$ to define norms in the first and the second lines of the right hand side of \eqref{form6464rev}. This definition is analogous to \eqref{form6464}.
The space of all $V$ as above with finite $\Vert\cdot\Vert_{W^2_{m,\delta}}$ norm which satisfies the boundary condition:
\[
  \hspace{3cm} V(z) \in T_{u'(z)}L \hspace{1cm} \forall z \in \partial \Sigma(\sigma_1,\sigma_2)
\]
forms a Hilbert space, which we denoted by: 
\begin{equation}\label{fcspace652}
W^2_{m,\delta}((\Sigma(\sigma_1,\sigma_2),\partial \Sigma(\sigma_1,\sigma_2));
(u^{\prime *}TX,u'\vert_{\partial}^*TL)).
\end{equation}

%We require boundary condition to $V$, that is, 
%$V(z) \in T_{u'(z)}L$ for $z \in \partial \Sigma(\sigma_1,\sigma_2)$.
%The set of all such $V$ whose norm defined by (\ref{form6464rev}) are finite
%becomes a Hilbert space. We denote this Hilbert space by
%\begin{equation}\label{fcspace652}
%W^2_{m,\delta}((\Sigma(\sigma_1,\sigma_2),\partial \Sigma(\sigma_1,\sigma_2));
%(u^{\prime *}TX,u'\vert_{\partial}^*TL)).
%\end{equation}
Next, let:
\[
  V \in L^2_{m}(\Sigma(\sigma_1,\sigma_2);u^{\prime *}TX\otimes \Lambda^{0,1})
\]
and define:
\begin{equation}\label{formula650}
\Vert V
\Vert_{L^2_{m,\delta}}^2 
=
\sum_{j=0}^m\int_{\Sigma(\sigma_1,\sigma_2)} e^{\delta}_{\sigma_1,\sigma_2}(z) 
\vert \nabla^jV(z)\vert^2 {\rm vol}_{\Sigma(\sigma_1,\sigma_2)}.
\end{equation}
We use the cylindrical metric on $\Sigma(\sigma_1,\sigma_2)$ and the metric $g$ on $X \setminus \mathcal D$ to define the norm and the volume element ${\rm vol}_{\Sigma(\sigma_1,\sigma_2)}$.
The set of all such $V$ with $\Vert V
\Vert_{L^2_{m,\delta}} < \infty$ forms  a Hilbert space, which we denote by
\begin{equation}
L^2_{m,\delta}(\Sigma(\sigma_1,\sigma_2);u^{\prime *}TX\otimes \Lambda^{0,1}).
\end{equation}
As a topological vector space, this is the same space as the standard space of Sobolev $L^2_m$ sections. However, the ratio between the above $L^2_{m,\delta}$ norm and the standard Sobolev $L^2_m$ norm 
is unbounded while $\sigma_1,\sigma_2$ go to $0$.

Finally, we can use the above Sobolev spaces, to define the linearization of the non-linear Cauchy-Riemann equation at $u'$, which is a Fredholm operator:
\begin{equation}\label{newform651}
	\aligned
	D_{u'}\overline{\partial} : 
	&W^2_{m+1,\delta}((\Sigma(\sigma_1,\sigma_2),\partial \Sigma(\sigma_1,\sigma_2));
	(u^{\prime *}TX,u'\vert_{\partial}^*TL)) \\
	&\to L^2_{m,\delta}(\Sigma(\sigma_1,\sigma_2);u^{\prime *}TX\otimes \Lambda^{0,1}).
	\endaligned
\end{equation}

\subsection{Pre-gluing}
\label{subsub:preglue}
Suppose $\xi = (u_{\rm d}^{\xi},u_{\rm D}^{\xi},u_{\rm s}^{\xi})$ is an element of the following space\footnote{Recall that $\Sigma_{\rm d}$ together with the marked points $z_0$ and $z_{\rm d}$ is already source stable and we did not need to introduce auxiliary marked points on this space. This is the reason that the first factor is $\mathcal U_{\rm d}$, rather than $\mathcal U_{\rm d}^+$.}:
\begin{equation}\label{fib-prod-str}
  \mathcal U_{\rm d} \,\,{}_{{\rm ev}_{\rm d}}\times_{{\rm ev}_{\rm D,d}}  \mathcal U^+_{\rm D}
   \,\,{}_{{\rm ev}_{\rm D,s}}\times_{{\rm ev}_{\rm s}}\mathcal U^+_{\rm s}
\end{equation}
In this subsection, for each choice of $\sigma_1$ and $\sigma_2$, we shall construct an approximate inconsistent solution and approximate the error for this approximate solution.

By assumption, the pull back bundle $(u_{\rm D}^{\xi})^*\mathcal N_{\mathcal D}(X)$ has a meromorphic section $\frak s^\xi$
which has poles of order $p_1$ and $p_2$ at $z_{\rm d}$ and $z_{\rm s}$, respectively. As in \eqref{U-D}, $\frak s^\xi$ gives rise to a map 
\begin{equation}\label{form653}
	U_{\rm D}^{\xi} : \Sigma_{\rm D} \setminus \{z_{\rm d},z_{\rm s}\} \to \mathcal N_{\mathcal D}(X)\backslash \mathcal D.
\end{equation}
A priori, the section $\frak s^\xi$ is well-defined up to the action of $\bbC_*$ and for each $\xi$ in \eqref{fib-prod-str}, we fix one such section such that $U_{\rm D}^{\xi}$ depends smoothly on $\xi$. Later we will pin down the choice of sections such that \eqref{formula634} is satisfied.
%We use various objects and notations introduced in the preceding subsections to start 
%gluing analysis in this subsection.
%Let 
%$$
%\xi \in \mathcal U^+_{\rm d} \,\,{}_{{\rm ev}_{\rm d}}\times_{{\rm ev}_{\rm D,d}}  \mathcal U^+_{\rm D}
%\,\,{}_{{\rm ev}_{\rm D,s}}\times_{{\rm ev}_{\rm s}}\mathcal U^+_{\rm s}.
%$$
%We denote
%$$
%\xi = (u_{\rm d}^{\xi},u_{\rm D}^{\xi},u_{\rm s}^{\xi}).
%$$
%The pull back bundle $(u_{\rm D}^{\xi})^*\mathcal N_{\mathcal D}(X)$ has a meromorphic section $\frak s$
%which has poles of order 2 (resp. 3) at $z_{\rm d}$ (resp. $z_{\rm s}$).
%We use it to obtain a map
%\begin{equation}\label{form653}
%U_{\rm D}^{\xi} : \Sigma_{\rm D} \setminus \{z_{\rm d},z_{\rm s}\} \to \mathcal N_{\mathcal D}(X).
%\end{equation}
%Note there is an ambiguity (via $\bbC_*$ action), for the choice of $U_{\rm D}^{\xi}$. 
%We fix one so that $U_{\rm D}^{\xi}$ depends smoothly on $\xi$.
%We will change the choice of this section at the last step of the proof such that \eqref{formula634} will be satisfied.
Recall that a neighborhood of the zero section in $\mathcal N_{\mathcal D}(X)$ is identified with the neighborhood $\frak U$ of $\mathcal D$ in $X$. For now, we assume that the section $\frak s^\xi$ is chosen such that the image of $U_{\rm D}^{\xi}$ on the domain $\Sigma_{\rm D}^+(\sigma_1,\sigma_2)$ belongs to this neighborhood of the zero section of $\mathcal N_{\mathcal D}(X)$. Recall that $\Sigma_{\rm D}^+(\sigma_1,\sigma_2)$ is defined in \eqref{formu647}.

Next, we shall glue the three maps $u^{\xi}_{\rm d}$, $u^{\xi}_{\rm s}$, $U_{\rm D}^{\xi}$ by a partition of unity. One should beware that the output of this construction is an approximate inconsistent solution. In particular, it will not be a globally well-defined map from $\Sigma(\sigma_1,\sigma_2)$ to $X$. In order to describe this process, we need to fix an {\it exponential} map.

In the following, we need a map 
\begin{equation} \label{eq653}
 {\rm Exp}:T(X\setminus \mathcal D)\to\frak N(\Delta)
\end{equation}
for a neighborhood $\frak N(\Delta)$ of the diagonal $\Delta$ in $(X\setminus \mathcal D)\times (X\setminus \mathcal D)$ that satisfies certain properties. Before stating the required properties for this map, we need to define partial $\bbC_*$ actions on a pair of a complex manifold and a submanifold. Recall that we defined partial $\bbC_*$ actions for a pair of an almost complex manifold $Y$ and a submanifold $D$ of (complex) codimension $1$ in \cite[Subsection 3.2]{part1:top}. This notion can be generalized to the case of complex submanifolds of arbitrary codimension in an obvious way. For example, the derivative of the partial $\bbC_*$ action for the pair $(X,\mathcal D)$ determines a partial $\bbC_*$ action for the pair $(TX,T\mathcal D)$. Moreover, the product of two copies of partial $\bbC_*$ actions for the pair $(X,\mathcal D)$ induces a partial $\bbC_*$ action on $(X\times X,\mathcal D\times \mathcal D)$. Now we are ready to state the properties of ${\rm Exp}$:
\begin{itemize}
	\item[(i)] For $p \in X\setminus \mathcal D$ and $V\in T_p(X\setminus \mathcal D)$, the first component 
	of ${\rm Exp}(p,V)$ is $p$.
	\item[(ii)] ${\rm Exp}$	 maps $(p,0)\in T_p(X\setminus \mathcal D)$ to 
			$(p,p)\in (X\setminus \mathcal D)\times (X\setminus \mathcal D)$. Moreover, 
			at the point $(p,0)$, the derivative of ${\rm Exp}$  in the fiber direction given by 
			$T_p(X\setminus \mathcal D)\subset T(X\setminus \mathcal D)$
			is equal to $(0,{\rm id})$ where ${\rm id}$ is the identity map from 
			$T_p(X\setminus \mathcal D)$ to itself.
	\item[(iii)] The map \eqref{eq653} is equivariant with respect to the partial 
			$\bbC_*$ actions on the domain and the target defined above.
%			\footnote{The $\bbC_*$ actions are given as follows. Recall that we defined partial $\bbC_*$ actions for 
%			a pair of an almost complex manifold $Y$
%			and a complex submanifold $D$ of (complex) codimension $1$ in 
%			\cite[Subsection 3.2]{part1:top}. 
%			This notion can be generalized to the case of complex submanifolds of 
%			arbitrary codimension in an obvious way. 
%			For example, the derivative of the partial $\bbC_*$
%			action for the pair $(X,\mathcal D)$ determines a partial $\bbC_*$ action for the pair 
%			$(TX,T\mathcal D)$.
%			Moreover, the product of two copies of partial $\bbC_*$ actions for the pair $(X,\mathcal D)$
%			induces a partial $\bbC_*$ action on $(X\times X,\mathcal D\times \mathcal D)$.}
                        %We moreover require that the map ${\rm Exp}$ is well-behaved with respect to partial $\bbC_*$ actions 
                        %in the following sense.
                        %We consider the restriction of a vector bundle $TX$ to a neighborhood of $\mathcal D$.
                        %Then the partial $\bbC_*$ action to $X$ on  a neighborhood of $\mathcal D$  induces 
                        %a partial $\bbC_*$ action to $(TX,TX\vert_{\mathcal D})$.
                        %
                        %On other hand we use the trivial action on the first factor of $X\times X$ and 
                        %the action we obtained in \cite[Lemma \ref{exsitsJ'}]{part1:top} on the second factor of $X\times X$
                        %to define a partial $\bbC_*$ action on $(X\times X,X \times \mathcal D)$.
                        %
                        %We require that the map (\ref{eq653}) preserves these two 
                        %partial $\bbC_*$ actions.
	\item[(iv)] For a positive real number $\kappa$, let $D_{\kappa}TL$ denote the tangent vectors 
			to $L$ whose norms are smaller than
			$\kappa$. There is $\kappa$ such that:
			\[{\rm Exp}(D_{\kappa}TL) \subset L \times L.\]
\end{itemize} 
%We take the following diffeomorphism to its image:
%\begin{equation}\label{eq653}
%{\rm Exp} : T(X\setminus \mathcal D) \to (X\setminus \mathcal D) \times (X\setminus \mathcal D)
%\end{equation}
%with the following properties. 
%We denote an element of
%$T(X\setminus \mathcal D)$ by $(p,V)$ where $p \in X\setminus \mathcal D$ and $V\in T_p(X\setminus \mathcal D)$. Then the first component of ${\rm Exp}(p,V)$ is $p$ and the derivative 
%of this map in $T_p(X\setminus \mathcal D)$ direction defines an identity map $T_p(X\setminus \mathcal D) \to T_p(X\setminus \mathcal D)$ 
%at $(p,0)$.
%The map $({\rm id},{\rm exp})$ (where ${\rm exp}$ is the exponential map 
%with respect to a Riemannian metric on $X\setminus \mathcal D$ which is of cylindrical 
%type at the end), has this property.
Let ${\rm exp}$ be the exponential map with respect to the metric $g$. The map $({\rm id},{\rm exp})$, defined on a neighborhood of the zero section of $T(X\setminus \mathcal D)$, satisfies (i)-(iii). We can modify this map and extend it to a map on $T(X\setminus \mathcal D)$ which satisfies (iv). We denote the inverse of \eqref{eq653} by
$$
{\rm E} : \frak N(\Delta) \to T(X\setminus \mathcal D).
$$

We now define $\rho_{i,(0)}^{\xi} \in \bbC_*$ ($i=1,2$) as follows.
We define the compositions 
$$
u^{\xi}_{\rm d} \circ \varphi_{\rm d} : D^2 \to X\setminus \mathcal D,
\qquad
\pi \circ u^{\xi}_{\rm d} \circ \varphi_{\rm d} : D^2 
\to \mathcal D.
$$
We take a (holomorphic) trivialization 
$\Pi:(\pi \circ u^{\xi}_{\rm d} \circ \varphi_{\rm d})^*\mathcal N_{\mathcal D}(X) \to \bbC$ of the pullback of the normal bundle $\mathcal N_{\mathcal D}(X)$ 
in a neighborhood of $z_{\rm d}$. Note that 
$u^{\xi}_{\rm d}(z_{\rm d}) \in \mathcal D$ is in a small neighborhood of
$u_{\rm d}^{0}(z_{\rm d})$. Therefore, $u^{\xi}_{\rm d} \circ \varphi_{\rm d}$ 
induces a holomorphic function
$$
{\Pi} \circ u^{\xi}_{\rm d} \circ \varphi_{\rm d} : D^2(o) \to \bbC
$$
for a small $o>0$. By assumption ${\Pi} \circ u^{\xi}_{\rm d} \circ \varphi_{\rm d}$ 
has a zero of order $p_1$ at $z_{\rm d}$.
We define $c^{\xi}_{\rm d} \in \bbC_*$ by
\begin{equation}\label{shiki655}
({\Pi} \circ u^{\xi}_{\rm d} \circ \varphi_{\rm d})(z) = c^{\xi}_{\rm d}z^{p_1} + f(z) z^{p_1+1}
\end{equation}
where $f(z)$ is holomorphic at $0$.

%Note when we define $U_{D}^{\xi}$ as in (\ref{form653}), we take a meromorphic section $\frak s$ of the pull back $u_{\rm D}^*\mathcal N_{\mathcal D}(X)$ of the normal bundle $\mathcal N_{\mathcal D}(X)$. 
Using the trivialization $\Pi$, we may regard the meromorphic section $\frak s^\xi\circ \varphi_{\rm d,D}$ as a meromorphic 
{\it function} which has a pole of order $2$ at $z_{\rm d}$. In particular, there is a constant $c^{\xi}_{\rm D,d} \in \bbC_*$ such that $\Pi\circ \frak s^\xi \circ \varphi_{\rm d,D}: D^2(o) \setminus \{0\} \to \bbC$ has the following form:
\begin{equation}\label{shiki656}
	(\Pi\circ \frak s^\xi \circ \varphi_{\rm D,d})(w) = c^{\xi}_{\rm D,d} w^{-{p_1}} + {g(w)}{w}^{-{p_1}+1},
\end{equation}
where $g$ is holomorphic at $0$. We now define:
\begin{equation}
\rho_{1,(0)}^{\xi}(\sigma_1,\sigma_2) =  \frac{c^{\xi}_{\rm d}\sigma_1^{p_1}}{c^{\xi}_{\rm D,d}}.
\end{equation}
Note that $\rho_{1,(0)}^{\xi}$ is independent of the choice of the trivialization of $\mathcal N_{\mathcal D}(X)$, because an alternative choice affects the numerator and the denominator of the right hand side by multiplying with the same number. The constant $\rho_{1,(0)}^{\xi}$ has the property that if $zw = \sigma_1$, then:
\begin{equation}\label{newnewform6555}
	(u^{\xi}_{\rm d} \circ \varphi_{\rm d})(z)\sim
	({\rm Dil}_{\rho_{1,(0)}^{\xi}}\circ U^{\xi}_{\rm D} \circ \varphi_{{\rm D},{\rm d}})(w) 
\end{equation}
where $\sim$ means the coincidence of the lowest order term.

We define $\rho_{2,(0)}^{\xi}$ in a similar way using the behavior of $u_{\rm s}$ and $u_{\rm D,s}$ in a neighborhood of $z_{\rm s}$. Namely, we replace \eqref{shiki655} and \eqref{shiki656} by:
\begin{equation}\label{shiki655rev}
({\Pi} \circ u^{\xi}_{\rm s} \circ \varphi_{\rm s})(z) = c^{\xi}_{\rm s}z^{p_2} + f(z) z^{p_2+1},
\end{equation}
\begin{equation}\label{shiki656rev}
(\Pi\circ \frak s^\xi \circ \varphi_{\rm D,s})(w) = c^{\xi}_{\rm D,s} w^{-{p_2}} + {g(w)}{w^{-{p_2}+1}},
\end{equation}
respectively and define:
\begin{equation}\label{newold664}
\rho_{2,(0)}^{\xi}(\sigma_1,\sigma_2) = \frac{c^{\xi}_{\rm s}\sigma_2^{p_2}}{c^{\xi}_{\rm D,s}}.
\end{equation} 
\par
Now we define a map
$$
u^{\prime, \xi,i}_{\sigma_1,\sigma_2,(0)} : \Sigma(\sigma_1,\sigma_2) \to X
$$
as follows. Roughly speaking, $u^{\prime, \xi,i}_{\sigma_1,\sigma_2,(0)} $ is obtained 
by gluing the three maps $u_{\rm d}^{\xi}$, $u_{\rm s}^{\xi}$, ${\rm Dil}_{\rho_{i,(0)}^{\xi}} \circ U_{\rm D}^{\xi}$, 
using bump functions $\chi_{i,\mathcal X}^{\leftarrow}$, $\chi_{\mathcal X}^{\rightarrow}$.
From now on, we write $\rho_{i,(0)}^{\xi}$ instead of $\rho_{i,(0)}^{\xi}(\sigma_1,\sigma_2)$ when the dependence on $\sigma_i$ is clear.
\begin{definition}
\begin{enumerate}
	\item If $z \in \Sigma_{\rm d}^-(\sigma_1,\sigma_2)$, then:
		\[
		  u^{\prime, \xi,1}_{\sigma_1,\sigma_2,(0)}(z) = u^{\prime, \xi,2}_{\sigma_1,\sigma_2,(0)}(z) = u_{\rm d}^{\xi}(z).
		\]
	\item If $z \in \Sigma_{\rm s}^-(\sigma_1,\sigma_2)$, then
		\[
		  u^{\prime, \xi,1}_{\sigma_1,\sigma_2,(0)}(z) = u^{\prime, \xi,2}_{\sigma_1,\sigma_2,(0)}(z) = u_{\rm s}^{\xi}(z).
		\]
	\item If $z \in \Sigma_{\rm D}^-(\sigma_1,\sigma_2)$, then:
		\[
		  u^{\prime, \xi,i}_{\sigma_1,\sigma_2,(0)}(z) = ({\rm Dil}_{\rho_{i,(0)}^{\xi}} \circ U_{\rm D}^{\xi})(z)
		\]
		for $i=1,2$.
	\item Suppose $z = (r_1,s_1) \in [-5T_1,5T_1]_{r_1} \times S^1_{s_1}$.
		%We put $p_1^{\xi,(0)} = u_{\rm d}^{\xi}(z_{\rm d}) = u_{\rm D}^{\xi}(z_{\rm d})$.
		We  define
		\[
		  u^{\prime, \xi,i}_{\sigma_1,\sigma_2,(0)}(z) \\
		= {\rm Exp_2}\left(u_{\rm d}^{\xi}(z),
		\chi_{\mathcal X}^{\rightarrow}(z)  {\rm E}(u_{\rm d}^{\xi}(z),({\rm Dil}_{\rho_{i,(0)}^{\xi}} \circ U_{\rm D}^{\xi})(z))
		\right).
		\]
		Here $ {\rm Exp_2}$ denotes the composition of $ {\rm Exp}$ and projection map from $(X\setminus \mathcal D)\times (X\setminus \mathcal D)$ to the second factor.
		\item
		Suppose $z = (r_2,s_2) \in [-5T_2,5T_2]_{r_2} \times S^1_{s_2}$.
%We put $p_2^{\xi,(0)} = u_{\rm s}^{\xi}(z_{\rm s}) = u_{\rm D}^{\xi}(z_{\rm s})$.
		We  define
		\[
		  u^{\prime, \xi,i}_{\sigma_1,\sigma_2,(0)}(z) \\= {\rm Exp_2}\left(u_{\rm s}^{\xi}(z),
		\chi_{\mathcal X}^{\rightarrow}(z)  {\rm E}(u_{\rm s}^{\xi}(z),({\rm Dil}_{\rho_{i,(0)}^{\xi}} \circ U_{\rm D}^{\xi})(z))\right).
		\]
\end{enumerate}
\end{definition}
\begin{remark}
%	\begin{enumerate}
%	\item 
	In part (4), if $r_1$ is close to $-5T_1$, then the right hand side is $u^{\xi}_{\rm d}$, and if $r_1$ is close to $5T_1$ then the right hand side is 
		${\rm Dil}_{\rho_{i,(0)}^{\xi}} \circ U_{\rm D}^{\xi}$. A similar property holds for the definition in part (5). In particular, our definition is well-defined.
%	\item The definition of $u^{\prime, \xi,i}_{\sigma_1,\sigma_2,(0)}(z)$ in parts (4) and (5) can be rewritten as follows:
%		\begin{enumerate}
%		\item[${\rm (4)}'$]
%			{\color{red}\AD{\color{red}It would hold if we used the exponential map.}
%			\[
%			u^{\prime, \xi,i}_{\sigma_1,\sigma_2,(0)}(z) \\= {\rm Exp}_2\left(({\rm Dil}_{\rho_{i,(0)}^{\xi}} \circ U_{\rm D}^{\xi})(z),
%			\chi_{1,\mathcal X}^{\leftarrow}(z) {\rm E}(({\rm Dil}_{\rho_{i,(0)}^{\xi}} \circ U_{\rm D}^{\xi})(z),u_{\rm d}^{\xi}(z))\right).
%			\]}
%		\item[${\rm (5)}'$]
%			{\color{red}\[
%			u^{\prime, \xi,i}_{\sigma_1,\sigma_2,(0)}(z) \\= {\rm Exp}_2\left(({\rm Dil}_{\rho_{i,(0)}^{\xi}} \circ U_{\rm D}^{\xi})(z),
%			\chi_{2,\mathcal X}^{\leftarrow}(z) {\rm E}(({\rm Dil}_{\rho_{i,(0)}^{\xi}} \circ U_{\rm D}^{\xi})(z),u_{\rm s}^{\xi}(z))\right).
%			\]}
%	\end{enumerate}
%\end{enumerate}
\end{remark}
\noindent{\bf (Step 0-3) (Error estimate)}
\footnote{The enumeration of the steps of this paper is the same as those in 
\cite[Section A1.4]{fooobook2} and \cite{foooexp}.}

The next lemma provides an estimate of  $\overline{\partial}u^{\prime, \xi,i}_{\sigma_1,\sigma_2,(0)}$ modulo the obstruction space
\[
  (E_{\rm d} \oplus E_{\rm s} \oplus E_{\rm D})(u^{\prime, \xi,i}_{\sigma_1,\sigma_2,(0)}).
\]
In the case that $\rho_{1,(0)}^{\xi} \ne \rho_{2,(0)}^{\xi}$, we need to restrict the domain in the following way 
to obtain an appropriate estimate.
We put
\begin{equation}
\Sigma(\sigma_1,\sigma_2)_{i}^-
= 
\begin{cases}
\Sigma(\sigma_1,\sigma_2)
\setminus ([-5T_2,5T_2]_{r_2} \times S^1_{s_2})
&\text{if $i=1$}, \\
\Sigma(\sigma_1,\sigma_2)
\setminus ([-5T_1,5T_1]_{r_1} \times S^1_{s_1})
&\text{if $i=2$}.
\end{cases}
\end{equation}
We consider the $L^2_{m,\delta}$ norm of the restriction of maps to $\Sigma(\sigma_1,\sigma_2)_{i}^-$ 
and denote it by $L^{2,i,-}_{m,\delta}$.

\begin{lemma}\label{lem623}
	There exist constants $\delta_1$, $C_m$ (for any integer $m$) and vectors 
	$\frak e_{{\rm d},(0)}^{\xi} \in E_{\rm d}(u^{\prime, \xi,i}_{\sigma_1,\sigma_2,(0)})$, $\frak e_{{\rm s},(0)}^{\xi} \in E_{\rm s}(u^{\prime, \xi,i}_{\sigma_1,\sigma_2,(0)})$,
	$\frak e_{{\rm D},(0)}^{\xi,i} \in E_{\rm D}(u^{\prime, \xi,i}_{\sigma_1,\sigma_2,(0)})$ such that $\delta_1$, $C_m$ are independent of $\sigma_1$, $\sigma_2$, $\xi$, and we have the following inequalities:
	\begin{enumerate}
		\item \[\Vert \overline{\partial}u^{\prime, \xi,1}_{\sigma_1,\sigma_2,(0)}  - \frak e_{{\rm d},(0)}^{\xi} - \frak e_{{\rm s},(0)}^{\xi}
			- \frak e_{{\rm D},(0)}^{\xi} \Vert_{L^{2,1,-}_{m,\delta}}\le C_m e^{-\delta_1T_1}.\]
		\item \[\Vert \overline{\partial}u^{\prime, \xi,2}_{\sigma_1,\sigma_2,(0)}  - \frak e_{{\rm d},(0)}^{\xi} - \frak e_{{\rm s},(0)}^{\xi}
			- \frak e_{{\rm D},(0)}^{\xi} \Vert_{L^{2,2,-}_{m,\delta}}\le C_m e^{-\delta_1T_2}.\]
	\end{enumerate}
\end{lemma}
We can be more specific about the value of the constant $\delta_1$ as in \eqref{form62} and \eqref{form62rev}. However, the actual choices do not matter for the details of our construction. So we do not give an exact value for this constant.
\begin{proof}
        We define:
        \begin{equation}\label{form6662}
        \aligned
        \frak e_{{\rm d},(0)}^{\xi} &:= \overline{\partial} u^{\xi}_{\rm d} \in E_{\rm d}(u^{\prime, \xi,i}_{\sigma_1,\sigma_2,(0)}), 
        \\
        \frak e_{{\rm s},(0)}^{\xi} &:= \overline{\partial} u^{\xi}_{\rm s} \in E_{\rm s}(u^{\prime, \xi,i}_{\sigma_1,\sigma_2,(0)}),
        \\
        \frak e_{{\rm D},(0)}^{\xi,i} &:= \overline{\partial} ({\rm Dil}_{\rho_{i,(0)}^{\xi}} \circ U^{\xi}_{\rm D}) \in E_{\rm D}(u^{\prime, \xi,i}_{\sigma_1,\sigma_2,(0)}).
        \endaligned
        \end{equation}
        Then by construction the support of 
        $\overline{\partial}u^{\prime, \xi,1}_{\sigma_1,\sigma_2,(0)}  
        - \frak e_{{\rm d},(0)}^{\xi} - \frak e_{{\rm s},(0)}^{\xi}
        - \frak e_{{\rm D},(0)}^{\xi}$
        is contained in 
        $([-5T_1,5T_1]_{r_1} \times S^1_{s_1}) \cup ([-5T_2,5T_2]_{r_2} \times S^1_{s_2})$.
        Therefore, it suffices to estimate 
        $\overline{\partial}u^{\prime, \xi,i}_{\sigma_1,\sigma_2,(0)}$
        on $[-5T_i,5T_i]_{r_i} \times S^i_{s_i}$. Below we discuss the case $i=1$. The other case is similar. 
        
        Let $z = \varphi_{\rm d}(r'_1,s'_1)$ be the coordinate on $\Sigma_{\rm d}$ used to denote points in a neighborhood of $z_{\rm d}$ and 
        $w = \varphi_{\rm D,d}(r''_1,s''_1)$ be the coordinate on $\Sigma_{\rm D}$ used to denote points in a neighborhood of $z_{\rm d}$.
        In order to obtain $\Sigma(\sigma_1,\sigma_2)$, the equation:
        $$
        zw = \sigma_1
        $$ 
        is used to glue $\Sigma_{\rm d}$ and $\Sigma_{\rm D}$.
        Note that the supports of the derivatives of the bump functions $\chi_{1,\mathcal X}^{\leftarrow}$, $\chi_{\mathcal X}^{\rightarrow}$
        are in $\mathcal X_{1,T_1}$. (Here we look at the restriction of the function 
        $\chi_{\mathcal X}^{\rightarrow}$ to $\Sigma(\sigma_1,\sigma_2)_{i}^-$. 
        Otherwise, part of the support of the derivate of this function is contained in $\mathcal X_{2,T_2}$.)
        Therefore, the support of $\overline{\partial}u^{\prime, \xi,i}_{\sigma_1,\sigma_2,(0)}$ is contained in the same subspace.
	
	Firstly we wish to show that the maps $f_1:=u^{\xi}_{\rm d}$ and $f_2:={\rm Dil}_{\rho_{1,(0)}^{\xi}} \circ U^{\xi}_{\rm D}$, as maps from
	from $\mathcal X_{1,T_1}$ to $\frak U\setminus \mathcal D\subset X\setminus \mathcal D$, 
	are close to each other in the $C^m$ metric. In fact, analogues of the inequalities in 
	\eqref{form62} and \eqref{form62rev} show that there are constants $C_m'$ and $\delta_0$ independent of $\sigma_1$, $\sigma_2$ and $\xi$ such that:
	\begin{equation} \label{C1-C2-dis}
	  d_{C^m}(f_1,f_2)\leq C_m'e^{- 5 \delta_0 T_1}
	\end{equation}
	where $d_{C^m}$ is computed with respect to the cylindrical metric $g$. To be a bit more detailed, this inequality holds because the leading terms of $f_1$ and $f_2$ agree with each other, and 
	$f_1$ and $f_2$ are both holomorphic. 
	
	Let $h_1,\,h_2:\mathcal X_{1,T_1} \to \frak U$ be maps such that their $C^0$ distance is less than or equal to a constant $\kappa$. If $\kappa$ is small enough,
	then the following map is well defined:
	\[
	  F(h_1,h_2)= {\rm Exp_2}\left(h_1,\chi_{\mathcal X}^{\rightarrow}\cdot {\rm E}(h_1,h_2)\right).
	\]
	Clearly there is a constant $K$ such that:
	\[
	  \Vert\overline \partial F(h_1,h_2)-\overline \partial F(h_1,h_2')\Vert_{L^2_m}\leq K \cdot d_{C^{m+1}}(h_2,h_2')
	\]
	Since $F(f_1,f_1)=f_1$, the above inequality together with \eqref{C1-C2-dis} implies that  there is a constant $C_m$ such that:
	\[
	  \Vert\overline{\partial}u^{\prime, \xi,1}_{\sigma_1,\sigma_2,(0)}\Vert_{L^2_m(\mathcal X_{1,T_1})} \leq C_me^{- 5 \delta_0 T_1}
	\]
	Therefore, if we pick $\delta$ and $\delta_1$ such that $\delta+\delta_1<5\delta_0$, then the desired inequality holds.
\end{proof}

\subsection{Why Inconsistent Solutions?}
\label{subsub:linsol1}

We already hinted at the necessity of inconsistent solutions at the end of Section \ref{sub:Obst}. In this section we elaborate on this point with an eye toward modifying the approximate solution of the previous section to a solution. We firstly sketch our approach for this modification which is based on {\it  Newton's iteration method}. Next, we explain the main point where the proof in the case of the RGW compactification diverges from the case of the stable map compactification. The discussion of this subsection is informal, and the actual proof will be carried out in the next two subsections.

Suppose $u^{\prime, \xi,1}_{\sigma_1,\sigma_2,(0)}$ and $u^{\prime, \xi,2}_{\sigma_1,\sigma_2,(0)}$ are the approximate solutions of the previous subsection associated to the element $\xi = (u_{\rm d}^{\xi},u_{\rm D}^{\xi},u_{\rm s}^{\xi})$ of \eqref{fib-prod-str}. We assume that $\sigma_1$ and $\sigma_2$ are chosen such that $\rho_{1,(0)}^{\xi}(\sigma_1,\sigma_2) = \rho_{2,(0)}^{\xi}(\sigma_1,\sigma_2)$. In particular, $u^{\prime, \xi,1}_{\sigma_1,\sigma_2,(0)}=u^{\prime, \xi,2}_{\sigma_1,\sigma_2,(0)}$ and we denote these maps by $u'$. Lemma \ref{lem623} gives the following estimate:
$$
\Vert\overline\partial u'\Vert_{L^2_{m,\delta}(\Sigma(\sigma_1,\sigma_2))/E(u')} \le C e^{-c \delta_1 \min\{T_1,T_2\}}.
$$
Here $E(u') = E_{\rm d}(u') \oplus E_{\rm s}(u') \oplus E_{\rm D}(u')$, and the norm on the left hand side is the induced norm on the quotient space. The next step would be to find: 
\[
  V \in W^2_{m+1,\delta}((\Sigma(\sigma_1,\sigma_2),\partial \Sigma(\sigma_1,\sigma_2));
  (u^{\prime *}TX,u'\vert_{\partial}^*TL)).
\]
which satisfies the equation:
\begin{equation}\label{form670new}
(D_{u'}\overline{\partial}) V + \overline\partial u' \equiv 0
\mod E(u').
\end{equation}
and 
$$
\Vert V\Vert_{W_{m+1,\delta}^2} \le C  \Vert\overline\partial u'\Vert_{L^2_m(\Sigma(\sigma_1,\sigma_2))/E(u')}.
$$
Then we could define our first modified approximate solution as follows:
$$
u''(z)={\rm Exp}(u'(z),V(z))
$$
This modified solution would satisfy the following inequality:
$$
\Vert\overline\partial u''\Vert_{L^2_{m,\delta}(\Sigma(\sigma_1,\sigma_2))/E(u'')} \le
\mu \Vert\overline\partial u'\Vert_{L^2_{m,\delta}(\Sigma(\sigma_1,\sigma_2))/E(u')}
$$
for a fixed $0<\mu<1$ if $\sigma_1, \sigma_2$ are sufficiently small 
(or equivalently, $T_1$, $T_2$ are sufficiently large).

We could then continue to obtain $u^{(i)}$ such that
\[
  \Vert\overline\partial u^{(i)}\Vert_{L^2_m(\Sigma(\sigma_1,\sigma_2))/E(u^{(i)}{\vert})} \le
  \mu^{i} \Vert\overline\partial u'\Vert_{L^2_m(\Sigma(\sigma_1,\sigma_2))/E(u')}
\]
and for fixed constants $C$ and $c$, the $W^2_{m+1,\delta}$-distance between $u^{(i)}$ and $u^{(i+1)}$ is bounded by $C\mu^{i} e^{-c \delta_1 \min\{T_1,T_2\}}$. Then $u^{(i)}$ would be convergent to a map $u$, and it would be the required solution of the equation:
\begin{equation}\label{eq671}
	\overline \partial u \equiv 0 \mod E(u).
\end{equation}
This is the standard Newton's iteration method to solve a nonlinear equation using successive solutions to the linearized equation. However, the RGW compactification is singular at the starting point of our construction, the element $\zeta$ of \eqref{fib-prod-str}.
So we cannot expect the above Newton's iteration method works without some adjustments. We fix our approach by thickening the solution set of \eqref{eq671} to the set of inconsistent solutions. 

The main reason that we will work with this larger moduli space lies in the step that we find the solution $V$ of the equation \eqref{form670new}. To solve this equation, we need to find a right inverse to the following operator modulo $E(u')$:
\[\aligned D_{u'}\overline{\partial} : &W^2_{m+1,\delta}((\Sigma(\sigma_1,\sigma_2),\partial \Sigma(\sigma_1,\sigma_2));
  (u^{\prime *}TX,u'\vert_{\partial}^*TL)) \\
  &\to L^2_{m,\delta}(\Sigma(\sigma_1,\sigma_2);u^{\prime *}TX\otimes \Lambda^{0,1})/E(u').
  \endaligned
\]
The standard approach to construct this right inverse is to glue the right inverses of the linearized operators $D_{u_{\rm d}}\overline{\partial}$, $D_{u_{\rm s}}\overline{\partial}$ and 
$D_{U_{\rm D}}\overline{\partial}$. The linearized operator $D_{u'}\overline{\partial}$ over the cylinder $[-5T_i,5T_i]_{r_i} \times S^1_{s_i}$ is modeled by an operator of the form 
$$
\frac{\partial}{\partial r_i} + P_{r_i}.
$$
The relevant operators $P_{r_i}$ in our setup have non-trivial kernel and our gluing construction is of ``Morse-Bott'' type. As it was clarified by Mrowka's Mayer-Vietoris principle \cite{Mr}, to have a well-behaved gluing problem we need to assume certain `mapping transversality conditions'.

To be more specific, the zero eigenspace of the operator $P_{r_i}$ can be identified with: 
\begin{equation}\label{form672}
	\aligned
	(\bbR\oplus \bbR)\oplus &T_{u_{\rm d}(z_{\rm d})}\mathcal D    \qquad \text{if $i=1$}, \\
	(\bbR\oplus \bbR)\oplus  &T_{u_{\rm s}(z_{\rm s})}\mathcal D  \qquad \text{if $i=2$}.
	\endaligned
\end{equation}
Here $\bbR \oplus \bbR$ is the tangent space to $\bbC_*$. The mapping transversality condition we introduced in Definition \ref{defn6969} concerns the summand  $T_{u_{\rm d}(z_{\rm d})}\mathcal D$. Therefore, it is {\it not} sufficient for the Mayer-Vietoris principle in our setup. However, working with inconsistent solutions allows us to enlarge the tangent spaces and obtain the required transversality condition. A byproduct of using inconsistent solutions is that we might end up with inconsistent solutions throughout Newton's iterations, even if the starting approximate solution has $\rho_1 = \rho_2$.

%{\color{red} We remark that to work out the Mayer-Vietoris principle in our case, the degeneration for $i=1$ and $i=2$ in (\ref{form672}) should be taken care of separately. It means that when we go from one step of Newton's iteration to the next, the numbers $\rho_1, \rho_2$ can change independently. Therefore, even if the starting approximate solution has $\rho_1 = \rho_2$, we might end up with inconsistent solutions throughout Newton's iterations. So we need to work with the whole family of inconsistent solutions to insure that Newton's iteration works.}

\subsection{Inconsistent Maps and Linearized Equations}
\label{subsub:linsol2}
In Section \ref{sub:statement}, the notion of holomorphic maps was extended to inconsistent solutions of the Cauchy-Riemann equation. It is also convenient to define generalizations of maps from $\Sigma(\sigma_1,\sigma_2)$ to $X\backslash \mathcal D$:
\begin{definition}\label{defn625625}
	A 7-tuple $\frak u'=(u'_{\rm d},u'_{\rm s},U'_{\rm D},\sigma_1,\sigma_2,\rho_1,\rho_2)$ is 
	an {\it inconsistent map} if it satisfies only parts (1) and (4) of Definition \ref{defn615inconsis}. 
	In other words, we do not require that the 7-tuple satisfies the Cauchy-Riemann equation in \eqref{eq630-def}
	and the constraint in \eqref{eq631-p}.
	We define equivalence of inconsistent maps in the same way as in Definition \ref{defn615inconsis}.
\end{definition}

An example of inconsistent maps can be constructed using the maps:
\[
  \hspace{3cm}u^{\prime, \xi,i}_{\sigma_1,\sigma_2,(0)} : \Sigma(\sigma_1,\sigma_2) \to X\hspace{1cm} i=1,\,2
\]
of Subsection \ref{subsub:preglue} which are associated to an element $\xi= (u_{\rm d}^{\xi},u_{\rm D}^{\xi},u_{\rm s}^{\xi})$ of \eqref{fib-prod-str}. We use these two maps to define:
$$
    \aligned
    u^{\xi,\prime}_{{\rm d},\sigma_1,\sigma_2,(0)}
    &:=
    u^{\prime,\xi,1}_{\sigma_1,\sigma_2,(0)}\vert_{\Sigma_{\rm d}^+(\sigma_1,\sigma_2)}\\
    u^{\xi,\prime}_{{\rm s},\sigma_1,\sigma_2,(0)}
    &:=
    u^{\prime,\xi,2}_{\sigma_1,\sigma_2,(0)}\vert_{\Sigma_{\rm s}^+(\sigma_1,\sigma_2)}
    \\
    U^{\xi,\prime}_{{\rm D},\sigma_1,\sigma_2,(0)}
    &:=
    \begin{cases}
            {\rm Dil}_{1/\rho_{1,(0)}^{\xi}} \circ u^{\prime,\xi,1}_{\sigma_1,\sigma_2,(0)}
            &\text{on 
            $\Sigma_{\rm d}^+(\sigma_1,\sigma_2) \cap 
            \Sigma_{\rm D}^+(\sigma_1,\sigma_2)$} \\
            {\rm Dil}_{1/\rho_{2,(0)}^{\xi}} \circ  u^{\prime,\xi,2}_{\sigma_1,\sigma_2,(0)}
            &\text{on 
            $\Sigma_{\rm s}^+(\sigma_1,\sigma_2) \cap 
            \Sigma_{\rm D}^+(\sigma_1,\sigma_2)$} \\
            {\rm Dil}_{1/\rho_{1,(0)}^{\xi}} \circ u^{\prime,\xi,1}_{\sigma_1,\sigma_2,(0)}
            &\text{on $\Sigma_{\rm D}^-(\sigma_1,\sigma_2)$}
    \end{cases}
    \endaligned
$$
Note that ${\rm Dil}_{1/\rho_{1,(0)}^{\xi}} \circ u^{\prime,\xi,1}_{\sigma_1,\sigma_2,(0)} ={\rm Dil}_{1/\rho_{2,(0)}^{\xi}} \circ u^{\prime,\xi,2}_{\sigma_1,\sigma_2,(0)}$ on $\Sigma_{\rm D}^-(\sigma_1,\sigma_2)$. The following lemma is obvious from the construction:
\begin{lemma}\label{lem62611}
	The 7-tuple: 
	\[\frak u^{\xi,\prime}_{\sigma_1,\sigma_2,(0)}:=
	(u^{\xi,\prime}_{{\rm d},\sigma_1,\sigma_2,(0)},u^{\xi,\prime}_{{\rm s},\sigma_1,\sigma_2,(0)},U^{\xi,\prime}_{{\rm D},\sigma_1,\sigma_2,(0)},\sigma_1,\sigma_2,
	\rho_{1,(0)}^{\xi},\rho_{2,(0)}^{\xi})\] 
	is an inconsistent map
\end{lemma}
The inconsistent map $\frak u^{\xi,\prime}_{\sigma_1,\sigma_2,(0)}$ of Lemma \ref{lem62611} is the approximate solution at the 0-th step. In order to obtain an actual inconsistent solution, we keep modifying this approximate solution into better approximate solutions. To be more detailed, we firstly use $\frak u^{\xi,\prime}_{\sigma_1,\sigma_2,(0)}$ and our bump functions to obtain a triple 
\[\frak u^{\xi,\prime\prime}_{\sigma_1,\sigma_2,(0)}=(u^{\xi,\prime\prime}_{{\rm d},\sigma_1,\sigma_2,(0)},u^{\xi,\prime\prime}_{{\rm s},\sigma_1,\sigma_2,(0)},U^{\xi,\prime\prime}_{{\rm D},\sigma_1,\sigma_2,(0)})\] 
such that
\begin{equation}\label{uprime20}
	u^{\xi,\prime\prime}_{{\rm d},\sigma_1,\sigma_2,(0)}:\Sigma_{\rm d} \setminus \{z_{\rm d}\} \to X \setminus \mathcal D,\hspace{1cm}
	u^{\xi,\prime\prime}_{{\rm s},\sigma_1,\sigma_2,(0)}:\Sigma_{\rm s} \setminus \{z_{\rm s}\} \to X \setminus \mathcal D,
\end{equation}
\begin{equation}\label{uprime21}
	U^{\xi,\prime\prime}_{{\rm D},\sigma_1,\sigma_2,(0)}:\Sigma_{\rm D} \setminus \{z_{\rm d},z_{\rm s}\} \to \mathcal N_{\mathcal D}(X).
\end{equation}
are close to $(u_{\rm d}^{\xi},u_{\rm s}^{\xi},U_{\rm D}^{\xi})$. In fact, the smaller the values of $\sigma_1$ and $\sigma_2$ are, the closer $\frak u^{\xi,\prime\prime}_{\sigma_1,\sigma_2,(0)}$ is to $\xi$. Thus we can exploit this to conclude that an appropriate version of the Cauchy-Riemann operator associated to $\frak u^{\xi,\prime\prime}_{\sigma_1,\sigma_2,(0)}$  has a right inverse. (See Lemma \ref{lem63631}.) This allows us to find a modified inconsistent map $\frak u^{\xi,\prime}_{\sigma_1,\sigma_2,(1)}$. We repeat the same process to construct a sequence of inconsistent maps $\{\frak u^{\xi,\prime}_{\sigma_1,\sigma_2,(i)}\}$ which are approximate solutions and they converge to an inconsistent solution. This sequence of modified inconsistent solution is constructed using Newton's iteration, and it also has some components of the  ``alternating method''.\footnote{See, for example, \cite[Sublemma 8.6]{fconnsum}.
Application of alternating method for gluing analysis of this kinds 
is initiated by Donaldson \cite{Don}. He applied alternating method directly to a nonlinear equation.} In this method we solve the equation in various pieces and glue them together.

%{\color{red} Our way to show Mayer-Vietoris principle is by using alternating method. It is a way to solve the equation in various pieces 
%and glue them.  Actual alternating method repeats this process infinitely many times 
%to obtain a solution of the  equation.
%For our purpose, we study linearized equation as one step of 
%Newton's iteration.\footnote{See, for example, \cite[Sublemma 8.6]{fconnsum}.
%Application of alternating method for gluing analysis of this kinds 
%is initiated by Donaldson \cite{Don}. (Donaldson applied alternating method directly 
%to a nonlinear equation.)} So it is enough to use 
%approximate solution of the linearized equation in place of 
%its exact solution.
%Therefore we solve linearized equation in various pieces,
%glue them, and use it to define approximate solution 
%(of the nonlinear equation) of the next step.
%In other words we use only the first step of the alternating method 
%for linearized equation.}

In order to carry out the above plan, we need to introduce norms to quantify the distance between two inconsistent maps and to measure how good an approximate solution is. Such norms are given in the following Definition \ref{tang-incon-map}:

        Let 
        $\frak u' =
        (u'_{{\rm d}},u'_{{\rm s}},U'_{{\rm D}},\sigma_1,\sigma_2,\rho_1,
        \rho_2)
        $
        be an inconsistent map.
        We consider a triple ${\rm V} = (V_{\rm d},V_{\rm s},V_{\rm D})$ with 
        \begin{align*}
        &V_{\rm d}\in L^2_{m+1}(\Sigma_{\rm d}^+(\sigma_1,\sigma_2),(u'_{\rm d})^*TX), \\
        &V_{\rm s}\in L^2_{m+1}(\Sigma_{\rm s}^+(\sigma_1,\sigma_2),(u'_{\rm s})^*TX), \\
        &V_{\rm D}\in L^2_{m+1}(\Sigma_{\rm D}^+(\sigma_1,\sigma_2),(U'_{\rm D})^*(T\mathcal N_{\mathcal D}(X))).
        \end{align*}
        We assume $V_{\rm d}(\frak z) \in T_{u'_{\rm d}(\frak z)}L$ if $\frak z \in 
        \partial \Sigma_{\rm d}^+(\sigma_1,\sigma_2)$.
        Moreover, we assume that there exist $(a_{\rm d},b_{\rm d}), (a_{\rm s},b_{\rm s})\in \bbR\oplus \bbR$ such that
        $$
        \aligned
        V_{\rm d} - V_{\rm D} &= (a_{\rm d},b_{\rm d}) \qquad \text{on $[-5T_1,5T_1]_{r_1} \times 
        S^1_{s_1}$} \\
        V_{\rm s} - V_{\rm D} &= (a_{\rm s},b_{\rm s}) \qquad \text{on $[-5T_2,5T_2]_{r_2} \times 
        S^2_{s_2}$}
        \endaligned
        $$
        Here we regard $\bbR\oplus \bbR$ as the vector field on the neighborhood $\frak U$ of $\mathcal D$ given by the $\bbC_*$ action.
\begin{definition}\label{tang-incon-map}
        Define $v_i = V_{\rm D}((0,0)_i)$ where $(0,0)_i$ is the same as in 
        \eqref{defnvivi}. We then define $\hat v_i$ in the same way as in \eqref{form651651newnew}. We now define
        $\Vert {\rm V}\Vert^2_{W^{2,\sim}_{m,\delta}}$ as follows:
        $$
        \aligned
        &\Vert V_{{\rm d}} \Vert^2_{L^2_m(\Sigma_{\rm d}^-(\sigma_1,\sigma_2))}
        +
        \Vert V_{{\rm s}}\Vert^2_{L^2_m(\Sigma_{\rm s}^-(\sigma_1,\sigma_2))}
        +
        \Vert V_{{\rm D}}\Vert^2_{L^2_m(\Sigma_{\rm D}^-(\sigma_1,\sigma_2))}
        \\
        &+
        \sum_{j=0}^m
        \int_{[-5T_1,5T_1]_{r_1} \times S^1_{s_1}} 
        e_{\delta}^{\sigma_1,\sigma_2} 
        \left\vert \nabla^j (V_{\rm D} - \hat v_1)\right\vert^2 dr_1ds_1
        \\
        &+\sum_{j=0}^m
        \int_{[-5T_2,5T_2]_{r_2} \times S^1_{s_2}} 
        e_{\delta}^{\sigma_1,\sigma_2} 
        \left\vert  \nabla^j (V_{\rm D} - \hat v_2)\right\vert^2 dr_2ds_2
        \\
        &+\vert v_1\vert^2+\vert v_2\vert^2.
        \endaligned
        $$
        Let ${\rm V} = (V_{\rm d},V_{\rm s},V_{\rm D})$, 
        ${\rm V}' = (V'_{\rm d},V'_{\rm s},V'_{\rm D})$ be as above.
        We say they are equivalent if $V_{\rm d} = V'_{\rm d}$, 
        $V_{\rm s} = V'_{\rm s}$ and $V_{\rm D} - V'_{\rm D} \in \bbR\oplus \bbR$, where $\bbR\oplus \bbR$ is the set of vector fields generated by the $\bbC_*$ action.
        We put
        $$
        \Vert {\rm V}\Vert^2_{W^{2}_{m,\delta}}
        = 
        \inf \{\Vert {\rm V}'\Vert^2_{W^{2,\sim}_{m,\delta}} 
        \mid \text{${\rm V}'$ is equivalent to ${\rm V}$}. \}
        $$
\end{definition}

Note in the above definition we take infimum over representatives. The ambiguity of 
the choice of representatives is determined by the shift of several components
by real numbers. If those shifts are very big the norm $\Vert V\Vert$ becomes 
far from minimum. Therefore, the domain we need to take infimum is 
bounded. It implies that even though we take infimum this norm is 
positive, as long as $V\neq 0$.

\begin{definition}\label{defbn7-8}
	For $j=1,2$, let $\frak u'_{(j)}$ be an inconsistent map. 
	We assume that there is a representative
	$(u'_{{\rm d},(j)},u'_{{\rm s},(j)},U'_{{\rm D},(j)},\sigma_1,\sigma_2,\rho_{1,(j)},\rho_{2,(j)})$ for 
	$\frak u'_{(j)}$ such that the triple $(u'_{{\rm d},(1)},u'_{{\rm s},(1)},U'_{{\rm D},(1)})$ 
	is $C^0$-close to $(u'_{{\rm d},(2)},u'_{{\rm s},(2)},U'_{{\rm D},(2)})$. 
	Define  $V_{\rm d}$, $V_{\rm s}$, $V_{\rm D}$ by the following properties:
	\begin{align}\label{vector-exp}
		\phantom{i + j + k}
  		&\begin{aligned}
			&{\rm Exp}(u'_{{\rm d},(1)},V_{\rm d}) = u'_{{\rm d},(2)}, \\
			  &{\rm Exp}(u'_{{\rm s},(1)},V_{\rm s}) = u'_{{\rm s},(2)}, \\
			  &{\rm Exp}(U'_{{\rm D},(1)},V_{\rm D}) = U'_{{\rm D},(2)}. 
		\end{aligned}
	\end{align}
	Let ${\rm V} = (V_{\rm d},V_{\rm s},V_{\rm D})$, and define:
	\[
	  d_{W^{2}_{m,\delta}}(\frak u'_{(1)},\frak u'_{(2)})= \inf\{\Vert {\rm V}\Vert_{W^{2}_{m,\delta}}\}.
	\]
	where the infimum is taken over all representatives for $\frak u'_{(1)}$ and $\frak u'_{(2)}$
	which are close enough to each other in the $C^0$ metric such that 
	the vectors in \eqref{vector-exp} exist.
	Therefore, $d_{W^{2}_{m,\delta}}$ is a well defined distance between two equivalence 
	classes of inconsistent maps.
	Note that $d_{W^{2}_{m,\delta}}(\frak u'_{(1)},\frak u'_{(2)})$
	is strictly positive if $\frak u'_{(1)} \ne \frak u'_{(2)}$.
	We can prove this fact in the same way as in Definition \ref{tang-incon-map}.
\end{definition}

For any inconsistent map $\frak u'=(u'_{\rm d},u'_{\rm s},U'_{\rm D},\sigma_1,\sigma_2,\rho_1,\rho_2)$, we may use a similar parallel transport construction as in Definition \ref{defn615inconsis} to define obstruction spaces for $\frak u'$. That is to say, we define maps $\mathcal {PAL}$ as in \eqref{PAL1} and \eqref{PAL2}. Then the images of $E_{\rm d}$ and $E_{\rm s}$ with respect to these maps give rise to the obstruction spaces $E_{\rm d}(u'_{\rm d})$ and $E_{\rm s}(u'_{\rm s})$. Similarly, we define $E_{\rm D}(U'_{\rm D})$ by replacing $u'_{\rm D}$ with $\pi\circ U'_{\rm D}$ in \eqref{PAL-D} and using the decomposition \eqref{decom-tan-bdle}. We will write $E(\frak u')$ for the direct sum of the vector spaces $E_{\rm d}(u'_{\rm d})$ and $E_{\rm s}(u'_{\rm s})$ and $E_{\rm D}(U'_{\rm D})$. Note that $E_{\rm d}(u'_{\rm d})$, $E_{\rm s}(u'_{\rm s})$ and $E_{\rm D}(U'_{\rm D})$ are identified with $E_{\rm d}$ and $E_{\rm s}$ and $E_{\rm D}$. Therefore, we drop $u'_{\rm d}$, $u'_{\rm s}$ and $U'_{\rm D}$ from our notation for these obstruction spaces if it does not make any confusion.

\begin{definition}
	Let $\frak u' =(u'_{{\rm d}},u'_{{\rm s}},U'_{{\rm D}},\sigma_1,\sigma_2,\rho_1,\rho_2)$
	be an inconsistent map and $\frak e = (\frak e_{\rm d},\frak e_{\rm s},\frak e_{\rm D})\in E_{\rm d}\oplus E_{\rm s} \oplus E_{\rm D}$. 
	Then we define $\Vert \overline \partial \frak u' - \frak e\Vert^2_{L^2_{m,\delta}}$ to be the following sum:
	\begin{align*}
	&\Vert \overline\partial u'_{{\rm d}} - \frak e_{\rm d}\Vert^2_{L^2_m(\Sigma_{\rm d}^-(\sigma_1,\sigma_2))}
	+\Vert \overline\partial u'_{{\rm s}} - \frak e_{\rm s}\Vert^2_{L^2_m(\Sigma_{\rm s}^-(\sigma_1,\sigma_2))} \\
	&+\Vert \overline\partial U'_{{\rm D}} - \frak e_{\rm D}\Vert^2_{L^2_m(\Sigma_{\rm D}^-(\sigma_1,\sigma_2))}\\
	&+\sum_{j=0}^m\int_{[-5T_1,5T_1]_{r_1} \times S^1_{s_1}} e_{\delta}^{\sigma_1,\sigma_2} \left\vert \nabla^j \overline\partial U'_{{\rm D}}\right\vert^2 dr_1ds_1\\
	&+\sum_{j=0}^m\int_{[-5T_2,5T_2]_{r_2} \times S^1_{s_2}} e_{\delta}^{\sigma_1,\sigma_2} \left\vert \nabla^j \overline\partial U'_{{\rm D}}\right\vert^2 dr_2ds_2.
	\end{align*}
\end{definition}
We remark that the first 3 terms in the above definition are the Sobolev norms of $\overline \partial \frak u' - \frak e$ in the {\it thick part}. The fourth and the fifth terms are its weighted Sobolev norms in the neck region. %By assumption $u'_{{\rm d}}$ or $u'_{{\rm s}}$ is obtained from $U'_{{\rm D}}$ by the (partial) $\bbC_*$ action on the neck region. 
Because of our choice of cylindrical metrics on $\frak U$, the partial $\bbC_*$-action induces isometries  and preserves the almost complex structure. Therefore, the above sum is well-defined and only depends on the equivalence class of $\frak u'$.

The process of the modifications of our approximate solutions are performed by finding solutions to the linearization of the modified Cauchy-Riemann equations in \eqref{eq630-def}. Since our equation has terms induced by the obstruction bundle, the linearized operator has an extra term in addition to $D_{u'}\overline{\partial}$. The equations in \eqref{eq630-def} can be regarded as an equation for an inconsistent map $\frak u' =(u'_{{\rm d}},u'_{{\rm s}},U'_{{\rm D}},\sigma_1,\sigma_2,\rho_1,\rho_2)$ and $(\frak e_{\rm d}, \frak e_{\rm s}, \frak e_{\rm D})\in E_{\rm d}\oplus E_{\rm s} \oplus E_{\rm D}$:
\begin{equation}\label{eq630-rep}
	\overline \partial u_{\rm d}'- \frak e_{\rm d}=0, \quad
	\overline \partial u_{\rm s}' -\frak e_{\rm s}=0, \quad
	\overline \partial U_{\rm D}' -\frak e_{\rm D}=0.
\end{equation}
Suppose ${\rm V} = (V_{\rm d},V_{\rm s},V_{\rm D})$ is an element of the Hilbert space introduced in Definition \ref{tang-incon-map}. For each real number $\tau$ with ${\vert}\tau{\vert}<1$,  let $\frak u^\tau$ be given by the triple $(u^\tau_{{\rm d}},u^\tau_{{\rm s}},U^\tau_{{\rm D}})$ defined as:
	\begin{align}\label{}
		\phantom{i + j + k}
  		&\begin{aligned}
		  &u^\tau_{{\rm d}}:={\rm Exp}(u'_{{\rm d}},\tau V_{\rm d}), \hspace{.6cm}
		  u^\tau_{{\rm s}}:={\rm Exp}(u'_{{\rm s}},\tau V_{\rm s}), \hspace{.6cm} \\
		  &U^\tau_{{\rm D}}:={\rm Exp}(U'_{{\rm D}},\tau V_{\rm D}). 
		\end{aligned}
	\end{align}
We use parallel transport along minimal geodesics to obtain:
\[
  \mathcal{PAL}^\tau_{u'_{\rm d}} : L^2_{m,\delta}(\Sigma_{\rm d}^+(\sigma_1,\sigma_2);
  u_{\rm d}^{\prime *}TX\otimes \Lambda^{0,1})\xrightarrow{\cong} 
  L^2_{m,\delta}(\Sigma_{\rm d}^+(\sigma_1,\sigma_2);u^{\tau *}_{\rm d}
  TX\otimes \Lambda^{0,1}).
\]
and maps $\mathcal{PAL}^\tau_{u'_{\rm s}}$ and $\mathcal{PAL}^\tau_{U'_{\rm D}}$. Then for $\frak e= (\frak e_{\rm d}, \frak e_{\rm s}, \frak e_{\rm D})\in E_{\rm d}\oplus E_{\rm s} \oplus E_{\rm D}$, we define:
\begin{equation} \label{der-ob-bun}
	\aligned
	(D_{u_{\rm d}'}E)(\frak e_{\rm d},V_{\rm d})= &
	\left.\frac{d}{d\tau}\right\vert_{\tau=0} ((\mathcal{PAL}^\tau_{u'_{\rm d}})^{-1}(\frak e_{\rm d})), \\
	(D_{u_{\rm s}'}E)(\frak e_{\rm s},V_{\rm s})= &
	\left.\frac{d}{d\tau}\right\vert_{\tau=0} ((\mathcal{PAL}^\tau_{u'_{\rm s}})^{-1}(\frak e_{\rm s})), \\
	(D_{U_{\rm D}'}E)(\frak e_{\rm D},V_{\rm D})=&
	\left.\frac{d}{d\tau}\right\vert_{\tau=0} ((\mathcal{PAL}^\tau_{U'_{\rm D}})^{-1}(\frak e_{\rm D})). \\
	\endaligned
\end{equation}
We also reserve the following notation for the triple given by the above vectors:
\begin{equation} \label{der-ob-bun-triple}
	(D_{\frak u'}E)(\frak e,{\rm V})=
	((D_{u_{\rm d}'}E)(\frak e_{\rm d},V_{\rm d}),
	(D_{u_{\rm s}'}E)(\frak e_{\rm s},V_{\rm s}),
	(D_{U_{\rm D}'}E)(\frak e_{\rm D},V_{\rm D}))
\end{equation}
The linearizations of the Cauchy-Riemann equations in \eqref{eq630-rep} at $(\frak u',\frak e)$ evaluated at $\rm V$ as above and $\frak f\in E_{\rm d}\oplus E_{\rm s} \oplus E_{\rm D}$ have the following form:
\begin{equation}\label{lin-eq630-rep}
	D_{\frak u'}\overline{\partial}({\rm V}) - (D_{\frak u'}E)(\frak e,{\rm V})-\frak f.
\end{equation}
where:
\[
  D_{\frak u'}\overline{\partial}({\rm V})=(D_{u_{\rm d}'}\overline{\partial}(V_{\rm d}),
  D_{u_{\rm s}'}\overline{\partial}(V_{\rm s}),D_{U_{\rm D}'}\overline{\partial}(V_{\rm D}) ).
\]

\subsection{Newton's Iteration}
\label{subsub:newton}

Now we are ready to carry out the strategy which is discussed in the previous subsection. In the following, we use the maps constructed in Subsection \ref{subsub:preglue}.
\par\smallskip
\noindent{\bf (Step 0-4) (Separating error terms into three parts)}
\par
We firstly fix notations for the error terms of our first approximation $\frak u^{\xi,\prime}_{\sigma_1,\sigma_2,(0)}$:
\begin{equation}\label{error-first-approx}
	\aligned
	{\rm Err}^{\xi}_{{\rm d},\sigma_1,\sigma_2,(0)}
	&= \chi_{1,\mathcal X}^{\leftarrow}(\overline \partial u^{\prime, \xi,1}_{\sigma_1,\sigma_2,(0)}- \frak e_{{\rm d},(0)}^{\xi}),\\
	{\rm Err}^{\xi}_{{\rm s},\sigma_1,\sigma_2,(0)}
	&= \chi_{2,\mathcal X}^{\leftarrow}(\overline \partial u^{\prime, \xi,2}_{\sigma_1,\sigma_2,(0)} - \frak e_{{\rm s},(0)}^{\xi}), \\
	{\rm Err}^{\xi}_{{\rm D},\sigma_1,\sigma_2,(0)}
	&= \chi_{\mathcal X}^{\rightarrow}(\overline \partial u^{\prime, \xi,1}_{\sigma_1,\sigma_2,(0)}- 
	\frak e_{{\rm D},(0)}^{\xi,1}).
	\endaligned
\end{equation}
where $\frak e_{{\rm d},(0)}^{\xi}$, $\frak e_{{\rm s},(0)}^{\xi}$, $\frak e_{{\rm D},(0)}^{\xi,i}$ are defined in \eqref{form6662}.

\par\smallskip
\noindent{\bf (Step 1-1) (Approximate solution for linearization)}
\par
Next we define:
\begin{equation}\label{new677770}
	\frak u^{\xi,\prime\prime}_{\sigma_1,\sigma_2,(0)}=(u^{\xi,\prime\prime}_{{\rm d},\sigma_1,\sigma_2,(0)},
	u^{\xi,\prime\prime}_{{\rm s},\sigma_1,\sigma_2,(0)}, 
	U^{\xi,\prime\prime}_{{\rm D},\sigma_1,\sigma_2,(0)})
\end{equation}
whose entries have the form given in \eqref{uprime20} and \eqref{uprime21}.
Let:
\begin{equation}
\aligned
&u^{\xi}_{\rm d}(z_{\rm d}) = p^{\xi}_{{\rm d},\sigma_1,\sigma_2,(0)} = p_{{\rm D,d},\sigma_1,\sigma_2,(0)},
\\
&u^{\xi}_{\rm s}(z_{\rm s}) = p^{\xi}_{{\rm s},\sigma_1,\sigma_2,(0)} = p_{{\rm D,s},\sigma_1,\sigma_2,(0)}.
\endaligned\end{equation}
We take $c_{\rm d}^{\xi}$, $c_{\rm s}^{\xi}$, $c_{\rm D,d}^{\xi}$, $c_{\rm D,s}^{\xi}$
as in (\ref{shiki655}), (\ref{shiki655rev}), (\ref{shiki656}), (\ref{shiki656rev}),
respectively.
We regard $c_{\rm d}^{\xi} z^{p_1}$ as an element of the fiber of $\mathcal N_{\mathcal D}(X)$ at $p^{\xi}_{{\rm d},\sigma_1,\sigma_2,(0)}$ and hence as an element of $X \setminus \mathcal D$. We define:
\begin{equation}\label{new682new}
	\aligned
	&u^{\xi,\prime\prime}_{{\rm d},\sigma_1,\sigma_2,(0)}(r_1,s_1) \\
	&:=
	{\rm Exp}(c_{\rm d}^{\xi} z^{p_1},\chi_{1,\mathcal B}^{\leftarrow}(r_1-T_1,s_1){\rm E}(c_{\rm d}^{\xi} z^{p_1},u^{\prime, \xi,1}_{\sigma_1,\sigma_2,(0)}(r_1,s_1)))
	\endaligned
\end{equation}
if $(r_1,s_1) \in [-5T_1,\infty)_{r_1} \times S^1_{s_1} \subset \Sigma_{\rm d}\setminus\{z_{\rm d}\}$.
If $\frak z \in \Sigma_{\rm d}\setminus\{z_{\rm d}\}$ is an element in the complement of $[-5T_1,\infty)_{r_1} \times S^1_{s_1}$,
then we define:
$$
u^{\xi,\prime\prime}_{{\rm d},\sigma_1,\sigma_2,(0)}(\frak z) :=u^{\prime, \xi,1}_{\sigma_1,\sigma_2,(0)}(\frak z).
$$
This completes the definition of $u^{\xi,\prime\prime}_{{\rm d},\sigma_1,\sigma_2,(0)}$ as a map from $\Sigma_{\rm d}\setminus\{z_{\rm d}\}$ to $X\setminus \mathcal D$.

Similarly, we define:
\begin{equation}\label{new683new}
	\aligned
	&u^{\xi,\prime\prime}_{{\rm s},\sigma_1,\sigma_2,(0)}(r_1,s_1) \\
	&:={\rm Exp}(c_{\rm s}^{\xi} z^{p_2},\chi_{2,\mathcal B}^{\leftarrow}(r_2-T_2,s_2){\rm E}(c_{\rm s}^{\xi} z^{p_2},u^{\prime, \xi,2}_{\sigma_1,\sigma_2,(0)}(r_2,s_2)))
	\endaligned
\end{equation}
if $(r_2,s_2) \in [-5T_2,\infty)_{r_2} \times S^1_{s_2} \subset \Sigma_{\rm s}\setminus\{z_{\rm s}\}$. Here we regard $c_{\rm s}^{\xi} z^{p_2}$ as an element of the fiber of $\mathcal N_{\mathcal D}(X)$ at $p^{\xi}_{{\rm s},\sigma_1,\sigma_2,(0)}$ and hence as an element of $X \setminus \mathcal D$. If $\frak z \in \Sigma_{\rm s}\setminus\{z_{\rm s}\}$ is an element in the complement of $[-5T_2,\infty)_{r_2} \times S^1_{s_2}$,
then we define:
\[
  u^{\xi,\prime\prime}_{{\rm s},\sigma_1,\sigma_2,(0)}(\frak z) :=u^{\prime, \xi,2}_{\sigma_1,\sigma_2,(0)}(\frak z).
\]
The next lemma is easy to prove.

\begin{lemma}\label{conds626}
	If the constants $T_1$, $T_2$ are sufficiently large, then $u''_{\rm d} :  (\Sigma_{\rm d}\setminus\{z_{\rm d}\},\partial \Sigma_{\rm d}) \to (X\setminus \mathcal D,L)$ 
	(resp. $u''_{\rm s} :  \Sigma_{\rm s}\setminus\{z_{\rm s}\} \to X\setminus \mathcal D$) satisfies the following properties:
	\begin{enumerate}
		\item $u''_{\rm d}$ (resp. $u''_{\rm s}$) maps  $[3T_1,\infty)_{r_1} \times S^1_{s_1}$ (resp. $[3T_2,\infty)_{r_2} \times S^1_{s_2}$) to $\frak U$.
			There exist $p_{\rm d}, p_{\rm s} \in \mathcal D$  such that the restriction of $\pi \circ u''_{\rm d}$ 
			(resp. $\pi \circ u''_{\rm s}$) to $[3T_1,\infty)_{r_1} \times S^1_{s_1}$ (resp. $[3T_2,\infty)_{r_2} \times S^1_{s_2}$) is a constant map to $p_{\rm d}$ (resp. $p_{\rm s}$).
		\item After an appropriate  trivialization of the 
		pull back of the normal bundle $\mathcal N_{\mathcal D}(X)$ at the points $p_{\rm d}, p_{\rm s}$,
			there exist $c_{\rm d}, c_{\rm s} \in \bbC_*$ such that the restriction of $u''_{\rm d} \circ \varphi_{\rm d}$ to $[3T_1,\infty)_{r_1} \times S^1_{s_1}$ 
			(resp. $u''_{\rm s} \circ \varphi_{\rm s}$ to $[3T_2,\infty)_{r_2} \times S^1_{s_2}$) is
			\begin{equation}\label{form6733}
				(u''_{\rm d} \circ \varphi_{\rm d})(z)=c_{\rm d} z^{p_1},\quad\text{(resp. $(u''_{\rm s} \circ \varphi_{\rm s})(z)= c_{\rm s} z^{p_2}$).}
			\end{equation}
	\end{enumerate}
\end{lemma}

Next, we define the map $U^{\xi,\prime\prime}_{{\rm D},\sigma_1,\sigma_2,(0)}:\Sigma_{\rm D} \setminus \{z_{\rm d},z_{\rm s}\} \to \mathcal N_{\mathcal D}(X)$. A trivialization of the fibers of $\mathcal N_{\mathcal D}(X)$ at the points $p^{\xi}_{{\rm d},\sigma_1,\sigma_2,(0)}$ and $p^{\xi}_{{\rm s},\sigma_1,\sigma_2,(0)}$ allow us to identify $c_{\rm D,d}^{\xi} w^{-{p_1}}$ and $c_{\rm D,s}^{\xi} w^{-{p_2}}$ with elements of $\mathcal N_{\mathcal D}(X)\setminus \mathcal D = \bbR_{\tau} \times S\mathcal N_{\mathcal D}(X)$. We define:
\begin{equation}\label{new685new}
	\aligned
	&U^{\xi,\prime\prime}_{{\rm D},\sigma_1,\sigma_2,(0)}(r_1,s_1) \\
	&={\rm Exp}(c_{\rm D,d}^{\xi} w^{-{p_1}}, \chi_{\mathcal A}^{\rightarrow}(r_1+T_1,s_1)\cdot\\
	&\qquad\qquad\qquad\qquad{\rm E}(c_{\rm D,d}^{\xi} w^{-{p_1}},(({\rm Dil}_{\rho_{1,(0)}^{\xi}})^{-1}  \circ u^{\prime, \xi,1}_{\sigma_1,\sigma_2,(0)})(r_1,s_1)))
	\endaligned
\end{equation}
if $(r_1,s_1) \in (-\infty,5T_1]_{r_1} \times S^1_{s_1} \subset\Sigma_{\rm D} \setminus \{z_{\rm d},z_{\rm s}\}$, and:
\begin{equation}\label{new686new}
	\aligned
	&U^{\xi,\prime\prime}_{{\rm D},\sigma_1,\sigma_2,(0)}(r_2,s_2) \\
	&={\rm Exp}(c_{\rm D,s}^{\xi} w^{-{p_2}},\chi_{\mathcal A}^{\rightarrow}(r_2+T_2,s_2)\cdot\\
	&\qquad\qquad\qquad\qquad{\rm E}(c_{\rm D,s}^{\xi} w^{-{p_2}},(({\rm Dil}_{\rho_{2,(0)}^{\xi}})^{-1}  \circ u^{\prime, \xi,2}_{\sigma_1,\sigma_2,(0)})(r_2,s_2)))
	\endaligned
\end{equation}
if $(r_2,s_2) \in (-\infty,5T_2]_{r_2} \times S^1_{s_2} \subset \Sigma_{\rm D} \setminus \{z_{\rm d},z_{\rm s}\}$.
If  $\frak z$ is an element of $\Sigma_{\rm D} \setminus \{z_{\rm d},z_{\rm s}\}$, that does not belong to the above cylinders, then we define: 
$$
  U^{\xi,\prime\prime}_{{\rm D},\sigma_1,\sigma_2,(0)}(\frak z) :=({\rm Dil}_{1/\rho_{1,(0)}^{\xi}}  \circ u^{\prime, \xi,1}_{\sigma_1,\sigma_2,(0)})(\frak z).
$$
Note that we can equivalently use the term $({\rm Dil}_{1/\rho_{2,(0)}^{\xi}}  \circ u^{\prime, \xi,2}_{\sigma_1,\sigma_2,(0)})(\frak z)$ on the right hand side of the above definition. 
We remark that the `highest order' terms of the maps $({\rm Dil}_{\rho_{i,(0)}^{\xi}})^{-1}  \circ u^{\prime, \xi,i}_{\sigma_1,\sigma_2,(0)}$ and $U_{\rm D}^{\xi}$ agree with each other on $[-5T_i,5T_i]_{r_i} \times S^2_{s_i}$. Similarly, $U_{\rm D}^{\xi}(\varphi_{\rm D,d}(w))$ (resp. $U_{\rm D}^{\xi}(\varphi_{\rm D,s}(w))$) and $c_{\rm D,d}^{\xi} w^{-2}$ (resp. $c_{\rm D,s}^{\xi} w^{-3}$) have the same highest order terms on $[-5T_i,5T_i]_{r_i} \times S^2_{s_i}$.

The following lemma can be verified in a straightforward way.
\begin{lemma}\label{conds627}
	If the constants $T_1$, $T_2$ are sufficiently large, then $U''_{\rm D} : \Sigma_{\rm D} \setminus \{z_{\rm d},z_{\rm s}\} \to X\setminus \mathcal D$ satisfies the following properties:
	\begin{enumerate}
		\item There exist $p_{\rm D,d}, p_{\rm D,s} \in \mathcal D$ such that the restriction of $\pi \circ U''_{D}$ to $(-\infty,-3T_1]_{r_1} \times S^1_{s_1}$ (resp. 
 			$(-\infty,-3T_2]_{r_2} \times S^1_{s_2}$) is a constant map to $p_{\rm D,d}$ (rsep. $p_{\rm D,d}$).
		\item  There exist $c_{\rm D,d}, c_{\rm D,s} \in \bbC_*$ such that the restriction of  
			$U''_{\rm D} \circ \varphi_{\rm D,d}$ to $(-\infty,-3T_1]_{r_1} \times S^1_{s_1}$ 
			(resp. $U''_{\rm D} \circ \varphi_{\rm D,s}$ to $(-\infty,-3T_2]_{r_2} \times S^1_{s_2}$) is
			\begin{equation}\label{form6744}
				(U''_{\rm D,d} \circ \varphi_{\rm D,d})(w)= c_{\rm D,d} w^{-{p_1}},\quad\text{(resp. $(U''_{\rm D} \circ \varphi_{\rm D,s})(w)=c_{\rm D,s} w^{-{p_2}}$).}
			\end{equation}
	\end{enumerate}
\end{lemma}

Let $\frak u''=(u''_{\rm d},u''_{\rm s},U''_{\rm D})$ be a triple of maps satisfying the properties in Lemmas \ref{conds626}, \ref{conds627}. We also assume:
\begin{equation}\label{form67554}
	p_{\rm d} = p_{\rm D,d},\qquad p_{\rm s} = p_{\rm D,s}.
\end{equation}

\begin{definition}\label{defn62888}
	Let $W^{2 \sim}_{m,\delta}(\frak u'',U''_{\rm D};TX)$ be the set of all 
	$\bf V=({\bf V}_{\rm d},{\bf V}_{\rm s},{\bf V}_{\rm D})$ satisfying the following properties:
	\begin{enumerate}
            	\item ${\bf V}_{\rm d} = (V_{\rm d},(\frak r_{\infty,\rm d},\frak s_{\infty,\rm d}),v_{\rm d})\in W^2_{m,\delta}(\Sigma_{\rm d} \setminus \{z_{\rm d}\};((u''_{\rm d})^*TX,(u_{\rm d}'')^*TL))$.
            		(This function space is introduced in Definition \ref{defn6262}.)
            	\item ${\bf V}_{\rm s} = (V_{\rm s},(\frak r_{\infty,\rm s},\frak s_{\infty,\rm s}),v_{\rm s})\in W^2_{m,\delta}(\Sigma_{\rm s} \setminus \{z_{\rm s}\};(u''_{\rm s})^*TX)$.
            		(This function space is introduced in Definition \ref{defn64444}.)
            	\item The tuple ${\bf V}_{\rm D} = (V_{\rm D},(\frak r_{\infty,\rm D,d},\frak s_{\infty,\rm D,d}),(\frak r_{\infty,\rm D,s},\frak s_{\infty,\rm D,s})),v_{\rm D,d},v_{\rm D,s})$ is an element of $
            		W^2_{m,\delta}(\Sigma_{\rm D} \setminus \{z_{\rm d},z_{\rm s}\};(U_{\rm D}'')^*T(\bbR_{\sigma} \times S\mathcal N_{\mathcal D}(X))).$
            		(This function space is introduced in Definition \ref{defn66666}.)
            	\item We assume 
            		\[
            		  v_{\rm d} = v_{\rm D,d},\qquad v_{\rm s} = v_{\rm D,s}.
            		\]
	\end{enumerate}
	The space $W^{2 \sim}_{m,\delta}(\frak u'';TX)$ is a linear subspace of finite codimension of the direct sum of 
	three Hilbert spaces defined in Definitions \ref{defn6262},  \ref{defn64444},  \ref{defn66666}. Therefore, it is also a Hilbert space.
	
	We regard $\bbR \oplus \bbR$ as the subspace of $W^2_{m,\delta}(\Sigma_{\rm D} \setminus \{z_{\rm d},z_{\rm s}\};(U_{\rm D}'')^*T(\bbR \times S\mathcal N_{\mathcal D}(X)))$
	given by constant sections with values in $\bbR\oplus \bbR\subset T(\bbR_{\sigma} \times S\mathcal N_{\mathcal D}(X))$. Thus $\bbR \oplus \bbR$ can be also regarded as a subspace of 
	$W^{2 \sim}_{k,\delta}(\frak u'';TX)$. We define 
	$W^{2}_{m,\delta}(\frak u'';TX)$ to be the quotient 
	space of $W^{2 \sim}_{m,\delta}(\frak u'';TX)$ by this copy of $\bbR \oplus \bbR$.
\end{definition}

\begin{remark}
	We do {\it not} assume $\frak r_{\infty,\rm d} = \frak r_{\infty,\rm D,d}$ or $\frak r_{\infty,\rm s} = \frak r_{\infty,\rm D,s}$.
	%We solve linearized equation in the function space $W^{2}_{m+1,\delta}(u''_{\rm d},u''_{\rm s},U''_{\rm D};TX)$.
	The fact that we might have $\frak r_{\infty,\rm d} \ne \frak r_{\infty,\rm D,d}$ or $\frak r_{\infty,\rm s} \ne \frak r_{\infty,\rm D,s}$ is related to the shift of $\rho_1$, $\rho_2$, 
	which we mentioned in Subsection \ref{subsub:linsol1}.
\end{remark}

\begin{definition}
	Let $L^{2}_{m,\delta}(\frak u'';TX\otimes \Lambda^{0,1})$ be the direct sum of the three Hilbert spaces:
	  \begin{align*}
	  &L^2_{m,\delta}(\Sigma_{\rm d} \setminus \{z_{\rm d}\};(u''_{\rm d})^*TX \otimes \Lambda^{0,1}) \\
	  &\oplus L^2_{m,\delta}(\Sigma_{\rm s} \setminus \{z_{\rm s}\};(u''_{\rm s})^*TX \otimes \Lambda^{0,1}) \\
	  &\oplus L^2_{m,\delta}(\Sigma_{\rm D} \setminus \{z_{\rm d},z_{\rm s}\};(U''_{\rm D})^*T(\bbR_{\tau} \times S\mathcal N_{\mathcal D}(X)) \otimes \Lambda^{0,1}),
	  \end{align*}
	introduced in  Definitions \ref{coker-weighted-sob},  \ref{defn64444}, \ref{defn66666}.
	The three operators (\ref{fredholmmap1}), (\ref{fredholmmap2ss}), (\ref{fredholmmap2ssrev}) together induce a Fredholm operator:
	\begin{equation}\label{form67575}
		D_{\frak u''}\overline\partial : W^{2}_{m+1,\delta}(\frak u'';TX)\to L^{2}_{m,\delta}(\frak u'';TX\otimes \Lambda^{0,1}).
	\end{equation}
\end{definition}

\begin{remark}
	If $u''_{\rm d},u''_{\rm s},U''_{\rm D}$ are $C^1$-close to $u_{\rm d}^{\xi}$, $u_{\rm s}^{\xi}$, $U_{\rm D}^{\xi}$, then 
	the surjectivity of (\ref{fredholmmap1}), (\ref{fredholmmap2ss}), (\ref{fredholmmap2ssrev}) (for $u_{\rm d}^{\xi}$, $u_{\rm s}^{\xi}$, $U_{\rm D}^{\xi}$)
	modulo $E_{\rm d}(u_{\rm d}^{\xi}) \oplus E_{\rm s}(u_{\rm s}^{\xi}) \oplus E_{\rm D}(u_{\rm D}^{\xi})$ and the mapping transversality condition of Definition \ref{defn6969}
	imply that (\ref{form67575}) is surjective modulo the obstruction space 
	$E_{\rm d}(u_{\rm d}'') \oplus E_{\rm s}(u_{\rm s}'') \oplus E_{\rm D}(u_{\rm D}'')$. (See also Lemma \ref{lem63631}.)  
\end{remark}

\begin{lemma}\label{lem63030}
	The triple:
	\[
	  {\rm Err}^{\xi}_{\sigma_1,\sigma_2,(0)}:=({\rm Err}^{\xi}_{{\rm d},\sigma_1,\sigma_2,(0)}, 
	  {\rm Err}^{\xi}_{{\rm s},\sigma_1,\sigma_2,(0)},
	  {\rm Err}^{\xi}_{{\rm D},\sigma_1,\sigma_2,(0)})
	 \] 
	determines an element of $L^{2}_{m,\delta}(\frak u^{\xi,\prime\prime}_{\sigma_1,\sigma_2,(0)};TX\otimes \Lambda^{0,1})$. The terms above are defined in \eqref{error-first-approx}.
\end{lemma}
\begin{proof}
	It follows from the fact that the map $u^{\xi,\prime\prime}_{{\rm d},\sigma_1,\sigma_2,(0)}$ (resp.
	$u^{\xi,\prime\prime}_{{\rm s},\sigma_1,\sigma_2,(0)}$, 
	$U^{\xi,\prime\prime}_{{\rm D},\sigma_1,\sigma_2,(0)}$)
	coincides with $u^{\prime, \xi,1}_{\sigma_1,\sigma_2,(0)}$ (resp. $u^{\prime, \xi,2}_{\sigma_1,\sigma_2,(0)}$, 
	$({\rm Dil}_{\rho_{1,(0)}^{\xi}})^{-1}  \circ u^{\prime, \xi,1}_{\sigma_1,\sigma_2,(0)}$) on the support of ${\rm Err}^{\xi}_{{\rm d},\sigma_1,\sigma_2,(0)}$
	(resp. ${\rm Err}^{\xi}_{{\rm s},\sigma_1,\sigma_2,(0)}$, ${\rm Err}^{\xi}_{{\rm D},\sigma_1,\sigma_2,(0)}$).
\end{proof}

\begin{lemma}\label{lem63631}
        Let the linear operator 
        \[L:W^{2}_{m+1,\delta}(\frak u^{\xi,\prime\prime}_{\sigma_1,\sigma_2,(0)};TX)\oplus 
        E_{\rm d}\oplus E_{\rm s}\oplus E_{\rm D}\to 
        L^{2}_{m,\delta}(\frak u^{\xi,\prime\prime}_{\sigma_1,\sigma_2,(0)};TX\otimes \Lambda^{0,1})\]
        be given as follows:
        \[
          L({\bf V},\frak f)=D_{\frak u^{\xi,\prime\prime}_{\sigma_1,\sigma_2,(0)}}\overline\partial({\bf V})
          -(D_{\frak u^{\xi,\prime\prime}_{\sigma_1,\sigma_2,(0)}}E)(
          \frak e_{(0)}^{\xi},{\bf V})-\frak f
        \]
        where the compoents of $\frak e_{(0)}^{\xi}:=
        (\frak e_{{\rm d},(0)}^{\xi},\frak e_{{\rm s},(0)}^{\xi},\frak e_{{\rm D},(0)}^{\xi,1})$ 
        are defined in \eqref{form6662}, and the term 
        $(D_{\frak u^{\xi,\prime\prime}_{\sigma_1,\sigma_2,(0)}}E)(\frak e_{(0)}^{\xi},{\bf V})$
        is defined similar to the corresponding term in \eqref{lin-eq630-rep}.
        If $\sigma_1$ and $\sigma_2$ are small enough, then there is a continuous operator 
        \[Q:L^{2}_{m,\delta}(\frak u^{\xi,\prime\prime}_{\sigma_1,\sigma_2,(0)};TX\otimes \Lambda^{0,1})
        \to W^{2}_{m+1,\delta}(\frak u^{\xi,\prime\prime}_{\sigma_1,\sigma_2,(0)};TX)\oplus 
        E_{\rm d}\oplus E_{\rm s}\oplus E_{\rm D}\] 
        which is a right inverse to $L$. Let $(\overline Q,Q_{\rm d},Q_{\rm s},Q_{\rm D})$
        be the components of $Q$ with respect to the decomposition of the target of $Q$.
        There is also constant $C$, independent of 
        $\sigma_1$, $\sigma_2$ and $\xi$, such that for any 
        $z\in L^{2}_{m,\delta}(\frak u^{\xi,\prime\prime}_{\sigma_1,\sigma_2,(0)};TX\otimes \Lambda^{0,1})$:
        \begin{equation} \label{estimate}
          \Vert \overline  Q(z) \Vert_{W^2_{m+1,\delta}}+{\vert}Q_{\rm d}(z){\vert}+{\vert}Q_{\rm s}(z){\vert}+{\vert}Q_{\rm D}(z){\vert}\le
          C \Vert z\Vert_{L^{2}_{m,\delta}}.
        \end{equation}
	Moreover, we can make this choose of $Q$ unique by demanding that its image is $L^2$-orthogonal\footnote{
	We use the $L^2$ norm on the target of $Q$ given by $W^2_{m,\delta}$ with $m=0$ and $\delta=0$.} 
	to the subspace ${\rm ker}(L)$.\footnote{The last condition is 
	similar to \cite[Definition 5.9]{foooexp}.}
\end{lemma}       
\begin{proof}
	Using Definition \ref{defn6868} (2), (3), (4) and Definition \ref{defn6969}, we can construct a continuous operator:
	\[ 
	    \aligned
            Q_0 = (Q_{0,\rm d},Q_{0,\rm s},Q_{0,\rm D}, Q_{0,E}):
            &L^2_{m}(\Sigma_{\rm d}\setminus\{z_{\rm d}\};u_{\rm d}^*TX\otimes \Lambda^{0,1}) \\
            &\quad \oplus 
            L^2_{m}(\Sigma_{\rm s}\setminus\{z_{\rm s}\};u_{\rm s}^*TX\otimes \Lambda^{0,1})  \\
            &\quad \oplus 
            L^2_{m}(\Sigma_{\rm D}\setminus\{z_{\rm d},z_{\rm s}\};U_{\rm D}^*TX\otimes \Lambda^{0,1}) \\
            &\to W^2_{m+1,\delta}(\Sigma_{\rm d} \setminus \{z_{\rm d}\};(u_{\rm d}^*TX,u_{\rm d}^*TL)) \\
            &\quad \oplus W^2_{m+1,\delta}(\Sigma_{\rm s} \setminus \{z_{\rm s}\};u_{\rm s}^*TX)\\
            & \quad \oplus W^2_{m+1,\delta}(\Sigma_{\rm D} \setminus \{z_{\rm d},z_{\rm s}\};U_{\rm D}^*T(\bbR_{\tau} 
            \times S\mathcal N_{\mathcal D}(X))) \\
            &\qquad \oplus E_{\rm d} \oplus  E_{\rm s} \oplus E_{\rm D}
            \endaligned
        \] 
        such that:
        \[
	  \mathcal{EV}_{0,\rm d}\circ Q_{0,\rm d} = \mathcal{EV}_{0,\rm d}\circ Q_{0,\rm D},\quad
	  \mathcal{EV}_{0,\rm s}\circ Q_{0,\rm s} = \mathcal{EV}_{0,\rm d}\circ Q_{0,\rm D}.
	\]
	and:
	\[
	  (D_{u_{\rm d}}\overline \partial Q_{0,\rm d},D_{u_{\rm s}}\overline \partial Q_{0,\rm s}, D_{U_{\rm D}}\overline \partial Q_{0,\rm D})=Q_{0,E}.
	\]

Note that since $u^{\xi,\prime\prime}_{{\rm d},\sigma_1,\sigma_2,(0)}$, 
$u^{\xi,\prime\prime}_{{\rm s},\sigma_1,\sigma_2,(0)}$, 
$U^{\xi,\prime\prime}_{{\rm D},\sigma_1,\sigma_2,(0)}$
are respectively close to $u_{\rm d}$, $u_{\rm s}$, $U_{\rm D}$,
we can identify the function spaces 
$W^2_{m+1,\delta}(\Sigma_{\rm d} \setminus \{z_{\rm d}\};(u_{\rm d}^*TX,u_{\rm d}^*TL))$
etc. with 
$W^2_{m+1,\delta}(\Sigma_{\rm d} \setminus \{z_{\rm d}\};((u^{\xi,\prime\prime}_{{\rm d},\sigma_1,\sigma_2,(0)})^*TX,u_{\rm d}^*TL))$ etc.
and 
$L^2_{m,\delta}(\Sigma_{\rm d} \setminus \{z_{\rm d}\};(u_{\rm d}^*TX,u_{\rm d}^*TL))$
etc. with 
$L^2_{m,\delta}(\Sigma_{\rm d} \setminus \{z_{\rm d}\};((u^{\xi,\prime\prime}_{{\rm d},\sigma_1,\sigma_2,(0)})^*TX,u_{\rm d}^*TL))$ etc.
by target parallel transport.
\par
Using this identification we obtain a map

	\[
	  Q_1 :  L^{2}_{m,\delta}(\frak u^{\xi,\prime\prime}_{\sigma_1,\sigma_2,(0)};TX\otimes \Lambda^{0,1})
          \to W^{2}_{m+1,\delta}(\frak u^{\xi,\prime\prime}_{\sigma_1,\sigma_2,(0)};TX)\oplus E_{\rm d}\oplus E_{\rm s}\oplus E_{\rm D}
	\]
	such that
	\begin{equation}
		 \Vert (L \circ Q_1)(z) - z  \Vert  \le C e^{-c\delta_1 \min \{T_1,T_2\}} \Vert z \Vert ,\qquad
		 \Vert Q_1(z) \Vert \le  C' \Vert z \Vert.
	\end{equation}
	Here $c, C,C' > 0$ are constants independent of $\sigma_1$ and $\sigma_2$.
	Thus for $\sigma_1$ and $\sigma_2$ small enough, we may define
	\[
	  Q_2 = \sum_{k=0}^{\infty} (-1)^k Q_1 \circ ({\rm id} - L \circ Q_1)^k.
	\]
	Then we have
	\begin{equation}
		(L \circ Q_2)(z) ={\rm id}, \qquad \Vert Q_2(z) \Vert \le  C'' \Vert z\Vert.
	\end{equation}
	(This formula is used there to estimate 
	derivatives of the right inverse with respect to the gluing parameter.)
	The operator $Q_2$ has the required properties except the last one.
	To obtain the right inverse which also satisfies the last condition, we compose $Q_2$ with projection to the orthogonal complement of
	the finite dimensional space ${\rm ker}(L)$.
\end{proof}  

\begin{remark}
The stable map compactification case of Lemma \ref{lem63631} 
is \cite[Lemma 5.7]{foooexp}.
The function space $W^{2}_{m+1,\delta}(\frak u^{\xi,\prime\prime}_{\sigma_1,\sigma_2,(0)};TX)$ is a subspace of the direct sum of the function spaces 
of the three irreducible components. By this reason the fact we are working on 
inconsistent maps does not affect the proof of Lemma \ref{lem63631}.
\end{remark}

Let $z:={\rm Err}^{\xi}_{\sigma_1,\sigma_2,(0)}$. By Lemma \ref{lem63030}, we know that $z$ belongs to the target of $L$. Therefore, $\overline Q(z)$ determines a triple as follows:
\begin{equation*}\label{Vssx}
  {\bf V}_{\sigma_1,\sigma_2,(1)}^{\xi} = ({\bf V}^{\xi}_{\rm d,\sigma_1,\sigma_2,(1)},
  {\bf V}^{\xi}_{\rm s,\sigma_1,\sigma_2,(1)},{\bf V}^{\xi}_{\rm D,\sigma_1,\sigma_2,(1)})\in
  W^2_{m+1,\delta}(\frak u^{\xi,\prime\prime}_{\sigma_1,\sigma_2,(0)};TX)
\end{equation*}
Moreover, we have:
\[
  \Delta\frak e_{{\rm d},\sigma_1,\sigma_2,(1)}^{\xi}:=Q_{\rm d}(z),\hspace{.6cm}
  \Delta\frak e_{{\rm s},\sigma_1,\sigma_2,(1)}^{\xi}:=Q_{\rm s}(z),\hspace{.6cm}
  \Delta\frak e_{{\rm D},\sigma_1,\sigma_2,(1)}^{\xi}:=Q_{\rm D}(z).
\]
Lemmas \ref{lem623} and \ref{lem63631} imply that:
\[
  \Vert {\bf V}^{\xi}_{\sigma_1,\sigma_2,(1)}\Vert_{W^2_{m+1,\delta}}\le C e^{-c\delta_1 \min \{T_1,T_2\}}
\]
and 
\[
  \vert \Delta\frak e_{{\rm d},\sigma_1,\sigma_2,(1)}^{\xi} \vert, 
  \vert \Delta\frak e_{{\rm s},\sigma_1,\sigma_2,(1)}^{\xi} \vert,
  \vert \Delta\frak e_{{\rm D},\sigma_1,\sigma_2,(1)}^{\xi} \vert
  \le C e^{-c\delta_1 \min \{T_1,T_2\}}.
\]
In summary, we obtained a solution of the linearized equation with appropriate decay properties.

\par\smallskip

\noindent{\bf (Step 1-2) (Gluing solutions)}
\par
In this step we will use ${\bf V}_{\sigma_1,\sigma_2,(1)}^{\xi}$ to obtain an improved approximate inconsistent solution. Suppose the entries of ${\bf V}_{\sigma_1,\sigma_2,(1)}^{\xi}$ are given as follows:
$$
\aligned
&{\bf V}_{{\rm d},\sigma_1,\sigma_2,(1)}^{\xi}
= 
(V^{\xi}_{{\rm d},\sigma_1,\sigma_2,(1)},(\frak r^{\xi}_{\infty,{\rm d},\sigma_1,\sigma_2,(1)},
\frak s^{\xi}_{\infty,{\rm d},\sigma_1,\sigma_2,(1)}),v^{\xi}_{\infty,{\rm d},\sigma_1,\sigma_2,(1)}),
\\
&{\bf V}_{{\rm s},\sigma_1,\sigma_2,(1)}^{\xi}
= 
(V^{\xi}_{{\rm s},\sigma_1,\sigma_2,(1)},(\frak r^{\xi}_{\infty,{\rm s},\sigma_1,\sigma_2,(1)},
\frak s^{\xi}_{\infty,{\rm s},\sigma_1,\sigma_2,(1)}),v^{\xi}_{\infty,{\rm s},\sigma_1,\sigma_2,(1)}),
\\
&{\bf V}_{{\rm D},\sigma_1,\sigma_2,(1)}^{\xi}
= 
(V^{\xi}_{{\rm D},\sigma_1,\sigma_2,(1)},(\frak r^{\xi}_{\infty,{\rm D,d},\sigma_1,\sigma_2,(1)},
\frak s^{\xi}_{\infty,{\rm D,d},\sigma_1,\sigma_2,(1)}),
\\
&\qquad\qquad\qquad\qquad\qquad\qquad(\frak r^{\xi}_{\infty,{\rm D,s},\sigma_1,\sigma_2,(1)},
\frak s^{\xi}_{\infty,{\rm D,s},\sigma_1,\sigma_2,(1)}),
\\
&\qquad\qquad\qquad\qquad\qquad\qquad v^{\xi}_{\infty,{\rm D,d},\sigma_1,\sigma_2,(1)},
v^{\xi}_{\infty,{\rm D,s},\sigma_1,\sigma_2,(1)}).
\endaligned
$$
We also have the following identities:
\[
  v^{\xi}_{\infty,{\rm d},\sigma_1,\sigma_2,(1)}=v^{\xi}_{\infty,{\rm D,d},\sigma_1,\sigma_2,(1)},
  \hspace{1cm}v^{\xi}_{\infty,{\rm s},\sigma_1,\sigma_2,(1)}
  =v^{\xi}_{\infty,{\rm D,s},\sigma_1,\sigma_2,(1)}.
\]
by Definition \ref{defn62888} (4).
%As we mentioned several times a similar 
%formula for 
%$\frak r^{\xi}_{\infty,{\rm d},\sigma_1,\sigma_2,(1)},
%\frak s^{\xi}_{\infty,{\rm d},\sigma_1,\sigma_2,(1)}$ etc. are not assumed.
We also define:
\begin{equation}
\aligned
\Delta\frak r^{\xi}_{\infty,{\rm d},\sigma_1,\sigma_2,(1)}
&= 
\frak r^{\xi}_{\infty,{\rm D,d},\sigma_1,\sigma_2,(1)}
-
\frak r^{\xi}_{\infty,{\rm d},\sigma_1,\sigma_2,(1)}, \\
\Delta \frak s^{\xi}_{\infty,{\rm d},\sigma_1,\sigma_2,(1)}
&= 
\frak s^{\xi}_{\infty,{\rm D,d},\sigma_1,\sigma_2,(1)}
-
\frak s^{\xi}_{\infty,{\rm d},\sigma_1,\sigma_2,(1)},
\endaligned
\end{equation}

\begin{equation}
\aligned
\Delta\frak r^{\xi}_{\infty,{\rm s},\sigma_1,\sigma_2,(1)}
&= 
\frak r^{\xi}_{\infty,{\rm D,s},\sigma_1,\sigma_2,(1)}
-
\frak r^{\xi}_{\infty,{\rm s},\sigma_1,\sigma_2,(1)}, \\
\Delta \frak s^{\xi}_{\infty,{\rm s},\sigma_1,\sigma_2,(1)}
&= 
\frak s^{\xi}_{\infty,{\rm D,s},\sigma_1,\sigma_2,(1)}
-
\frak s^{\xi}_{\infty,{\rm s},\sigma_1,\sigma_2,(1)}.
\endaligned
\end{equation}

\begin{definition}\label{defn633}
        We define $u^{\xi,\prime}_{{\rm d},\sigma_1,\sigma_2,(1)} : \Sigma_{\rm d}^+(\sigma_1,\sigma_2) 
        \to X \setminus \mathcal D$ as follows.
        \begin{enumerate}
        \item
        If  $\frak z \in \Sigma_{\rm d}^-(\sigma_1,\sigma_2)$ then
        $$
        u^{\xi,\prime}_{{\rm d},\sigma_1,\sigma_2,(1)}(\frak z)
        =
        {\rm Exp}(u^{\xi,\prime\prime}_{{\rm d},\sigma_1,\sigma_2,(0)}(\frak z),V^{\xi}_{{\rm d},\sigma_1,\sigma_2,(1)})
        $$
        \item
        If $\frak z = (r_1,s_1) \in [-5T_1,5T_1]_{r_1} \times S^1_{s_1}$ then
        $$
        u^{\xi,\prime}_{{\rm d},\sigma_1,\sigma_2,(1)}(\frak z)
        =
        {\rm Exp}(u^{\xi,\prime\prime}_{{\rm d},\sigma_1,\sigma_2,(0)}(r_1,s_1),\diamondsuit)
        $$
        where
        $$
        \aligned
        \diamondsuit
        =  
        &\chi_{1,\mathcal B}^{\leftarrow}(r_1,s_1)
        (V^{\xi}_{{\rm d},\sigma_1,\sigma_2,(1)}
        - (\frak r^{\xi}_{\infty,{\rm d},\sigma_1,\sigma_2,(1)},\frak s^{\xi}_{\infty,{\rm d},\sigma_1,\sigma_2,(1)})
        - \hat v^{\xi}_{\infty,{\rm d},\sigma_1,\sigma_2,(1)}) \\
        &+\chi_{\mathcal A}^{\rightarrow}(r_1,s_1)
        (V^{\xi}_{{\rm D,d},\sigma_1,\sigma_2,(1)}
        - (\frak r^{\xi}_{\infty,{\rm D,d},\sigma_1,\sigma_2,(1)},\frak s^{\xi}_{\infty,{\rm D,d},\sigma_1,\sigma_2,(1)})
        - \hat v^{\xi}_{\infty,{\rm D},{\rm d},\sigma_1,\sigma_2,(1)})
        \\
        &+ (\frak r^{\xi}_{\infty,{\rm d},\sigma_1,\sigma_2,(1)},\frak s^{\xi}_{\infty,{\rm d},\sigma_1,\sigma_2,(1)})
        + \hat v^{\xi}_{\infty,{\rm d},\sigma_1,\sigma_2,(1)}.
        \endaligned
        $$
        Here and in Items (4), (6), (7), we extend $v^{\xi}_{\infty,{\rm d},\sigma_1,\sigma_2,(1)}$ to $\hat v^{\xi}_{\infty,{\rm d},\sigma_1,\sigma_2,(1)}$ in the same way as in Definition \ref{defn6262}.
        \end{enumerate}
        \par
        We define $u^{\xi,\prime}_{{\rm s},\sigma_1,\sigma_2,(1)} : \Sigma_{\rm s}^+(\sigma_1,\sigma_2) 
        \to X \setminus \mathcal D$ as follows.
        \begin{enumerate}
        \item[(3)]
        If  $\frak z \in \Sigma_{\rm s}^-(\sigma_1,\sigma_2)$ then
        $$
        u^{\xi,\prime}_{{\rm s},\sigma_1,\sigma_2,(1)}(\frak z)
        =
        {\rm Exp}(u^{\xi,\prime\prime}_{{\rm s},\sigma_1,\sigma_2,(0)}(\frak z),V^{\xi}_{{\rm s},\sigma_1,\sigma_2,(1)})
        $$
        \item[(4)]
        If $\frak z = (r_2,s_2) \in [-5T_2,5T_2]_{r_2} \times S^1_{s_2}$ then
        $$
        u^{\xi,\prime}_{{\rm s},\sigma_1,\sigma_2,(1)}(\frak z)
        =
        {\rm Exp}(u^{\xi,\prime\prime}_{{\rm s},
        \sigma_1,\sigma_2,(0)}(r_2,s_2),\clubsuit)
        $$
        where
        $$
        \aligned
        \clubsuit 
        =  
        &\chi_{2,\mathcal B}^{\leftarrow}(r_2,s_2)
        (V^{\xi}_{{\rm s},\sigma_1,\sigma_2,(1)}
        - (\frak r^{\xi}_{\infty,{\rm s},\sigma_1,\sigma_2,(1)},\frak s^{\xi}_{\infty,{\rm s},\sigma_1,\sigma_2,(1)})
        - \hat v^{\xi}_{\infty,{\rm s},\sigma_1,\sigma_2,(1)}) \\
        &+\chi_{\mathcal A}^{\rightarrow}(r_2,s_2)
        (V^{\xi}_{{\rm D,s},\sigma_1,\sigma_2,(1)}
        - (\frak r^{\xi}_{\infty,{\rm D,s},\sigma_2,\sigma_2,(1)},\frak s^{\xi}_{\infty,{\rm D,s},\sigma_1,\sigma_2,(1)})
        - \hat v^{\xi}_{\infty,{\rm D},{\rm s},\sigma_2,\sigma_2,(1)})
        \\
        &+ (\frak r^{\xi}_{\infty,{\rm s},\sigma_1,\sigma_2,(1)},\frak s^{\xi}_{\infty,{\rm s},\sigma_1,\sigma_2,(1)})
        + \hat v^{\xi}_{\infty,{\rm s},\sigma_1,\sigma_2,(1)}.
        \endaligned
        $$
        \end{enumerate}
        We next define 
        $U^{\xi,\prime}_{{\rm D},\sigma_1,\sigma_2,(1)} : \Sigma_{\rm D}^+(\sigma_1,\sigma_2) 
        \to X \setminus \mathcal D$ as follows.
        \begin{enumerate}
        \item[(5)]
        If  $\frak z \in \Sigma_{\rm D}^-(\sigma_1,\sigma_2)$ then:
        $$
        U^{\xi,\prime}_{{\rm D},\sigma_1,\sigma_2,(1)}(\frak z)
        =
        {\rm Exp}(U^{\xi,\prime\prime}_{{\rm D},\sigma_1,\sigma_2,(0)}(\frak z),V^{\xi}_{{\rm D},\sigma_1,\sigma_2,(1)})
        $$
        \item[(6)]
        If $\frak z = (r_1,s_1) \in [-5T_1,5T_1]_{r_1} \times S^1_{s_1}$ then:
        $$
        U^{\xi,\prime}_{{\rm D},\sigma_1,\sigma_2,(1)}(\frak z)
        =
        {\rm Exp}(U^{\xi,\prime\prime}_{{\rm D},\sigma_1,\sigma_2,(0)}(r_1,s_1),\heartsuit)
        $$
        where:
        $$
        \aligned
        \heartsuit
        =  
        &\chi_{1,\mathcal B}^{\leftarrow}(r_1,s_1)
        (V^{\xi}_{{\rm d},\sigma_1,\sigma_2,(1)}
        - (\frak r^{\xi}_{\infty,{\rm d},\sigma_1,\sigma_2,(1)},\frak s^{\xi}_{\infty,{\rm d},\sigma_1,\sigma_2,(1)})
        - \hat v^{\xi}_{\infty,{\rm d},\sigma_1,\sigma_2,(1)}) \\
        &+\chi_{\mathcal A}^{\rightarrow}(r_1,s_1)
        (V^{\xi}_{{\rm D,d},\sigma_1,\sigma_2,(1)}
        - (\frak r^{\xi}_{\infty,{\rm D,d},\sigma_1,\sigma_2,(1)},\frak s^{\xi}_{\infty,{\rm D,d},\sigma_1,\sigma_2,(1)})
        - \hat v^{\xi}_{\infty,{\rm D},{\rm d},\sigma_1,\sigma_2,(1)})
        \\
        &+ (\frak r^{\xi}_{\infty,{\rm D,d},\sigma_1,\sigma_2,(1)},\frak s^{\xi}_{\infty,{\rm D,d},\sigma_1,\sigma_2,(1)})
        + \hat v^{\xi}_{\infty,{\rm D}{\rm d},\sigma_1,\sigma_2,(1)}.
        \endaligned
        $$
        We remark that: 
        \begin{equation}\label{formdiahear}
        \heartsuit-\diamondsuit = (\Delta\frak r^{\xi}_{\infty,{\rm d},\sigma_1,\sigma_2,(1)},
        \Delta \frak s^{\xi}_{\infty,{\rm d},\sigma_1,\sigma_2,(1)}).
        \end{equation}
        \item[(7)]
        If $\frak z = (r_2,s_2) \in [-5T_2,5T_2]_{r_2} \times S^1_{s_2}$ then:
        $$
        U^{\xi,\prime}_{{\rm D},\sigma_1,\sigma_2,(1)}(\frak z)
        =
        {\rm Exp}(U^{\xi,\prime\prime}_{{\rm D},\sigma_1,\sigma_2,(0)(r_2,s_2)},\spadesuit)
        $$
        where:
        $$
        \aligned
        \spadesuit
        =  
        &\chi_{2,\mathcal B}^{\leftarrow}(r_2,s_2)
        (V^{\xi}_{{\rm s},\sigma_1,\sigma_2,(1)}
        - (\frak r^{\xi}_{\infty,{\rm s},\sigma_1,\sigma_2,(1)},\frak s^{\xi}_{\infty,{\rm s},\sigma_1,\sigma_2,(1)})
        - \hat v^{\xi}_{\infty,{\rm s},\sigma_1,\sigma_2,(1)}) \\
        &+\chi_{\mathcal A}^{\rightarrow}(r_2,s_2)
        (V^{\xi}_{{\rm D,s},\sigma_1,\sigma_2,(1)}
        - (\frak r^{\xi}_{\infty,{\rm D,s},\sigma_2,\sigma_2,(1)},\frak s^{\xi}_{\infty,{\rm D,s},\sigma_1,\sigma_2,(1)})
        - \hat v^{\xi}_{\infty,{\rm D},\rm s,\sigma_2,\sigma_2,(1)})
        \\
        &+ (\frak r^{\xi}_{\infty,{\rm D,s},\sigma_1,\sigma_2,(1)},\frak s^{\xi}_{\infty,{\rm D,s},\sigma_1,\sigma_2,(1)})
        + \hat v^{\xi}_{\infty,{\rm D},\rm s,\sigma_1,\sigma_2,(1)}.
        \endaligned
        $$
        We remark that: 
        \begin{equation}\label{speclub}
        \spadesuit-\clubsuit = (\Delta\frak r^{\xi}_{\infty,{\rm s},\sigma_1,\sigma_2,(1)},
        \Delta \frak s^{\xi}_{\infty,{\rm s},\sigma_1,\sigma_2,(1)}).
        \end{equation}
        \end{enumerate}
        Let:
        \begin{equation}
 	       \aligned
		\rho_{1,(1)}^{\xi,\Delta} &= \exp(-(\Delta\frak r^{\xi}_{\infty,{\rm d},\sigma_1,\sigma_2,(1)}
		+\sqrt{-1}\Delta \frak s^{\xi}_{\infty,{\rm d},\sigma_1,\sigma_2,(1)})) \in \bbC_* \\
		\rho_{2,(1)}^{\xi,\Delta} &= \exp(-(\Delta\frak r^{\xi}_{\infty,{\rm s},\sigma_1,\sigma_2,(1)}
		+\sqrt{-1}\Delta \frak s^{\xi}_{\infty,{\rm s},\sigma_1,\sigma_2,(1)})) \in \bbC_*
        		\endaligned
        \end{equation}
        and
        \begin{equation}
		\rho_{i,(1)}^{\xi} = \rho_{i,(0)}^{\xi}\rho_{i,(1)}^{\xi,\Delta} \in \bbC_*,
        \end{equation}
        for $i=1,2$. Finally, we define:
        \begin{equation}
        \frak u^{\xi,\prime}_{\sigma_1,\sigma_2,(1)}: =
        (u^{\xi,\prime}_{{\rm d},\sigma_1,\sigma_2,(1)},u^{\xi,\prime}_{{\rm s},\sigma_1,\sigma_2,(1)},U^{\xi,\prime}_{{\rm D},\sigma_1,\sigma_2,(1)},\sigma_1,\sigma_2,\rho_{1,(1)}^{\xi},
        \rho_{2,(1)}^{\xi}).
        \end{equation}
\end{definition}
\begin{lemma}
	The 7-tuple $\frak u^{\xi,\prime}_{\sigma_1,\sigma_2,(1)}$
	 is an inconsistent map in the sense of Definition \ref{defn625625}.
\end{lemma}
\begin{proof}
This is a consequence of Lemma \ref{lem62611} and 
(\ref{formdiahear}), (\ref{speclub}).
\end{proof}
\begin{remark}
	We remark that if we change ${\bf V}^{\xi}_{{\rm D,s},\sigma_1,\sigma_2,(1)}$ 
	by an element of $\bbR \oplus \bbR$ (the tangent vector generated by the $\bbC_*$ action), 
	then $V^{\xi}_{{\rm D,s},\sigma_1,\sigma_2,(1)}$, 
	$(\frak r^{\xi}_{\infty,{\rm D,d},\sigma_2,\sigma_2,(1)},
	\frak s^{\xi}_{\infty,{\rm D,d},\sigma_1,\sigma_2,(1)})$ and 
	$(\frak r^{\xi}_{\infty,{\rm D,s},\sigma_2,\sigma_2,(1)},
	\frak s^{\xi}_{\infty,{\rm D,s},\sigma_1,\sigma_2,(1)})$  change
	by the same amount. Therefore, $\diamondsuit$ and $\clubsuit$ do not change.
	On the other hand, $\heartsuit$ and $\spadesuit$ change by the same element in $\bbR \oplus \bbR$.
	This implies that the equivalence class of 
	$\frak u^{\xi,\prime}_{\sigma_1,\sigma_2,(1)}$ is fixed among all representatives for 
	${\bf V}^{\xi}_{{\rm D,s},\sigma_1,\sigma_2,(1)}$.
\end{remark}
\par\smallskip
\noindent{\bf (Step 1-3) (Error estimate)}

$\frak u^{\xi,\prime}_{\sigma_1,\sigma_2,(1)}$ is our next approximate solution. Lemma \ref{lem636} quantifies to what extent this inconsistent map improves the previous approximate solution $\frak u^{\xi,\prime}_{\sigma_1,\sigma_2,(0)}$.

\begin{lemma}\label{lem636}
	There is a constant $C$ and for any positive number $\mu$, there is a constant $\eta$ such that the following holds. If $\sigma_1,\sigma_2$ 
	are smaller than $\eta$, 
	then there exists $\frak e_{\sigma_1,\sigma_2,(1)}^{\xi}= (\frak e_{{\rm d},\sigma_1,\sigma_2,(1)}^{\xi}, \frak e_{{\rm s},\sigma_1,\sigma_2,(1)}^{\xi},
	\frak e_{{\rm D},\sigma_1,\sigma_2,(1)}^{\xi})$ with $\frak e_{{\rm d},\sigma_1,\sigma_2,(1)}^{\xi} \in E_{\rm d}$,
	$\frak e_{{\rm s},\sigma_1,\sigma_2,(1)}^{\xi} \in E_{\rm s}$, 
	$\frak e_{{\rm D},\sigma_1,\sigma_2,(1)}^{\xi} \in E_{\rm D}$
	such that the following holds:
	\begin{enumerate}
		\item \begin{equation}
				\Vert \overline \partial\frak u^{\xi,\prime}_{\sigma_1,\sigma_2,(1)} - \frak e_{\sigma_1,\sigma_2,(1)}^{\xi}
				\Vert_{L^2_{m,\delta}}\le\mu
				\Vert \overline \partial\frak u^{\xi,\prime}_{\sigma_1,\sigma_2,(0)} - \frak e_{(0)}^{\xi}\Vert_{L^2_{m,\delta}}
			\end{equation}
			where $\frak e_{(0)}^{\xi}= (\frak e_{{\rm d},(0)}^{\xi}, \frak e_{{\rm s},(0)}^{\xi}, \frak e_{{\rm D},(0)}^{\xi})$ is as in \eqref{form6662}.
		\item $$\Vert \frak e_{\sigma_1,\sigma_2,(1)}^{\xi} - \frak e_{(0)}^{\xi}\Vert \le \mu C.$$ 
			The square of the left hand side, by definition, is the sum of the squares of the factors associated to ${\rm d}$, ${\rm s}$,  ${\rm D}$.\footnote{
			This estimate is provided for the first step of Newton's iteration. In the $\frak i$-th step, a similar estimate appears where 
			$\mu C$ is replaced by $\mu^{\frak i}C$. It is 
			important that $C$ is independent of $\frak i$.}
	\end{enumerate}
\end{lemma}
We define:
\begin{equation}
	\aligned
	\frak e_{\rm{d},\sigma_1,\sigma_2,(1)}^{\xi}
	&=
	\frak e_{\rm{d},\sigma_1,\sigma_2,(0)}^{\xi}
	+\Delta \frak e_{\rm{d},\sigma_1,\sigma_2,(0)}^{\xi} \\
	\frak e_{\rm{s},\sigma_1,\sigma_2,(1)}^{\xi}
	&=
	\frak e_{\rm{s},\sigma_1,\sigma_2,(0)}^{\xi}+\Delta \frak e_{\rm{s},\sigma_1,\sigma_2,(0)}^{\xi} \\
	\frak e_{\rm{D},\sigma_1,\sigma_2,(1)}^{\xi}
	&=\frak e_{\rm{D},\sigma_1,\sigma_2,(0)}^{\xi}+\Delta \frak e_{\rm{D},\sigma_1,\sigma_2,(0)}^{\xi}.
	\endaligned
\end{equation}
The proof of the estimates in the lemma is based on Lemma \ref{lem63631}. That is to say, we use the estimate \eqref{estimate} of Lemma \ref{lem63631} and the fact that ${\rm V}_{\sigma_1,\sigma_2,(1)}^{\xi}$ is given by solving the linearized equation.
%$$
%\overline \partial u'_{{\rm d}} \in E_{\rm d},
%\quad
%\overline \partial u'_{{\rm s}} \in E_{\rm s}, 
%\quad
%\overline \partial U'_{{\rm D}} \in E_{\rm D}
%$$
%and the estimate \eqref{estimate} of Lemma \ref{lem63631}.
The details of this estimate is similar to the proof of 
\cite[Proposition 5.17]{foooexp} and is omitted. In particular, to estimate the effect of the bump function appearing in Definition \ref{defn633} (2), (4), (6) and (7), we use the `drop of the weight' argument, which is explained in detail in \cite[right above Remark 5.21]{foooexp}.

Lemma \ref{lem6399} below concerns the estimate of the difference 
between $\frak u^{\xi,\prime}_{\sigma_1,\sigma_2,(1)}$
and $\frak u^{\xi,\prime}_{\sigma_1,\sigma_2,(0)}$.

\begin{lemma}\label{lem6399}
	Let $\sigma_1$, $\sigma_2$ be small enough such that Lemma \ref{lem63631} holds. 
	There is a fixed constant\footnote{It is important that we can take the same constant for all the steps of inductive construction of 
	Newton's iteration. The dependence of those constants to various 
	choices are studied in detail in \cite{foooexp}. So we do not repeat it here.} $C$, independent of $\sigma_1$ and $\sigma_2$, such that:
	$$
	  d_{W^{2}_{m+1,\delta}}(\frak u^{\xi,\prime}_{\sigma_1,\sigma_2,(1)},\frak u^{\xi,\prime}_{\sigma_1,\sigma_2,(0)})\le
	  C  \Vert \overline \partial\frak u^{\xi,\prime}_{\sigma_1,\sigma_2,(0)} - \frak e_{\sigma_1,\sigma_2,(0)}^{\xi}\Vert_{L^2_{m,\delta}}.
	$$
\end{lemma}	
\begin{proof}
	This is a consequence of Lemma \ref{lem63631} and definitions.
\end{proof}

\par\smallskip
\noindent{\bf (Step 1-4) (Separating error terms into three parts)}
\par
Define
\begin{equation}
\aligned
{\rm Err}^{\xi}_{{\rm d},\sigma_1,\sigma_2,(1)}
&= \chi_{1,\mathcal X}^{\leftarrow}
(\overline \partial u^{\xi,\prime}_{{\rm d},\sigma_1,\sigma_2,(1)}- \frak e_{{\rm d},\sigma_1,\sigma_2,(1)}^{\xi}),\\
{\rm Err}^{\xi}_{{\rm s},\sigma_1,\sigma_2,(1)}
&= \chi_{2,\mathcal X}^{\leftarrow}
(\overline \partial u^{\xi,\prime}_{{\rm s},\sigma_1,\sigma_2,(1)} - \frak e_{{\rm s},\sigma_1,\sigma_2,(1)}^{\xi}), \\
{\rm Err}^{\xi}_{{\rm D},\sigma_1,\sigma_2,(1)}
&= \chi_{\mathcal X}^{\rightarrow}(\overline \partial u^{\xi,\prime}_{{\rm D},\sigma_1,\sigma_2,(1)}
- \frak e_{{\rm D},\sigma_1,\sigma_2,(1)}^{\xi}).
\endaligned
\end{equation}
\par\smallskip
\noindent{\bf (Step 2-1) (Approximate solution for linearization)}
\par
We will next define 
\begin{equation}\label{new67777}
\frak u^{\xi,\prime\prime}_{\sigma_1,\sigma_2,(1)}
=
(u^{\xi,\prime\prime}_{{\rm d},\sigma_1,\sigma_2,(1)},
u^{\xi,\prime\prime}_{{\rm s},\sigma_1,\sigma_2,(1)}, 
U^{\xi,\prime\prime}_{{\rm D},\sigma_1,\sigma_2,(1)})
\end{equation}
satisfying the properties 
in Lemmas \ref{conds626}, \ref{conds627}. This step is essentially the same as (Step 1-1). We mention a few points where the two steps slightly differ.
\par
Let $v^{\xi}_{\infty,{\rm d},\sigma_1,\sigma_2,(1)} \in T_{p^{\xi}_{{\rm d},\sigma_1,\sigma_2,(0)}}
\mathcal D$ 
and $v^{\xi}_{\infty,{\rm s},\sigma_1,\sigma_2,(1)}\in T_{p^{\xi}_{{\rm s},\sigma_1,\sigma_2,(0)}}
\mathcal D$
be the element appearing at the beginning of Step 1-2.
We put
\begin{equation}
\aligned
p^{\xi}_{{\rm d},\sigma_1,\sigma_2,(1)} 
&= {\rm Exp}(p^{\xi}_{{\rm d},\sigma_1,\sigma_2,(0)},v^{\xi}_{\infty,{\rm d},\sigma_1,\sigma_2,(1)})
\\
p^{\xi}_{{\rm s},\sigma_1,\sigma_2,(1)} 
&= {\rm Exp}(p^{\xi}_{{\rm s},\sigma_1,\sigma_2,(0)},v^{\xi}_{\infty,{\rm s},\sigma_1,\sigma_2,(1)}).
\endaligned
\end{equation}
We next define
$$
\aligned
c^{\xi}_{{\rm d},(1)} &= c^{\xi}_{{\rm d}} \exp(-(\frak r^{\xi}_{\infty,{\rm d},\sigma_1,\sigma_2,(1)}+
\sqrt{-1} \frak s^{\xi}_{\infty,{\rm d},\sigma_1,\sigma_2,(1)})) \\
c^{\xi}_{{\rm D,d},(1)} &= c^{\xi}_{{\rm D,d}} \exp(-(\frak r^{\xi}_{\infty,{\rm D,d},\sigma_1,\sigma_2,(1)}+
\sqrt{-1} \frak s^{\xi}_{\infty,{\rm D,d},\sigma_1,\sigma_2,(1)})) \\
c^{\xi}_{{\rm s},(1)} &= c^{\xi}_{{\rm s}} \exp(-(\frak r^{\xi}_{\infty,{\rm s},\sigma_1,\sigma_2,(1)}+
\sqrt{-1} \frak s^{\xi}_{\infty,{\rm s},\sigma_1,\sigma_2,(1)})) \\
c^{\xi}_{{\rm D,s},(1)} &= c^{\xi}_{{\rm D,s}} \exp(-(\frak r^{\xi}_{\infty,{\rm D,s},\sigma_1,\sigma_2,(1)}+
\sqrt{-1} \frak s^{\xi}_{\infty,{\rm D,s},\sigma_1,\sigma_2,(1)})).
\endaligned
$$
Then formulas similar to (\ref{shiki655}), (\ref{shiki656}), (\ref{shiki655rev}) and (\ref{shiki656rev}) hold.
\par
In (\ref{new682new}),(\ref{new683new}),(\ref{new685new}),(\ref{new686new}), we replace $c^{\xi}_{{\rm d}}$ with $c^{\xi}_{{\rm d},(1)}$ and so on. In these formulas, we also replace $(0)$ with $(1)$. We thus define
$u^{\xi,\prime\prime}_{{\rm d},\sigma_1,\sigma_2,(1)}$, $u^{\xi,\prime\prime}_{{\rm s},\sigma_1,\sigma_2,(1)}$, $U^{\xi,\prime\prime}_{{\rm D},\sigma_1,\sigma_2,(1)}$  and $\frak u^{\xi,\prime\prime}_{\sigma_1,\sigma_2,(1)}$.
Then
$({\rm Err}^{\xi}_{{\rm d},\sigma_1,\sigma_2,(1)}, 
{\rm Err}^{\xi}_{{\rm s},\sigma_1,\sigma_2,(1)},
{\rm Err}^{\xi}_{{\rm D},\sigma_1,\sigma_2,(1)})$
determines an element of the space
$L^{2}_{k,\delta}(\frak u^{\xi,\prime\prime}_{\sigma_1,\sigma_2,(1)};TX\otimes \Lambda^{0,1})$.
(Lemma \ref{lem63030}.)

We can then formulate an analogue of Lemma \ref{lem63631} where $\frak u^{\xi,\prime\prime}_{\sigma_1,\sigma_2,(0)}$ is replaced by $\frak u^{\xi,\prime\prime}_{\sigma_1,\sigma_2,(1)}$. Using this lemma, we can obtain ${\rm V}_{\sigma_1,\sigma_2,(2)}^{\xi}$, $\Delta\frak e_{{\rm d},\sigma_1,\sigma_2,(2)}^{\xi}$, $\Delta\frak e_{{\rm s},\sigma_1,\sigma_2,(2)}^{\xi}$, $\Delta\frak e_{{\rm D},\sigma_1,\sigma_2,(2)}^{\xi}$. The counterpart of the estimate in \eqref{estimate} can be used to give appropriate bounds for these four terms. This completes (Step 2-1). (Step 2-2) and (Step 2-3) can be carried out in the same way as in (Step 1-2) and (Step 1-3).

More generally, we can perform (Step $\frak i$-1), (Step $\frak i$-2) and (Step $\frak i$-3) in the case that $\vert \sigma_1\vert ,\vert \sigma_2\vert $ are smaller than a positive number $\epsilon_0$ and obtain a sequence of inconsistent maps:
$$
\frak u^{\xi,\prime}_{\sigma_1,\sigma_2,(\frak i)}: =
(u^{\xi,\prime}_{{\rm d},\sigma_1,\sigma_2,(\frak i)},u^{\xi,\prime}_{{\rm s},\sigma_1,\sigma_2,(\frak i)},U^{\xi,\prime}_{{\rm D},\sigma_1,\sigma_2,(\frak i)},\sigma_1,\sigma_2,\rho_{1,(\frak i)}^{\xi},
\rho_{2,(\frak i)}^{\xi}),
$$
and a triple: 
\[
  \frak e_{\sigma_1,\sigma_2,(\frak i)}^{\xi}= (\frak e_{{\rm d},\sigma_1,\sigma_2,(\frak i)}^{\xi}, \frak e_{{\rm s},\sigma_1,\sigma_2,
  (\frak i)}^{\xi}, \frak e_{{\rm D},\sigma_1,\sigma_2,(\frak i)}^{\xi})\in E_{\rm d}\oplus E_{\rm s} \oplus E_{\rm D}
\] 
such that
\begin{equation}\label{form697}
\aligned
\Vert \overline \partial\frak u^{\xi,\prime}_{\sigma_1,\sigma_2,(\frak i + 1)} - \frak e_{\sigma_1,\sigma_2,(\frak i + 1)}^{\xi}
\Vert^2_{L^2_{m,\delta}}
&\le
\mu
\Vert \overline \partial\frak u^{\xi,\prime}_{\sigma_1,\sigma_2,(\frak i)} - \frak e_{\sigma_1,\sigma_2,(\frak i)}^{\xi}\Vert^2_{L^2_{m,\delta}}
\\
&\le
\mu^{\frak i+1}
\Vert \overline \partial\frak u^{\xi,\prime}_{\sigma_1,\sigma_2,(0)} - \frak e_{\sigma_1,\sigma_2,(0)}^{\xi}\Vert^2_{L^2_{m,\delta}}
\endaligned
\end{equation}
and 
\begin{equation}\label{form698}
	\Vert \frak e_{\sigma_1,\sigma_2,(\frak i+1)}^{\xi} - \frak e_{\sigma_1,\sigma_2,(\frak i)}^{\xi}
	\Vert \le \mu^{\frak i} C.
\end{equation}
Moreover, we have:
\begin{equation}\label{form699}
	d_{W^{2}_{m+1,\delta}}(\frak u^{\xi,\prime}_{\sigma_1,\sigma_2,(\frak i+1)},\frak u^{\xi,\prime}_{\sigma_1,\sigma_2,(\frak i)})
	\le C  \Vert \overline \partial\frak u^{\xi,\prime}_{\sigma_1,\sigma_2,(\frak i)} - \frak e_{\sigma_1,\sigma_2,(\frak i)}^{\xi}\Vert_{L^2_{m,\delta}}.
\end{equation}

We make a remark that the constants $\epsilon_0$ and $C$ may be taken independent of $\frak i$. But these constants might depend on $m$, the exponent in the weighted  Sobolev space $L^2_{m,\delta}$.

The estimates in (\ref{form697}) and (\ref{form699}) imply that the sequence $\{\frak u^{\xi,\prime}_{\sigma_1,\sigma_2,(\frak i)}\}_\frak i$ converges in $W^{2}_{m+1,\delta}$. We denote the limit by:
\begin{equation}\label{constructedfamilyof}
        \aligned
        \frak u^{\xi,\prime}_{\sigma_1,\sigma_2,(\infty)}: =
        (u^{\xi,\prime}_{{\rm d},\sigma_1,\sigma_2,(\infty)},
        &u^{\xi,\prime}_{{\rm s},\sigma_1,\sigma_2,(\infty)},U^{\xi,\prime}_{{\rm D},\sigma_1,\sigma_2,(\infty)},\\
        &\sigma_1,\sigma_2,\rho_{1,(\infty)}^{\xi},
        \rho_{2,(\infty)}^{\xi}).
        \endaligned
\end{equation}

The estimate in (\ref{form698}) implies that $\frak e_{{\rm d},\sigma_1,\sigma_2,(\frak i)}^{\xi} \in E_{\rm d}$, $\frak e_{{\rm s},\sigma_1,\sigma_2,(\frak i)}^{\xi} \in E_{\rm s}$, $\frak e_{{\rm D},\sigma_1,\sigma_2,(\frak i)}^{\xi} \in E_{\rm D}$ converges as $\frak i$ goes to infinity. We denote the limit by:
$$
\frak e_{\sigma_1,\sigma_2,(\infty)}^{\xi}
= (\frak e_{{\rm d},\sigma_1,\sigma_2,(\infty)}^{\xi}, 
\frak e_{{\rm s},\sigma_1,\sigma_2,(\infty)}^{\xi},
\frak e_{{\rm D},\sigma_1,\sigma_2,(\infty)}^{\xi}).
$$
As a consequence of (\ref{form697}), we have:
$$
\Vert
\overline\partial\frak u^{\xi,\prime}_{\sigma_1,\sigma_2,(\infty)} - 
\frak e_{\sigma_1,\sigma_2,(\infty)}^{\xi}\Vert_{L^2_{m},\delta} = 0.
$$
In other words, $\frak u^{\xi,\prime}_{\sigma_1,\sigma_2,(\infty)}$ satisfies (\ref{eq630}). Thus $\frak u^{\xi,\prime}_{\sigma_1,\sigma_2,(\infty)}$ satisfies 
the requirements of an inconsistent solution (Definition \ref{defn615inconsis}) except possibly (\ref{eq631}) (the transversal constraint).

\subsection{Completion of the Proof}
\label{subsub:glucomple}

We are now in the position to complete the proof of 
Proposition \ref{prop617}. For any $\xi \in \mathcal U^+_{\rm d} \,\,{}_{{\rm ev}_{\rm d}}\times_{{\rm ev}^+_{\rm D,d}}  \mathcal U^+_{\rm D}
\,\,{}_{{\rm ev}_{\rm D,s}}\times_{{\rm ev}_{\rm s}}\mathcal U^+_{\rm s}$ and sufficiently small $\sigma_1,\sigma_2 \in \bbC$, we defined an inconsistent map 
$\frak u^{\xi,\prime}_{\sigma_1,\sigma_2,(\infty)}$ in (\ref{constructedfamilyof}). For $(\sigma_1,\sigma_2)$ we define

$$
{\rm EVw}_{\sigma_1,\sigma_2} : 
\mathcal U^+_{\rm d} \,\,{}_{{\rm ev}_{\rm d}}\times_{{\rm ev}^+_{\rm D,d}}  \mathcal U^+_{\rm D}
\,\,{}_{{\rm ev}_{\rm D,s}}\times_{{\rm ev}_{\rm s}}\mathcal U^+_{\rm s}
\to \mathcal D \times X^2
$$
by:
\begin{align*}
	{\rm EVw}_{\sigma_1,\sigma_2}&(\xi)=\\
	&((\pi \circ U^{\xi,\prime}_{{\rm D},\sigma_1,\sigma_2,(\infty)})(w_{{\rm D},1}),
	(\pi \circ U^{\xi,\prime}_{{\rm D},\sigma_1,\sigma_2,(\infty)})(w_{{\rm D},2}),
	U^{\xi,\prime}_{{\rm s},\sigma_1,\sigma_2,(\infty)}(w_{{\rm s}})).
\end{align*}
Here $w_{{\rm D},1}$, $w_{{\rm D},2}$, $w_{{\rm s}}$ are as in 
Definition \ref{conds610}.
\begin{lemma}
${\rm EVw}_{\sigma_1,\sigma_2}$ is transversal to 
$\mathcal N_{\rm D,1} \times \mathcal N_{\rm D,2} \times \mathcal N_{\rm s}$
for sufficiently small $\sigma_1,\sigma_2$.
\end{lemma}
\begin{proof}
	The map ${\rm EVw}_{\sigma_1,\sigma_2}$ converges to ${\rm EVw}_{0,0}$ in the $C^1$ sense as $\sigma_1,\sigma_2 \to 0$. 
	Moreover, ${\rm EVw}_{0,0}$ is
	transversal to $\mathcal N_{\rm D,1} \times \mathcal N_{\rm D,2} \times \mathcal N_{\rm s}$ by 
	assumption (Definition \ref{defn6969}). The lemma follows from these observations.
\end{proof}
By definition 
$$
\bigcup_{\sigma_1,\sigma_2} ({\rm EVw}_{\sigma_1,\sigma_2})^{-1}(\mathcal N_{\rm D,1} \times \mathcal N_{\rm D,2} \times \mathcal N_{\rm s})
\times \{(\sigma_1,\sigma_2)\}
$$
can be identified with $\mathcal U$, the set of inconsistent solutions. We also have:
$$
 ({\rm EVw}_{0,0})^{-1}(\mathcal N_{\rm D,1} \times \mathcal N_{\rm D,2} \times \mathcal N_{\rm s})
 \cong 
 \mathcal U_{\rm d} \,\,{}_{{\rm ev}_{\rm d}}\times_{{\rm ev}_{\rm D,d}}  \mathcal U_{\rm D}
\,\,{}_{{\rm ev}_{\rm D,s}}\times_{{\rm ev}_{\rm s}}\mathcal U_{\rm s}
$$
by definition. Proposition \ref{prop617} is a consequence of these facts.

\par\smallskip

Once Proposition \ref{prop617} is proved the proof of Proposition \ref{prop618} is similar to the proof of \cite[Theorem 6.4]{foooexp}. We have written the proof of Proposition \ref{prop617} so that the construction of the inconsistent solutions are parallel to the gluing construction in  \cite[Section 5]{foooexp}. Therefore, the proof of \cite[Section 6]{foooexp} can be applied with almost no change to prove Proposition \ref{prop618}. This completes the construction of the Kuranishi chart at the point $[\Sigma,z_0,u]$.

\section{Kuranishi Charts: the General Case}
\label{sub:kuraconst}

Up to point, we constructed a Kuranishi chart of the space $\mathcal M_1^{\rm RGW}(L;\beta)$ at the particular point $[\Sigma,z_0,u]$ described in Section \ref{subsec:gluing1}. In this section, we explain how this construction generalizes to an arbitrary point $\frak u$ of $\mathcal M_{k+1}^{\rm RGW}(L;\beta)$. There is a DD-ribbon tree $\mathcal R = (R,c,\alpha,m,\lambda)$ such that $\frak u$ belongs to ${\mathcal M}^{0}(\mathcal R)\subset \mathcal M_{k+1}^{\rm RGW}(L;\beta)$. (See \cite[Subsection 3.4]{part1:top}). Let $((\Sigma_{v},\vec z_v,u_{v});v \in C^{\rm int}_0(\mathcal R))$ be a representative for $\frak u$. In the case that $c(v)={\rm D}$, the image of $u_v$ is contained in $\mathcal D$, and we are given a meromorphic section $U_v$ of $u_v^*\mathcal N_{\mathcal D}(X)$. Recall that for each $i$, the set of all sections $\{U_v\}_{\lambda(v)=i}$ is well-defined up to an action of $\bbC_*$ \cite[Formula (3.78)]{part1:top}. We firstly, associate a combinatorial object to $\frak u$ which is called a {\it very detailed DD-ribbon tree} and is the refinement of the notion of detailed DD-ribbon trees defined in \cite[Subsection 3.4]{part1:top}. 

Let $\hat R$ be the detailed tree associated to $\mathcal R$. Recall that each interior vertex $v$ of $\hat R$ corresponds to a possibly nodal Riemann surface $\Sigma_v$. (See, for example, \cite[Figure 8]{part1:top}.) We refine the detailed DD-ribbon tree $\hat R$ further to the very detailed DD-ribbon tree $\check R$ so that each vertex of $\check R$ corresponds to an irreducible component of $\Sigma$.
To be more detailed, for each $v \in C^{\rm int}_0(\check R)$, we form a tree $\mathcal Q_v$ such that the following holds.
\begin{enumerate}
	\item  Each vertex corresponds to either an irreducible component of $\Sigma_{v}$
		or a marked point on it. The latter corresponds to an edge of $\hat R$, which 
		contains $v$. We call any such vertex an exterior vertex.
	\item There are two types of edges in $\mathcal Q_v$. An edge of the first type
		joins two edges such that the corresponding irreducible components intersect.
		An edge of the second type is called an exterior edge and connects a vertex corresponding 
		to a marked point to the vertex corresponding to the irreducible component 
		containing the marked point.
\end{enumerate}
We replace each interior vertex $v$ of the detailed tree $\hat R$ with $\mathcal Q_v$ and identify exterior edges of $\mathcal Q_v$ with the corresponding edges of $\hat R$ containing  $v$. We thus obtain a  tree $\check R$, called the very detailed DD-ribbon tree associated to $\frak u$, or the very detailed tree associated to $\frak u$ for short. Figure \ref{surfaceforverydetail} sketches an element $\frak u$ of our moduli space. The associated detailed DD-ribbon tree $\hat R$ and the very detailed DD-ribbon tree $\check R$ are given in Figures \ref{treRRRR} and \ref{verydetailedtree}.

\begin{figure}[h]
\includegraphics[scale=0.6]{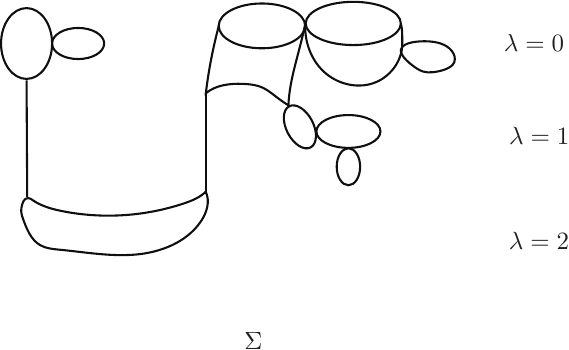}
\caption{An element of the moduli space $\mathcal M_{k+1}^{\rm RGW}(L;\beta)$}
\label{surfaceforverydetail}
\end{figure}
\begin{figure}[h]
\includegraphics[scale=0.6]{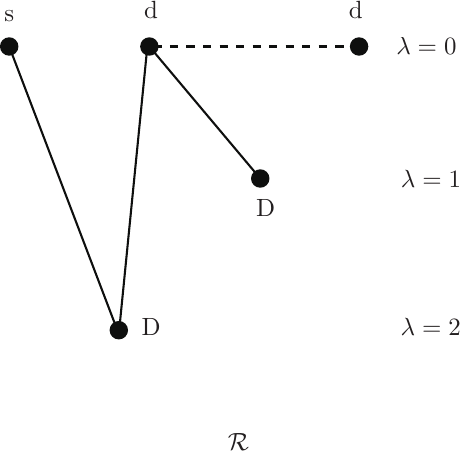}
\caption{The detailed DD-ribbon tree $\hat R$}
\label{treRRRR}
\end{figure}
\begin{figure}[h]
\includegraphics[scale=0.6]{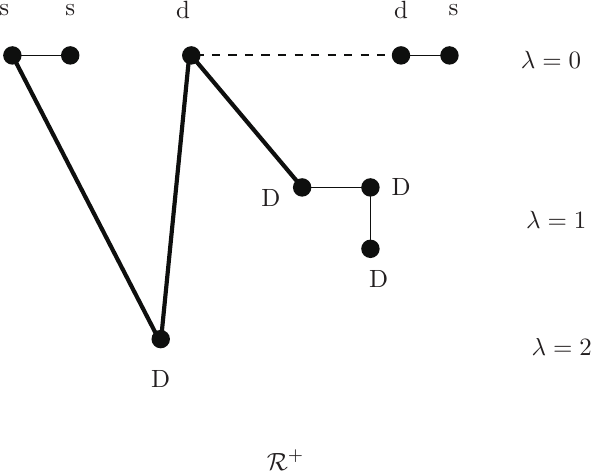}
\caption{The very detailed DD-ribbon tree $\check R$}
\label{verydetailedtree}
\end{figure}

We say an edge of $\check R$ is a {\it fine edge} if it does not correspond to an edge of the detailed DD-ribbon tree $\hat R$. In Figure \ref{verydetailedtree}, the fine edges are illustrated by narrow lines and level $0$ edges are illustrated by dotted lines. An edge of $\check R$, which is not fine, is called a {\it thick edge}. We denote by $C^{\rm int}_{\rm fi}(\check R)$ and $C^{\rm int}_{\rm th}(\check R)$ the set of all fine and thick edges of $\check R$, respectively. The level of a vertex of $\check R$ induced by a vertex of $\mathcal Q_v$ is defined to be $\lambda(v)$. We do not associate a multiplicity number to a fine edge. Homology class of a vertex is the homology class of the map $u$ on this component.
The color of an interior vertex $v$ of $\check R$, denoted by $c(v)$,  is ${\rm D}$ if its level is positive. If this vertex has level $0$, then its color is either ${\rm s}$ or ${\rm d}$ depending on whether $\Sigma_v$ is a sphere or a disk.

The notion of level shrinking and level 0 edge shrinking for very detailed DD-ribbon trees can be defined as in the case of detailed DD-ribbon trees. We define a {\it fine edge shrinking} as follows. We remove a fine edge $e$ and identify the two vertices connected to each other by $e$. For two very detailed DD-ribbon trees $\check R$, $\check R'$, we say $\check R' \le \check R$ if $\check R$ is obtained from $\check R'$ by a sequence of level shrinkings, level 0 edge shrinkings and fine edge shrinkings. Note that there might be a fine edge joining two vertices of level $0$. We do {\it not} call any such edge a level $0$ edge. The level $0$ edges are limited to those joining vertices of color ${\rm d}$.

We can stratify the moduli space $\mathcal M^{\rm RGW}_{k+1}(L;\beta)$ using very detailed DD-ribbon trees $\check R$. Namely, we define $\mathcal M^{\rm RGW}_{k+1}(L;\beta)(\check R)$ to be the subset of $\mathcal M^{\rm RGW}_{k+1}(L;\beta)$ consists of elements $\frak u$ whose associated very detailed DD-ribbon tree is $\check R$. If $\check R' \le \check R$, then the closure of $\mathcal M^{\rm RGW}_{k+1}(L;\beta)(\check R)$ contains $\mathcal M^{\rm RGW}_{k+1}(L;\beta)(\check R')$. 

%Given an element $\frak u=(\Sigma,\vec z,u))$ of $\mathcal M^{\rm RGW}_{k+1}(L;\beta)(\check R)$, we denote the irreducible component of $\Sigma$ corresponding to the interior vertex $v$ of $\check R$ with $\Sigma_v$. The restriction of the pseudo-holomorphic map $u$ to $\Sigma_v$ is denoted by $u_v$. We will write $\vec z_v$ for the union of nodal points and the boundary marked points of $\Sigma_{v}$. If $c(v)={\rm D}$, then $u_v$ maps $\Sigma_v$ to $\mathcal D$. In this case, $u_v$ also determines a map to $\mathcal N_{\mathcal D}(X)$, which we denote by $U_v$. By the definition of $\mathcal M^{\rm RGW}_{k+1}(L;\beta)$, the triple $\frak u_v = (\Sigma_{v},\vec z_v,u_{v})$ is stable for any $v$. 

Let $\frak u=((\Sigma_{v},\vec z_v,u_{v});v \in C^{\rm int}_0(\check R))$ be an element of $\mathcal M^{\rm RGW}_{k+1}(L;\beta)(\check R)$ as above. For an interior vertex $v$ of $\check R$, the triple $\frak u_v = (\Sigma_{v},\vec z_v,u_{v})$ is stable by definition. If $(\Sigma_{v},\vec z_v)$ is not stable, then we may add auxiliary interior marked points $\vec w_{v}$ so that $(\Sigma_{v},\vec z_v\cup \vec w_v)$ is stable. We assume that $\vec w_v$ is chosen such that the following symmetry assumption holds. Suppose $\Gamma_{\frak u}$ denotes the group of automorphisms of $\frak u$. Given $\gamma \in \Gamma_{\frak u}$, for each interior vertex $v$, there exists a vertex $\gamma(v)$ and a bi-holomorphic map $\gamma_v: (\Sigma_{v},\vec z_v) \to  (\Sigma_{\gamma(v)},\vec z_{\gamma(v)})$ such that $u_{\gamma(v)}\circ \gamma_v = u_v$. We assume $\vec w_v$ is mapped to $\vec w_{\gamma(v)}$ via $\gamma_v$. Note that the case $\gamma(v) = v$ is also included. For each member $w_{v,i}$ of $\vec w_v$, we take a codimension 2 submanifold $\mathcal N_{v,i}$ of $X$ (resp. $\mathcal D$) if $c(v) = {\rm d}$ or $\rm s$ (resp. if $c(v) = {\rm D}$.) We assume that the same condition as Condition \ref{conds610} holds for these choices of {\it transversals}. If $\gamma \in \Gamma_{\frak u}$ and $\gamma_v(w_{v,i})=w_{v',i'}$, then we require that $\mathcal N_{v,i} = \mathcal N_{v',i'}$.

In order to define Cauchy-Riemann operators, we introduce function spaces similar to those of Section \ref{sub:Fred}. For an interior vertex $v$ of $\check R$, if the color of $v$ is ${\rm d}$, ${\rm s}$ or ${\rm D}$, the Hilbert space $W^2_{m,\delta}(\frak u_v;T)$ is respectively defined as in Definition \ref{defn6262}, Definition \ref{defn64444} or Definition \ref{defn66666}. Here $T$ is a placeholder for the pull-back of the tangent bundle of $X$ (if $c(v)={\rm d}$,  ${\rm s}$) or $\mathcal N_{\mathcal D}X \setminus \mathcal D$ (if $c(v)={\rm D}$). Similarly, for each $v$, we define the Weighted Sobolev spaces $L^2_{m,\delta}(\frak u_v;T \otimes \Lambda^{0,1})$ as in Section \ref{sub:Fred}.

\begin{remark}
	In  the case that $c(v) = \rm d$, the space $\Sigma_{v}$ may have boundary nodes.
	In that case we take cylindrical coordinates on a neighborhood of each 
	boundary node and use a cylindrical metric on this neighborhood. 
	The approach here is very similar to the case of interior nodes which is discussed in 
	Section \ref{sub:Fred}. 
	See \cite{foooexp} for the case of boundary nodes 
	in the context of the stable map compactification.
	
	We also need to fix cylindrical coordinates for nodes corresponding to 
	fine edges. In this case the target
	of the corresponding cylindrical end is contained in a compact subset of 
	$X \setminus \mathcal D$ (for fine edges connecting level $0$ vertices)
	or $\mathcal N_{\mathcal D}X \setminus \mathcal D$ 
	(for fine edges connecting positive level vertices).
	In the first case we use the metric $g$ given in Section \ref{sub:Fred}. 
	In the latter case, we use the metric
	on $\mathcal N_{\mathcal D}X \setminus \mathcal D$ with the form given in \eqref{g-cylinder-end}.
\end{remark}

Let $W^{2,\sim}_{m,\delta}(\frak u;T)$ be the subspace of the direct sum:
\begin{equation}\label{form6101}
	\bigoplus_{v \in C^{\rm int}_0(\check R)} W^2_{m,\delta}(\frak u_v;T)
\end{equation}
consisting of elements $(V_{v};v \in  C^{\rm int}_0(\check R))$ with the following properties. Let $e$ be an interior edge of $\check R$ joining $v_1$ and $v_2$. The source curve of the element $\frak u_{v_i}$ contains a nodal point $z_{v_i,e}$ corresponding to the edge $e$. Suppose $e$ is not a level 0 edge or a fine edge. By definition $V_{v_i}$ has an asymptotic value $(\frak r_{v_i,e},\frak s_{v_i,e},v_{v_i,e})\in \bbR\oplus \bbR \oplus T_{p_{v_i,e}}\mathcal D$ where $p_{v_i,e}$ is the point of $\mathcal D$ such that $u_v(z_{v_i,e}) = p_{v_i,e}$ and $\bbR\oplus \bbR$ corresponds to the tangent space of the partial $\bbC_*$-action. (See Definition \ref{defn6262}.) We require:
\[
v_{v_1,e} = v_{v_2,e}.
\]
This condition is the counterpart of part (4) of Definition \ref{defn62888}. In the case of a level 0 edge (resp. a fine edge), the corresponding asymptotic values are tangent vectors of $L$ (resp. tangent vectors of $X$ or $\mathcal N_{X}\mathcal D \setminus \mathcal D$) and we require that these two tangent vectors agree with each other. (See \cite[Definition 3.4]{foooexp}.) 

Analogous to Definition \ref{defn62888}, there is an action of $(\bbR\oplus \bbR)^{{\vert}\lambda{\vert}}$ on $W^{2,\sim}_{m,\delta}(\frak u;T)$ with ${\vert}\lambda{\vert}$ being the number of levels of $\check R$. We define $W^{2}_{m,\delta}(\frak u;T)$ to be the quotient space with respect to this action. We also write $L^2_{m,\delta}(\frak u;T\otimes \Lambda^{0,1})$ for the direct sum of $L^2_{m,\delta}(\frak u_v;T \otimes \Lambda^{0,1})$ 
for $v  \in C^{\rm int}_0(\check R)$.
%Let $W^2_{m,\delta}(\frak u;T)$ be the subspace of the direct sum:
%\begin{equation}\label{form6101}
%	\bigoplus_{v \in C^{\rm int}_0(\check R)} W^2_{m,\delta}(\frak u_v;T)
%\end{equation}
%consisting of elements $(V_{v};v \in  C^{\rm int}_0(\check R))$ with the following properties. Let $e$ be an interior edge of $\check R$ joining $v_1$ and $v_2$. The source curve of the element $\frak u_{v_i}$ contains a nodal point $z_{v_i,e}$ corresponding to the edge $e$. Suppose $e$ is not a level 0 edge or a fine edge. By definition $V_{v_i}$ has an asymptotic value $(\frak r_{v_i,e},\frak s_{v_i,e},v_{v_i,e})\in \bbR\oplus \bbR \oplus T_{p_{v_i,e}}\mathcal D$ where $p_{v_i,e}$ is the point of $\mathcal D$ such that $u_v(z_{v_i,e}) = p_{v_i,e}$ and $\bbR\oplus \bbR$ corresponds to the tangent space of the partial $\bbC_*$-action. (See Definition \ref{defn6262}.) We require:
%\[
%v_{v_1,e} = v_{v_2,e}.
%\]
%This condition is the counterpart of part (4) of Definition \ref{defn62888}. In the case of a level 0 edge (resp. a fine edge), the corresponding asymptotic values are tangent vectors of $L$ (resp. tangent vectors of $X$ or $\mathcal N_{X}\mathcal D \setminus \mathcal D$) and we require that these two tangent vectors agree with each other. (See \cite[Definition 3.4]{foooexp}.) We also write $L^2_{m,\delta}(\frak u;T\otimes \Lambda^{0,1})$ for the direct sum of $L^2_{m,\delta}(\frak u_v;T \otimes \Lambda^{0,1})$ 
%for $v  \in C^{\rm int}_0(\check R)$.

The linearization of the Cauchy-Riemann equation associated to each vertex $v$ of the very detailed tree $\check R$, defines the linear operator: 
\[
  D_{\frak u_v}\overline \partial:W^2_{m+1,\delta}(\frak u_v;T) \to 
  L^2_{m,\delta}(\frak u_v;T \otimes \Lambda^{0,1}).
\]
The direct sum of these operators together determines a Fredholm operator:
\begin{equation}\label{form6103}
	D_{\frak u}\overline\partial :
	W^2_{m+1,\delta}(\frak u;T) \to L^2_{m,\delta}(\frak u;T \otimes \Lambda^{0,1}).
\end{equation}
In the case that this operator is not surjective, we need to introduce obstruction spaces as in Section \ref{sub:Obst}. 
\begin{definition}\label{cond643}
{\it Obstraction data} $E: = \{E_v\}$ assign,
for each interior vertex $v$, a vector space $E_v$ such that the following conditions are satisfied:

\begin{enumerate}
	\item $E_{v}$ is a finite dimensional subspace of 
	$L^2_{m,\delta}(\Sigma_v\setminus \vec z_v;u_v^*T(X\setminus \mathcal D) \otimes \Lambda^{0,1})$
	if $c(v)={\rm d}$ or ${\rm s}$, and is a finite dimensional subspace of 
	$L^2_{m}(\Sigma_v;u_v^*T\mathcal D \otimes \Lambda^{0,1})$
	if $c(v)={\rm D}$. Moreover, $E_{v}$ consists of smooth sections.
	Using the decomposition in \eqref{decom-tan-bdle}, we can also regard $E_v$ as a subspace of
	$L^2_{m,\delta}(\frak u_v;T \otimes \Lambda^{0,1})$.
	\item Elements of $E_v$ have compact supports away from  nodal points and boundary.
	\item If $u_{v}$ is a constant map, then $E_v$  is $0$.
	\item If $\gamma \in \Gamma_{\frak u}$, then
		$$
		  (\gamma_v)_* E_v = E_{\gamma(v)}.
		$$ 
		Here $(\gamma_v)_* : L^2_{m,\delta}(\frak u_v;T \otimes \Lambda^{0,1})
		\to L^2_{m,\delta}(\frak u_{\gamma(v)};T \otimes \Lambda^{0,1})$
		is the map induced by $\gamma_v : \Sigma_v \to \Sigma_{\gamma(v)}$. 
		(Recall that $u_{\gamma(v)} \circ \gamma_v = u_v$.)
		\item The operator $D_{\frak u}\overline\partial$ in \eqref{form6103} is transversal to:
			\begin{equation}
				E_0 = \bigoplus_{v \in C^{\rm int}_0(\check R)} E_v 
				\subset L^2_{m,\delta}(\frak u;T \otimes \Lambda^{0,1}).
			\end{equation}
\end{enumerate}
\end{definition}

It is straightforward to see that there are obstruction data satisfying the conditions in Definition \ref{cond643}. Since each operator $D_{\frak u_v}\overline \partial$ is Fredholm, we can fix $E_v$, which satisfies part (1). This choice also would imply the required transversality in part (5). In the case that $u_v$ is constant, we can pick $E_v$ to be the trivial vector space because $\Sigma_v$ has genus $0$. (It is either a disk or a sphere.) Unique continuation implies that we can assume that the supports of the elements of $E_v$ is contained in a compact subset of $\Sigma_v$ away from the nodal points and boundary. By taking direct sums over the action of $\Gamma_{\frak u}$ if necessary, we may also assume that (4) holds. Using $E_0$, we can define a thickened moduli space which gives a Kuranishi neighborhood of $\frak u$ in a stratum of $\mathcal M_{k+1}^{\rm RGW}(L;\beta)$ which contains $\frak u$. (See Definition \ref{defn64545}.) In the upcoming sections, we give a systematic construction of the obstruction spaces $E_0$ which satisfy further compatibility assumptions.

We next discuss the process of gluing the irreducible components of $\frak u$. We firstly need to explain how the deformation of source curves is parametrized. The mathematical content here is classical and we follow the approaches in \cite[Section 8]{foooexp} and \cite[Section 3]{fooo:const1}.

For an interior vertex $v$, we consider $(\Sigma_v,\vec z_v \cup \vec w_v)$. This is a disk or a sphere with marked points, which is stable. We may regard it as an element of the moduli space $\mathcal M_{v}^{\rm source}$. The space $\mathcal M_{v}^{\rm source}$ is metrizable and we fix one metric on it for our purposes later. The moduli space $\mathcal M_{v}^{\rm source}$ comes with a universal family: 
\begin{equation}\label{sourceunifami}
	\pi : \mathcal {C}_{v}^{\rm source}\to \mathcal M_{v}^{\rm source}
\end{equation}
and $\#\vec z_v + \#\vec w_v$ sections which are in correspondence with the marked points. (See, for example, \cite[Section 2]{fooo:const1}.) For $\frak x \in \mathcal M_{v}^{\rm source}$, the fiber $\pi^{-1}(\frak x)$ together with the values of the sections at $\frak x$ determines a representative for $\frak x$, which we denote it by $(\Sigma_{\frak x,v},\vec z_{\frak x,v} \cup \vec w_{\frak x,v})$.

Since $\Sigma_v$ has no singularity, \eqref{sourceunifami} is a $C^\infty$-fiber bundle near the point $(\Sigma_v,\vec z_v \cup \vec w_v)$. We fix a neighborhood $\mathcal V_{v}^{{\rm source}}$ of 
$[\Sigma_{v},\vec z_{v} \cup \vec w_{v}]$ %$[\Sigma_{\frak x,v},\vec z_{\frak x,v} \cup \vec w_{\frak x,v}]$ 
and a trivialization 
\begin{equation}\label{form10505}
\phi_v : \mathcal V_{v}^{{\rm source}} \times \Sigma_v \to \mathcal {C}_{v}^{\rm source}.
\end{equation}
of  \eqref{sourceunifami} over this neighborhood.
%Here a trivialization means a local trivialization of the fiber bundle in $C^{\infty}$ category.
We assume that these trivializations are compatible with the automorphisms of $\frak u$. For $\frak x_{v} \in \mathcal V_{v}^{{\rm source}}$, we define a complex structure $j_{\frak x_v}$ on $\Sigma_v$ such that 
the restriction of the trivialization \eqref{form10505} to $\{\frak x_{v}\} \times \Sigma_v$ defines a bi-holomorphic map 
$(\Sigma_v,j_{\frak x_v}) \to \pi^{-1}(\frak x_{v})$.

%We next define a family of coordinates at the nodal points.
Let $e$ be an interior edge of $\check R$ containing $v$. There is a nodal point of  $(\Sigma_v,\vec z_v)$ associated to $e$. Let $\frak s_{v,e}$ be the section of \eqref{sourceunifami}
corresponding to this marked point.
In the case of an interior node, an 
{\it analytic family of coordinates} at this nodal point is a holomorphic map
\begin{equation}\label{form19009}
	\varphi_{v,e} :   \mathcal V_{v}^{{\rm source}}  \times {\rm Int} (D^2)\to \mathcal {C}_{v}^{\rm source}
\end{equation}
such that for each $\frak x\in \mathcal V_{v}^{{\rm source}} $, we have $\varphi_{v,e}(\frak x,0) = \frak s_{v,e}(\frak x)$ and the restriction of $\varphi_{v,e}$ to $\{\frak x\} \times {\rm Int} D^2$ determines a holomorphic coordinate for $\Sigma_{\frak x} = \pi^{-1}(\frak x)$ around  $\varphi_{v,e}(\frak x,0) = \frak s_{v,e}(\frak x)$. Thus $\varphi_{v,e}$ commutes with the projection map to $\mathcal {M}_{v}^{\rm source}$ and $\varphi_{v,e}$ is a bi-holomorphic map onto an open subset of $\mathcal {C}_{v}^{\rm source}$. When $z_{v,e}$ is a boundary node, we replace $D^2$ by $D^2_+ = \{z \in D^2 \mid {\rm Im} z \ge 0\}$ and define the notion of analytic family of coordinates in a similar way. (See \cite[Section 3]{fooo:const1} for more details.) We require that the images of the maps $\varphi_{v,e}$ are disjoint and away from the image of the sections of $\mathcal {C}_{v}^{\rm source}$ corresponding to the auxiliary marked points $\vec w_{v}$. We also assume that the chosen analytic families are compatibile with automorphisms of $\frak u$.

We use analytic family of coordinates $\varphi_{v,e}$ to desingularize the nodal points as follows. Fix an element:
\begin{equation}\label{vectorsigma}
	\vec\sigma = (\sigma_e ; e \in  C^{\rm int}_{1}(\check R)) \in\prod_{e  \in  C^{\rm int}_{1}(\check R)}\mathcal V_{e}^{{\rm deform}}.
\end{equation}
Here $\sigma_e \in D^2 =: \mathcal V_{e}^{{\rm deform}}$ if $z_{v,e}$ is an interior node, and $\sigma_e \in [0,1] =:  \mathcal V_{e}^{{\rm deform}}$ if $z_{v,e}$ is a boundary node.\footnote{Note that $z_{v,e}$ is a  boundary node if and only if $e$ is a level $0$ edge.} Let
\begin{equation}\label{defnuniverse}
	\vec{\frak x} = (\frak x_v ; v  \in  C^{\rm int}_0(\check R))\in \prod_{v  \in  C^{\rm int}_0(\check R)}\mathcal V_{v}^{{\rm source}}.
\end{equation}
We put
\begin{equation}\label{form610}
\aligned
\Sigma^+_{v}(\vec{\frak x},\vec \sigma)
= 
\Sigma_{\frak x_v,v} &\setminus \bigcup_{(v,e) : 
\text{$e$ is not a level $0$ edge}} \varphi_{v,e}(\frak x_v,D^2(\vert \sigma_{e}\vert))
\\
&\setminus \bigcup_{(v,e) : \text{$e$ is a level $0$ edge}} \varphi_{v,e}(\frak x_v,D_+^2(\sigma_{e}))
\\
\Sigma^-_{v}(\vec{\frak x},\vec \sigma) = 
\Sigma_{\frak x_v,v} &\setminus \bigcup_{(v,e) : \text{$e$ is not a level $0$ edge}} \varphi_{v,e}(\frak x_v,D^2(1))
\\
&\setminus \bigcup_{(v,e) : \text{$e$ is a level $0$ edge}} \varphi_{v,e}(\frak x_v,D_+^2(1)).
\endaligned
\end{equation}
Recall that a fine edge $e$ connecting two level $0$ vertices is not a level $0$ edge by definition.The auxiliary marked points $\vec w_v$ determine a set of marked points on $\Sigma^-_{v}(\vec{\frak x},\vec \sigma)$, which is also denoted by $\vec w_v$.

We define an equivalence relation $\sim$ on:
\begin{equation}\label{form6107}
	\bigcup_{v \in C^{\rm int}_0(\check R)}\Sigma^+_{v}(\vec{\frak x},\vec \sigma)
\end{equation}
as follows. Let $e$ be an edge which is not a level $0$ edge and connects the vertices $v_1,v_2$. Suppose $z_1,z_2 \in {\rm Int}(D^2)$ with $z_1z_2 = \sigma_e$. Then:
$$
\varphi_{v_1,e}(\frak x_{v_1},z_1) \sim \varphi_{v_2,e}(\frak x_{v_2},z_2).
$$
Let $e$ be a level $0$ edge connecting the vertices $v_1,v_2$. Suppose $z_1,z_2 \in {\rm Int}(D_+^2)$ with $z_1z_2 =-\sigma_e$. Then:
$$
\varphi_{v_1,e}(\frak x_{v_1},z_1) \sim \varphi_{v_2,e}(\frak x_{v_2},z_2).
$$
We divide the space (\ref{form6107}) by the equivalence relation
$\sim$ and denote the quotient space by
\begin{equation}
\Sigma(\vec{\frak x},\vec \sigma).
\end{equation}

Let:
$$
\aligned
\sigma_e &= \exp(-(10T_e + \theta_e\sqrt{-1})) \qquad\qquad
&\text{$e$ is not a level $0$ edge}, \\
\sigma_e &= \exp(-10T_e) \qquad\qquad
&\text{$e$ is a level $0$ edge}.
\endaligned
$$
For each $e \in C^{\rm int}_1(\check R)$, there is a corresponding neck region in $\Sigma(\vec{\frak x},\vec \sigma)$. We define coordinates $r_e$, $s_e$ on this region as follows. Suppose $e$ is not a level $0$ edge.
We choose $v_1,v_2$ so that $v_1$ and the root of $\check R$ (corresponding to the zero-th exterior marked point $z_0$ of $\frak u$) are in the same connected component of $\check R \setminus e$. Let:
$$
\Sigma^+_{v}(\vec{\frak x},\vec \sigma) 
\setminus \Sigma^-_{v}(\vec{\frak x},\vec \sigma)
=
[-5T_e,5T_e]_{r_e} \times S^1_{s_e}
$$
where
$$
  \hspace{3cm}(-5T_e,s_e)  \in {\rm Closure}\,(\Sigma^-_{v_1}(\vec{\frak x},\vec \sigma))\hspace{1cm}\forall s_e\in S^1.
$$
The coordinate $r_e, s_e$ is defined in the same way as in (\ref{form643}), (\ref{form644}). 

If {$e$ is a level $0$ edge}, then we take $v_1$, $v_2$ so that $v_1$ and the root of $\check R$ are in the same connected component of $\check R \setminus e$. Then  $\Sigma^+_{v_1}(\vec{\frak x},\vec \sigma) \setminus \Sigma^-_{v_1}(\vec{\frak x},\vec \sigma)$ has a connected component corresponding to each edge which is incident to $v_1$. We identify the connected component corresponding to the edge $e$ with:
\begin{equation}\label{rectangle}
[-5T_e,5T_e]_{r_e} \times [0,\pi]_{s_e}
\end{equation}
where the point $\varphi_{v_1,e}(\frak x_v,\exp(-(r_e+5T_e)-\sqrt{-1}s_e))$ is identified with $(r_e,s_e)$ in \eqref{rectangle}. Similarly, $\Sigma^+_{v_2}(\vec{\frak x},\vec \sigma) \setminus \Sigma^-_{v_2}(\vec{\frak x},\vec \sigma)$ has a connected component corresponding to each edge which is incident to $v_2$. We identify the connected component corresponding to the edge $e$ with \eqref{rectangle} where the point $\varphi_{v_2,e}(\frak x_v,\exp((r_e-5T_e)+\sqrt{-1}s_e))$ is identified with $(r_e,s_e)$. These identifications are compatible with the equivalence relation $\sim$.

We thus have the decomposition:
\begin{equation}\label{form6110}
\aligned
\Sigma(\vec{\frak x},\vec \sigma)
= 
&\bigcup_{v\in  C^{\rm int}_0(\check R)}\Sigma^-_{v}(\vec{\frak x},\vec \sigma)\\
&\cup
\bigcup_{e\in  C^{\rm int}_1(\check R), \text{$e$ is not a level $0$ edge}}[-5T_e,5T_e]_{r_e} \times S^1_{s_e}
\\
&\cup 
\bigcup_{e\in  C^{\rm int}_1(\check R), \text{$e$ is a level $0$ edge}}[-5T_e,5T_e]_{r_e} \times [0,\pi]_{s_e}.
\endaligned
\end{equation}
This is the thick and thin decomposition which is used frequently in various kinds of Gromov-Witten theory. The inclusion of $\vec w_v$ in $\Sigma^-_{v}(\vec{\frak x},\vec \sigma)$ induces a set of marked points in $\Sigma(\vec{\frak x},\vec \sigma)$, which is also denoted by $\vec w_v$.

\begin{definition}\label{TSD}
We call $\Xi= (\vec w_v,(\mathcal N_{v,i}),(\phi_{v}),(\varphi_{v,e}),\kappa)$ 
	a choice of {\it trivialization and stabilization data} ({\it TSD}) for $\frak u$.
	Here $w_{v,i}$ is  the choice of the additional marked points
	and $\mathcal N_{v,i}$ the set of transversals, $\phi_{v}$ 
	the trivializations of the universal family
	 and analytic family of coordinates $\varphi_{v,e}$.
	The size of $\Xi$ is the sum of a small number $\kappa$, 
	the diameters of $\mathcal V_v^{\rm source}$ and the images of the maps $\varphi_{v,e}$. 
	When we say $\Xi$ is {\it small enough},
	we mean that the size of $\Xi$ is small enough.
	The way they are used is explained in Definition \ref{defn64545}
	below.
	\par
	We call a pair $(\Xi,E)$ of trivialization and stabilization data 
	$\Xi$ and obstraction data $E: = \{E_v\}$ 
	(as in Definition \ref{cond643}), 
	trivialization, stalibization and obstruction data. 
	(TSO).
	\par
	We remark in Section \ref{sub:Obst} we introduced 
	the notions of stabilization data and of stabilization 
	and obstruction data. In the situation of Section \ref{sub:Obst}
	there is an obvious choice of trivializations and coordinates at the nodes.
\end{definition}

Now we introduce several thickened moduli spaces which are used in the definition of our Kuranishi structures. We firstly define the stratum corresponding to $\check R$:
\begin{definition}\label{defn64545}
        Given a TSO $\Upsilon = (\Xi,E)$ 
        ($\Xi= (\vec w_v,(\mathcal N_{v,i}),(\phi_{v}),(\varphi_{v,e}),\kappa)$), the space 
        $\widetilde{\mathcal U}(\frak u,{\check R},\Upsilon)$
        consists of triples $(\vec{\frak x},u',U')$
        with the following properties:
        \begin{enumerate}
        \item $\vec{\frak x} = (\frak x_v ; v  \in  C^{\rm int}_0(\check R))\in 
        \prod_{v  \in  C^{\rm int}_0(\check R)}\mathcal V_{v}^{{\rm source}}$.
         For each interior vertex $v$, $\frak x_v $ belongs to the 
	$\kappa$-neighborhood of the point of $\mathcal V_{v}^{{\rm source}}$ 
	induced by $\frak u$. 
        A representative 
        $(\Sigma_{v},\vec z_{\frak x,v} \cup \vec w_{\frak x,v})$ of $\frak x_v$ is also given
        where the irreducible component $\Sigma_{v}$ is equipped with an almost complex structure 
        $j_{\frak x_v}$.
        \item $u' : \Sigma \to X$ is a continuous map,
        whose restriction $u'_v$ to $\Sigma_v$ is smooth.
	If $c(v)={\rm D}$, then we require that $u'(\Sigma_v)\subset \mathcal D$.
        Moreover, a meromorphic\footnote{The complex line bundle $(u_v')^*(\mathcal N_{\mathcal D}(X))$
        becomes a holomorphic line bundle by using the $(0,1)$ part of the (given) $U(1)$ connection. This 
        is because the base space is complex one dimensional.} section $U_v'$ of $(u_v')^*(\mathcal N_{\mathcal D}(X))$
        is also fixed such that the data of the zeros and poles of $U_v'$ is determined by 
        the multiplicity of the thick edges connected to $v$.
        If $c(v)={\rm d}$, then the restriction of $u'_v$ to the boundary of the
        disc $\Sigma_v$ is mapped to $L$.
       \item 
        The $C^2$-distance\footnote{If $c(v)={\rm d}$ or ${\rm s}$, then the $C^2$-distance is defined
        using the metric $g$ on $X$, and if $c(v)={\rm D}$, then the $C^2$-distance is defined
        using the metric $g'$ on $\mathcal D$.} between $u'_v$ and $u_v$ is less than $\kappa$.
        If $c(v)={\rm D}$, then the $C^2$-distance\footnote{The 
        $C^2$-distance is defined with respect to a 
        metric which has the form given in \eqref{g-cylinder-end}. 
        Note that the set of sections $\{U_v\}_v$ is defined up to action of $\bbC_*^{{\vert}\lambda{\vert}}$, and 
        here we mean than there is a representative for $\{U_v\}_v$ such that 
        the distance between $U'_v$ and $U_v$ is less than $\kappa$.} 
        between $U'_v$ and $U_v$ is less than $\kappa$.
        \item
        We require
        \begin{equation}
        \overline\partial_{j_{\frak x_v}} u'_v \in E_v(u'_v).
        \end{equation}
	Here $j_{\frak x_v}$ is the complex structure of $\Sigma_v$ corresponding 
        to $\frak x_v$. (See the discussion proceeding (\ref{form10505}).)
       Using the complex structure $j_{\frak x_v}$ on $\Sigma_v$, we may define the target parallel transportation 
        in the same way as in Section \ref{sub:Obst}, and obtain $E_v(u'_v)$ from $E_v$.
        (We will explain the definition of $E_v(u'_v)$ more after Definition \ref{defn64545}.)
        We also require:
        \begin{equation}
        		\overline\partial_{j_{\frak x_v}} U'_v \in E_v(u'_v).
        \end{equation}
        if $c(v)={\rm D}$. Here we use \eqref{decom-tan-bdle} to regard $E_v(u'_v)$ as a subspace
        of the function space $L^2_{m,\delta}(\Sigma_v;(U_v')^*T\mathcal N_{\mathcal D}(X) \otimes \Lambda^{0,1})$.
        \item
        We require
        \begin{equation}\label{trans-cons}
        u'_v(w_{v,i}) \in \mathcal N_{v,i}.
        \end{equation}
        \item
        If $e$ is a fine edge connecting vertices $v_1$ and $v_2$ with color d or s, then 
        the values of $u_{v_1}'$ and $u_{v_2}'$ at the node $\Sigma_{v_1}\cap \Sigma_{v_2}$
        are equal to each other. 
	If $e$ is a fine edge connecting vertices $v_1$ and $v_2$ with color D, then 
        the values of $U_{v_1}'$ and $U_{v_2}'$ at the node $\Sigma_{v_1}\cap \Sigma_{v_2}$
        are equal to each other.
%        For each interior vertex $v$ the number $c_v$ is a nonzero complex number.
%        If $v,v'$ are joined by a fine edge then $c_v = c_{v'}$.\AD{\color{red} Is it also defined in the case that the color of $v$ is not ${\rm D}$? }
        \end{enumerate}
        \par
        We define an equivalence relation $\sim$ on 
        $\widetilde{\mathcal U}(\frak u, {\check R},\Upsilon)$ as follows.
        Let $\vert\lambda\vert$ be the number of levels of the DD-ribbon tree associated to $\frak u$. 
        For $i=1,\dots,\vert\lambda\vert$, we take $a_i \in \bbC_*$.
        We define $(\vec{\frak x},u',U_{(0)}') \sim (\vec{\frak x},u',U_{(1)}')$ 
        \begin{equation}
		U_{(1),v}' = {\rm Dil}_{a_{\lambda(v)}}\circ U_{(0),v}'.
        \end{equation}
        We denote the quotient space with respect to this equivalence relation by 
        $\widehat{\mathcal U}(\frak u,{\check R},\Upsilon)$.
        The group of automorphisms $\Gamma_{\frak u}$ acts on 
        $\widehat{\mathcal U}(\frak u,{\check R},\Upsilon)$ in an obvious way.
        We write ${\mathcal U}(\frak u,{\check R},\Upsilon)$ for
      	the quotient space 
        $\widehat{\mathcal U}(\frak u,{\check R},\Upsilon)/\Gamma_{\frak u}$.
\par
We replace the condition (4) to  
\begin{equation}
        \overline\partial_{j_{\frak x_v}} u'_v =0,
        \qquad
        \overline\partial_{j_{\frak x_v}} U'_v =0
\end{equation}
to define $\widetilde{\frak U}(\frak u,{\check R},\Xi)$, $\widehat{\frak U}(\frak u,{\check R},\Xi)$
and ${\frak U}(\frak u,{\check R},\Xi)$.
Note that they depend only on TSD $\Xi$ and is independent of 
the choice of $E$.
\end{definition}

We now elaborate on the definition of $E_v(u'_v)$.
Let $z \in \Sigma_v$ be a point in the support of $E_v$. 
There exists a map $I^0_v: {\rm Supp} E_v \to \Sigma'_v$ since $\Sigma'_v$ is close to
$\Sigma_v$. 
The map $I^0_v$ is not much canonical. However we can 
modify it to more canonical map $I^t_v : {\rm Supp} E_v \to \Sigma$ as follows.
There exists a unique point $I^t_v(z)$ in $\Sigma$ which is close to $I^0_v$ and 
that the minimal geodesic joining $u_v(z)$ to $u'_v(I_v(z))$ is perpendicular to the image of $u_v$.
Note that this condition is a generalization of Conditions (1)(2) 
of Definition \ref{defn481}, where the
identity map (that is not canonical) played the role of $I^0_v$.
Now using $I^t_v$ we define $E_v(u'_v)$ in the same way as (\ref{formula621}).
\par
%The numbers $c_v$ parametrize the degree of freedom to associate $U_v : \Sigma_v \to \mathcal N_{\mathcal D}X$ to a given $u_v :  \Sigma_v \to \mathcal D$ in the case that $c(v) = {\rm D}$. The equivalence relation $\sim$ is similar to \cite[\eqref{quotientbyci}]{part1:top}.
%\par
The space ${\mathcal U}(\frak u,{\check R},\Upsilon)$ is a generalization of $\mathcal U_{\rm d} \,\,{}_{{\rm ev}_{\rm d}}\times_{{\rm ev}_{\rm D,d}}  \mathcal U_{\rm D}\,\,{}_{{\rm ev}_{\rm D,s}}\times_{{\rm ev}_{\rm s}}\mathcal U_{\rm s}$ appearing in (\ref{form633}) and is a thickened version of a neighborhood of $\frak u$ in the stratum ${\mathcal M}^0(\mathcal R)$ of $\mathcal M^{\rm RGW}_{k+1}(L;\beta)$ defined in  \cite[(3.78)]{part1:top}. The following lemma is a consequence of Definition \ref{cond643} and the implicit function theorem.

\begin{lemma}
	If $\kappa$ is small enough, then $\widehat{\mathcal U}(\frak u,{\check R},\Upsilon)$ is a smooth 
	manifold and ${\mathcal U}(\frak u,{\check R},\Upsilon)$ is a smooth orbifold.
\end{lemma}

We next introduce the generalization of the space 
$\mathcal U_0$ in Definition \ref{defnn614}.

\begin{definition}\label{defn647}
        Let $\Xi= (\vec w_v,(\mathcal N_{v,i}),(\phi_{v}),(\varphi_{v,e}),\kappa)$ be a TSD at
        the element $\frak u$ of $\mathcal M_{k+1}^{\rm RGW}(L,\beta)$
        and $\Upsilon = (\Xi,E)$ 
        a TSO.
        The space $\widehat{\mathcal U}_0(\frak u,\Upsilon)$ consists of 
        $(\vec{\frak x},\vec{\sigma},u')$ with the following properties:
        \begin{enumerate}
        			\item $\vec{\frak x} (\frak x_v ; v  \in  C^{\rm int}_0(\check R))\in 
			\prod_{v  \in  C^{\rm int}_0(\check R)}\mathcal V_{v}^{{\rm source}}$ and 
		        	$\vec\sigma = (\sigma_e ; e \in  C^{\rm int}_{1}(\check R)) \in
			\prod_{e  \in  C^{\rm int}_{1}(\check R)}\mathcal V_{e}^{{\rm deform}}$. 
			 Furthermore, for each interior vertex $v$, $\frak x_v $ belongs to the 
			$\kappa$-neighborhood of the point of $\mathcal V_{v}^{{\rm source}}$ 
			induced by $\frak u$. Similarly, for each $e$, we have ${\vert}\sigma_e{\vert}<\kappa$.
	        \item $u' : (\Sigma(\vec{\frak x},\vec{\sigma}),\partial(\Sigma(\vec{\frak x},\vec{\sigma}))) 
		        \to (X\backslash \mathcal D,L)$ is a continuous map
        			and is smooth on each irreducible component.
	        \item If $e$ is (resp. is not) a level $0$ edge, 
	        		then the image of the restriction to $u'$ to
			$[-5T_e,5T_e]_{r_e} \times [0,\pi]_{s_e}$ (resp. $[-5T_e,5T_e]_{r_e} \times S^1_{s_e}$)
			has a diameter\footnote{The diameter is defined
			with respect to the metric $g_{NC}$.} less than $\kappa$. 
			If $c(v)={\rm D}$, then the restriction of $u'$ to  $\Sigma^+_{v}(\vec{\frak x},\vec \sigma)$
			is included in the open neighborhood $\frak U$ of $\mathcal D$.%\footnote{\color{red}
			%We take our coordinate $\varphi_{v,e}$ such that if $\lambda(v) > 0$ 
        			%then the image of $u_{v} \circ \varphi_{v,e}$ lies in a 
			%sufficiently small neighborhood of $\mathcal D$. This fact is necessary so that 
		        %there exists $u'$ satisfying this condition (4).}
                \item If $c(v)={\rm d}$ or ${\rm s}$, then the $C^2$-distance between the 
			restrictions of $u'$ and $u_v$ to $\Sigma^-_{v}(\vec{\frak x},\vec \sigma)$ is less than 
			$\kappa$. If $c(v)={\rm D}$, then the previous part implies that the restriction of 
			$u'$ to $\Sigma^-_{v}(\vec{\frak x},\vec \sigma)$ may be regarded as a map to 
			$\mathcal N_{\mathcal D}(X)\setminus \mathcal D$. We also demand that 
			the $C^2$-distance between this map and $U_v$ is less than $\kappa$.\footnote{
			Here again we use the convention that the distance between an object and 
			$\{U_v\}_{\lambda(v)>0}$, is defined to be the minimum of the relevant distance 
			between that object and all representatives of $\{U_v\}_{\lambda(v)>0}$.}
		\item We require
        			\begin{equation}
			        \overline\partial_{j_{\vec{\frak x},\vec \sigma}} u' \in E_0(u').
		        \end{equation}
		        Here $j_{\vec{\frak x},\vec \sigma}$ is the complex structure of 
		        $\Sigma(\vec{\frak x},\vec \sigma)$, and $E_0(u')$ is defined from $E_v$ by 
		        target parallel transportation in the same way as in Section \ref{sub:Obst}.
		\item We require
        			\begin{equation}
			        u'(w_{v,i}) \in \mathcal N_{v,i}.
		        \end{equation}
%		        Note that we may regard $w_{v,i}$ as an element of 
%		        $\Sigma(\vec{\frak x},\vec \sigma)$   by (\ref{form6110})
	\end{enumerate}
        The group of automorphisms $\Gamma_{\frak u}$ acts on 
        $\widehat{\mathcal U}_0(\frak u,\Upsilon)$ in the obvious way.
        We write ${\mathcal U}_0(\frak u,\Upsilon)$
        for the quotient space $\widehat{\mathcal U}_0(\frak u,\Upsilon)/\Gamma_{\frak u}$.
\end{definition}

\begin{remark}
	The above definition needs to be slightly modified if some 
	of the components of $\vec\sigma$ are zero. 
	Let $e$ be an edge connecting a vertex of level $i$ to a vertex of level $i+1$ such 
	that $\sigma_e=0$. If $e'$ is another edge that 
	connects a vertex of level $i$ to a vertex of level $i+1$, 
	then $\sigma_{e'}=0$. Next, we decompose $\check R$ 
	into several blocks such that $\sigma_e = 0$ for the edges $e$ joining two different blocks 
	and $\sigma_e \ne 0$ for an edge $e$, which is inside a block and is not a fine edge.
	In each block, we use Definition \ref{defn647} and 
	join spaces associated to various blocks in the same way as in Definition \ref{defn64545}.
	We omit the details of this process because the actual space we use for 
	the definition of our Kuranishi structure is not ${\mathcal U}_0(\frak u,\Upsilon)$
	but ${\mathcal U}(\frak u,\Upsilon)$, introduced in Definition \ref{defn6488}.
	We can also define ${\mathcal U}_0(\frak u,\Upsilon)$ as a subspace of ${\mathcal U}(\frak u,\Upsilon)$.
	We brought firstly Definition \ref{defn647} because its geometric meaning is more clear.
\end{remark}

The space ${\mathcal U}_0(\frak u,\Upsilon)$ in general is singular (not an orbifold). We introduce the notion of inconsistent solutions to thicken ${\mathcal U}_0(\frak u,\Upsilon)$ into an orbifold.

\begin{definition}\label{defn6488}
	Let $\Xi= (\vec w_v,(\mathcal N_{v,i}),(\phi_{v}),(\varphi_{v,e}),\kappa)$ be a TSD at
        the element $\frak u$ of $\mathcal M_{k+1}^{\rm RGW}(L,\beta)$
        and $\Upsilon = (\Xi,E)$ 
        a TSO.
	We say $(\vec{\frak x},\vec{\sigma},(u'_{v}),(U'_{v}),(\rho_e),(\rho_i))$ is an 
	{\it inconsistent solution near $\frak u$ with respect to $\Upsilon$}
	 if it satisfies the following properties:
        \begin{enumerate}
        \item $\vec{\frak x} (\frak x_v ; v  \in  C^{\rm int}_0(\check R))\in 
        \prod_{v  \in  C^{\rm int}_0(\check R)}\mathcal V_{v}^{{\rm source}}$,
		$\vec\sigma = (\sigma_e ; e \in  C^{\rm int}_{1}(\check R)) \in
		\prod_{e  \in  C^{\rm int}_{1}(\check R)}\mathcal V_{e}^{{\rm deform}}$ and 
	        $\rho_e \in \bbC$ for each edge $e \in C^{\rm int}_{\rm th}(\check R)$ 
	        that is not a level $0$ edge, and 
	        $\rho_i \in D^2$ for each level $i =1 ,\dots,\vert\lambda\vert$.
	        Furthermore, for each interior vertex $v$, $\frak x_v $ belongs to the 
		$\kappa$-neighborhood of the point of $\mathcal V_{v}^{{\rm source}}$ 
		induced by $\frak u$. Similarly, for each $e$, we have ${\vert}\sigma_e{\vert}<\kappa$.
        \item
        If $c(v) = {\rm d}$ (resp. ${\rm s}$), then 
        $u'_{v} : (\Sigma^+_{v}(\vec{\frak x},\vec \sigma),\partial \Sigma^+_{v}(\vec{\frak x},\vec \sigma))
        \to (X\setminus\mathcal D,L)$ (resp. $u'_{v} : \Sigma^+_{v}(\vec{\frak x},\vec \sigma)\to X \setminus\mathcal D$)
        is a smooth map.
        \item If $c(v) = {\rm D}$, then $U'_{v} : \Sigma^+_{v}(\vec{\frak x},\vec \sigma)
        \to \mathcal N_{\mathcal D}X \setminus\mathcal D$
        is a smooth map, and $u'_{v}=\pi\circ U'_{v}$.
        \item $\rho_e = 0$ if and only if $\sigma_e = 0$. 
        \item Suppose $e$ is an edge connecting vertices $v_0$ and $v_1$ 
        such that $\lambda(v_0)=0$ and $\lambda(v_1)\geq1$.
	Then we require:
        \begin{equation}\label{form621}
    	    u'_{v_0} = {\rm Dil}_{\rho_e} \circ U'_{v_1}
        \end{equation}
        on $[-5T_e,5T_e]_{r_e} \times S^1_{s_e}
        = \Sigma^+_{v_1}(\vec{\frak x},\vec \sigma)\cap \Sigma^+_{v_2}(\vec{\frak x},\vec \sigma)$
        if $\sigma_e \ne 0$. In particular, we assume that the restriction of $u'_{v_0}$ to 
        $[-5T_e,5T_e]_{r_e} \times S^1_{s_e}$ is contained in the open neighborhood $\frak U$ of
        $\mathcal D$.
        If $\sigma_e = 0$, then the values of $u'_{v_0}$ and 
        $\pi \circ U'_{v_1}$ at the nodal points corresponding to $e$ are equal to each other.
        \item Suppose $e$ is an edge connecting vertices $v_1$ and $v_2$ such that 
        $\lambda(v_1)=i>0$ and $\lambda(v_2)\geq i+1$.
        We require
        \begin{equation}\label{form6222}
        		U'_{v_1} = {\rm Dil}_{\rho_e} \circ U'_{v_2}
        \end{equation}
        on $[-5T_e,5T_e]_{r_e} \times S^1_{s_e}
        = \Sigma^+_{v_1}(\vec{\frak x},\vec \sigma)\cap \Sigma^+_{v_2}(\vec{\frak x},\vec \sigma)$
        if $\sigma_e \ne 0$.
        If $\sigma_e = 0$, then the values of $U'_{v_1}$ and $\pi\circ U'_{v_2}$
        at the nodal points corresponding to $e$ are equal.
        \item Suppose $e$ is a level $0$ edge connecting the vertices $v_1$ and $v_2$. If $\sigma_e \ne 0$, then we require:
	\begin{equation}\label{form623}
	        u'_{v_1} = u'_{v_2}
	\end{equation}
        on $[-5T_e,5T_e]_{r_e} \times [0,1]_{s_e}
        = \Sigma^+_{v_1}(\vec{\frak x},\vec \sigma)\cap \Sigma^+_{v_2}(\vec{\frak x},\vec \sigma)$. If $\sigma_e =0$, then \eqref{form623} holds at the nodal 
        point corresponding to $e$.
        \item Suppose $e$ is a fine edge connecting the vertices $v_1$ and $v_2$ with level zero (resp. with the same positive level). If $\sigma_e \ne 0$, then we require:
	\begin{equation}\label{form623rev}
	        \hspace{1cm} u'_{v_1} = u'_{v_2}  \quad {\rm (resp.} \,\,\, \quad 
        U'_{v_1} = U'_{v_2})
	\end{equation}
        on  $[-5T_e,5T_e]_{r_e} \times S^1_{s_e}= \Sigma^+_{v_1}(\vec{\frak x},\vec \sigma)\cap \Sigma^+_{v_2}(\vec{\frak x},\vec \sigma)$. If $\sigma_e =0$, then \eqref{form623rev} holds at the nodal 
        point corresponding to $e$.
	\item If $e$ is (resp. is not) a level $0$ edge, 
	        	then the image of the restriction to $u_v'$ to
		$[-5T_e,5T_e]_{r_e} \times [0,\pi]_{s_e}$ (resp. $[-5T_e,5T_e]_{r_e} \times S^1_{s_e}$)
		has a diameter\footnote{The diameter is defined
		with respect to the metric $g_{NC}$.} less than $\kappa$.
	\item If $c(v)={\rm d}$ or ${\rm s}$, then the $C^2$-distance between the 
		restrictions of $u_v'$ and $u_v$ to $\Sigma^-_{v}(\vec{\frak x},\vec \sigma)$ is less than 
		$\kappa$. If $c(v)={\rm D}$, then we demand that 
		the $C^2$-distance between the restrictions of $U_v'$ and $U_v$ 
		to $\Sigma^-_{v}(\vec{\frak x},\vec \sigma)$ is less than $\kappa$.\footnote{
		Here again we use the convention that the distance between an object 
		and $\{U_v\}_{\lambda(v)>0}$, is defined to be the minimum of the relevant distance 
		between that object and all representatives of $\{U_v\}_{\lambda(v)>0}$.}
	        \item
        If $c(v) = {\rm d}$ or $\rm s$, then we require:
        \begin{equation}\label{form6119}
        \overline\partial_{j_{\vec{\frak x},\vec \sigma}} u'_v \in E_v(u'_v).
        \end{equation}
        Here $j_{\vec{\frak x},\vec \sigma}$ is the complex structure of $\Sigma(\vec{\frak x},\vec \sigma)$, and
        $E_v(u'_v)$ is defined from $E_v$ by target parallel transportation in the same way as in Section \ref{sub:statement}.
        \item
        If $c(v) = {\rm D}$, then we require:
        \begin{equation}\label{form6120}
       		\overline\partial_{j_{\frak x_v}} U'_v \in E_v(U'_v).
        \end{equation}
        Here $j_{\vec{\frak x},\vec \sigma}$ is the complex structure of 
        $\Sigma(\vec{\frak x},\vec \sigma)$, and
        we use \eqref{decom-tan-bdle} to obtain $E_v(U'_v)$ from $E_v(u'_v)$ as a subspace
        of $L^2_{m,\delta}(\Sigma_v;(U_v')^*T\mathcal N_{\mathcal D}(X) \otimes \Lambda^{0,1})$.
        \item We have:
        \begin{equation}\label{form6121}
	        u'_v(w_{v,i}) \in \mathcal N_{v,i}.
        \end{equation}
        \end{enumerate}

        We denote by $\widetilde{\mathcal U}(\frak u,\Upsilon)$ the set of all
        $(\vec{\frak x},\vec{\sigma},(u'_{v}),(U'_{v}),(\rho_e),(\rho_i))$ satisfying the above properties.
        We define an equivalence relation $\sim$ on $\widetilde{\mathcal U}(\frak u,\Upsilon)$ in the following way. 
        Let ${\bf x}_j =(\vec{\frak x}_{(j)},\vec{\sigma}_{(j)},(u'_{v,(j)}),(U'_{v,(j)}),(\rho_{e,(j)}),(\rho_{i,(j)}))$
        be elements of $\widetilde{\mathcal U}(\frak u,\Upsilon)$ for $j=1,2$. We say that 
        ${\bf x}_1 \sim {\bf x}_2$ if there exists $a_i \in \bbC_*$ ($i=1,\dots,\vert\lambda\vert$)
        with the following properties.
        Let:
        $$
        b_i = a_1 \cdots  a_i \in \bbC_*.
        $$
        \begin{enumerate}
        \item[(i)]
        $\vec{\frak x}_{(1)} = \vec{\frak x}_{(2)}$, $\vec{\sigma}_{(1)} = \vec{\sigma}_{(2)}$, 
        $u'_{v,(1)} = u'_{v,(2)}$.
        \item[(ii)] $\rho_{i,(2)} = a_i \rho_{i,(1)}$.
        \item[(iii)] $U'_{v,(1)} = {\rm Dil}_{b_{\lambda(v)}}\circ U'_{v,(2)}$
        \item[(iv)] Suppose $e$ is an edge connecting a vertex $v_0$ with $\lambda(v_0)=0$ to 
        a vertex $v_1$ with $\lambda(v_1)\geq 1$. Then we require:
        $$
        \rho_{e,(2)} =b_{\lambda(v_2)}\rho_{e,(1)}.
        $$
        \item [(v)] Suppose $e$ is an edge connecting a vertex $v_1$ with $\lambda(v_1)\geq 1$ to 
        a vertex $v_2$ with $\lambda(v_2)\geq 2$. Then we require:
        $$
        \rho_{e,(2)} = a_{\lambda(v_1)+1}\cdots a_{\lambda(v_2)}\rho_{e,(1)}.
        $$
        \end{enumerate}
        We denote by  $\widehat{\mathcal U}(\frak u,\Upsilon)$ the quotient space  
        $\widetilde{\mathcal U}(\frak u,\Upsilon)/\sim$. The group $\Gamma_{\frak u}$ acts on 
        $\widehat{\mathcal U}(\frak u,\Upsilon)$ in an obvious way.
        We denote by ${\mathcal U}(\frak u,\Upsilon)$ the quotient space 
        $\widehat{\mathcal U}(\frak u,\Upsilon)/\Gamma_{\frak u}$.
        We say an element of  ${\mathcal U}(\frak u,\Upsilon)$ is an {\it inconsistent solution} near $\frak u$
        with respect to $\Upsilon$.
        When it does not make any confusion, the elements of 
        $\widehat{\mathcal U}(\frak u,\Upsilon)$ or $\widetilde{\mathcal U}(\frak u,\Upsilon)$ 
        are also called inconsistent solutions near $\frak u$ with respect to $\Upsilon$.
        \par
        We replace Conditions (11) and (12) by 
        pseudo-holomorphicity to define 
        $\widehat{\frak U}(\frak u,\Xi)$, $\widetilde{\frak U}(\frak u,\Xi)$
        and ${\frak U}(\frak u,\Xi)$.
\end{definition}
\begin{remark}
	Initially it might seem that the complex numbers $\rho_i$ do not play any role in the definition of 
	the elements of $\mathcal U(\frak u,\Upsilon)$.
	However, later they make it slightly easier for us to define the obstruction maps.
\end{remark}
 
Our generalization of Proposition \ref{prop617} claims that ${\mathcal U}(\frak u,\Upsilon)$ is a smooth orbifold.
Before stating this result, we elaborate on the relationship between $\mathcal U(\frak u,\Upsilon)$ and $\mathcal U_0(\frak u,\Upsilon)$.

\begin{definition}\label{inconssitentmap} 
In the situation of Definition \ref{defn6488},
we say $(\vec{\frak x},\vec{\sigma},(u'_{v}),(U'_{v}),(\rho_e),(\rho_i))$ 
is an 
	{\it inconsistent map near $\frak u$ with respect to TSD $\Xi$}
	if Conditions (1)-(10), (13) of Definition \ref{defn6488}
	is satisfied. (We do not require (11) and (12) here.)
\end{definition}
 
\begin{definition}
	Let $(\vec{\frak x},\vec{\sigma},(u'_{v}),(U'_{v}),(\rho_e),(\rho_i))$ be 
	an inconsistent solution near $\frak u$ with respect to $\Upsilon$. We say that it satisfies 
	{\it consistency equation} if for each edge $e$ connecting vertices $v_1$ and $v_2$ with
	$0\leq\lambda(v_1)< \lambda(v_2)$, we have:
	\begin{equation}\label{form6123}
		\rho_{e} = \rho_{\lambda(v_1)+1}\cdots \rho_{\lambda(v_2)}.
	\end{equation}
	It is easy to see that the consistency equation (\ref{form6123}) is independent of the 
	choice of the representative with respect to the relations given by $\sim$ and the action of 
	$\Gamma_{\frak u}$.
\end{definition}
\begin{lemma}\label{lem650}
	The set of inconsistent solutions near $\frak u$ satisfying consistency equation 
	can be identified with ${\mathcal U}_0(\frak u,\Upsilon)$.
\end{lemma}
\begin{proof}
	Let  $(\vec{\frak x},\vec{\sigma},(u'_{v}),(U'_{v}),(\rho_e),(\rho_i))$ be
	an inconsistent solution near $\frak u$ satisfying consistency equations.
	For the simplicity of exposition, 
	we consider the case that all the components of $\vec{\sigma}$ are nonzero.\footnote{
	This is the case that we gave a detailed definition of ${\mathcal U}_0(\frak u,\Upsilon)$, after all.
	For other cases, this lemma can be used as the definition.}
	Define $\tau_i$ to be the product $\rho_1\cdot \rho_2 \dots \rho_i$. For each vertex $v$ 
	with $c(v) = D$, we also define:
	\[
	  U^{\frak m}_{v} = {\rm Dil}_{\tau_{\lambda(v)}} \circ U^{\prime}_{v}.
	\]
	Then the maps $U^{\frak m}_{v}$
	for $c(v)={\rm D}$ and $u'_{v}$ for $c(v)={\rm d}$ or ${\rm s}$ are compatible on the overlaps
	and by gluing them together, we obtain an element of ${\mathcal U}_0(\frak u,\Upsilon)$.
	The reverse direction is clear.
\end{proof}
%\begin{proof}
%Let  $(\vec{\frak x},\vec{\sigma},(u'_{v}),(U'_{v}),(\rho_e),(\rho_i))$ be
%an inconsistent solution near $\frak u$ satisfying (\ref{form6123}).
%We will construct an element of ${\mathcal U}_0(\frak u)$.
%We consider the case where all the components of $\vec{\sigma}$ are nonzero.
%\par
%For each vertex $v$ with $c(v) = D$ and 
%$\frak y = (\vec{\frak x},u',(c_v)) \in \widetilde{\mathcal U}(\widehat{\mathcal S},\frak u)$
%we take a lift
%$$
%U^{\prime,\frak y}_{v} : \Sigma_v \to \mathcal N_{\mathcal D}X
%$$
%of $u'_{v}$. We require that there exists $U^{\prime}_{v}$
%independent of $(c_v)$ such that
%$$
%U^{\prime,\frak y}_{v} = {\rm Dil}_{c_v} \circ U^{\prime}_{v}.
%$$
%We require that $U^{\prime,\frak y}_{v}$ depends smoothly on $\frak y$
%(the set of the pair  $(\vec{\frak x},u')$). We also require invariance of 
%$\Gamma_{\frak u}$ action in an obvious sense.
%\par
%We put
%$
%R_i = \rho_1 \dots \rho_i
%$
%and 
%$$
%U''_{v} = {\rm Dil}_{R_{\lambda(v)}} \circ U'_v :   \Sigma^+_{v}(\vec{\frak x},\vec \sigma)
%\to \mathcal N_{\mathcal D}X \setminus\mathcal D.
%$$
%Definition \ref{defn6488} (9)(10)(11) and (\ref{form6123}) imply that various 
%$U''_{v}$ and $u'_v$ coincides on the overlapped part.
%Therefore they gives a globally well-defined map:
%$$
%u' : (\Sigma(\vec{\frak x},\vec{\sigma}),\partial
%(\Sigma(\vec{\frak x},\vec{\sigma})) \to (X\setminus \mathcal D,L)
%$$
%It is easy to see that  $(\vec{\frak x},\vec{\sigma},u')$ satisfies 
%Definition \ref{defn647} (1)-(6).
%The construction of the opposite direction is similar.
%\end{proof}
\begin{example}
	We consider the case of detailed DD-ribbon tree in Figure \ref{FIgsec6-1}.
	This tree has two edges (whose multiplicities are $p_1$ and $p_2$, respectively). We denote them 
	by $e_{\rm d}$ and $e_{\rm s}$, respectively. 
	Two parameters $\rho_{\rm d}$ and $\rho_{\rm s}$ are associated to these edges. 
	(In Section \ref{sub:proofmain}, $\rho_{\rm d}$ and $\rho_{\rm s}$ are denoted by 
	$\rho_1$ and $\rho_2$, respectively). 
	The total number of levels is 1. So there is a parameter $\rho$ associated to this level.
	The consistency equation (\ref{form6123}) implies that:
	$$
	  \rho_{\rm d}=\rho =\rho_{\rm s}.
	$$
	which is the same as the equation in \eqref{rho1rho2hito}.
\end{example}
%\begin{example}
%	We consider the case of detailed DD-ribbon tree in Figure \ref{FIgsec6-1}.
%	This tree has two edges (whose multiplicities are 2 and 3, respectively). We denote them 
%	$e_{\rm d}$ and $e_{\rm s}$, respectively. 
%	Two parameters $\rho_{\rm d}$ and $\rho_{\rm s}$ are associated to these edges. 
%	(In Section \ref{sub:proofmain}, $\rho_{\rm d}$ and $\rho_{\rm s}$ are denoted by 
%	$\rho_1$ and $\rho_2$, respectively). 
%	The total number of levels is 1 so there is a parameter $\rho$ associated to this level.
%	The consistency equation (\ref{form6123}) implies that:
%	$$
%	  \rho_{\rm d}= \rho = \rho_{\rm s}.
%	$$
%	The equation given in (\ref{rho1rho2hito}) is (using the notations here)
%	$$
%	  \rho_{\rm d}= \rho_{\rm s}.
%	$$
%	After eliminating the extra parameter $\rho$, two descriptions coincide each other.
%\end{example}
For any $\ell \le m-2$, we fix a $C^{\ell}$ structure on $\widehat{\mathcal U}(\frak u,\Upsilon)$ in the following way. For an interior vertex $v$ of $\Sigma_v$, let $\Sigma_v^-$ be the space $\Sigma_v^-(\vec{\frak x}, \vec{\sigma})$ in the case that $\vec{\sigma}=0$ and $\vec{\frak x}$ is induced by $\frak u$. The trivialization of the universal family allows us also to identify $\Sigma_v^-$ with $\Sigma_v^-(\vec{\frak x}, \vec{\sigma})$ for different choices of $\vec{\frak x}$, $\vec{\sigma}$. Define maps:
\begin{align}
  {\rm Res}_v : \widetilde{\mathcal U}(\frak u,\Upsilon)&\to L^2_{m+1}(\Sigma_v^-,X\setminus \mathcal D)
  \qquad &\text{if $\lambda(v) = 0$}, \\
  {\rm Res}_v : \widetilde{\mathcal U}(\frak u,\Upsilon)
  &\to L^2_{m+1}(\Sigma_v^-,\mathcal N_{\mathcal D}X \setminus \mathcal D)
  \qquad &\text{if $\lambda(v) > 0$},\label{lambda-v-p}
\end{align}
such that ${\rm Res}_v(\vec{\frak x},\vec{\sigma},(u'_{v}),(U'_{v}),(\rho_e),(\rho_i))$ is the restriction of $u'_v$ or $U'_v$ to $\Sigma_v^-(\vec{\frak x},\vec{\sigma}) \cong \Sigma_v^-$. By unique continuation, ${\rm Res}_v$ and the obvious projection maps induce an embedding:
\begin{equation}\label{ccc}
	\aligned
	\widetilde{\mathcal U}(\frak u,\Upsilon)\to 
	&\prod_{v  \in  C^{\rm int}_0(\check R)}\mathcal V_{v}^{{\rm source}} 
	\times \prod_{e  \in  C^{\rm int}_1(\check R)}\mathcal V_{e}^{{\rm deform}} \times 
	(D^2)^{\vert\lambda\vert}\\
	&\times\prod_{v  \in  C^{\rm int}_0(\check R), \lambda(v) = 0}L^2_{m+1}
	(\Sigma_v^-,X\setminus \mathcal D)\\
	&\times\prod_{v  \in  C^{\rm int}_0(\check R), \lambda(v) > 0}L^2_{m+1}
	(\Sigma_v^-,\mathcal N_{\mathcal D}X\setminus \mathcal D)
	\endaligned
\end{equation}

We use this embedding to fix a $C^{\ell}$-structure on $\widetilde{\mathcal U}(\frak u,\Upsilon)$. The group $\bbC_*^{{\vert}\lambda{\vert}}$ acts freely on the target and the domain of \eqref{ccc}, and the above embedding is equivariant with respect to this action. We use the induced map at the level of the quotients to define a $C^{\ell}$-structure on $\widehat{\mathcal U}(\frak u,\Upsilon)$. Note that we can define a slice for $\widehat{\mathcal U}(\frak u,\Upsilon)$ using the following idea. For each $1\leq i \leq {\vert}\lambda{\vert}$, we fix an interior vertex $v_i$ with $\lambda(v_i)=i$ and a base point $x_i\in\Sigma_{v_i}^{-}$. We also trivialize the bundle $\mathcal N_{\mathcal D}(X)$ in a neighborhood of $U_{v_i}(x_i)$. Each element of $\widehat{\mathcal U}(\frak u,\Upsilon)$ has a unique representative $(\vec{\frak x},\vec{\sigma},(u'_{v}),(U'_{v}),(\rho_e),(\rho_i))$ such that $U_{v_i}'(x_i)=1\in \bbC$. Here we assume that $\kappa$ is small enough such that $U_{v_i}'(x_i)$ belongs to the neighborhood of $U_{v_i}(x_i)$ that 
the pull back of $\mathcal N_{\mathcal D}(X)$ is trivialized.
\begin{prop}\label{prop652}
	The space $\widehat{\mathcal U}(\frak u,\Upsilon)$ is a $C^{\ell}$-manifold 
	and ${\mathcal U}(\frak u,\Upsilon)$ is a $C^{\ell}$-orbifold.
	There exists a $\Gamma_{\frak u}$-invariant open $C^{\ell}$-embedding for $\ell \le m-2$:
	\[
	  \Phi : \prod_{e  \in  C^{\rm int}_1(\check R)}\mathcal V_{e}^{{\rm deform}} 
	  \times \widehat{\mathcal U}(\frak u,{\check R},\Upsilon)\times D^2(\epsilon)^{\vert\lambda\vert}
	  \to \widehat{\mathcal U}(\frak u,\Upsilon)
	\]
	with the following properties:
	\begin{enumerate}
		\item
			\[
			  \Phi(\vec{\sigma},\xi,(\rho_i))=[\vec{\frak x}, \vec{\sigma},(u'_{\vec{\sigma},\xi,v}),
			  (U'_{\vec{\sigma},\xi,v}),(\rho_{e}(\vec{\sigma},\xi)),(\rho_i)]
			\]
		Namely, the gluing parameters $\vec{\sigma}$  
		are preserved by the map $\Phi$. Moreover, the deformation parameter $\vec{\frak x}$ is 
		the same as the one for the source curve of $\xi$.
		\item
		For each edge $e \in C^{\rm int}_{\rm th}(\check R)$ 
	        that is not a level $0$ edge, there exists a nonzero smooth function $f_e$ such that: 
		\[
		  \rho_{e}(\vec{\sigma},\xi)=f_e(\vec{\sigma},\xi) \sigma_e^{m(e)} 
		\]
		where $m(e)$ is the multiplicity of the edge $e$.
%		\item 
%		 {\color{red}\AD{\color{red} How is this part related to $\Phi$?}
%		For $v$, let $\Sigma_v^-$ be the space $\Sigma_v^-(\vec{\sigma})$ in the case that
%		 $\vec{\sigma}=0$. Define maps:\AD{\color{red} Why the second map is well-defined?}
%		\[
%		  \aligned
%		  {\rm Res}_v : \widehat{\mathcal U}(\frak u)&\to L^2_{m+1}(\Sigma_v^-,X\setminus \mathcal D)
%		  \qquad &\text{if $\lambda(v) = 0$}, \\
%		  {\rm Res}_v : \widehat{\mathcal U}(\frak u) 
%		  &\to L^2_{m+1}(\Sigma_v^-,\mathcal N_{\mathcal D}X \setminus \mathcal D)
%		  \qquad &\text{if $\lambda(v) > 0$},
%		  \endaligned
%		\]
%		such that ${\rm Res}_v(\vec{\frak x},\vec{\sigma},(u'_{v}),(U'_{v}),(\rho_e))$
%		is the restriction of $u'_v$ or $U'_v$ to $\Sigma_v^-(\vec{\sigma}) \cong \Sigma_v^-$.
%		(The diffeomorphism is induced by the trivialization of the 
%		universal family we fixed before.)
%		Then ${\rm Res}_v$ and obvious projections induce a $C^{\ell}$ embedding:
%		\begin{equation}
%			\aligned
%			\widehat{\mathcal U}(\frak u)\to 
%			&\prod_{v  \in  C^{\rm int}_0(\mathcal R)}\mathcal V_{v}^{{\rm source}} 
%			\times \prod_{e  \in  C^{\rm int}_1(\mathcal R)}\mathcal V_{e}^{{\rm deform}} \times 
%			{\color{blue}D^2(\epsilon)^{\vert\lambda\vert}}\\
%			&\times\prod_{v  \in  C^{\rm int}_0(\mathcal R), \lambda(v) = 0}L^2_{m+1}
%			(\Sigma_v^-,X\setminus \mathcal D)\\
%			&\times\prod_{v  \in  C^{\rm int}_0(\mathcal R), \lambda(v) > 0}L^2_{m+1}
%			(\Sigma_v^-,\mathcal N_{\mathcal D}X\setminus \mathcal D)
%			\endaligned
%		\end{equation}}
		\item Let $\xi = (\vec{\frak x},u',U')\in \widehat{\mathcal U}(\frak u,{\check R},\Upsilon)$  and 
			$\vec{\sigma}_0$ be the vector that $\sigma_e$ are all zero.
			Then we have:
			\[
			  \Phi(\vec{\sigma}_0,\xi,(\rho_i))=[\vec{\frak x},\vec{\sigma}_0,(u'_{\vec{\sigma}_0,\xi,v}),
			  (U'_{\vec{\sigma}_0,\xi,v}),(\rho_{e}(\vec{\sigma}_0,\xi)),(\rho_i)],
			\]
			where $u'_{\vec{\sigma}_0,\xi,v}$ is the restriction $u'_v$ of $u'$,
			$U'_{\vec{\sigma}_0,\xi,v}$ is the restriction of $U^{\prime}_{v}$ 
			and $\rho_{e}(\vec{\sigma}_0,\xi)=0$.
	\end{enumerate}
\end{prop}
The proof of Proposition \ref{prop652} is essentially the same as the proof of Proposition \ref{prop617}, and it is only notationally more involved.

We next state a generalization of Proposition \ref{prop618}. For a thick edge $e$ which is not of level $0$, we define $T_e$, $\theta_e$, $\frak R_e$, $\eta_e$ using the following identities:
\begin{equation}\label{form635rev}
	\aligned
	\sigma_e &= \exp(-(T_e+\sqrt{-1}\theta_e)), \\
	\rho_e  &= \exp(-(\frak R_e+\sqrt{-1} \eta_e)).
	\endaligned
\end{equation}
If $e$ is a level $0$ edge, then we define $T_e$ using:
\begin{equation}\label{form635rev-2}
	\sigma_e = \exp(-T_e).
\end{equation}
We may also define $T_e$ and $\theta_e$ for a fine edge as in \eqref{form635rev}. Using $\Phi$, we regard $\frak R_e$, $\eta_e$ as functions of $T_{e'}$, $\theta_{e'}$ and $\xi$.
We again use the trivialization of the universal family to identify $\Sigma_v^-(\vec{\frak x}_0,\vec{\sigma})$
(see \eqref{form610}.) for various choices of $\vec{\frak x}_0$, $\vec{\sigma}$. For the purpose of the next proposition, we also regard $u'_{\vec{\sigma},\xi,v}$, $U'_{\vec{\sigma},\xi,v}$ as maps
  \begin{align*}
  &u'_{\vec{\sigma},\xi,v} :\Sigma_v^-(\vec{\sigma}) \to X \setminus \mathcal D \\
  &U'_{\vec{\sigma},\xi,v} :\Sigma_v^-(\vec{\sigma}) \to \mathcal N_{\mathcal D}X \setminus \mathcal D. 
  \end{align*}
In particular, the domain of these maps are independent of $T_e$, $\theta_e$ and $\xi$.
\begin{prop}\label{prop653}
	Let $\ell$ be an arbitrary positive integer
	and $k_e,k'_e$ be non-negative integers. Let
	$\upsilon_e = 0$ if $k_e,k'_e = 0$. Otherwise we define $\upsilon_e = 1$.
	\begin{enumerate}
		\item We have the following exponential decay estimates:
			\begin{align}\label{}
				  \phantom{i + j + k}
				  &\begin{aligned}
   				&\left\Vert
				\prod_e  \frac{\partial^{k_e}}{\partial^{k_e} T_e}\frac{\partial^{k'_e}}
				{\partial^{k'_e} \theta_e}u'_{\vec{\sigma},\xi,v}
				\right\Vert_{L^2_{\ell}(\Sigma_v^-(\vec{\sigma}))}\le
				C \exp(-c\sum \upsilon_eT_e).\\
				&\left\Vert\prod_e  \frac{\partial^{k_e}}{\partial^{k_e} T_e}\frac{\partial^{k'_e}}
				{\partial^{k'_e} \theta_e}U'_{\vec{\sigma},\xi,v}
				\right\Vert_{L^2_{\ell}(\Sigma_v^-(\vec{\sigma}))}\le C \exp(-c\sum \upsilon_eT_e).
				  \end{aligned}
			\end{align}
			Here $C,c$ are positive constants depending on $\ell$, $k_e$, $k_e'$. 
			The same estimate holds for the $\xi$ derivatives of 
			$u'_{v,\vec{\sigma},\xi}$, $U'_{v,\vec{\sigma},\xi}$.
		\item For any thick edge $e_0$ which is not a level $0$ edge, 
			we also have the following exponential decay estimates:
						\begin{align}\label{}
				  \phantom{i + j + k}
				  &\begin{aligned}
				                              &\left\vert
                                \prod_e  \frac{\partial^{k_e}}{\partial^{k_e} T_e}
                                \frac{\partial^{k'_e}}{\partial^{k'_e} \theta_e}
                                {(\frak R_{e_0} - m({e_0})T_{e_0})}
                                \right\vert
                                \le
                                C \exp(-c\sum \upsilon_eT_e) \\
                                &\left\vert
                                \prod_e  \frac{\partial^{k_e}}{\partial^{k_e} T_e}
                                \frac{\partial^{k'_e}}{\partial^{k'_e} \theta_e}
                                {(\eta_{e_0} - m({e_0})\theta_{e_0})}
                                \right\vert
                                \le
                                C \exp(-c\sum \upsilon_eT_e).
				  \end{aligned}
			\end{align}
			Here $C,c$ are positive constants depending on $\ell$, $k_e$, $k_e'$. 
			The same estimate holds for the $\xi$ derivatives of $\frak R_{e_0}$, $\eta_{e_0}$.
	\end{enumerate}
\end{prop}
Similar to Proposition \ref{prop618}, Proposition \ref{prop653} can be verified using the same argument as in the proof of \cite[Section 6]{foooexp}. 

We now use Propositions \ref{prop652} and \ref{prop653} to produce a Kuranishi chart at $\frak u$.
Let $\frak y = (\vec{\frak x},\vec{\sigma},(u'_{v}),(U'_{v}),(\rho_e),(\rho_i))$ be a representative of an element of $\widehat{\mathcal U}(\frak u,\Upsilon)$. Recall that we fix vector spaces $E_v$ for each $\frak u$, and use target parallel transportation to obtain the vector spaces $E_v(u'_{v})$ and $E_v(U'_{v})$. We define:
\begin{equation}\label{form63131}
\mathcal E_{0,\frak u}(\frak y)
= 
\bigoplus_{v \in C_0^{\rm int}(\check R),\,\lambda(v) = 0} E_v(u'_{v}) 
\oplus
\bigoplus_{v \in C_0^{\rm int}(\check R),\,\lambda(v) > 0} E_v(U'_{v}).
\end{equation}
Using Proposition \ref{prop653}, it is easy to see that (\ref{form63131}) defines a $\Gamma_{\frak u}$-equivariant $C^{\ell}$ vector bundle on $\widehat{\mathcal U}(\frak u,\Upsilon)$.

We define the other part of the obstruction bundle as follows. Let $e$ be a thick edge which connects vertices $v_1$ and $v_2$ with $0\leq \lambda(v_1)<\lambda(v_2)$. We fix the trivial line bundle $\bbC_e$ on $\widetilde{\mathcal U}(\frak u,\Upsilon)$. Let ${\bf x}_1 \sim {\bf x}_2$ and $a_i \in \bbC_*$ ($i=1,\dots,\vert\lambda\vert$) be as in Definition \ref{defn6488} (i)-(v). Then define an equivalence relation on $\bbC_e$ where $({\bf x}_1,V_1) \sim ({\bf x}_2,V_2)$ if:
\[
  V_2 = a_{\lambda(v_1)+1}\dots a_{\lambda(v_2)}  V_1.
\]
We thus obtain a line bundle $\mathscr L_e$ on $\widehat{\mathcal U}(\frak u,\Upsilon)$. The group $\Gamma_{\frak u}$ acts on $\bigoplus_e\mathscr L_e$ in an obvious way. Our obstruction bundle $\mathcal E_{\frak u}$ on $\widehat{\mathcal U}(\frak u,\Upsilon)$ is defined to be:
\begin{equation}\label{form1162}
	\mathcal E_{\frak u} = \mathcal E_{0,\frak u} \oplus 
	\bigoplus_{e \in C^{\rm int}_{\rm th}(\check R),\, \lambda(e) > 0} \mathscr L_e.
\end{equation}
It induces an orbi-bundle on ${\mathcal U}(\frak u,\Upsilon)$. By an abuse of notation, this orbi-bundle is also denoted by $\mathcal E_{\frak u}$.

Next, we define Kuranishi maps. If $v$ is a vertex with $\lambda(v) = 0$, then we define:
\begin{equation}\label{form630}
	\frak s_{\frak u,v}(\vec{\frak x},\vec{\sigma},(u'_{v}),(U'_{v}),(\rho_e),(\rho_i))= 
	\overline{\partial} u'_v \in  E_v(u'_{v}).
\end{equation}
If $v$ is a vertex with $\lambda(v) >0$, then we define:
\begin{equation}\label{form631}
\frak s_{\frak u,v}(\vec{\frak x},\vec{\sigma},(u'_{v}),(U'_{v}),(\rho_e),(\rho_i))
= 
\overline{\partial} U'_v \in  E_v(U'_{v}).
\end{equation}
If $e$ is a thick edge connecting vertices $v_1$ and $v_2$ with $0\leq \lambda(v_1)<\lambda(v_2)$, then we define:
\begin{equation}\label{form632}
	\frak s_{\frak u,e}(\vec{\frak x},\vec{\sigma},(u'_{v}),(U'_{v}),(\rho_e),(\rho_i))
	= \rho_{e} - \rho_{\lambda(v_1)+1}\cdots \rho_{\lambda(v_2)}.
\end{equation}
We define:
\[
  \frak s_{\frak u} = ((\frak s_{\frak u,v};v \in C^{\rm int}_0(\check R)),
  (\frak s_{\frak u,e};e \in C^{\rm int}_{\rm th}(\check R), \lambda(e) > 0)).
\]
It is easy to see that $\frak s_{\frak u}$  induces a $\Gamma_{\frak u}$-invariant section of $\mathcal E_{\frak u}$.
Using Proposition \ref{prop653}, we can show that the section $\frak s_{\frak u}$ is smooth.

Suppose $\frak s_{\frak u}(\vec{\frak x},\vec{\sigma},(u'_{v}),(U'_{v}),(\rho_e),(\rho_i)) = 0$.
By (\ref{form632}) and Lemma \ref{lem650},
it induces an element $(\vec{\frak x},\vec{\sigma},u')$ of $\widehat{\mathcal U}_0(\frak u,\Upsilon)$. By (\ref{form630}) and (\ref{form631}) the map $u'$ is pseudo-holomorphic. Therefore, $(\Sigma({\vec{\frak x}},\vec{\sigma}),u')$ and marked points on $\Sigma_{\vec{\frak x}}$ determine an element of $\mathcal M_{k+1}^{\rm RGW}(L;\beta)$. This element does not change if we change $(\vec{\frak x},\vec{\sigma},(u'_{v}),(U'_{v}),(\rho_e),(\rho_i))$ by the $\Gamma_{\frak u}$-action.
We thus obtained:
\begin{equation}\label{paramatp}
\psi _{\frak u}: \frak s_{\frak u}^{-1}(0)/\Gamma_{\frak u} \to \mathcal M_{k+1}^{\rm RGW}(L;\beta),
\end{equation}
which is a homeomorphism onto an open neighborhood of $\frak u$.
We thus proved:
\begin{theorem}\label{pro65411}
$\mathcal U_{\frak u} = (\mathcal U(\frak u),\mathcal E_{\frak u},\Gamma_{\frak u},\frak s_{\frak u},\psi_{\frak u})$ givecs a Kuranshi chart for the 
moduli space $\mathcal M_{k+1}^{\rm RGW}(L;\beta)$ at $\frak u$.
\end{theorem}

\section{Construction of Kuranishi Structures}
\label{sub:kuracont}

So far, we constructed a Kuranishi chart at each point of $\mathcal M_{k+1}^{\rm RGW}(L;\beta)$.
In this section we construct a global Kuranishi structure. We follow similar arguments as in \cite{foootech, fooo:const1}. However, there are certain points that our treatment is different. We discuss the construction emphasizing on those differences. 

\subsection{Compatible Trivialization and Stabilization Data}
\label{subsub:Nplustri}

Throughout this subsection% and the next {\color{red} two} subsections
, we fix:
$$
\hspace{3cm}\frak u_{(j)} = ((\Sigma_{(j),v},\vec z_{(j),v},u_{(j),v});v \in C^{\rm int}_0(\check R_{(j)}))\hspace{1cm} j=1,\,2
$$
an element of $\mathcal M_{k+1}^{\rm RGW}(L;\beta)$ which is contained in the stratum corresponding to the very detailed DD-ribbon trees $\check R_{(j)} = (c_{(j)},\alpha_{(j)},m_{(j)},\lambda_{(j)})$. We denote the union of the irreducible components $\Sigma_{(j),v}$ by $\Sigma_{(j)}$. The map $u_{(j)}$ are also defined similarly, and $\vec z_{(j)}$ is the set of boundary marked points of $\Sigma_{(j)}$. We use a similar convention several times in this section.  We assume that $\frak u_{(2)}$ belongs to a small neighborhood of $\frak u_{(1)}$ in the RGW-topology. To be more precise, for %each element $\frak u_{(j)}$ of  $\mathcal M_{k+1}^{\rm RGW}(L;\beta)$, we fix a small TSD 
$\frak u_{(j)}$, let $\Xi_{(j)}= (\vec w_{(j)},(\mathcal N_{(j),v,i}),(\phi_{(j),v}),(\varphi_{(j),v,e}),\kappa_{(j)})$ be a fixed TSD. We assume that $\frak u_{(2)}$ is represented by an element of the space ${\frak U}(\frak u_{(1)},\Xi_{(1)})$. 

This assumption implies that $\check R_{(2)}$ is obtained from $\check R_{(1)}$ by level shrinkings, level $0$ edge shrinkings and fine edge shrinkings. In particular, we may regard:
\[
  C^{\rm int}_1(\check R_{(2)})\subseteq C^{\rm int}_1(\check R_{(1)}).
\]
There also exists a surjective map $\pi : \check R_{(1)} \to \check R_{(2)}$ inducing:
\[
  \pi : C^{\rm int}_0(\check R_{(1)}) \to C^{\rm int}_0(\check R_{(2)})
\]
such that the irreducible component corresponding to $v \in C^{\rm int}_0(\check R_{(2)})$ is obtained by gluing the irreducible components corresponding to $\hat v \in \pi^{-1}(v) \subset C^{\rm int}_0(\check R_{(1)})$.  There also exists a surjective map:
\[
  \nu : \{0,1,\dots,\vert\lambda_{(1)}\vert\} \to \{0,1,\dots,\vert\lambda_{(2)}\vert\}
\]
such that $i\le j$ implies $\nu(i) \le \nu(j)$, and $\lambda_{(2)}(\pi(\hat v)) =\nu(\lambda_{(1)}(\hat v))$  for inside vertices $\hat v$ of $\check R_{(1)}$. The maps $\pi$ and $\nu$ are the analogue of ${\rm treesh}$ and ${\rm levsh}$ in \cite[Lemma 4.47]{part1:top} defined for detailed trees.

%In the construction of Kuranishi charts at an element $\frak u= ((\Sigma_{v},\vec z_{v},u_{v});v \in C^{\rm int}_0(\check R))$, we made several choices. These choices include the choice of additional marked points $\vec w_{v}$, transversals $\mathcal N_{v,i}$ and trivializations of the universal family:
%$$
%\phi_{(j),v} : \mathcal V_{(j),v}^{{\rm source}} \times \Sigma_{(j),v} \to \mathcal {C}_{(j),v}^{\rm source}.
%$$
%We also made choices of analytic families of coordinates 
%\begin{equation}\label{form19009rev}
%\varphi_{(j),v,e} :   \mathcal V_{(j),v}^{{\rm source}}  \times {\rm Int} D^2
%\to \mathcal {C}_{(j),v}^{\rm source}
%\end{equation}
%at the nodal points, and similar analytic families for boundary nodal points.
%\begin{definition}
%	We call $\Xi= (\vec w_v,(\mathcal N_{v,i}),(\phi_{v}),(\varphi_{v,e}))$ 
%	a choice of {\it stabilization and trivialization data} ({\it TSD}) for $\frak u$.
%\end{definition}

To describe the coordinate change, it is convenient to start with the case that the TSDs $\Xi_{(1)}$ and $\Xi_{(2)}$ satisfy some compatibility conditions. In this subsection% and the next {\color{red} two or three} subsections
, we discuss these compatibility conditions and in Subsection \ref{subsub:coordinatechage1}, we explain how a coordinate change can be constructed assuming these conditions. In Subsection \ref{subsub:coordinatechage2}, we consider the case that $\Xi_{(1)}$ and $\Xi_{(2)}$ are two (not necessarily compatible) TSDs associated to the same element of the moduli space. We combine the results of Subsections \ref{subsub:coordinatechage1} and \ref{subsub:coordinatechage2} in Subsection \ref{subsub:coordinatechage3} to define coordinate changes in the general case and verify the co-cycle condition for these coordinate changes.

%Let $\Xi_{(1)}= (\vec w_{(1)},\mathcal N_{(1),v,i},\phi_{(1),v},\varphi_{(1),v,e})$ be a fixed choice of stabilization and trivialization data at the point $\frak u_{(1)}$. Then 
The assumption that $\frak u_{(2)}$ belongs to a small neighborhood of $\frak u_{(1)}$ implies that we can find:
$$\vec\sigma_0 = (\sigma_{0,e} ; e \in  C^{\rm int}_1(\check R_{(1)})) \in
\prod_{e  \in  C^{\rm int}_1(\check R_{(1)})}\mathcal V_{(1),e}^{{\rm deform}}
$$
and
$$
\vec{\frak x}_{0} = (\frak x_{0,v} ; v  \in  C^{\rm int}_0(\check R_{(1)}))
\in 
\prod_{v  \in  C^{\rm int}_0(\check R_{(1)})}\mathcal V_{(1),v}^{{\rm source}}
$$
such that the inconsistent map:
\begin{equation}\label{incon-basis-change}
	(\Sigma_{(1)}(\vec{\frak x}_0,\vec \sigma_0),\vec z_{(1)}(\vec{\frak x}_0,\vec{\sigma}_0),u_{(1)}
	(\vec{\frak x}_0,\vec{\sigma}_0))
\end{equation}
is isomorphic to $(\Sigma_{(2)},\vec z_{(2)},u_{(2)})$. Although it is not clear from the notation, the map $u_{(1)}(\vec{\frak x}_0,\vec{\sigma}_0)$ in \eqref{incon-basis-change} depends on $\frak u_{(2)}$ and not just on $\vec{\frak x}_0$ and $\vec{\sigma}_0$. We assume that the additional marked points and transversals in $\Xi_{(2)}$ satisfy the following conditions:

\begin{conds}\label{choi655}
	Since \eqref{incon-basis-change} is induced by an element of ${\mathfrak U}(\frak u_{(1)},\Xi_{(1)})$,
	there is a set of marked points $\vec w_{(1)}(\vec{\frak x}_0,\vec \sigma_0)\subset 
	\Sigma_{(1)}^{-}(\vec{\frak x}_0,\vec \sigma_0)$ determined by $\vec w_{(1)}$, $\vec{\frak x}_0$ and
	$\vec \sigma_0$.
 	Then we require that the marked points $\vec w_{(2)}$ of $\Xi_{(2)}$ are chosen such that:
	\begin{equation}\label{6134form}
		(\Sigma_{(1)}(\vec{\frak x}_0,\vec \sigma_0),\vec z_{(1)}(\vec{\frak x}_0,\vec \sigma_0)
		\cup \vec w_{(1)}(\vec{\frak x}_0,\vec \sigma_0))\cong
		(\Sigma_{(2)},\vec z_{(2)}\cup \vec w_{(2)}).
	\end{equation}
	Furthermore, if $w_{(2),v,i}$ corresponds to $w_{(1),\hat v,\hat i}$, then we require:\footnote{
	Here we use the correspondence between $\vec w_{(1)}$ $\vec w_{(2)}$ given 
	by the identification in \eqref{6134form} and the correspondence between the elements of 
	$\vec w_{(1)}$ and $\vec w_{(1)}(\vec{\frak x}_0,\vec \sigma_0)$.}
	\begin{equation}
		\mathcal N_{(2),v',i'} = \mathcal N_{(1),v,i}.
	\end{equation}	
%	
%	
%	{\color{red} We also demand that $u'_{(2),v}$ (resp. $U'_{(2),v}$), restricted to 
%	$\Sigma_{(1),v}^{-}(\vec{\frak x}_0,\vec \sigma_0)$ (see \eqref{form610}.) is $C^0$-close to $u'_{(1),v}$ 
%	(resp. $U'_{(1),v}$)}. (Here we use $\phi_{(j),v}$ and $\varphi_{(j),v,e}$ for various $v$, $e$ and $j=1,2$ 
%	to identify the domains of the two maps.)
%	Suppose $w_{(1),v,i}$ corresponds to $w_{(1),v',i'}$ by 
%	the above isomorphism. Then we require: 
%	\begin{equation}
%		\mathcal N_{(2),v',i'} = \mathcal N_{(1),v,i}.
%	\end{equation}
\end{conds}

Next, we impose some constraints on the choices of the maps $\phi_{(2),v}$ and $\varphi_{(2),v,e}$. Let $v$ be an interior vertex of $\check R_{(2)}$. We consider the moduli space $\mathcal M^{\rm source}_v$ of deformation of the irreducible component $(\Sigma_{(2),v},\vec z_{(2),v}\cup \vec w_{(2),v})$. We firstly fix a neighborhood $\mathcal V^{\rm source}_{(2),v}$ of $[\Sigma_{(2),v},\vec z_{(2),v}\cup \vec w_{(2),v}]$ in $\mathcal M^{\rm source}_v$  as follows. The Riemann surface $\Sigma_{(2),v}$ is obtained by gluing spaces $\Sigma_{(1),\hat v}$ for $\hat v \in \pi^{-1}(v)$. Here the complex structure on $\Sigma_{(1),\hat v}$ is given by $\frak x_{0,\hat v}$ and the gluing parameters are $\sigma_{(0),\hat e}$ for edges $\hat e$ in $\pi^{-1}(v)$.
There is a neighborhood $\mathcal U_{(1),\hat v}^{{\rm source}}$ of $\frak x_{0,\hat v}$ in $\mathcal V_{(1),\hat v}^{{\rm source}}$ and a neighborhood $\mathcal V_{(1),(2),e}^{{\rm deform}}$ of $\sigma_{(0),\hat e}\in \mathcal V_{(1),e}^{{\rm deform}}$ such 
that the following map:
\[	
\prod_{\hat v \in C^{\rm int}_0(\check R_{(1)}),\, \pi(\hat v) = v}
\mathcal U_{(1),\hat v}^{{\rm source}}\times \prod_{e \in C^{\rm int}_1(\check R_{(1)}),\, \pi(e) = v}
\mathcal V_{(1),(2),e}^{{\rm deform}}\to \mathcal M^{\rm source}_v.
\]
is an isomorphism onto an open neighborhood of the point determined by $(\Sigma_{(2),v},\vec z_{(2),v}\cup \vec w_{(2),v})$. Therefore, we may define:
\begin{equation}\label{form6143}
	\mathcal V_{(2),v}^{{\rm source}}:=\prod_{\hat v \in C^{\rm int}_0(\check R_{(1)}),\, \pi(\hat v) = v}
	\mathcal U_{(1),\hat v}^{{\rm source}}\times \prod_{\hat e \in C^{\rm int}_1(\check R_{(1)}),\, \pi(\hat e) = v}
	\mathcal V_{(1),\hat e}^{{\rm deform}}
\end{equation}

Let $\frak x_{2,v} = ((\frak x_{1,\hat v}),(\sigma_{1,\hat e}))$ be an element of \eqref{form6143}. Then $\Sigma_{(2),v}(\frak x_{2,v})$, the Riemann surface $\Sigma_{(2),v}$ with the complex structure induced by $\frak x_{2,v}$, has the following decomposition:
\begin{align}
	\Sigma_{(2),v}(\frak x_{2,v})=
	&\coprod_{\hat v \in C^{\rm int}_0(\check R_{(1)}),\, \pi(\hat v) = v} \Sigma^-_{(1),\hat v}(\frak x_{1,\hat v})
	\label{form614444pre}\\
	&\cup \coprod_{e \in C^{\rm int}_1(\check R_{(2)}),\, v \in \partial e} D^2 
	\label{form614444pre-2}\\
	&\cup \coprod_{\hat e \in C^{\rm int}_1(\check R_{(1)}),\, \pi(\hat e) = v} [-5T_{\hat e},5T_{\hat e}] 
	\times S^1.
	\label{form614444pre-3}
\end{align}	
The following comments about the above decomposition is in order. In \eqref{form614444pre}, $\Sigma^-_{(1),\hat v}(\frak x_{(1),\hat v})$ denotes the subspace of $\Sigma_{(1),\hat v}(\frak x_{(1),\hat v})$ given by the complements of $\varphi_{(1),\hat v,\hat e}(\frak x_{(1),\hat v},D^2(1))$ where $\hat e$ runs among the edges of $\check R_{(1)}$ which are connected to $\hat v$. For each edge $e \in C^{\rm int}_1(\check R_{(2)})$ which is incident to $v\in C^{\rm int}_0(\check R_{(2)})$, there is a unique edge $\hat e\in C^{\rm int}_1(\check R_{(1)})$ which is mapped to $e$. In particular, one of the endpoints of $\hat e$, denoted by $\hat v$, is mapped to $v$. The disc corresponding to $e$ in \eqref{form614444pre-2} is given by the space $\varphi_{(1),\hat v,\hat e}(\frak x_{(1),\hat v},D^2(1))$. Finally if an edge $\hat e\in C^{\rm int}_1(\check R_{(1)})$ is mapped to a vertex $v\in C^{\rm int}_0(\check R_{(2)})$ by $\pi$, then the space in \eqref{form614444pre-3} is identified with the neck region associated to $\hat e$. In particular, the positive number $T_{\hat e}$ is determined by $\sigma_{(1),\hat e}$. The union of the spaces in \eqref{form614444pre} and \eqref{form614444pre-2} is called the thick part of $\Sigma_{(2),v}(\frak x_{(2),v})$, and the spaces in \eqref{form614444pre-3} form the thin part of $\Sigma_{(2),v}(\frak x_{(2),v})$. The above decomposition can be used in an obvious way to define the map $\varphi_{(2),v,e}$ on $\mathcal V_{(2),v}^{{\rm source}}\times {\rm Int}(D^2)$.

We have the following decomposition of $ \Sigma_{(2),v}$ as a special case of the above decomposition applied to the point $((\frak x_{0,\hat v}),(\sigma_{0,\hat e}))$:
\begin{align}
	\Sigma_{(2),v}=&\coprod_{\hat v \in C^{\rm int}_0(\check R_{(1)}),\, \pi(\hat v) = v} 
	\Sigma^-_{(1),v}(\frak x_{0,\hat v})\label{form614444}\\
	&\cup \coprod_{e \in C^{\rm int}_1(\check R_{(2)}),\, v \in \partial e} D^2%{\rm Im}(\varphi_{(1),v,e})
	\label{form614444-2} \\
	&\cup \coprod_{\hat e \in C^{\rm int}_1(\check R_{(1)}),\, \pi(\hat e) = v} [-5T'_{\hat e},5T'_{\hat e}] 
	\times S^1.
	\label{form614444-3}
\end{align}
The trivialization $\phi_{(2),v}$ that we intend to define is a family (parametrized by $\frak x_{2,v}$) of diffeomorphisms from $\Sigma_{(2),v}$ to $\Sigma_{(2),v}(\frak x_{2,v})$. The trivialization $\phi_{(1),v}$ defines a diffeomorphism between the subspaces in \eqref{form614444pre} and \eqref{form614444}. We then use the coordinate at nodal points, $\varphi_{(1),\hat v,\hat e}$, to extend it to a diffeomorphism from the unions of the subspaces in \eqref{form614444} and \eqref{form614444-2} to the union of the subsapces in \eqref{form614444pre} and \eqref{form614444pre-2}. Finally we extend this family of diffeomorphisms in an arbitrary way to the neck region to complete the construction of $\phi_{(2),v}$. This construction of the maps  $\varphi_{(2),v,e}$ and $\phi_{(2),v}$ is analogous to \cite[Sublemma 10.15]{fooo:const1}. 

\begin{conds}\label{choi655-2}
	We require that the maps $\phi_{(2),v}$ and $\varphi_{(2),v,e}$ of the TSD $\Xi_{(2)}$ 
	are obtained form the TSD $\Xi_{(1)}$ as above.
\end{conds}

\begin{definition}
We say that a TSD $\Xi_{(2)}$ is {\it induced} from a TSD $\Xi_{(1)}$
if Conditions \ref{choi655} and \ref{choi655-2} are satisfied.
\end{definition}

\begin{definition}\label{incon-map-near-u-gen}
	Let $\frak u$ be an element of $\mathcal M_{k+1}^{\rm RGW}(L;\beta)$ and $\Xi$ be a TSD at $\frak u$. An {\it inconsistent map near $\frak u$ with respect to $\Xi$} is 
	an object similar to the ones in Definition \ref{inconssitentmap} where we do 
	not require the Cauchy-Riemann equations.
\end{definition}
This definition is almost a straightforward generalization of Definition \ref{defn625625} with the difference that we also include transversal constraints in \eqref{eq631} as one of the requirements for an inconsistent map near $\frak u$. 

Suppose $\Xi_{(2)}$ is induced from $\Xi_{(1)}$. For any:
\begin{equation}\label{form136}
	\aligned
	\vec\sigma_{(2)} &= (\sigma_{2,e} ; e \in  C^{\rm int}_1(\check R_{(2)})) \in
	\prod_{e  \in  C^{\rm int}_1(\check R_{(2)})}\mathcal V_{(2),e}^{{\rm deform}}\\
	\vec{\frak x}_{(2)} &= (\frak x_{2,v} ; v  \in  C^{\rm int}_0(\check R_{(2)}))\in 
	\prod_{v  \in  C^{\rm int}_0(\check R_{(2)})}\mathcal V_{(2),v}^{{\rm source}},
	\endaligned
\end{equation}
satisfying Definition \ref{defn6488} (1) with respect to $\Xi_{(2)}$, there exist:
\begin{equation}
	\aligned
	\vec\sigma_{(1)} &= (\sigma_{1,e} ; e \in  C^{\rm int}_1(\check R_{(1)})) \in
	\prod_{e  \in  C^{\rm int}_1(\check R_{(1)})}\mathcal V_{(1),e}^{{\rm deform}}\\
	\vec{\frak x}_{(1)} &= (\frak x_{1,v} ; v  \in  C^{\rm int}_0(\check R_{(1)}))\in 
	\prod_{v  \in  C^{\rm int}_0(\check R_{(1)})}\mathcal V_{(1),v}^{{\rm source}}.
	\endaligned
\end{equation}
satisfying Definition \ref{defn6488} (1) with respect to $\Xi_{(1)}$ such that:
\begin{equation}\label{form6138}
	(\Sigma_{(2)}(\vec{\frak x}_{(2)},\vec \sigma_{(2)}),\vec z_{(2)}(\vec{\frak x}_{(2)},
	\vec \sigma_{(2)})\cup \vec w_{(2)}(\vec{\frak x}_{(2)},\vec \sigma_{(2)}))\cong\hspace{2cm}
\end{equation}
\begin{equation*}
	\hspace{2cm}(\Sigma_{(1)}(\vec{\frak x}_{(1)},\vec \sigma_{(1)}),\vec z_{(1)}(\vec{\frak x}_{(1)},\vec
	\sigma_{(1)}) \cup \vec w_{(1)}(\vec{\frak x}_{(1)},\vec \sigma_{(1)})).
\end{equation*}
Here $\vec{\frak x}_{(1)},\vec \sigma_{(1)}$ depend on $\vec{\frak x}_{(2)},\vec \sigma_{(2)}$. (See Figure \ref{Figuresec6p164}.)

\begin{figure}[h]
\centering
\includegraphics[scale=0.5]{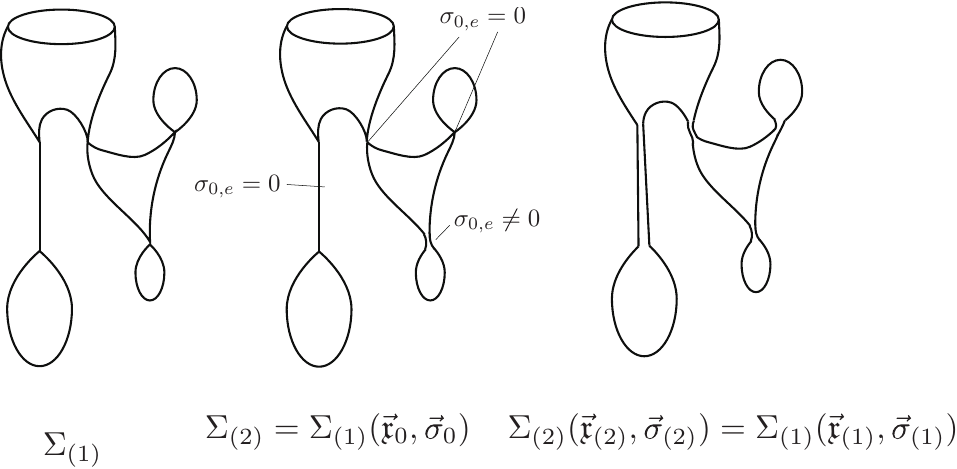}
\caption{$\Sigma_{(1)}$, $\Sigma_{(2)}$ and $\Sigma_{(2)}(\vec {\frak x}_{(2)},{\vec \sigma}_{(2)})$.}
\label{Figuresec6p164}
\end{figure}

Next, let $\frak y_{(2)} = (\vec{\frak x}_{(2)},\vec{\sigma}_{(2)},(u'_{(2),v}),(U'_{(2),v}),(\rho_{(2),e}),(\rho_{(2),i}))$ be an inconsitent map near $\frak u_{(2)}$ with respect to $\Xi_{(2)}$. Let $\vec{\frak x}_{(1)}$ and $\vec\sigma_{(1)}$ be chosen as in the previous paragraph. Let $\hat v$ be an interior vertex of $\check R_{(1)}$. Identification in (\ref{form6138}) and Condition \ref{choi655-2} imply:
\begin{equation}\label{form614343bef}
	\Sigma^+_{(1),\hat v}(\vec{\frak x}_{(1)},\vec{\sigma}_{(1)})\subseteq \Sigma^+_{(2),\pi(\hat v)}
	(\vec{\frak x}_{(2)},\vec{\sigma}_{(2)}).
\end{equation}
We write $I_{\hat v}$ for this inclusion map. Define:
\begin{equation}\label{form614343}
	\aligned
	U'_{(1),\hat v} &= 
	\begin{cases}
	u'_{(2),\pi(\hat v)}\circ I_{\hat v} &\text{if $\lambda_{(2)}(\pi(\hat v)) = 0$} \\
	U'_{(2),\pi(\hat v)}\circ I_{\hat v} &\text{if $\lambda_{(2)}(\pi(\hat v)) > 0$}
	\end{cases}
	\\
	u'_{(1),\hat v} &=u'_{(2),\pi(\hat v)}\circ I_{\hat v}.
	\endaligned
\end{equation}
We also define:
\begin{equation}\label{newform614154}
	\rho_{(1),e}= 
	\begin{cases}
		\rho_{(2),e} &\text{if $e \in C^{\rm int}_{\rm th}(\check R_{(2)})\subset
		C^{\rm int}_{\rm th}(\check R_{(1)})$} \\
		&\text{and $e$ is not a level $0$ edge,}\\
		1 &\text{otherwise}.
	\end{cases}
\end{equation}
We next define:
\begin{equation}
	\rho_{(1),i} = 
	\begin{cases}
		1   \qquad &\text{if $ \nu(i-1) = \nu(i)$} \\
		\rho_{(2),i} \qquad &\text{otherwise}.
        \end{cases}
\end{equation}
It is easy to check that $\frak y_{(1)} = (\vec{\frak x}_{(1)},\vec{\sigma}_{(1)},(u'_{(1),\hat v}),(U'_{(1),\hat v}),(\rho_{(1),e}),(\rho_{(1),i}))$ is an inconsistent map near $\frak u_{(1)}$ with respect to $\Xi_{(1)}$. In fact, \eqref{form621}-\eqref{form623rev} for $\frak y_{(1)}$ follow from the corresponding identities for $\frak y_{(2)}$ and the definition. This discussion is summarized in the following lemma:
\begin{lemma}\label{lem65865777}
	Suppose $\frak u_{(1)}$ and $\frak u_{(2)}$ are as above and $\Xi_{(j)}$ is a TSD at 
	$\frak u_{(j)}$. We assume that $\Xi_{(2)}$ is induced from $\Xi_{(1)}$.
	Then an inconsistent map near $\frak u_{(2)}$ with respect to $\Xi_{(2)}$ can be regarded as 
	an inconsistent map near $\frak u_{(1)}$ with respect to $\Xi_{(1)}$.
\end{lemma}

\begin{remark}
	The equality (\ref{newform614154}) in the second case shows that the consistency equation is satisfied for such edges.
\end{remark}

Using the above argument, we can also verify the following lemma:
\begin{lemma}
	Suppose $\frak u_{(1)}$, $\frak u_{(2)}$, $\Xi_{(1)}$ and $\Xi_{(2)}$ are 
	given as in Lemma \ref{lem65865777}.
	In particular, $\frak u_{(2)}$ can be regarded as an inconsistent solution\footnote
	{More precisely it is an element of $\widetilde{\frak U}(\frak u,\Xi)$.}: 
	$$(\vec{\frak x}_{(0)},\vec{\sigma}_{(0)},(u'_{(0),\hat v}),(U'_{(0),\hat v}),(\rho_{(0),e}),(\rho_{(0),i}))$$
	with respect to $\Xi_{(1)}$.
	An inconsistent map: 
	$$\frak y_{(1)} = (\vec{\frak x}_{(1)},\vec{\sigma}_{(1)},(u'_{(1),\hat v}),(U'_{(1),\hat v}),(\rho_{(1),e}),(\rho_{(1),i}))$$
	near $\frak u_{(1)}$ with respect to $\Xi_{(1)}$ may be regarded as an inconsistent map near 
	$\frak u_{(2)}$ with respect to $\Xi_{(2)}$
	if and only if the following conditions hold:
	\begin{itemize}
		\item[(i)] For each vertex $\hat v\in C^{\rm int}_0(\hat R_{(1)})$, 
			the distance between ${\frak x}_{(0),\hat v}$ and ${\frak x}_{(1),\hat v}$ is 
			less than $\kappa_{(2)}$. 
			(Recall that $\kappa_{(2)}$ is the size of $\Xi_{(2)}$).
		\item[(ii)] For each edge $e\in C^{\rm int}_1(\hat R_{(2)})\subset C^{\rm int}_1(\hat R_{(1)})$, 
		we have ${\vert}\sigma_{(1),e}{\vert}<\kappa_{(2)}$. 
		\item[(iii)] If $e\in C^{\rm int}_1(\hat R_{(2)})\subset C^{\rm int}_1(\hat R_{(1)})$ is (resp. is not) a level $0$ edge, 
	        	then the image of the restriction of $u'_{(1),\hat v}$ to
		$[-5T_e,5T_e]_{r_e} \times [0,\pi]_{s_e}$ (resp. $[-5T_e,5T_e]_{r_e} \times S^1_{s_e}$)
		has a diameter less than $\kappa_{(2)}$.
	\item[(iv)] If $c(v)={\rm d}$ or ${\rm s}$, then the $C^2$-distance between the 
		restrictions of $u'_{(0),\hat v}$ and $u'_{(1),\hat v}$ to $\Sigma^-_{v}(\vec{\frak x},\vec \sigma)$ is less than 
		$\kappa_{(2)}$. If $c(v)={\rm D}$, then we demand that 
		the $C^2$-distance between the restrictions of $U'_{(0),\hat v}$ and $U'_{(1),\hat v}$ 
		to $\Sigma^-_{v}(\vec{\frak x},\vec \sigma)$ is less than $\kappa_{(2)}$.
		\item[(v)] The consistency equation 
				$$
				  \rho_{(1),e} = \rho_{(1),\lambda_{(1),\hat v_1}+1}\dots \rho_{(1),\lambda_{(1),\hat v_2}}
				$$
				are satisfied for any edge 
				$e \in C^{\rm int}_{\rm th}(\check R_{(1)})\setminus C^{\rm int}_{\rm th}(\check R_{(2)})$ 
				which is not a level $0$ edge.
	\end{itemize}
\end{lemma}

\subsection{The Choice of Obstruction Spaces}
\label{subsub:chooseobst}
In order to define an inconsistent {\it solution}, we need to fix obstruction spaces. For any inconsistent map $\frak u$ and for any other inconsistent map $\frak y$ which is close to $\frak u$, we explained how to make a choice of $\mathcal E_{0,\frak u}(\frak y)$ in Section \ref{sub:kuraconst}. We need to insure that we can arrange for such choices such that they satisfy some nice properties when we move $\frak u$. More precisely, we need to pick them so that they are {\it semi-continuous} with respect to $\frak u$. We will prove this property in the next subsection. In this subsection, we explain how we modify our choice of obstruction spaces.

For $j=1,2$, let $\frak u_{(j)} = ((\Sigma_{(j),v},\vec z_{(j),v},u_{(j),v});v \in C^{\rm int}_0(\check R_{(j)}))$
be representatives of two elements of $\mathcal M_{k+1}^{\rm RGW}(L;\beta)$ in the strata corresponding to very detailed DD-ribbon trees $\check R_{(j)}$. We fix a TSD $\Xi_{(j)}$ at $\frak u_{(j)}$. We do {\it not} assume that they are related as in Subsection \ref{subsub:Nplustri}. We also fix 
obstruction bundle data $E = \{E_{\frak u_{(1)},v}\}$, where
$E_{\frak u_{(1)},v} \subset L^2_{m,\delta}(\frak u_{(1),v})$ satisfying the conditions in 
Definition \ref{cond643}. We wish to use $E$ to define obstruction spaces for an inconsistent map with respect to $\Xi_{(2)}$, under the assumption that $\frak u_{(2)}$ is close to $\frak u_{(1)}$. In particular, we assume that $\check R_{(2)}$ is obtained from $\check R_{(1)}$ by level shrinking, level $0$ edge shrinking and fine edge shrinking. As in the previous subsection, we may define a surjective map $\pi : \check R_{(1)} \to \check R_{(2)}$. Note that in Section \ref{sub:kuraconst0}, we studied the case $\frak u_{(2)} = \frak u_{(1)}$. 

The following lemma is a straightforward consequence of the implicit function theorem: (See \cite[Lemma 9.9]{fooo:const1}.)

\begin{lemma}\label{lemma660}
	If $\frak u_{(2)}$ is close enough to $\frak u_{(1)}$ with respect to the $C^1$ distance,
	then for any $\hat v\in C_0^{\rm int}(\check R_{(1)})$, there exists a unique choice of 
	$\vec w_{(2),(1),\hat v} \subset \Sigma_{(2),v}$ with $v$ being $\pi(\hat v)$ such that:
	\begin{enumerate}
		\item $(\Sigma_{(2),v},\vec z_{(2),v} \cup \coprod_{\pi(\hat v)=v} \vec w_{(2),(1),v})$
		is close to: 
		$$\coprod_{\pi(\hat v)=v} (\Sigma_{(1),\hat v},\vec z^{\,\prime}_{(1),\hat v} \cup \vec w_{(1),\hat v})$$
		in the moduli space of stable curves with marked points. 
		Here $\vec z^{\,\prime}_{(1),\hat v}$ is the subset of $\vec z_{(1),\hat v}$ 
		that is the set of all marked points on $\Sigma_{(1),v}$ and nodal points 
		on $\Sigma_{(1),v}$ which correspond to the edges $e$ incident to $v$ with 
		$\pi(e) \ne v$.
%		Moreover, we have a partition
%		of the marked points $\vec w_{(2),(1),v}$ of the following form:
%		\[
%		   \vec w_{(2),(1),v}=\coprod_{\pi(\hat v)=v}\vec w_{(2),(1),\hat v}
%		\]
%		where $\vec w_{(2),(1),\hat v}$ is the set of marked points on $\Sigma_{(2),v}$ 
%		in correspondence with the marked points $\vec w_{(1),\hat v}$ on $\Sigma_{(1),\hat v}$.
		\item $u_{(2),v}(w_{(2),(1),\hat v,i}) \in \mathcal N_{(1),\hat v, i}$.
	\end{enumerate}
\end{lemma}

%\begin{lemma}\label{lemma660}
%	If $\frak u_{(2)}$ is close enough to $\frak u_{(1)}$ with respect to the $C^1$ distance,
%	then for any $v\in C_0^{\rm int}(\check R_{(2)})$, there exists a unique choice of 
%	$\vec w_{(2),(1),v} \subset \Sigma_{(2),v}$ such that:
%	\begin{enumerate}
%		\item $(\Sigma_{(2),v},\vec z_{(2),v} \cup \vec w_{(2),(1),v})$
%		is close to $(\Sigma_{(1),\hat v},\vec z_{(1), \hat v} \cup \vec w_{(1),\hat v})$
%		in the moduli space of stable curves with marked points. Moreover, we have a partition
%		of the marked points $\vec w_{(2),(1),v}$ of the following form:
%		\[
%		   \vec w_{(2),(1),v}=\coprod_{\pi(\hat v)=v}\vec w_{(2),(1),\hat v}
%		\]
%		where $\vec w_{(2),(1),\hat v}$ is the set of marked points on $\Sigma_{(2),v}$ 
%		in correspondence with the marked points $\vec w_{(1),\hat v}$ on $\Sigma_{(1),\hat v}$.
%		\item $u_{(2),v}(w_{(2),(1),v,i}) \in \mathcal N_{(1),i}$.
%	\end{enumerate}
%\end{lemma}

From now on, we assume that $\frak u_{(2)}$ is close enough to $\frak u_{(1)}$ such that the claim in Lemma \ref{lemma660} holds. Furthermore,  let also an inconsistent map:
$$
\frak y = (\vec{\frak x},\vec{\sigma},(u'_{v}),(U'_{v}),(\rho_{e}),(\rho_{i}))
$$
with respect to  $\Xi_{(2)} = (\vec w_{(2)},(\mathcal N_{(2),v,i}),(\phi_{(2),v}),(\varphi_{(2),v,e}),\kappa_{(2)})$ be fixed. Suppose also $(\Sigma(\vec{\frak x},\vec{\sigma}),\vec z(\vec{\frak x},\vec{\sigma}) \cup \vec w(\vec{\frak x},\vec{\sigma}))$ denotes the representative of $\vec{\frak x}$, as a part of the data of $\frak y$.

Let $\hat v$ be a vertex of $\check R_{(1)}$ and $v = \pi(\hat v)$. Lemma \ref{lemma660} allows us to find $w_{(2),(1),\hat v,i} \in \Sigma_{(2),v}$. If $\Xi_{(2)}$ is small enough, then we can regard $w_{(2),(1),\hat v,i}$ as an element of $\Sigma_{(2),v}^-$, and hence an element of $\Sigma_{(2),v}^-(\vec{\frak x},\vec{\sigma})$. This implies that if we replace $\vec w(\vec{\frak x},\vec{\sigma})$ with the points $w_{(2),(1),\hat v,i}$ to obtain: 
\begin{equation}\label{surgered-element}
	(\Sigma_v(\vec{\frak x},\vec{\sigma}),\vec z_v(\vec{\frak x},\vec{\sigma}) \cup \coprod_{\pi(\hat v)=v}
	\vec w_{(2),(1),\hat v})
\end{equation}
then \eqref{surgered-element} is close to $\coprod_{\pi(\hat v)=v} (\Sigma_{(1),\hat v},\vec z^{\,\prime}_{(1),\hat v} \cup \vec w_{(1),\hat v})$.

%Let $\hat v$ be a vertex of $\check R_{(1)}$ and $v = \pi(\hat v)$. Lemma \ref{lemma660} allows us to find $w_{(2),(1),\hat v,i} \in \Sigma_{(2),v}$ and regard it as an element of $\Sigma^-_{(1),\hat v}(\vec{\frak x},\vec{\sigma})$. We forget marked points corresponding to $\vec q_{v}$ and add $w_{(1),(0),\hat v,i}$ for all choices of $\hat v$ and $i$ with $\pi(\hat v) = v$. {\color{red} Then $(\Sigma_v^-(\vec{\frak x},\vec{\sigma}),\vec z_v\cup \vec w_{(1),(0),v})$. This is close to the union $\bigcup_{\hat v: \pi(\hat v) = v}(\Sigma_{(1),\hat v}^-,\vec z^{\,\prime}_{(1),\hat v},\vec w_{(1),\hat v})$.
%(Here $\vec z^{\,\prime}_{(0),\hat v}$ is the subset of $\vec z_{(0),\hat v}$ that is the set of 
%all marked points on $\Sigma_{(0),v}$ and nodal points 
%on $\Sigma_{(0),v}$ which corresponds to the edges $e$ with 
%$\pi(e) \ne v$.)}

We use this fact and the target parallel transportation in the same way as in Section \ref{sub:Obst} to obtain the following map for any $\hat v \in C^{\rm int}_0(\check R_{(1)})$:
$$
\mathcal P_{\hat v} : E_{\frak u_{(1)},\hat v} \to L^2_{m,\delta}(\Sigma^-_{(2),v}(\vec{\frak x},\vec{\sigma})).
$$
We then define for any $v\in \check R_{(2)}$:
$$
E_{\frak u_{(2)},\frak u_{(1)},v}(u'_v) = \bigoplus_{\hat v, \pi(\hat v) = v} \mathcal P_{\hat v}(E_{\frak u_{(1)},\hat v})
$$
for $\lambda(v) = 0$.
We also define $E_{\frak u_{(2)},\frak u_{(1)},v}(U'_v)$ for $\lambda(v) > 0$ by a similar formula.

Now we replace \eqref{form63131} by
\begin{equation}\label{form6313122}
	\aligned
	\mathcal E_{0,\frak u_{(2)},\frak u_{(1)}}(\frak y)= 
	&\bigoplus_{v \in C_0^{\rm int}(\check R_{(2)}),\,\lambda(v) = 0} E_{\frak u_{(2)},\frak u_{(1)},v}(u'_v)  \\
	&\oplus\bigoplus_{v \in C_0^{\rm int}(\check R_{(2)}),\,\lambda(v) > 0} E_{\frak u_{(2)},\frak u_{(1)},v}(U'_v).
	\endaligned
\end{equation}
\begin{lemma}\label{lemn6661}
	Suppose $\Xi_{(1)}$ is small enough such that we can apply the construction of 
	Section \ref{sub:kuraconst} to $(\Xi_{(1)},E)$	to obtain a 
	Kuranishi chart at $\frak u_{(1)}$. 
	Then for $\Xi_{(2)}$ small enough, 
	applying the construction of Section \ref{sub:kuraconst} to 
	$\{E_{\frak u_{(2)},\frak u_{(1)},v}\}$ (instead of 
	$\{E_{\frak u_{(2)},v}\}$) gives rise 
	to a Kuranishi chart at $\frak u_{(2)}$.
\end{lemma}
This is immediate from the construction of Section \ref{sub:kuraconst}. In fact, the choice of obstruction bundles we take here satisfies the `smoothness' condition. See \cite[Definition 3.11 (5)]{part3:FH} or \cite[Definition 5.1 (2)]{foooexp}. (Smoothness here means smoothness with respect to $\frak y$.)
\begin{remark}
We remark the obstruction spaces we used in Section \ref{sub:kuraconst} to 
construct a Kuranishi neighborhood (in the case of $
\frak u = \frak u_{(2)}$) is slightly different from our 
choice here.
The obstruction space in Section \ref{sub:kuraconst} 
is defined by a target parallel transformation 
from the obstruction space defined on $\frak u_{(2)}$.
Here the obstruction space (\ref{form6313122}) 
is defined by a target parallel transformation 
from the obstruction space defined on $\frak u_{(1)}$.
\end{remark}
%\begin{remark}
%	{\color{red} 
%	Note that the conclusion of Lemma \ref{lemn6661} holds if 
%	$D_{\frak u_{(2)}}\overline\partial$ is transversal to 
%	$\mathcal E_{0,\frak u_{(2)};\frak u_{(1)}}(\frak u_{(1)})$.}
%\end{remark}

For each $\frak p \in \mathcal M^{\rm RGW}_{k+1}(L;\beta)$, 
we consider $\Upsilon_{\frak p} = (\Xi_{\frak p},E_{\frak p})$, where $E_{\frak p}:=\{E_{\frak p,v}\}$ gives obstruction spaces at $\frak p$ and $\Xi_{\frak p}=(\vec w_{\frak p}, (\mathcal N_{\frak p,v,i}), (\phi_{\frak p,v}),(\varphi_{\frak p,v,e}),\kappa_{\frak p})$ is a TSD for $\frak p$. We assume that $\Xi_{\frak p}$ is small enough such that the assumption of Lemma \ref{lemn6661} holds. Let $\frak U(\frak p) := 
\frak U(\frak p,\Xi_{\frak p}) \cap \mathcal M^{\rm RGW}_{k+1}(L;\beta)$ be a neighborhood of $\frak p$ in $\mathcal M^{\rm RGW}_{k+1}(L;\beta)$ determined by the TSD $\Xi_{\frak p}$. We also fix a compact neighborhood $\mathcal K(\frak p)$ of $\frak p$ which is a subset of $\frak U(\frak p)$.

%For each $\frak p \in \mathcal M^{\rm RGW}_{k+1}(L;\beta)$, we fix a choice of stabilization and trivialization data $\Xi_{\frak p}=(\vec w_{\frak p}, (\mathcal N_{\frak p,v,i}), (\phi_{\frak p,v}),(\varphi_{\frak p,v,e}))$ and the obstruction space $E_{\frak p,\hat v}$. Using $\Xi_{\frak p}$, we fix an open neighborhood $\frak U(\frak p)$ of $\frak p$ in $\mathcal M^{\rm RGW}_{k+1}(L;\beta)$ such that the conclusion of Lemma \ref{lemn6661} holds for $\frak u_{(1)}=\frak p$ and any $\frak u_{(2)} \in \frak U(\frak p)$. We also fix a compact neighborhood $\mathcal K(\frak p)$ of $\frak p$ which is a subset of $\frak U(\frak p)$.

Compactness of $\mathcal M^{\rm RGW}_{k+1}(L;\beta)$ implies that we can find a finite subset 
\begin{equation}\label{formfrakJ}
{\frak J} = \{\frak p_j : j=1,\dots,J\} \subset \mathcal M^{\rm RGW}_{k+1}(L;\beta)
\end{equation}
such that
\begin{equation}
\mathcal M^{\rm RGW}_{k+1}(L;\beta)
\subseteq 
	\bigcup_{j=1}^{J} {\rm Int}\,\mathcal (K(\frak p_j)).
\end{equation}
For $\frak u \in \mathcal M^{\rm RGW}_{k+1}(L;\beta)$, we define:
\begin{equation}
{\frak J}(\frak u) = \{ \frak p_j \mid \frak u \in \mathcal K(\frak p_j)\}.
\end{equation}
\begin{lemma} \label{sum=direct-sum}
	Let $\check R_j$ be the very detailed tree associated to $\frak p_j$.
	We can perturb $\{E_{\frak p_j,v} \mid v \in C^{\rm int}_0(\check R_j)\}$
%	\footnote{{\color{red} 
%	The very detailed DD-ribbon tree $\mathcal R$ is one defining the stratum containing $\frak p$.}}
	by an arbitrary small amount so
	that the following holds. For any $\frak u \in \mathcal M^{\rm RGW}_{k+1}(L;\beta)$,
	the vector spaces $\mathcal E_{0,\frak u,\frak p_j}(\frak u)$ for $\frak p_j \in {\frak J}(\frak u)$
	are transversal, i.e.,
	the sum of $\mathcal E_{0,\frak u,\frak p_j}(\frak u)$ for $\frak p_j \in {\frak J}(\frak u)$
	is the direct sum:
	\begin{equation}
		\bigoplus_{\frak p_j \in {\frak J}(\frak u)} \mathcal E_{0,\frak u,\frak p_j}(\frak u).
	\end{equation}
\end{lemma}
\begin{proof}
	The proof is the same as the proof of the analogous statement in the case of 
	the stable map compactification.
	See the proof of \cite[Lemma 11.7]{fooo:const1} in \cite[Subsection 11.4]{fooo:const1}.
\end{proof}

Now we define a Kuranishi chart at each point $\frak u \in \mathcal M^{\rm RGW}_{k+1}(L;\beta)$ as follows:
\begin{definition}\label{Kura-chart}
	Let $\Xi_{\frak u} = (\vec w_{\frak u},(\mathcal N_{\frak u,v,i},\kappa_{\frak u}), 
	(\phi_{\frak u,v}),(\varphi_{\frak u,v,e}),\kappa_{\frak u})$ be a TSD, which is small enough such that 
	the conclusion of Lemma \ref{lemn6661} holds for $\frak u_{(1)}=\frak p_j$, $\frak u_{(2)}=\frak u$,
	$\Xi_{(1)}=\Xi_{\frak p_j}$ and $\Xi_{(2)}=\Xi_{\frak u}$ with $\frak p_j$ being an arbitrary element in 
	${\frak J}(\frak u)$.
	We put $\Upsilon_{(j)} = (\Xi_{(j)},E_{(j)})$,
	where $E_{(j)}$ is the obstruction space at $\frak u_{(1)}=\frak p_j$.
We define $E_{\frak u} = \{E_{(j)} \mid \frak p_j \in {\frak J}(\frak u)\}$	
and $\Upsilon_{\frak u} = (\Xi_{\frak u},E_{\frak u})$.
Hereafter we always take this choice of $E_{\frak u}$.
	The Kuranishi neighborhood $\mathcal U(\frak u,\Upsilon_{\frak u})$
	is the set of the equivalence classes of inconsistent maps 
	$\frak y = (\vec{\frak x},\vec{\sigma},(u'_{v}),(U'_{v}),(\rho_e),(\rho_i))$ near $\frak u$
	such that
	\begin{equation}\label{form6157}
		\aligned
		\overline\partial_{j_{\vec{\frak x},\vec \sigma}} u'_v &\in 
		\bigoplus_{\frak p_j \in {\frak J}(\frak u)} \mathcal E_{0,\frak u,\frak p_j}(\frak y),\\
		\overline\partial_{j_{\vec{\frak x},\vec \sigma}} U'_v &\in 
		\bigoplus_{\frak p_j \in {\frak J}(\frak u)} \mathcal E_{0,\frak u,\frak p_j}(\frak y).
		\endaligned
	\end{equation}
	In other words, it is the set of inconsistent solutions (Definition \ref{defn6488}) where 
	\eqref{form6119} and \eqref{form6120}
	are replaced with \eqref{form6157}.
	
	We also define $\widehat{\mathcal U}(\frak u,\Upsilon_{\frak u})$,
	$\widetilde{\mathcal U}(\frak u,\Upsilon_{\frak u})$
	in the same way as Definition \ref{defn6488}.
	
	The obstruction bundle $\mathcal E_{\frak u}$ is defined in the same way as in
	\eqref{form1162} in the following way:
	\begin{equation}\label{form116200}
		\mathcal E_{0,\frak u}(\frak y)=
		\bigoplus_{\frak p_j \in {\frak J}(\frak u)} \mathcal E_{0,\frak u,\frak p_j}(\frak y) 
	\end{equation}
	\begin{equation}\label{form1162rev}
		\mathcal E_{\frak u,\Xi_{\frak u}}(\frak y) = \mathcal E_{0,\frak u}(\frak y) 
		\oplus \bigoplus_{e \in C^{\rm int}_1(\check R),\, \lambda(e) > 0} \mathscr L_e.
	\end{equation}
	where $\check R$ is the very detailed tree associated to $\frak u$.
	The Kuranishi map $\frak s_{\frak u}$ is defined in the same way as in 
	(\ref{form630}), (\ref{form631}), (\ref{form632}), and the parametrization map $\psi _{\frak u}$ is defined 
	in the same way as in \eqref{paramatp}. 
\end{definition}

\subsection{Construction of Coordinate Change I}
\label{subsub:coordinatechage1}

In this and the next subsections, we construct coordinate changes. The next two lemmas state the semi-continuity of our obstruction spaces, a property that we hinted at the beginning of the last subsection.\footnote{Compare to \cite[Definition 5.1 (4)]{fooo:const1} or \cite[Definition 3.11  (4)]{part3:FH}.}

\begin{lemma}\label{loem668}
	For any $\frak u_{(1)} \in \mathcal M^{\rm RGW}_{k+1}(L;\beta)$, there exists a neighborhood 
	$U(\frak u_{(1)})$ of $\frak u_{(1)}$ in $\mathcal M^{\rm RGW}_{k+1}(L;\beta)$ such that
	for any $\frak u_{(2)} \in U(\frak u_{(1)})$:
	\begin{equation}\label{form6162}
		{\frak J}(\frak u_{(2)}) \subseteq {\frak J}(\frak u_{(1)}).
	\end{equation}
\end{lemma} 
\begin{proof}
	This is obvious because we pick the subspaces $\mathcal K(\frak p_j)$ to be closed.
\end{proof}

\begin{lemma}\label{lem6666}
	Let $\frak u_{(2)} \in U(\frak u_{(1)})$. Let 
	$\Xi_{(j)} = (\vec w_{(j)},(\mathcal N_{(j),v,i}), (\phi_{(j),v}),(\varphi_{(j),v,e}))$, for $j=1,2$,  be a TSD.
	We assume that $\Xi_{(2)}$ is induced from $\Xi_{(1)}$.
	Let
	$$
	\frak y_{(2)} = (\vec{\frak x}_{(2)},\vec{\sigma}_{(2)},(u'_{(2),v}),(U'_{(2),v}),(\rho_{(2),e}),(\rho_{(2),i}))
	$$
	be an inconsistent map near $\frak u_{(2)}$ with respect to $\Xi_{(2)}$ and 
	$$
	\frak y_{(1)} = (\vec{\frak x}_{(1)},\vec{\sigma}_{(1)},(u'_{(1),\hat v}),(U'_{(1),\hat v}),
	(\rho_{(1),e}),(\rho_{(1),i}))
	$$
	be the inconsistent map near $ \frak u_{(1)}$ with respect to $\Xi_{(1)}$, 
	constructed by Lemma \ref{lem65865777}.
	Let $\frak p_j \in {\frak J}(\frak u_{(2)})$. Then we have an isomorphism
	\begin{equation}\label{formnew6163}
		\mathcal E_{0,\frak u_{(2)};\frak p_j}(\frak y_{(2)})\cong 
		\mathcal E_{0,\frak u_{(1)};\frak p_j}(\frak y_{(1)}).
	\end{equation}
\end{lemma}
\begin{proof}
	Let $I_{\hat v}$ be the inclusion  map (\ref{form614343bef}). This is a holomorphic embedding.
	Then (\ref{form614343}) induces the required isomorphism.
	The fact that transversality constraint  (\ref{form6121}) is preserved is 
	a consequence of (\ref{form614343})
	and our choice of transversals $\mathcal N_{(j),v,i}$.
\end{proof}
\begin{lemma}
	Suppose $\frak u_{(1)}$, $\frak u_{(2)}$, $\Xi_{(1)}$, $\Xi_{(2)}$, $\frak y_{(1)}$ and $\frak y_{(2)}$ are given as
	 in Lemma \ref{lem6666}. 
	 We put 
	 $\Upsilon_{\frak u_{(j)}} = (\Xi_{\frak u_{(j)}},E_{\frak u_{(j)}})$,
	 where $E_{\frak u_{(j)}}$ is as in Definition \ref{Kura-chart}.
	  If $\frak y_{(2)}$ is an element of  
	 $\widehat {\mathcal U}(\frak u_{(2)},\Upsilon_{(2)})$,
	then $\frak y_{(1)}$ is an element of  $\widehat{\mathcal U}(\frak u_{(1)},\Upsilon_{(1)})$.
\end{lemma}
\begin{proof}
	The isomorphism induced by $I_{\hat v}$
	send $\overline\partial u'_{(2),\hat v}$ (resp.  $\overline\partial U'_{(2),\hat v}$)
	to $\overline\partial u'_{(1),\hat v}$ (resp. $\overline\partial U'_{(1),\hat v}$).
	This is a consequence of (\ref{form614343}).
	Therefore, if $\frak y_{(2)}$ satisfies (\ref{form6157})
	then $\frak y_{(1)}$ satisfies (\ref{form6157}).
\end{proof}
We thus constructed a $\Gamma_{\frak u_{(2)}}$-invariant map:
$$
\varphi_{\frak u_{(1)}\frak u_{(2)}} :  \widehat{\mathcal U}(\frak u_{(2)},\Upsilon_{(2)})
\to \widehat{\mathcal U}(\frak u_{(1)},\Upsilon_{(1)}).
$$
It is clear from the construction that the above map can be lifted to a map $\widetilde{\varphi}_{\frak u_{(1)}\frak u_{(2)}}$ from $\widetilde{\mathcal U}(\frak u_{(2)},\Upsilon_{(2)})$ to $\widetilde{\mathcal U}(\frak u_{(1)},\Upsilon_{(1)})$.

\begin{lemma}\label{lem668888}
	The maps  $\varphi_{\frak u_{(1)}\frak u_{(2)}}$ and $\widetilde \varphi_{\frak u_{(1)}\frak u_{(2)}}$ are 
	$C^{\ell}$ embeddings.
\end{lemma}
\begin{proof}
	It follows from the definition of $\widetilde \varphi_{\frak u_{(1)},\frak u_{(2)}}$ and the choices of 
	$\Xi_{(j)}$, $\Upsilon_{(j)}$ that the following two diagrams commute:
	\begin{equation}\label{dia6164}
	\begin{tikzcd}
		\widetilde{\mathcal U}(\frak u_{(2)},\Upsilon_{(2)})\ar[r,"F_1"]\ar[dd,"\widetilde{\varphi}_{\frak u_{(1)}\frak u_{(2)}}"] 
		& \displaystyle{\prod_{v  \in  C^{\rm int}_0(\check R_{(2)})}\mathcal V_{(2),v}^{\rm source}\times \prod_{e  \in  C^{\rm int}_1(\check R{(2)})}\mathcal V_{(2),e}^{{\rm deform}} 
		\times (D^2)^{\vert\lambda_{(2)}\vert}} \ar[dd,"R_1"]\\ \\
		 \widetilde{\mathcal U}(\frak u_{(1)},\Upsilon_{(1)}) \ar[r]& 
		 \displaystyle{
		 \prod_{v  \in  C^{\rm int}_0(\check R_{(1)})}\mathcal V_{(1),v}^{\rm source} \times \prod_{e  \in  C^{\rm int}_1(\check R{(1)})}\mathcal V_{(1),e}^{{\rm deform}} 
		\times (D^2)^{\vert\lambda_{(1)}\vert}} 
	\end{tikzcd}
	\end{equation}
%	\begin{equation}\label{dia6164}
%	\begin{CD}
%		\widetilde{\mathcal U}(\frak u_{(2)},\Upsilon_{(2)}) @>{F_1}>> 
%		\displaystyle{\prod_{v  \in  C^{\rm int}_0(\check R_{(2)})}\mathcal V_{(2),v}^{\rm source} 
%		\atop \qquad\qquad\quad
%		\times \prod_{e  \in  C^{\rm int}_1(\check R{(2)})}\mathcal V_{(2),e}^{{\rm deform}} \times 
%		(D^2)^{\vert\lambda_{(2)}\vert}}\\
%		@VV{\widetilde{\varphi}_{\frak u_{(1)}\frak u_{(2)}}}V  @VV{R_1}V \\
%		\widetilde{\mathcal U}(\frak u_{(1)},\Upsilon_{(1)}) @>{}>> 
%		\displaystyle{\prod_{v  \in  C^{\rm int}_0(\check R_{(1)})}\mathcal V_{(1),v}^{\rm source} 
%		\atop \qquad\qquad\quad \times \prod_{e  \in  C^{\rm int}_1(\check R{(1)})}\mathcal V_{(1),e}^{{\rm deform}} 
%		\times (D^2)^{\vert\lambda_{(1)}\vert}}
%	\end{CD}
%	\end{equation}
	\begin{equation}\label{dia6165}
	\begin{tikzcd}
		\widetilde{\mathcal U}(\frak u_{(2)},\Upsilon_{(2)})\ar[r,"F_2"]\ar[dd,"\widetilde{\varphi}_{\frak u_{(1)}\frak u_{(2)}}"] 
		& \displaystyle{\prod_{v  \in  C^{\rm int}_0(\check R_{(2)}), \lambda_{(2)}(v) = 0} L^2_{m+1}(\Sigma_{(2),v}^{-},X\setminus \mathcal D)\atop \quad\quad%\qquad\quad 
		\times \prod_{v  \in  C^{\rm int}_0(\check R_{(2)}), \lambda_{(2)}(v) > 0} L^2_{m+1}(\Sigma_{(2),v}^-, \mathcal N_{\mathcal D}X\setminus \mathcal D)} \ar[dd,"R_2"]\\ \\
		 \widetilde{\mathcal U}(\frak u_{(1)},\Upsilon_{(1)}) \ar[r]& 
		 \displaystyle{\prod_{v  \in  C^{\rm int}_0(\check R_{(1)}), \lambda_{(1)}(v) = 0} L^2_{m+1}(\Sigma_{(1),v}^-,X\setminus \mathcal D)\atop \quad\quad
		 \times\prod_{v  \in  C^{\rm int}_0(\check R_{(1)}), \lambda_{(1)}(v) > 0}L^2_{m+1}(\Sigma_{(1),v}^-,\mathcal N_{\mathcal D}X\setminus \mathcal D)} 
	\end{tikzcd}
	\end{equation}
	Here the horizontal arrows $F_1$ and $F_2$ are as in \eqref{ccc}.
	The right vertical arrow $R_1$ of Diagram \eqref{dia6164} is obtained by 
	requiring (\ref{form6138}) and is a smooth embedding.
	Diagram (\ref{dia6164}) commutes since $\widetilde \varphi_{\frak u_{(1)}\frak u_{(2)}}$
	does not change the conformal structure of source (marked) curves.
	The right vertical arrow $R_2$ of Diagram (\ref{dia6165}) is obtained by restriction 
	of domain and is a smooth map.
	Diagram (\ref{dia6165}) commutes because of Condition \ref{choi655-2}.
%	the commutativity of Diagram (\ref{digaram145}).}
	Now the definitions of the $C^{\ell}$ structures on $\widetilde {\mathcal U}(\frak u_{(1)},\Upsilon_{(1)})$ and 
	$\widetilde{\mathcal U}(\frak u_{(2)},\Upsilon_{(2)})$ imply that 
	$\widetilde \varphi_{\frak u_{(1)}\frak u_{(2)}}$ is a $C^{\ell}$ map.
	Unique continuation implies that the differential of the map $(R_1\circ F_1,R_2\circ F_2)$ is injective. In particular, this 
	 implies that $\widetilde \varphi_{\frak u_{(1)}\frak u_{(2)}}$ is an embedding.
	A similar argument applies to the map $\varphi_{\frak u_{(1)}\frak u_{(2)}}$.
\end{proof}

We next define a bundle map $\overline \varphi_{\frak u_{(1)}\frak u_{(2)}} : \mathcal E_{\frak u_{(2)}} \to \mathcal E_{\frak u_{(1)}}$ which lifts $\varphi_{\frak u_{(1)}\frak u_{(2)}}$. Using (\ref{form6162})  and (\ref{formnew6163}), we obtain a linear embedding:
\begin{equation}\label{form166633}
	\bigoplus_{\frak p_j \in {\frak J}(\frak u_{(2)})} \mathcal E_{0,\frak u_{(2)},\frak p_j}(\frak y_{(2)})\to
	\bigoplus_{\frak p_j \in {\frak J}(\frak u_{(1)})} \mathcal E_{0,\frak u_{(1)},\frak p_j}(\frak y_{(1)})
\end{equation}
if $\frak y_{(1)} = \varphi_{\frak u_{(1)}\frak u_{(2)}}(\frak y_{(2)})$. The map :
\begin{equation}\label{form166634}
	\bigoplus_{e \in C^{\rm int}_{\rm th}(\check R_{(2)}), \lambda(e) > 0} \mathscr L_e\to 
	\bigoplus_{e \in C^{\rm int}_{\rm th}(\check R_{(1)}), \lambda(e) > 0} \mathscr L_e
\end{equation}
is defined as identity on  $e \in C^{\rm int}_{\rm th}(\check R_{(2)}) \subset C^{\rm int}_{\rm th}(\check R_{(1)})$ and is zero on the other factors. The bundle map $\overline \varphi_{\frak u_{(1)}\frak u_{(2)}}$ is defined using \eqref{form166633} and \eqref{form166634}. Analogous to Lemma \ref{lem668888}, we can prove that $\overline \varphi_{\frak u_{(1)}\frak u_{(2)}}$ is $C^{\ell}$.
\begin{lemma}
$$
\frak s_{(1)} \circ \varphi_{\frak u_{(1)}\frak u_{(2)}} 
= \overline \varphi_{\frak u_{(1)}\frak u_{(2)}} \circ \frak s_{(2)}. 
$$
\end{lemma}
\begin{proof}
	On the factor in \eqref{form166633} this is a consequence of the 
	definitions of the map $\varphi_{\frak u_{(1)}\frak u_{(2)}}$ and \eqref{form166633}.
	Namely, it follows from \eqref{form614343} and the fact that $I_{\hat v}$ is bi-holomorphic.
	For the factor in \eqref{form166634}, this is a consequence of 
	(\ref{newform614154}).
\end{proof}
Compatibility of the parametrization map with $\varphi_{\frak u_{(1)}\frak u_{(2)}}$ is also an immediate consequence of the definitions. We thus proved that:

\begin{prop}
	Let  $\frak u_{(2)} \in U(\frak u_{(1)})$.
	We assume that $\Xi_{(2)}$ is induced from $\Xi_{(1)}$.
	Then the pair $(\varphi_{\frak u_{(1)}\frak u_{(2)}},\overline\varphi_{\frak u_{(1)}\frak u_{(2)}})$
	is a coordinate change of Kuranishi charts. 
\end{prop}

\subsection{Construction of Coordinate Change II}
\label{subsub:coordinatechage2}

Let $\frak u=((\Sigma_{v},\vec z_{v},u_{v});v \in C^{\rm int}_0(\check R))$ be an element of $\mathcal M^{\rm RGW}_{k+1}(L;\beta)$.
We fix two TSDs
$$
   \hspace{3cm} \Xi_{(j)} = (\vec w_{(j)},(\mathcal N_{(j),v,i}),(\phi_{(j),v}),(\varphi_{(j),v,e}),\kappa_{(j)})
   \hspace{1cm} j=1,2
$$ 
at $\frak u$ such that we can use Definition \ref{Kura-chart}, to form Kuranishi charts
$$
\mathcal U_{\frak u,\Upsilon_{(j)}} = (\mathcal U(\frak u,\Upsilon_{(j)}),\mathcal E_{\frak u,\Upsilon_{(j)}},\Gamma_{\frak u},\frak s_{\frak u,\Upsilon_{(j)}},\psi_{\frak u,\Upsilon_{(j)}}),
$$
where $\Upsilon_{(j)} = (\Xi_{(j)},E)$ and $E= \{E_{\frak u}\}$ is as in Definition \ref{Kura-chart}.

These Kuranishi charts depend on the choices of the subset $\{\frak p_j\}$ of the moduli space $ \mathcal M^{\rm RGW}_{k+1}(L;\beta)$, the TSOs $\{\Upsilon_{\frak p_j}\}$, the vector spaces $E_{\frak p_j,v}$ and the open sets $\mathcal K(\frak p_j)$. We assume that these choices agree with each other for the above two charts. In this subsection, we will construct a coordinate change from $\mathcal U_{\frak u,\Upsilon_{(2)}}$ to $\mathcal U_{\frak u,\Upsilon_{(1)}}$.\footnote{In fact, these two coordinate charts are isomorphic after possibly shrinking $\mathcal U(\frak u,\Upsilon_{(j)})$ into appropriate open subspaces.} The TSD $\Xi_{(j)}$ determines the subspace $\Sigma_{(j),v}^-$ of $\Sigma_{v}$ for each interior vertex $v$ of $\check R$. We assume that $\Xi_{(2)}$ is small enough such that $\vec w_{(1)} \cap\Sigma_v$ is a subset of $\Sigma^-_{(2),v}$.

We pick an inconsistent map with respect to $\Upsilon_{(2)}$ denoted by:
$$
\frak y_{(2)} = (\vec{\frak x}_{(2)},\vec{\sigma}_{(2)},(u'_{(2),v}),(U'_{(2),v}),(\rho_{(2),e}),(\rho_{(2),i}))
$$
Associated to $\frak y_{(2)}$, we have $\Sigma_{(2),v}^-(\vec{\frak x}_{(2)},\vec{\sigma}_{(2)})$, which comes with marked points:%\footnote{unlike the marked points $\vec z_{v}(\vec{\frak x}_{(2)},\vec{\sigma}_{(2)})$, the choice of $\vec w_{(2),v}(\vec{\frak x}_{(2)},\vec{\sigma}_{(2)})$ depends on $\frak y_{(2)}$, and not just $\vec{\frak x}_{(2)}$ and $\vec{\sigma}_{(2)}$. This is not reflected in our notation.} 
\[\vec z_{(2),v}(\vec{\frak x}_{(2)},\vec{\sigma}_{(2)}) \cup \vec w_{(2),v}(\vec{\frak x}_{(2)},\vec{\sigma}_{(2)}).\] 
Here the elements of $\vec z_{(2),v}(\vec{\frak x}_{(2)},\vec{\sigma}_{(2)})$ are in correspondence with the boundary marked points of $\Sigma_v$ and $\vec w_{(2),v}(\vec{\frak x}_{(2)},\vec{\sigma}_{(2)})$ are in correspondence with the additional marked points $\vec w_{(2),v}$ given by $\Xi_{(2)}$. We will write $\vec z_{(2)}(\vec{\frak x}_{(2)},\vec{\sigma}_{(2)})$ for the union of %$\vec z_{(2),v}(\vec{\frak x}_{(2)},\vec{\sigma}_{(2)})$ over the interior vertices $v$ of $\check R$. 
all boundary marked points of $\Sigma_{(2)}(\vec{\frak x}_{(2)},\vec{\sigma}_{(2)})$.
The following lemma is the analogue of Lemma \ref{lemma660}:

\begin{lemma}\label{lem67111}
	There exists 
	$\vec w_{(2);(1),v}(\frak y_{(2)}) \subset \Sigma^-_{(2),v}(\vec{\frak x}_{(2)},\vec{\sigma}_{(2)})$ 
	such that:
	\begin{enumerate}
	\item $w_{(2);(1),v,i}(\frak y_{(2)})$ is close to $w_{(1),v,i}$. Here we identify 
	$\Sigma^-_{(2),v}(\vec{\frak x}_{(2)},\vec{\sigma}_{(2)})$ and $\Sigma^-_{(2),v}$ using $\Xi_{(2)}$.
	\item $u'_{(2),v}(w_{(2);(1),v,i}(\frak y_{(2)})) \in \mathcal N_{(1),v,i}$.
%	If $\lambda(v)=0$, then $u'_{(2),v}(w_{(2);(1),v,i}(\frak y_{(2)})) \in \mathcal N_{(1),v,i}$, 
%	and if $\lambda(v)>0$, then
%	$U'_{(2),v}(w_{(2);(1),v,i}(\frak y_{(2)})) \in \mathcal N_{(1),v,i}$. 
	\end{enumerate}
\end{lemma}

We define:
\[
  \vec w_{(2);(1)}(\vec{\frak x}_{(2)},\vec{\sigma}_{(2)}) =
  \bigcup_{v\in C_0^{\rm int}(\check R)} \vec w_{(2);(1),v}(\vec{\frak x}_{(2)},\vec{\sigma}_{(2)}).
\]
Then $(\Sigma_{(2)}(\vec{\frak x}_{(2)},\vec{\sigma}_{(2)}),\vec z_{(2)}(\frak y_{(2)}) \cup \vec w_{(2);(1)}(\frak y_{(2)}))$
is close to $(\Sigma,\vec z \cup \vec w_{(1)})$ in the moduli space of bordered nodal curves.
Therefore, there exists $\vec{\frak x}_{(1)}, \vec\sigma_{(1)}$ such that:
\begin{equation}\label{form6168}
	(\Sigma_{(2)}(\vec{\frak x}_{(2)},\vec{\sigma}_{(2)}),\vec z_{(2)}(\vec{\frak x}_{(2)},\vec{\sigma}_{(2)})
	\cup \vec w_{(2);(1)}(\vec{\frak x}_{(2)},\vec{\sigma}_{(2)})) 
	\cong \hspace{2cm}
\end{equation}
\[	
  \hspace{2cm}(\Sigma_{(1)}(\vec{\frak x}_{(1)},\vec{\sigma}_{(1)}),\vec z_{(1)}(\vec{\frak x}_{(1)},\vec{\sigma}_{(1)})\cup \vec w_{(1)}(\vec{\frak x}_{(1)},\vec{\sigma}_{(1)})). 
\]
Here we use $\Xi_{(1)}$ to define the right hand side. Let $I$ be an isomorphism from the right hand side of \eqref{form6168} to the left hand side. Note that the choices of $\vec{\frak x}_{(1)}, \vec\sigma_{(1)}$ and $I$ are unique up to an element of $\Gamma_{\frak u}$.

We consider decompositions:
\begin{equation}\label{form614444prerev1}
	\aligned
	\Sigma_{(j)}(\vec{\frak x}_{(j)},\vec{\sigma}_{(j)})=
	&\coprod_{v \in C^{\rm int}_0(\check R)}  \Sigma^-_{(j),v}(\vec{\frak x}_{(j)},\vec{\sigma}_{(j)})\\
	&\cup \coprod_{e \in C^{\rm int}_1(\check R)} [-5T_{(j),e},5T_{(j),e}] \times S^1,
	\endaligned
\end{equation}
for $j=1,2$. 
In the above identity, we define $T_{(j),e}$ by requiring $e^{-10T_{(j),e}} = \vert\sigma_{(j),e}\vert$. Here for simplicity, we assume that $\sigma_{(2),e}$ is non-zero for all interior edges $e$ of $\check R$. A similar discussion applies to the case that $\sigma_{(2),e}=0$ with minor modifications. For example, in \eqref{form614444prerev1} we need to include two half cylinders for each $e$ that $\sigma_{(2),e}=0$. We also have:
\begin{equation}\label{form614444prerev12}
	\aligned
	\Sigma_{(j)}(\vec{\frak x}_{(j)},\vec{\sigma}_{(j)})=
	\bigcup_{v \in C^{\rm int}_0(\check R)}\Sigma^+_{(j),v}(\vec{\frak x}_{(j)},\vec{\sigma}_{(j)}).
	\endaligned
\end{equation}
Although $I$ in \eqref{form6168} is an isomorphism, it does not respect the decompositions in \eqref{form614444prerev1} or \eqref{form614444prerev12} for $j=1,2$.  This is because $\Xi_{(1)} \ne \Xi_{(2)}$.\footnote{Conditions \ref{choi655} and \ref{choi655-2} are used in Subsection \ref{subsub:coordinatechage1} to show the compatibility of the similar decompositions. We do not assume them here.} Nevertheless, one can easily prove:
\begin{lemma}
	If $\Xi_{(2)}$ is small enough, then $I$ can be chosen such that the following holds. Let $v \in C^{\rm int}_0(\check R)$ and 
	$\frak z \in \Sigma^+_{(1),v}(\vec{\frak x}_{(1)},\vec{\sigma}_{(1)})$.
	Then at least one of the following conditions holds:
	\begin{enumerate}
		\item [(I)] $I(\frak z) \in \Sigma^+_{(2),v}(\vec{\frak x}_{(2)},\vec{\sigma}_{(2)})$.
		\item[(II)] There exists $e \in C^{\rm int}_1(\check R)$ with $\partial e = \{v,v'\}$ such that
				$I(\frak z) \in \Sigma^+_{(2),v',\frak x_{(2),v'}}$.
	\end{enumerate}
\end{lemma}
This is the consequence of the fact that the decomposition (\ref{form614444prerev12}) is `mostly preserved' by $I$. Now we define $u'_{(1),v}$, $U'_{(1),v}$ as follows. If $\lambda(v)=0$, we have:
\begin{equation}\label{form17166}
u'_{(1),v}(\frak z) = 
\begin{cases}
u'_{(2),v}(\frak z) \quad &\text{if (I) holds,} 
\\
U'_{(2),v'}(\frak z)\quad &\text{if (II) holds, $\lambda(v) < \lambda(v')$,}
\\
u'_{(2),v'}(\frak z)\quad &\text{if (II) holds,  $\lambda(v)=\lambda(v')$},
\end{cases}
\end{equation}
and if $\lambda(v)>0$, we have:
\begin{equation}\label{form17266}
	U'_{(1),v}(\frak z) = 
	\begin{cases}
		U'_{(2),v}(\frak z) \quad &\text{if (I) holds,} \\
		U'_{(2),v'}(\frak z)\quad &\text{if (II) holds, $\lambda(v)=\lambda(v')$,}\\
		 {\rm Dil}_{\rho_{{(2)},e}} \circ U'_{(2),v'}(\frak z)\quad &\text{if (II) holds, $\lambda(v) < \lambda(v')$,}\\
		{\rm Dil}_{1/\rho_{{(2)},e}} \circ U'_{(2),v'}(\frak z)\quad &\text{if (II) holds,  
		$0 <\lambda(v') < \lambda(v)$}, \\
		{\rm Dil}_{1/\rho_{{(2)},e}} \circ u'_{(2),v'}(\frak z)\quad &\text{if (II) holds, 
		$0 =\lambda(v') < \lambda(v)$.}
	\end{cases}
\end{equation}
%{\color{red} In the above formula, we define $\rho_{(2),e} = 1$ for a fine edge $e$.}

Using the fact that $\frak y_{(2)}$ 
satisfies (\ref{form621}), (\ref{form6222}), (\ref{form623}),
we can easily check that in the case that (I) and (II) are both satisfied the right hand sides 
coincide.
\par
We also define
$$
\rho_{(1),e} = \rho_{(2),e},
\qquad \rho_{(1),i} = \rho_{(2),i}.
$$
\begin{lemma}\label{lem676}
The 6-tuple
$$
\frak y_{(1)} = (\vec{\frak x}_{(1)},\vec{\sigma}_{(1)},(u'_{(1),v}),(U'_{(1),v}),(\rho_{(1),e}),(\rho_{(1),i}))
$$
is an inconsistent solution near $\frak u$ with respect to $\Xi_{(1)}$.
\end{lemma}
\begin{proof}
	Definition \ref{defn6488} (1), (2), (3) are obvious. (4)-(8), (11) and (12) follow from the definition of $\frak y_{(1)}$ and the corresponding conditions 
	for $\frak y_{(2)}$. (9) and (10) hold by shrinking the size of $\Xi_{(2)}$ if necessary. 
	(13) is a consequence of Lemma \ref{lem67111}.
%		Definition \ref{defn6488} (1), (2), (3) are obvious, and (4), (5) follow from the corresponding conditions 
%	for  $\frak y_{(2)}$. (6), (7) are consequence of Lemma \ref{lem67111}.
%(8) follows from definition.
%(9)(10)(11) follow from corresponding conditions for  $\frak y_{(2)}$
%and (\ref{form17166}), (\ref{form17266}).
%(12)(13) follow from corresponding conditions for  $\frak y_{(2)}$.	
\end{proof}
Thus after shrinking the size of $\Xi_{(2)}$ if necessary, we may define:
\begin{equation}
	\widetilde \varphi_{(\frak u,\Xi_{(1)})(\frak u,\Upsilon_{(2)})}:\widetilde {\mathcal U}(\frak u,\Upsilon_{(2)}) \to \widetilde{\mathcal U}(\frak u,\Upsilon_{(1)})
\end{equation}
by:
\begin{equation}
	\widetilde \varphi_{(\frak u,\Upsilon_{(1)})(\frak u,\Upsilon_{(2)})}(\frak y_{(2)})   = \frak y_{(1)}.
\end{equation}
Similarly, we can define $\varphi_{(\frak u,\Upsilon_{(1)})(\frak u,\Upsilon_{(2)})}:\widehat {\mathcal U}(\frak u,\Upsilon_{(2)}) \to \widehat{\mathcal U}(\frak u,\Upsilon_{(1)})$.
\begin{lemma}
	The maps $\widetilde \varphi_{(\frak u,\Upsilon_{(1)})(\frak u,\Upsilon_{(2)})}$ and $\varphi_{(\frak u,\Upsilon_{(1)})(\frak u,\Upsilon_{(2)})}$ are $C^{\ell}$
	 diffeomorphisms into their images. 
\end{lemma}
\begin{proof}
	We cannot apply the same proof as in Lemma \ref{lem668888}. In fact, Diagram \eqref{dia6165} does not commute anymore because
	our TSD $\Xi_{(2)}$ may not be induced from $\Xi_{(1)}$. 
	In order to resolve this issue, we need to modify the definition of right vertical arrow in Diagram \eqref{dia6165}.

	Assuming $\Xi_{(2)}$ is small enough, we define a map:
	\[
	  \frak I_{v_0} : \prod_{v  \in  C^{\rm int}_0(\check R)}\mathcal V_{(2),v}^{\rm source} \times
	  \prod_{e  \in  C^{\rm int}_1(\check R)}\mathcal V_{(2),e}^{{\rm deform}}\times \Sigma_{(1),v_0}^{-}\to \Sigma_{(2),v_0}^{-}
	\]
	for any interior vertex $v_0$ of $\check R$ as follows. Fix an element  $(\vec{\frak x}_{(2)},\vec{\sigma}_{(2)})$ of 
	$\prod_{v  \in  C^{\rm int}_0(\check R)}\mathcal V_{(2),v}^{\rm source} \times\prod_{e  \in  C^{\rm int}_1(\check R)}\mathcal V_{(2),e}^{{\rm deform}}$, and let 
	$(\vec{\frak x}_{(1)},\vec{\sigma}_{(1)})\in \prod_{v  \in  C^{\rm int}_0(\check R)}\mathcal V_{(1),v}^{\rm source} \times\prod_{e  \in  C^{\rm int}_1(\check R)}\mathcal V_{(1),e}^{{\rm deform}}$
	satisfy \eqref{form6168}. By taking $\Xi_{(2)}$ small enough, we can form the following composition:
	\begin{equation}\label{form6174}
		\Sigma_{(1),v_0}^{-} \to\Sigma^-_{(1),v_0}(\vec{\frak x}_{(1)},\vec{\sigma}_{(1)})\to \Sigma^-_{(2),v_0}(\vec{\frak x}_{(2)},\vec{\sigma}_{(2)}) \to \Sigma^-_{(2),v_0}
	\end{equation}
	Here the first map is defined using $\Xi_{(1)}$, 
	the second map is induced by the isomorphism \eqref{form6168},
	and the last map is defined using $\Xi_{(2)}$. For $\frak z\in \Sigma_{(1),v_0}^{-}$, we define 
	$\frak I_{v_0}(\vec{\frak x}_{(2)},\vec{\sigma}_{(2)},\frak z)$ to be the image of $\frak z$ by the map \eqref{form6174}. It is clear that $\frak I_{v_0}$ is a smooth map.

	For a vertex $v_0$ with $\lambda(v_0)=0$, define:
	\[
	  \aligned
	  \frak I_{v_0}^* : 
	  &\prod_{v  \in  C^{\rm int}_0(\mathcal R)}\mathcal V_{(2),v}^{\rm source} \times\prod_{e  \in  C^{\rm int}_1(\mathcal R)}\mathcal V_{(2),e}^{{\rm deform}}\\
	  &\times L^2_{m+\ell+1}(\Sigma_{(2),v_0}^-,X\setminus \mathcal D)\to 
	  L^2_{m+1}(\Sigma_{(1),v_0}^{-},X\setminus \mathcal D)
	  \endaligned	
	\]
	as follows:
	\[
	  \frak I_{v_0}^*(\vec{\frak x}_{(2)},\vec{\sigma}_{(2)},u')(\frak z)=u'(\frak I_{v_0}(\vec{\frak x}_{(2)},\vec{\sigma}_{(2)},\frak z)).
	\]
	Note that we pick different Sobolev exponents for the Sobolev spaces on the domain and the target of $\frak I_{v_0}^*$.
	This allows us to obtain a $C^{\ell}$ map $\frak I_{v_0}^*$. 
	Similarly, for a vertex $v_0$ with $\lambda(v_0)>0$, we can define a $C^{\ell}$ map:
	\[
	  \aligned
	  \frak I_{v_0}^* : 
	  &\prod_{v  \in  C^{\rm int}_0(\mathcal R)}\mathcal V_{(2),v}^{\rm source} \times\prod_{e  \in  C^{\rm int}_1(\mathcal R)}\mathcal V_{(2),e}^{{\rm deform}}\\
	  &\times L^2_{m+\ell+1}(\Sigma_{(2),v_0}^-,\mathcal N_{\mathcal D}X\setminus \mathcal D)\to 
	  L^2_{m+1}(\Sigma_{(1),v_0}^{-},\mathcal N_{\mathcal D}X\setminus \mathcal D)
	  \endaligned	
	\]
	Now we replace Diagram \eqref{dia6165} with the following:
	\begin{equation}\label{dia6165rev}
	\xymatrix{
	\widetilde{\mathcal U}(\frak u,\Upsilon_{(2)}) \ar[dd]^{\widetilde \varphi_{(\frak u,\Upsilon_{(1)})(\frak u,\Upsilon_{(2)})}}\ar[rr]&
	&\txt{$\displaystyle{\prod_{v  \in  C^{\rm int}_0(\check R)}\mathcal V_{(2),v}^{\rm source} \times\prod_{e  \in  C^{\rm int}_1(\check R)}\mathcal V_{(2),e}^{{\rm deform}}}$\\
	$\displaystyle{\times \prod_{v  \in  C^{\rm int}_0(\check R), \atop \lambda(v) > 0}L^2_{m+\ell+1}(\Sigma_{(2),v}^{-},\mathcal N_{\mathcal D}X\setminus \mathcal D)}$\\
	$\displaystyle{\times \prod_{v  \in  C^{\rm int}_0(\check R), \atop \lambda(v) > 0}L^2_{m+\ell+1}(\Sigma_{(2),v}^{-},\mathcal N_{\mathcal D}X\setminus \mathcal D)}$
	}\ar[dd]^{\frak I_{v}^*}\\
        &&\\
        \widetilde{\mathcal U}(\frak u,\Upsilon_{(1)})\ar[rr]&&\txt{$\displaystyle{\prod_{v  \in  C^{\rm int}_0(\check R), \atop \lambda(v) = 0}L^2_{m+1}(\Sigma_{(1),v}^{-},X\setminus \mathcal D)}$\\
	$\displaystyle{\times \prod_{v  \in  C^{\rm int}_0(\check R), \atop \lambda(v) > 0}L^2_{m+1}(\Sigma_{(1),v}^{-},\mathcal N_{\mathcal D}X\setminus \mathcal D)}$}\\
	}
        \end{equation}
	Here horizontal arrows are defined as in \eqref{ccc}.
	Commutativity of (\ref{dia6165rev}) is immediate from the definition.
	We can also form a diagram similar to Diagram \eqref{dia6164}, which is commutative by the same reason as in Lemma \ref{lem668888}.
	Commutativity of these two diagrams and the fact that $\frak I_{v}^*$ is $C^{\ell}$ implies that $\widetilde \varphi_{(\frak u,\Xi_{(1)})(\frak u,\Xi_{(2)})}$ is also 
	$C^{\ell}$. A similar argument applies to $\varphi_{(\frak u,\Xi_{(1)})(\frak u,\Xi_{(2)})}$.
	
	By changing the role of $\Xi_{(1)}$ and $\Xi_{(2)}$, we can similarly obtain $C^{\ell}$ maps in different directions.
	To be more precise we can define maps 
	$\widetilde \varphi_{(\frak u,\Upsilon_{(2)})(\frak u,\Upsilon_{(1)}')}$ and $\varphi_{(\frak u,\Upsilon_{(2)})(\frak u,\Upsilon_{(1)}')}$
	where $\Upsilon_{(1)}' = (\Xi'_{(1)},E_{\frak u_{(1)}})$ 
	and $\Xi'_{(1)}$ is given by a small enough shrinking of $\Xi_{(1)}$. The compositions:
	\[
	  \widetilde \varphi_{(\frak u,\Upsilon_{(1)})(\frak u,\Upsilon_{(2)})}\circ \widetilde \varphi_{(\frak u,\Upsilon_{(2)})(\frak u,\Upsilon_{(1)}')}\hspace{1cm}
	  \varphi_{(\frak u,\Upsilon_{(1)})(\frak u,\Upsilon_{(2)})} \circ \varphi_{(\frak u,\Upsilon_{(2)})(\frak u,\Upsilon_{(1)}')}
	\]
	are equal to the identity map. Moreover, the compositions 
	\[
	   \widetilde \varphi_{(\frak u,\Xi_{(2)})(\frak u,\Upsilon_{(1)}')} \circ \widetilde \varphi_{(\frak u,\Upsilon_{(1)})(\frak u,\Upsilon_{(2)})}\hspace{1cm}
	  \varphi_{(\frak u,\Upsilon_{(2)})(\frak u,\Upsilon_{(1)}')} \circ \varphi_{(\frak u,\Upsilon_{(1)})(\frak u,\Upsilon_{(2)})} 
	\]
	are also equal to the identity map, wherever they are defined. This implies that 
	$\widetilde \varphi_{(\frak u,\Upsilon_{(1)})(\frak u,\Upsilon_{(2)})}$ and $\varphi_{(\frak u,\Upsilon_{(1)})(\frak u,\Upsilon_{(2)})}$ are 
	diffeomorphisms, after possibly shrinking $\Xi_{(2)}$.
\end{proof}
\begin{remark}
	By following an argument similar to the case of stable maps, one can show that 
	$\widetilde \varphi_{(\frak u,\Upsilon_{(1)})(\frak u,\Upsilon_{(2)})}$ and $\varphi_{(\frak u,\Upsilon_{(1)})(\frak u,\Upsilon_{(2)})}$ 
	are $C^{\infty}$ using the above $C^{\ell}$ property for all values of $\ell$. We omit this argument 
	and refer the reader to \cite[Section 12]{foooexp}
	for details of the proof. (See also \cite[Remark 3.19]{part3:FH}.)
\end{remark}

We thus constructed a $C^{\ell}$ embedding $\varphi_{(\frak u,\Upsilon_{(1)})(\frak u,\Upsilon_{(2)})}$. One can easily define a lift $\overline \varphi_{(\frak u,\Upsilon_{(1)})(\frak u,\Upsilon_{(2)})}$ of $ \varphi_{(\frak u,\Upsilon_{(1)})(\frak u,\Upsilon_{(2)})}$ and obtain embedding of obstruction bundles. The compatibility of the Kuranishi maps and the parametrization maps with the maps $\varphi_{(\frak u,\Upsilon_{(1)})(\frak u,\Upsilon_{(2)})}$ and $\overline\varphi_{(\frak u,\Upsilon_{(1)})(\frak u,\Upsilon_{(2)})}$ are immediate from the construction. In summary, we have coordinate change:
\[
  \Phi_{(\frak u,\Upsilon_{(1)})(\frak u,\Upsilon_{(2)})} = (\varphi_{(\frak u,\Upsilon_{(1)})(\frak u,\Upsilon_{(2)})},\hat\varphi_{(\frak u,\Upsilon_{(1)})(\frak u,\Upsilon_{(2)})})
 :\mathcal U_{\frak u,\Upsilon_{(2)}} \to \mathcal U_{\frak u,\Upsilon_{(1)}}.
\]

\subsection{Co-cycle Condition for Coordinate Changes}
\label{subsub:coordinatechage3}

For $j=1,\,2$, let $\frak u_{(j)}$ be an element of $\mathcal M^{\rm RGW}_{k+1}(L;\beta)$, and $\Xi_{(j)}$ be a TSD at $\frak u_{(j)}$. We assume that $\Xi_{(j)}$ is small enough such that we can form the Kuranish chart $\mathcal U_{\frak u_{(j)},\Upsilon_{(j)}}$ as in Definition \ref{Kura-chart}. 
(Here $\Upsilon_{(j)} = (\Xi_{(j)},E)$.) We also assume that $\frak u_{(2)}$ is sufficiently close to $\frak u_{(1)}$ in the sense that it belongs to the open subset of $\mathcal M^{\rm RGW}_{k+1}(L;\beta)$ determined by $\mathcal U_{\frak u_{(1)},\Upsilon_{(1)}}$. Therefore, we may use the constructions of Subsection \ref{subsub:Nplustri} to obtain a TSD $\Xi_{(2);(1)}$ at $\frak u_{(2)}$ which is compatible with  $\Xi_{(1)}$, namely, it satisfies Conditions \ref{choi655} and \ref{choi655-2}.
We put $\Upsilon_{(2);(1)} = (\Xi_{(2);(1)},E)$. Finally by shrinking $\Xi_{(2)}$, we can assume that we can define the coordinate change $\Phi_{(\frak u_{(2)},\Upsilon_{(2);(1)})(\frak u_{(2)},\Upsilon_{(2)})}$ following the construction of the previous subsection. Now we define:
\begin{definition}\label{defn6766}
	We define the coordinate change:
	\[
	  \Phi_{(\frak u_{(1)},\Upsilon_{(1)})(\frak u_{(2)},\Upsilon_{(2)})} : \mathcal U_{\frak u_{(2)},\Upsilon_{(2)}} \to \mathcal U_{\frak u_{(1)},\Upsilon_{(1)}}.
	\]
	as the composition
	\begin{equation}
		\Phi_{(\frak u_{(1)},\Upsilon_{(1)})(\frak u_{(2)},\Upsilon_{(2)})} = \Phi_{(\frak u_{(1)},\Upsilon_{(1)})(\frak u_{(2)},\Upsilon_{(2);(1)})}
		\circ \Phi_{(\frak u_{(2)},\Xi_{(2);(1)})(\frak u_{(2)},\Upsilon_{(2)})}.
	\end{equation}
	Here
	\[
	  \Phi_{(\frak u_{(1)},\Upsilon_{(1)})(\frak u_{(2)},\Upsilon_{(2);(1)})} : \mathcal U_{\frak u_{(2)},\Upsilon_{(2);(1)}} \to \mathcal U_{\frak u_{(1)},\Upsilon_{(1)}}
	\]
	is defined in Subsection \ref{subsub:coordinatechage1} and 
	\[
	  \Phi_{(\frak u_{(2)},\Upsilon_{(2);(1)})(\frak u_{(2)},\Upsilon_{(2)})} : \mathcal U_{\frak u_{(2)},\Upsilon_{(2)}} \to\mathcal U_{\frak u_{(2)},\Upsilon_{(2);(1)}} 
	\]
	is defined in Subsection \ref{subsub:coordinatechage2}.
\end{definition}
To complete the construction of the Kuranishi structure on $\mathcal M_{k+1}^{\rm RGW}(L;\beta)$, we need to prove the next lemma.
\begin{lemma}
	For $j=1,\,2,\,3$, let $\frak u_{(j)} \in \mathcal M_{k+1}^{\rm RGW}(L;\beta)$, and $\Xi_{(j)}$ be a TSD at 
	$\frak u_{(j)}$ 
	such that we can use Definition \ref{defn6766}, to define the coordinate changes $\Phi_{(\frak u_{(1)},\Upsilon_{(1)})(\frak u_{(2)},\Upsilon_{(2)})}$, $\Phi_{(\frak u_{(2)},\Upsilon_{(2)})(\frak u_{(3)},\Upsilon_{(3)})}$,
	$\Phi_{(\frak u_{(1)},\Upsilon_{(1)})(\frak u_{(3)},\Upsilon_{(3)})}$.
	Then we have:
	\begin{equation}\label{formula6177}
		\Phi_{(\frak u_{(1)},\Upsilon_{(1)})(\frak u_{(2)},\Upsilon_{(2)})} \circ \Phi_{(\frak u_{(2)},\Upsilon_{(2)})(\frak u_{(3)},\Upsilon_{(3)})}=
		\Phi_{(\frak u_{(1)},\Upsilon_{(1)})(\frak u_{(3)},\Upsilon_{(3)})}  .
	\end{equation}
\end{lemma}
\begin{proof}
	We use the constructions of Subsection \ref{subsub:Nplustri} 
	to find TSDs $\Xi_{(3);(1)}$, $\Xi_{(3);(2)}$ at $\frak u_{(3)}$
	such that the pairs $(\Xi_{(1)},\Xi_{(3);(1)})$ $(\Xi_{(2)},\Xi_{(3);(2)})$ 
	both satisfy Conditions \ref{choi655} and \ref{choi655-2}. We similarly choose the TSD $\Xi_{(2);(1)}$ at $\frak u_{(2)}$.
	We can easily check the following three formulas:
	\[
	  \aligned
	  \Phi_{(\frak u_{(3)},\Xi_{(3);(1)})(\frak u_{(3)},\Upsilon_{(3);(2)})} \circ \Phi_{(\frak u_{(3)},\Upsilon_{(3);(2)})(\frak u_{(3)},\Upsilon_{(3)})}
	  &=\Phi_{(\frak u_{(3)},\Upsilon_{(3);(1)})(\frak u_{(3)},\Upsilon_{(3)})} \\
	  \Phi_{(\frak u_{(1)},\Upsilon_{(1)})(\frak u_{(2)},\Upsilon_{(2);(1)})}\circ \Phi_{(\frak u_{(2)};\Upsilon_{(2);(1)})(\frak u_{(3)},\Upsilon_{(3);(1)})} 
	  &=\Phi_{(\frak u_{(1)},\Upsilon_{(1)})(\frak u_{(3)},\Upsilon_{(3);(1)})}
	  \endaligned
	\]
	\[
	  \aligned
	  &\Phi_{(\frak u_{(2)},\Upsilon_{(2);(1)})(\frak u_{(3)},\Upsilon_{(3);(1)})}\circ \Phi_{(\frak u_{(3)},\Upsilon_{(3);(1)})(\frak u_{(3)},\Upsilon_{(3);(2)})} \\
	  &= \Phi_{(\frak u_{(2)},\Upsilon_{(2);(1)})(\frak u_{(2)},\Upsilon_{(2)})}\circ \Phi_{(\frak u_{(2)},\Upsilon_{(2)})(\frak u_{(3)},\Upsilon_{(3);(2)})}
	  \endaligned
	\]
	Then \eqref{formula6177} is a consequence of these three formulas and Definition \ref{defn6766}. See the diagram below. In this diagram, the notation $\Phi_{(\frak u_{(3)},\Upsilon_{(3)})(\frak u_{(3)},\Upsilon_{(3);(2)})}$
	is simplified to $\Phi_{\Upsilon_{(3)}\Upsilon_{(3);(2)}}$. Similar notations for other coordinate changes are used.
        \begin{equation}
        \tiny{
        \xymatrix{
        \mathcal U_{\frak u_{(3)},\Upsilon_{(3)}} 
        \ar[rr]_{\Phi_{\Upsilon_{(3);(2)}\Upsilon_{(3)}}} 
        \ar@/^30pt/[rrrrr]^{\Phi_{\Upsilon_{(3);(1)}\Upsilon_{(3)}}}
        \ar@/_20pt/[rrdd]_{\Phi_{\Upsilon_{(2)}\Upsilon_{(3)}}}
        \ar@/_100pt/[rrrrrdddd]_{\Phi_{\Upsilon_{(1)}\Upsilon_{(3)}}}
        &&    \mathcal U_{\frak u_{(3)},\Upsilon_{(3);(2)}}  
        \ar[rrr]_{\Phi_{\Upsilon_{(3);(1)}\Upsilon_{(3);(2)}}}
        \ar[dd]_{\Phi_{\Upsilon_{(2)}\Upsilon_{(3);(2)}}}
        &&&   \mathcal U_{\frak u_{(3)},\Upsilon_{(3);(1)}}
        \ar[dd]_{\Phi_{\Upsilon_{(2);(1)}\Upsilon_{(3);(1)}}} 
        \ar@/^30pt/[dddd]^{\Phi_{\Upsilon_{(1)}\Upsilon_{(1);(3)}}}
        \\
        \\
        &&
        \mathcal U_{\frak u_{(2)},\Upsilon_{(2)}}
        \ar[rrr]_{\Phi_{\Upsilon_{(2);(1)}\Upsilon_{(2)}}}
        \ar@/_20pt/[rrrdd]_{\Phi_{\Upsilon_{(1)}\Upsilon_{(2)}}}
        &&&
        \mathcal U_{\frak u_{(2)},\Upsilon_{(2);(1)}}
        \ar[dd]_{\Phi_{\Upsilon_{(1)}\Upsilon_{(2);(1)}}}
        \\
        \\
        &&&&&
        \mathcal U_{\frak u_{(1)},\Upsilon_{(1)}}
        }
        \nonumber}
        \end{equation}
\end{proof}

This lemma completes the proof of the following result.
\begin{theorem}\label{sing-Kuranishi-str}
	The space $\mathcal M_{k+1}^{\rm RGW}(L;\beta)$ carries a Kuranishi structure.
\end{theorem}

We can prove the existence of Kuranishi structure for 
$\mathcal M_{k_1,k_0}^{\rm RGW}(L_1,L_0;p,q;\beta)$ in the same way.
The proof of Theorem \ref{Kura-exists}
is now complete.
\qed

\bibliography{references}
\bibliographystyle{alpha.bst}
\end{document}

%% file: init.tex
\DeclareMathOperator{\ind}{ind}

\renewcommand{\epsilon}{\varepsilon}

\def\({\mathopen{}\left(}
\def\){\right)\mathclose{}}
\def\<{\mathopen{}\left<}
\def\>{\right>\mathclose{}}

\usepackage{multicol, color}

\definecolor{gold}{rgb}{0.85,.66,0}
\definecolor{cherry}{rgb}{0.9,.1,.2}
\definecolor{burgundy}{rgb}{0.8,.2,.2}
\definecolor{orangered}{rgb}{0.85,.3,0}
\definecolor{orange}{rgb}{0.85,.4,0}
\definecolor{olive}{rgb}{.45,.4,0}
\definecolor{lime}{rgb}{.6,.9,0}
\definecolor{green}{rgb}{.2,.7,0}
\definecolor{grey}{rgb}{.4,.4,.2}
\definecolor{brown}{rgb}{.4,.3,.1}

\def\makeautorefname#1#2{\AtBeginDocument{\expandafter\def\csname#1autorefname\endcsname{#2}}}

\newcommand{\mynewtheorem}[2]{
  \newaliascnt{#1}{equation}          
  \newtheorem{#1}[#1]{#2}
  \aliascntresetthe{#1}
  \makeautorefname{#1}{#2}
}
\mynewtheorem{theorem}{Theorem}
\mynewtheorem{prop}{Proposition}
\mynewtheorem{cor}{Corollary}
\mynewtheorem{construction}{Construction}
\mynewtheorem{lemma}{Lemma}
\mynewtheorem{conjecture}{Conjecture}
\mynewtheorem{shitu}{Situation}
\mynewtheorem{conds}{Condition}

\numberwithin{substep}{step}
\makeautorefname{step}{Step}
\makeautorefname{substep}{Step}

\numberwithin{subcase}{case}
\makeautorefname{case}{Case}
\makeautorefname{subcase}{case}

\theoremstyle{remark}
\mynewtheorem{remark}{Remark}

\theoremstyle{definition}
\mynewtheorem{definition}{Definition}
\mynewtheorem{example}{Example}
\mynewtheorem{exercise}{Exercise}
\mynewtheorem{convention}{Convention}
\newtheorem*{convention*}{Convention}
\newtheorem*{conventions*}{Conventions}
\mynewtheorem{question}{Question}
        
\makeautorefname{chapter}{Chapter}
\makeautorefname{section}{Section}
\makeautorefname{subsection}{Section}
\makeautorefname{subsubsection}{Section}